\documentclass[10pt,namelimits,sumlimits]{amsart}
\usepackage{appendix}
\usepackage{amssymb,amsmath}
\usepackage[mathscr]{eucal}
\usepackage{color}
\usepackage{enumitem}
\usepackage[utf8]{inputenc}
\usepackage[leftcaption]{sidecap}
\usepackage{tikz, pgfplots}
 \usepackage{tikz-cd}
 \usepackage{cancel}
 \usepackage{stmaryrd}
\usetikzlibrary{matrix,arrows,decorations.pathmorphing}
\usetikzlibrary{positioning}
\usetikzlibrary{matrix,arrows,decorations.pathmorphing, shapes.geometric}
\tikzset{>=latex}

\usepackage{psfrag}

\DeclareFontFamily{U}{mathx}{\hyphenchar\font45}
\DeclareFontShape{U}{mathx}{m}{n}{
      <5> <6> <7> <8> <9> <10>
      <10.95> <12> <14.4> <17.28> <20.74> <24.88>
      mathx10
      }{}
\DeclareSymbolFont{mathx}{U}{mathx}{m}{n}
\DeclareFontSubstitution{U}{mathx}{m}{n}
\DeclareMathAccent{\widecheck}{0}{mathx}{"71}

\makeatletter
\newcounter{savesection}
\newcounter{apdxsection}
\renewcommand\appendix{\par
  \setcounter{savesection}{\value{section}}%
  \setcounter{section}{\value{apdxsection}}%
  \setcounter{subsection}{0}%
  \gdef\thesection{\@Alph\c@section}}
\newcommand\unappendix{\par
  \setcounter{apdxsection}{\value{section}}%
  \setcounter{section}{\value{savesection}}%
  \setcounter{subsection}{0}%
  \gdef\thesection{\@arabic\c@section}}
\makeatother

\usepackage{epstopdf}
\usepackage{subcaption}
\usepackage{graphicx}
\usepackage{verbatim}
\usepackage{amsmath,amsthm}
\usepackage{graphics}
\usepackage{color}
\usepackage{epsfig}
\usepackage{fullpage}
\usepackage{amssymb,amsmath}
\usepackage[mathscr]{eucal}
\usepackage{ upgreek }
\usepackage{tensor}

\newtheorem{theorem}[equation]{Theorem}
\newtheorem{lemma}[equation]{Lemma}
\newtheorem{prop}[equation]{Proposition}
\newtheorem{cor}[equation]{Corollary}
\newtheorem{corollary}[equation]{Corollary}

\newtheorem{definition}[equation]{Definition}

\newtheorem{thmx}{Theorem}

\theoremstyle{remark}
\newtheorem{remark}[equation]{Remark}
\newtheorem{notation}[equation]{Notation}
\newtheorem{example}[equation]{Example}
\newtheorem{convention}[equation]{Convention}
\newtheorem{assumption}[equation]{Assumption}

\numberwithin{equation}{section}

\newcommand{\dd}{{\boldsymbol{\partial}\mspace{.8mu}\!\!\!\!\boldsymbol{/}}}
\newcommand{\kvec}{\vec{v}}
\newcommand{\wvec}{\vec{w}}
\newcommand{\evec}{\vec{e}}
\newcommand{\ddiv}{\mathrm{div}}

\newcommand{\Thetacyl}{Y_{\cyl}}

\newcommand{\tb}{{\underline{\tau}}}

\newcommand{\vbreve}{\breve{v}}
\newcommand{\uhat}{\widehat{u}}
\newcommand{\Ehat}{\widehat{E}}

\newcommand{\Etilde}{\widetilde{E}}

\newcommand{\Etildehigh}{\widetilde{E}_{\mathrm{high}}}

\newcommand{\tauo}{\tau_{1}}

\newcommand{\tautilde}{\underline{a}}
\newcommand{\alphatilde}{\widetilde{\alpha}}
\newcommand{\alphahat}{\widehat{\alpha}}
\newcommand{\nablahat}{\widehat{\nabla}}
\newcommand{\betashat}{\beta^{\widehat{\sharp}}}

\newcommand{\Hcal}{{\mathcal{H}}}

\newcommand{\gammahat}{{\widehat{\gamma}}} 
\newcommand{\betahat}{{\widehat{\beta}}} 
\newcommand{\phicat}{\varphi_{\mathrm{cat}}}
\newcommand{\phicatp}{\varphi^+_{\mathrm{cat}}}
\newcommand{\phicatm}{\varphi^-_{\mathrm{cat}}}
\newcommand{\phicatpm}{\varphi^{\pm}_{\mathrm{cat}}}
\newcommand{\varphihat}{\widehat{\varphi}}
\newcommand{\varphiover}{\overline{\varphi}}
\newcommand{\varphigl}{\varphi^{gl}}

\newcommand{\vunder}{\underline{v}}

\newcommand{\kcir}{k_{\mathrm{\circ}}}
\newcommand{\kcirmin}{k_{\circ}^{\mathrm{min}}}
\newcommand{\kminev}{k_{\mathrm{min}}^{\mathrm{ev}}}
\newcommand{\kminodd}{k_{\mathrm{min}}^{\mathrm{odd}}}

\newcommand{\Acirc}{\mathring{A}}
\newcommand{\inj}{{\mathrm{inj}}}

\newcommand{\deltaunder}{{\underline{\delta}}}
\newcommand{\deltaunderp}{{\underline{\delta}'}}

\newcommand{\Mcore}{M_{\mathrm{core}}}
\newcommand{\Mend}{M_{\mathrm{end}}}
\newcommand{\varphiend}{\varphi_{\mathrm{end}}}
\newcommand{\catcore}{\Sigma_{\mathrm{core}}}
\newcommand{\catend}{\Sigma_{\mathrm{end}}}

\newcommand{\phie}{{\phi_{\mathrm{even}}}}
\newcommand{\phio}{{\phi_{\mathrm{odd}}}}
\newcommand{\phiend}{{\phi_{\mathrm{end}}}}

\newcommand{\phicrit}{{\phi_{\mathrm{crit}}}}

\newcommand{\gammagl}{{\alpha}}

\newcommand{\osc}{{{\mathrm{osc}}}}
\newcommand{\ave}{{{\mathrm{avg}}}}

\newcommand{\asym}{{{\mathrm{asym}}}}

\newcommand{\Xtilde}{{\widetilde{X}}}
\newcommand{\tildeE}{{\widetilde{E}}}

\newcommand{\gtilde}{{\widetilde{g}}}
\newcommand{\uutilde}{{\widetilde{u}}}
\newcommand{\uhatlow}{\widehat{u}_{\mathrm{low}}}
\newcommand{\uhatann}{\widehat{u}_{\mathrm{ann}}}
\newcommand{\uhathigh}{\widehat{u}_{\mathrm{high}}}
\newcommand{\Gtilde}{{\widetilde{G}}}

\newcommand{\chitilde}{{\widetilde{\chi}}}
\newcommand{\stilde}{{\widetilde{\sss}}}
\newcommand{\sroot}{{\sss_{\mathrm{root}}}}

\newcommand{\sbar}{\underline{\sss}}

\newcommand{\thetatilde}{{\widetilde{\theta}}}

\newcommand{\Btilde}{{{B}}}

\newcommand{\rhat}{{\hat{r}}}

\newcommand{\mutilde}{{\widetilde{\mu}}}
\newcommand{\muhat}{\widehat{\mu}}

\newcommand{\RRR}{\mathsf{R}}
\newcommand{\DDD}{\mathsf{D}}
\newcommand{\SSS}{\mathsf{S}}

\newcommand{\Rcap}{\mathsf{R}}

\newcommand{\ptilde}{\tilde{p}}

\newcommand{\gcyl}{\mathscr{G}}
\newcommand{\grot}{\mathscr{H}}

\newcommand{\grouprotcyl}{O(2)\times \Z_2}  
  
\newcommand{\groupmcyl}{\mathscr{G}_{m}}

\newcommand{\Omegatilde}{\widetilde{\Omega}}  
\newcommand{\varphitilde}{\widetilde{\varphi}}  
  
\newcommand{\phicheck}{{\widecheck{\phi}}}

\newcommand{\upphihat}{{\breve{\upphi}}}
\newcommand{\skernel}{\mathscr{K}}
\newcommand{\skernelv}{\widehat{\mathscr{K}}}
\newcommand{\val}{\mathscr{V}}
\newcommand{\valtop}{\mathscr{V}^\top}
\newcommand{\valperp}{\mathscr{V}^\perp}

\newcommand{\vspan}{\operatorname{span}}

\newcommand{\C}{\mathbb{C}}
\newcommand{\R}{\mathbb{R}}

\newcommand{\Z}{\mathbb{Z}}
\newcommand{\N}{\mathbb{N}}

\newcommand{\Sph}{\mathbb{S}}
\newcommand{\B}{\mathbb{B}}
\newcommand{\Spheq}{\mathbb{S}^2_{\mathrm{eq}}}

\newcommand{\Sphshr}{\mathbb{S}^2_{\mathrm{shr}}}
\newcommand{\Sphtwo}{\mathbb{S}^2}

\newcommand{\cat}{{\widecheck{K}}}
\newcommand{\Kcore}{{\underline{K}}}

\newcommand{\ccat}{\mathbb{K}_{\partial}}
\newcommand{\tildecat}{{\mathbb{K}}}
\newcommand{\KKcore}{{\underline{\mathbb{K}}}}
\newcommand{\KK}{{\mathbb{K}}} 

\newcommand{\PiSig}{\Pi_{\Sigma}}

\newcommand{\KM}{\tildecat_M}
\newcommand{\PiY}{\Pi_{\tildecat}}

\newcommand{\Ric}{\operatorname{Ric}}
\newcommand{\Rm}{\operatorname{Rm}}
\newcommand{\Rend}{\operatorname{R}}

\newcommand{\sech}{\operatorname{sech}}

\newcommand{\OZ}{{O(2)\times\Z_2}}  
\newcommand{\Ltilde}{\widetilde{L}}

\newcommand{\cunder}{\underline{c}\,}
\newcommand{\Eunder}{\underline{E}\,}
\newcommand{\kappaunder}{{\underline{\kappa}}}
\newcommand{\kappaunderbold}{{\boldsymbol{\underline{\kappa}}}}
\newcommand{\kappaperp}{\kappa^{\perp}}
\newcommand{\kappatilde}{{{\kappa}}}

\newcommand{\intprod}{\mathbin{\raisebox{\depth}{\scalebox{1}[-1]{$\lnot$}}}}

\newcommand{\Stilde}{\widetilde{S}}
\newcommand{\Stildep}{\widetilde{S}'}
\newcommand{\Shat}{\widecheck{K}}

\newcommand{\phiunder}{{\underline{\phi}}}

\newcommand{\junder}{{\underline{j}}}

\newcommand{\varphiunder}{{\underline{\varphi}}}

\newcommand{\epsilonunder}{\underline{\epsilon}}

\newcommand{\Lcal}{{\mathcal{L}}}
\newcommand{\Pcal}{{\mathcal{P}}}
\newcommand{\Zcal}{\mathcal{Z}}

\newcommand{\Ccal}{{\mathcal{C}}}
\newcommand{\Lie}{{\mathscr{L}}}

\newcommand{\Lcaltilde}{{\widetilde{\mathcal{L}}}}

\newcommand{\Lcalchi}{{\mathcal{L}_{\chi}}}

\newcommand{\Acal}{{\mathcal{A}}}

\newcommand{\Mcal}{{\mathcal{M}}}
\newcommand{\Rcal}{{\mathcal{R}}}
\newcommand{\Jcal}{{\mathcal{J}}}
\newcommand{\Ecal}{{\mathcal{E}}}
\newcommand{\Ecalunder}{\underline{{\mathcal{E}}}}

\newcommand{\Lchi}{{\mathcal{L}_\chi}}
\newcommand{\Lchitilde}{\mathcal{L}_{\widetilde{\chi}}}

\newcommand{\dbold}{{\mathbf{d}}}
\newcommand{\lambdabold}{\boldsymbol{\lambda}}
\newcommand{\zetabold}{{\boldsymbol{\zeta}}}
\newcommand{\domzb}{{B}_{\Pcal}}
\newcommand{\domku}{{B}_{\val\llbracket \zerobold\rrbracket}}
\newcommand{\zetaboldunder}{\underline{\zetabold}}

\newcommand{\zetaboldhat}{{\breve{\boldsymbol{\zeta}}}}
\newcommand{\zetaboldhatunder}{{\breve{\zetaboldunder}}}

\newcommand{\taubold}{{\boldsymbol{\tau}}}

\newcommand{\mubold}{{\boldsymbol{\mu}}}
\newcommand{\muboldtilde}{{\widetilde{\mubold}}}
\newcommand{\zerobold}{{\boldsymbol{0}}}
\newcommand{\zeroboldunder}{\underline{\zerobold}}
\newcommand{\xibold}{{\boldsymbol{\xi}}}

\newcommand{\xitilde}{{\widetilde{\xi}}}
\newcommand{\xitildecir}{{\widetilde{\xi}^{\circ}}}
\newcommand{\sigmatilde}{{\widetilde{\sigma}}}

\newcommand{\psihat}{\widecheck{\psi}}

\newcommand{\avg}{\operatornamewithlimits{avg}}
\newcommand{\sss}{{{\ensuremath{\mathrm{s}}}}}
\newcommand{\sbold}{\boldsymbol{\sss}}

\newcommand{\sym}{\mathrm{sym}}
\newcommand{\xx}{{{\ensuremath{\mathrm{x}}}}}

\newcommand{\ssshat}{{{\widehat{\sss}}}}
\newcommand{\shat}{{{\widehat{\sss}}}}
\newcommand{\shatbar}{{{\widehat{\sbar}}}}

\newcommand{\sssunder}{{\underline{\sss}}}

\newcommand{\yy}{\ensuremath{\mathrm{y}}}
\newcommand{\zz}{\ensuremath{\mathrm{z}}}
\newcommand{\rr}{\ensuremath{\mathrm{r}}}
\newcommand{\xxtilde}{{\ensuremath{\widetilde{\mathrm{x}}}}}
\newcommand{\yytilde}{{\ensuremath{\widetilde{\mathrm{y}}}}}
\newcommand{\zztilde}{{\ensuremath{\widetilde{\mathrm{z}}}}}

\newcommand{\Acalsssi}{{\Acal_{\sss_i}\!}}
\newcommand{\Rcalsssi}{\mathcal{R}_{\sss_i}\!}
\newcommand{\cyl}{\ensuremath{\mathrm{Cyl}}}

\newcommand{\partialx}{\partial}

\newcommand{\Pp}{{\widehat{G}}}
\newcommand{\Ghat}{{\widehat{G}}}
\newcommand{\Phat}{{\widehat{\Phi}}}
\newcommand{\phat}{{\widehat{\phi}}}

\newcommand{\Phip}{\Phi'}
\newcommand{\group}{{\mathscr{G}}}  
  
\newcommand{\groupcyl}{{\groupmcyl}}  

\newcommand{\conf}{\omega}

\newcommand{\T}{\mathbb{T}}

\newcommand{\Dhat}{\widehat{D}}

\newcommand{\Mhat}{\breve{M}}
\newcommand{\Omegahat}{\breve{\Sigma}}

\newcommand{\Dbreve}{\breve{D}}

\newcommand{\Sp}{S'}
\newcommand{\ghat}{\widehat{g}}
\newcommand{\gpar}{g^{\Sigma,{\zz}}}
\newcommand{\gpeuc}{\mathring{g}}
\newcommand{\gcir}{\mathring{g}}
\newcommand{\Hcirc}{\mathring{H}}

\newcommand{\nuhat}{\widehat{\nu}}

\newcommand{\Fcal}{{\mathcal{F}}}
\newcommand{\sssX}{\widehat{\mathsf{X}}}
\newcommand{\YYY}{\widehat{\mathsf{Y}}}

\newcommand{\bsigma}{\boldsymbol{\sigma}}
\newcommand{\bsigmaunder}{\underline{\bsigma}}
\newcommand{\bsigmaslash}{\boldsymbol{\sigma}\mspace{.1mu}\!\!\!\! /\,}
\newcommand{\bsigmaunderslash}{\bsigmaunder\mspace{.8mu}\!\!\!\! /}
\newcommand{\sigmaslash}{\sigma\mspace{.8mu}\!\!\!\! /}

\newcommand{\abold}{\boldsymbol{a}}
\newcommand{\Fbold}{\boldsymbol{F}}
\newcommand{\Fboldunder}{\underline{\Fbold}}
\newcommand{\xbar}{\underline{\sssX}}
\newcommand{\ybar}{\underline{\YYY}}

\newcommand{\xbartilde}{{\underline{\mathsf{X}}}}
\newcommand{\ybartilde}{{\underline{\mathsf{Y}}}}
\newcommand{\zbartilde}{{\underline{\mathsf{Z}}}}
\newcommand{\thetabar}{{\underline{\mathsf{\Theta}}}}
\newcommand{\thetasf}{{\mathsf{\Theta}}}
 
\newcommand{\Sbar}{{\underline{\mathsf{S}}}}

\newcommand{\vecu}{{\vec{u}}}

\newcommand{\tr}{\operatorname{tr}}
\newcommand{\arccosh}{\operatorname{arccosh}}
\newcommand{\Sim}{\displaystyle\operatornamewithlimits{{\scalebox{1.296}{{$\sim$}}}}}  
\newcommand{\Times}{\displaystyle\operatornamewithlimits{{\scalebox{1.636}{{$\times$}}}}}

\newcommand{\rtop}{{\overline{r}}}

\newcommand{\disjun}{\textstyle\bigsqcup}
\newcommand{\lem}{\:\le\:}

\newcommand{\Lmer}{{L_{\mathrm{mer}}}}
\newcommand{\Lpar}{{L_{\mathrm{par}}}}

\newcommand{\psicut}{{\psi_{\mathrm{cut}}}}
\newcommand{\Psibold}{{\boldsymbol{\Psi}}}

\newcommand{\pointp}{{(F_1, \bsigmaunder)}}

\newcommand{\Tor}{\mathbb{T}}

\newcommand{\munder}{\underline{m}}
\newcommand{\mbold}{\boldsymbol{m}}

\newcommand{\Isom}{{\mathrm{Isom}}}

\newcommand{\mmax}{m_{\mathrm{max}}}

\newcommand{\Glap}{G_\infty}
\newcommand{\const}{c'}
\newcommand{\Fmax}{F^{\phie}_{\max}}

\newcommand{\graph}{\operatorname{Graph}}
\newcommand{\tilt}{\operatorname{Tilt}}
\newcommand{\pert}{P}
\newcommand{\Vtilde}{\widetilde{V}}

\newcommand{\dom}{\mathrm{dom}}
\newcommand{\Omegain}{\Omega}

\newcommand{\Ethree}{E^3}
\newcommand{\Etwo}{E^2}
\newcommand{\ubreve}{\breve{u}}
\newcommand{\catbreve}{\breve{K}}
\newcommand{\RMa}{\Rcal_M^{appr}}

\begin{document}

\title[Generalizing Linearized Doubling]{Generalizing the Linearized Doubling approach, I: \\ General theory and new minimal surfaces and self-shrinkers}

\author[N.~Kapouleas]{Nikolaos~Kapouleas}
\author[P.~McGrath]{Peter~McGrath}

\address{Department of Mathematics, Brown University, Providence, RI 02912}  
\email{nicolaos\_kapouleas@brown.edu}

\address{Department of Mathematics, North Carolina State University, Raleigh, NC 27695}
\email{pjmcgrat@ncsu.edu}

\date{\today}


\begin{abstract}
In Part I of this article we generalize the Linearized Doubling (LD)
approach, introduced in earlier work by NK, by proving a general theorem
stating that if $\Sigma$ is a closed minimal surface embedded in a Riemannian
three-manifold $(N,g)$ and its Jacobi operator has trivial kernel, then given a
suitable family of LD solutions on $\Sigma$, a minimal surface $\breve{M}$
resembling two copies of $\Sigma$ joined by many small catenoidal bridges can
be constructed by PDE gluing methods. (An LD solution $\varphi$ on $\Sigma$ is
a singular solution of the Jacobi equation with logarithmic singularities which
in the construction are replaced by catenoidal bridges.) We also determine the
first nontrivial term in the expansion for the area $|\breve{M}|$ of
$\breve{M}$ in terms of the sizes of its catenoidal bridges and confirm that it
is negative; $|\breve{M}| < 2 | \Sigma|$ follows.

We demonstrate the applicability of the theorem by first constructing new
doublings of the Clifford torus. We then construct in Part II families of LD
solutions for general $(O(2)\times \mathbb{Z}_2)$-symmetric backgrounds
$(\Sigma, N,g)$. Combining with the theorem in Part I this implies the
construction of new minimal doublings for such backgrounds. (Constructions for
general backgrounds remain open.) This generalizes our earlier work for
$\Sigma=\mathbb{S}^2 \subset N=\mathbb{S}^3$ providing new constructions even
in that case.

In Part III, applying the results of Parts I and II---appropriately modified
for the catenoid and the critical catenoid---we construct new self-shrinkers
of the mean curvature flow via doubling the spherical self-shrinker or the
Angenent torus, new complete embedded minimal surfaces of finite total
curvature in the Euclidean three-space via doubling the catenoid, and new free
boundary minimal surfaces in the unit ball via doubling the critical catenoid.
\end{abstract}
\maketitle


\section{Introduction}
\label{S:intro}
\nopagebreak

\subsection*{The general framework}
\nopagebreak

Existence results for minimal surfaces have played a fundamental role in the development of the theory of minimal surfaces and more generally of Differential Geometry. 
Particularly important are the cases of embedded minimal (hyper)surfaces in Euclidean spaces or their quotients, 
embedded closed minimal (hyper)surfaces in the round spheres, 
properly embedded compact free boundary minimal (hyper)surfaces in Euclidean balls,  
closed embedded self-shrinkers for the mean curvature flow, 
and general closed embedded minimal (hyper)surfaces in closed Riemannian manifolds. 
Geometers have worked intensely on these directions 
and it is worth mentioning indicatively a sample of non-gluing results: 
by Scherk \cite{S}, 
by Lawson \cite{L2}, 
by Hsiang \cite{hsiang1}, 
by Kar\-cher-Pinkall-Sterling \cite{KPS},
by Hoffman-Meeks \cite{HM3},  
by Fraser-Schoen \cite{fraser-schoen:2},  
by Hoffman-Traizet-White \cite{white:helicoid},   
by Marques-Neves \cite{neves:yau}, 
by Song \cite{song},  
and by Chodosh-Mantoulidis \cite{mantoulidis:annals}. 

Gluing constructions by Partial Differential Equations (PDE gluing) methods have been very successful as well and hold further great promise. 
They are of two kinds: 
\emph{desingularization constructions} \cite{kapouleas:finite,kapouleas:compact,nguyenIII,kapshrinker,kapouleas:wiygul:toridesingularization,kapli} 
where the new surfaces resemble the union of given minimal surfaces intersecting along curves except in the vicinity 
of the intersection curves where they resemble singly periodic Scherk surfaces, 
and \emph{doubling constructions} \cite{kapouleas:clifford,Wiygul,FPZ,kapwiygul,kap,kapmcg} 
where the new surfaces resemble two (or more) copies of a given minimal surface joined by small catenoidal bridges; 
see also the survey articles \cite{kapouleas:survey,alm20}.  

We enumerate now some of the advantages of these gluing constructions. 
First, they provide new minimal surfaces which are almost explicit with well understood topology and geometry. 
In particular they are well suited for establishing the existence of infinitely many topological types of minimal surfaces in various situations. 
Second, the minimal surfaces constructed have low area, close to the total area of the ingredients, 
and so are important in classifications by increasing area. 
Third, the constructions are flexible, so they can be adjusted to apply to various different settings. 
Finally, doubling constructions hold great promise in high dimensions 
(for example \cite{kapouleas:Sn}) 
where very few existence results are currently known: 
even in Euclidean spaces the only complete embedded minimal hypersurfaces of finite geometry are the classical ones 
(hyperplane and high-dimensional catenoid). 
Note that new minimal hypersurfaces obtained via doubling are smooth in all dimensions by construction, 
similarly to the CMC hypersurfaces constructed in \cite{breiner:kapouleas:high}.  

Historically, PDE gluing methods have been applied extensively
and with great success in Gauge Theories by Donaldson \cite{donaldson1986}, Taubes \cite{taubes1982,taubes1984,taubes1988}, and others.
The particular kind of methods discussed here originate from Schoen's 
\cite{schoen}
and NK's \cite{kapouleas:annals},  
especially as they evolved and were systematized in
\cite{kapouleas:wente:announce,kapouleas:wente,kapouleas:imc}.
In the first doubling constructions \cite{kapouleas:clifford} the catenoidal bridges were attached to parallel copies of the given minimal surface 
to construct the initial surfaces, 
one of which was perturbed then to minimality. 
This approach turned out to be sufficient in some highly symmetric cases \cite{kapouleas:clifford,Wiygul,kapwiygul}  
where the symmetry does not allow horizontal forces and the surface modulo the symmetry is simple enough---although 
the constructions were still highly nontrivial.  

In most cases however this approach is not sufficient and for this reason NK introduced a powerful new approach called Linearized Doubling (LD) \cite{kap}. 
The LD approach was originally applied to construct doublings of a great two-sphere $\Sph^2$ in the round three-sphere $\Sph^3$ but was described 
for any given minimal surface $\Sigma$ \cite[Remark 3.21]{kap} embedded in a Riemannian three-manifold $N$ 
with an isometry of $N$ fixing $\Sigma$ pointwise and exchanging its sides. 

Given now such a $\Sigma$ let $\Lcal_\Sigma$ be its Jacobi operator (see \ref{NT}\ref{N:A}). 
The first step in the LD approach is to construct on $\Sigma$ a suitable family of Linearized Doubling (LD) solutions:  
an LD solution $\varphi$ is a singular solution of $\Lcal_\Sigma \varphi=0$ with logarithmic singularities; 
equivalently $\varphi$ can be considered as a Green's function for $\Lcal_\Sigma$ 
with multiple singularities of various strengths.  
In the second step the LD solutions are converted to approximately minimal ``initial surfaces''   
with the aid of chosen finite dimensional obstruction spaces $\skernelv[L]\subset C^\infty(\Sigma)$.  
The initial surface $M$ corresponding to an LD solution $\varphi$ 
consists of catenoidal bridges 
smoothly joined to the graphs of $\varphi+\vunder$ and $-\varphi-\vunder$ for some $\vunder\in \skernelv[L]$ 
chosen to optimize the matching of the bridges with the graphs.  
Each bridge is located in the vicinity of a singular point of $\varphi$ and its size is given by the strength of the logarithmic singularity of $\varphi$ at the point.  
In the final step one of the initial surfaces is perturbed to exact minimality providing the desired new minimal surface. 

The LD approach effectively reduces doubling constructions to constructions of suitable families of LD solutions. 
This is similar in spirit to the reduction of constructions of CMC (hyper)surfaces \cite{schoen,kapouleas:annals,kapouleas:1991,breiner:kapouleas:low,breiner:kapouleas:high} 
to constructions of suitable families of approximately balanced graphs, the LD solutions playing the role of the graphs. 
The LD solutions used are also approximately balanced in the sense that they approximately satisfy a finite number of ``matching conditions'',  
some nonlinear \cite[Definitions 3.3 and 3.4]{kap}. 
Not surprisingly, because of the PDE's involved, 
the construction of approximately balanced LD solutions is much harder than the construction of balanced graphs. 

In the original article \cite{kap} the construction was carried out only in two cases: 
when the singularities lie on two parallel circles of $\Sph^2$, and when they lie on the equatorial circle and the poles. 
Subsequently in \cite{kapmcg} this was extended to an arbitrary number of circles, optionally including the poles. 
In both cases the constructions of the LD solutions make heavy use of the 
$\OZ$ symmetry of the background. 
Actually in \cite{kap,kapmcg} the construction of LD solutions is reduced to the construction of what we called 
\emph{rotationally invariant linearized doubling (RLD) solutions}  
\cite[Definition 3.5]{kapmcg}, 
which being $O(2)$-invariant, 
satisfy an ODE instead of a PDE and can be understood by using appropriate flux quantities.

\subsection*{Brief discussion of the results}
\nopagebreak

In Part I of this article we generalize the LD approach to apply to general situations by proving Theorem \ref{Ttheory}, 
which we proceed to describe informally after stating a helpful general definition. 

\begin{definition}[Surface doublings] 
\label{Ddoubling} 
Given a Riemannian three-manifold $(N,g)$ and a two-sided 
surface $\Sigma$ in $N$, 
we define a \emph{(surface) doubling $\Mhat$ over $\Sigma$ in $N$} (equivalently we say \emph{$\Mhat$ doubles $\Sigma$ in $N$}) 
to be a smooth surface $\Mhat$ in $N$ satisfying the following.  
\begin{enumerate}[label=\emph{(\roman*)}]
\item 
\label{dprojec} 
The nearest point projection $\Pi_\Sigma$ to $\Sigma$ in $N$ is well defined on $\Mhat$. 
\item 
$\Omegahat:=\Pi_\Sigma(\Mhat) \subset \Sigma $ is closed with smooth boundary $\partial\Omegahat$. 
\item 
\label{dgraphs} 
$\Mhat$ is the union of the graphs of $\ubreve^+$ and $-\ubreve^- \in C^0(\Omegahat) \cap C^\infty( \Omegahat \setminus\partial \Omegahat)$.  
\item 
\label{dbgraphs} 
$\ubreve^+ + \ubreve^- = 0$ on $\partial \Omegahat$, where the two graphs join smoothly with vertical tangent planes, and $\ubreve^+ + \ubreve^- > 0$ close to $\partial \Omegahat$ in $\Omegahat$. 
\item 
By 
the above 
$\left.\Pi_\Sigma\right|_{\Mhat}$ covers $\Omegahat \setminus \partial \Omegahat$ twice, 
$\partial\Omegahat$ once, 
and misses $\Sigma\setminus\Omegahat $. 
\end{enumerate} 
We call $(\Sigma,N,g)$ \emph{the background of the doubling $\Mhat$}, $\Sigma$ its \emph{base surface}, 
and each connected component of $\Sigma\setminus\Omegahat$ a \emph{doubling hole of $\Mhat$ over $\Sigma$}. 
Finally if $\Sigma$ and $\Mhat$ are minimal we call the doubling $\Mhat$ \emph{minimal}. 
In this article, unless stated otherwise, $\Sigma$ and $\Mhat$ are assumed embedded and connected, and so $\ubreve^+ + \ubreve^- > 0$ 
on $\Omegahat$ ($\ubreve^+=\ubreve^->0$ in the special case of symmetric sides).   
\end{definition}

\begin{thmx}[Theorem \ref{Ttheory}]  
\label{TA}
We assume given a \emph{background} $(\Sigma,N,g)$ 
with the \emph{base surface} $\Sigma$ a closed minimal two-sided surface embedded in the Riemannian three-manifold $(N,g)$   
with Jacobi operator $\Lcal_\Sigma$ (see \ref{NT}\ref{N:A}) of trivial kernel on $\Sigma$ (see \ref{Ddoubling}, \ref{background} and \ref{cLker}).  
We assume given also a family of LD solutions on $\Sigma$ with 
appropriately uniform features, sufficiently small singularity strengths, and prescribable---when small---``unbalancing content'' (see \ref{A:FLD} for precise statements). 
There is then a \emph{smooth closed embedded minimal surface $\Mhat$ doubling $\Sigma$ in $N$} as in \ref{Ddoubling}  
satisfying the following.  
\begin{enumerate}[label=\emph{(\roman*)}]
\item 
There is an LD solution $\varphi$ in the given family 
with finite singular set $L\subset\Sigma$, 
such that $\forall p\in L$ 
and $\tau_p>0$ the strength of the logarithmic singularity of $\varphi$ at $p$, 
there is a catenoidal bridge $\catbreve_p \subset \Mhat $ in the vicinity of $p$ in $N$,  
with $\catbreve_p$ a small perturbation of the image by the Fermi exponential map $\exp^{\Sigma, N, g}_p$ (see \ref{dexp}) 
of a truncated catenoid in $T_pN$ of size (waist radius) $\tau_p$.   
\item 
\label{ld} 
$\Omegahat = \Sigma \setminus \textstyle{ \bigsqcup_{p \in L} \Dbreve_{p} } $, 
where each \emph{doubling hole} $\Dbreve_p\subset\Sigma$ 
is a small smooth perturbation of a geodesic disc in $\Sigma$ of center $p$ and radius $\tau_p$. 
\item 
The complement of the catenoidal bridges in $\Mhat$ is described graphically by small perturbations of $\pm\varphi$,  
or more precisely of $\pm( \varphi+\vunder_{\pm} )$, with $\vunder_\pm \in \skernelv[L]$ chosen in \ref{Dinit} to optimize the matching of the catenoidal bridges with the $\varphi$-graphical part 
and $\skernelv[L]\subset C^\infty(\Sigma)$ a chosen (as in \ref{aK}) finite dimensional obstruction space. 
\item 
The genus of $\Mhat$ is $2g_\Sigma-1+|L|$ where $g_\Sigma$ is the genus of $\Sigma$.    
\item 
$
| \Mhat | = 2 |\Sigma| - \pi \sum_{p \in L } \tau^2_p \left( 1+ O(\, \tau^{1/2}_p | \log \tau_p| \, ) \right), 
$ 
which implies also $|\Mhat| < 2 | \Sigma|$,  
where $ | \Mhat |$ and $|\Sigma|$ denote the areas of $\Mhat$ and $\Sigma$. 
\end{enumerate} 
\end{thmx}

$\Mhat$ is constructed in the proof of Theorem \ref{TA} as a small perturbation of one of the \emph{initial surfaces} $M[\varphi, \kappaunderbold]$ defined in \ref{Dinit} and 
parametrized by the given LD solutions $\varphi$ and parameters $\kappaunderbold$ satisfying \eqref{dalpha}. 
The construction of the initial surfaces is similar but more involved than in \cite{kap,kapmcg} where no $\kappaunderbold$ parameters are needed. 
The main new features are that each catenoidal bridge can be elevated and tilted relative to $\Sigma$ as prescribed by $\kappaunderbold$, 
and that $\vunder_+\ne\vunder_-$ when $\kappaunderbold\ne \zerobold$.  
In \cite{kap,kapmcg} $\kappaunderbold\ne\zerobold$ would violate the symmetry exchanging the two sides of the base surface;   
here however it introduces dislocations which (consistently with the \emph{geometric principle} \cite{kapouleas:survey,alm20})  
allow us to deal with the antisymmetric (with respect to approximate exchange of the sides of $\Sigma$) component of the obstructions involved.

Surprisingly the asymmetry of the sides of $\Sigma$ does not affect the nature or study of the families of LD solutions required, 
or the definition of the mismatch operator in \ref{Dmismatch}. 
The construction and study of the initial surfaces however presents new challenges related to the introduction of new parameters $\kappaunderbold$,  
and the estimation of mean curvature induced by a general Riemannian metric. 

Theorem \ref{TA} (or \ref{Ttheory}) not only generalizes the LD approach to the general case, 
but also makes the reduction to LD solutions explicit and systematic, 
unlike in \cite{kap,kapmcg}, 
where the reduction was described case by case. 
It is therefore a very powerful tool 
reducing doubling constructions to constructions of appropriate families of LD solutions, 
a much easier---but still very hard and open in general---problem. 

Note that although in Theorem \ref{TA} (or \ref{Ttheory}) $\Sigma$ is assumed to be a closed surface,  
the theorem can be modified to apply to other situations, 
as for example those in sections \ref{S:Cat} or \ref{S:ccat}.   
Moreover in Theorem \ref{TA}(v) we determine in full generality   
the first nontrivial term in the expansion of the area $|\Mhat|$ of $\Mhat$ 
in terms of the sizes of its catenoidal bridges, a new result even for the earlier doubling constructions. 
Finally we expect that Theorem \ref{TA} (or \ref{Ttheory}) will be an important step in proving a ``general'' doubling theorem 
asserting without any symmetry assumptions that any base surface $\Sigma$ with $|A|^2+\Ric(\nu,\nu)>0$ has infinitely many minimal doublings.  

As an example we next apply Theorem \ref{TA} to construct doublings of the Clifford torus $\T^2$ in Section \ref{S:clifford}. 
Recovering the doublings already known \cite{kapouleas:clifford,Wiygul} is fairly straightforward (see Remarks \ref{R:oldT} and \ref{R:uniqueness}). 
The catenoidal bridges in these doublings are located at the points of a $k\times m$ rectangular lattice $L$ with $k,m$ large ($m/k$ a priori bounded).  
We construct new doublings by allowing any $k\ge3$ 
(see Theorem \ref{Tcmain1} and for $k=1,2$ see Remark \ref{R:k12}), 
or by arranging for three bridges per fundamental domain when $k,m$ large (see Theorem \ref{Tcmain2}).  
Further results not discussed in this article are possible \cite{douT}, with more bridges per fundamental domain and any $k\ge3$, 
and also different symmetry groups, including constructions generalizing \cite[Example 13]{PRu} (related to torus knots).  
Note that the case of the Clifford torus is unusual because the background has $O(2)\times O(2)$ symmetry; 
the $O(2)\times\Z_2$-symmetric backgrounds on which we concentrate in Part II are less symmetric but more common.  

In Part II we construct families of LD solutions for $O(2)\times\Z_2$-symmetric backgrounds $(\Sigma,N,g)$,  
which are then used to construct minimal doublings via Theorem \ref{TA}. 
This generalizes our earlier work in \cite{kap,kapmcg} where families of LD solutions are constructed in the case $\Sigma=\Sph^2 \subset N=\Sph^3$ and used 
to construct minimal doublings of $\Sph^2$. 
The assumptions on the background we choose in \ref{Aimm} 
are general enough to allow many interesting applications.  
They imply that the base surface $\Sigma$ is diffeomorphic to a sphere or torus 
and the nontrivial orbits of the action of $O(2)$ on $\Sigma$ are circles (see Lemma \ref{LAconf});  
we call these circles \emph{parallel}. 
Calling $\Sbar$ the generator of the $\Z_2$ factor, it follows that $\Sbar$ fixes exactly one parallel circle when $\Sigma$ is a sphere and exactly two when $\Sigma$ is a torus; 
we call these circles \emph{equatorial}.

All constructions in \cite{kap,kapmcg} and in Parts II and III of this article are symmetric under a subgroup $\gcyl_m < O(2)\times\Z_2$ of order $4m$; 
more precisely $\gcyl_{ m     } = D_{2 m     } \times \Z_2$ with $D_{2 m     }<O(2)$ a dihedral subgroup of order $2 m     $ (see \ref{dHcyl}).  
The singularities of the LD solutions in these constructions concentrate on a prescribed number $\kcir$ of parallel circles 
and we assume $m$ large in terms of $\kcir$; 
we expect that other constructions are possible (beyond the scope of this article)  where $\kcir$ is large and $m$ is small or comparable to $\kcir$. 

Unlike in \cite{kap,kapmcg} we do allow different numbers of singularities in the $\kcir$ parallel circles but in a limited way: 
we allow $m$ or $2m$ singularities on the various circles (see \ref{Ambold}). 
Note that we use $\mbold=(m_i)_{i=1}^{ \lceil\kcir/2\rceil } $ to prescribe the numbers $|m_i|$ of singularities for the various circles,  
with the sign of $m_i$ choosing one of the two possible alignments with respect to $\gcyl_m$ (see \ref{dL} and \ref{RLcard}). 
Although not presented in this article, this can be further generalized to allowing the numbers $m_i$ to be multiples of $m$ by any small factors. 

\begin{thmx}[Theorem \ref{Trldldgen}] 
\label{TB}
Given a background $(\Sigma,N,g)$ satisfying Assumption \ref{Aimm} there is a minimum $\kcirmin\in\N$ (see \ref{Dkmin}) such that for each 
$\kcir \in\N$ with $\kcir \geq \kcirmin$ and any $\mbold \in \{m, -m, -2m\}^{ \lceil\kcir/2\rceil } $, 
with $m$ large enough in terms of $\kcir$, 
there is a family of LD solutions satisfying the required assumptions (see \ref{A:FLD}) in Theorem \ref{TA}, 
with the singularities concentrating along $\kcir$ parallel circles and 
the alignment and number of singularities at each circle prescribed by the entries of $\mbold$. 
\end{thmx}

Combining this with Theorem \ref{TA} (or \ref{Ttheory}) we obtain 

\begin{thmx}[Theorem \ref{Tconstruct}] 
\label{TC}
Given 
$(\Sigma,N,g)$,  
$\kcir\in\N$, and any $\mbold$ as in Theorem \ref{TB}, 
there is a minimal doubling containing one catenoidal bridge close to each singularity of one of the LD solutions in Theorem \ref{TB}
and satisfying (i)-(v) in Theorem \ref{TA}.
Moreover as $m\to\infty$ with fixed $\kcir$ the corresponding minimal doublings converge in the appropriate sense to $\Sigma$ covered twice. 
\end{thmx}

In Part III of this article we apply Theorem \ref{TC} (that is \ref{Tconstruct}) 
to construct new self-shrinkers of the mean curvature flow via doubling the spherical self-shrinker in Theorem \ref{Tmainsph} or 
via doubling the Angenent torus \cite{angenent} in Theorem \ref{Tmaintor}.   
By adjusting the results and proofs in Parts II and III,  
we also construct new complete embedded minimal surfaces of finite total curvature with four catenoidal ends in the Euclidean three-space 
via doubling the catenoid in Theorem \ref{Tmaincat},  
and new free boundary embedded minimal surfaces in the unit ball via doubling the critical catenoid in Theorem \ref{Tmainccat}.  

\subsection*{Outline of strategy and main ideas}
\nopagebreak

In this article we define in \ref{DbM} catenoidal bridges $\cat[p, \tau_p, \kappaunder_p] \subset M[\varphi, \kappaunderbold]$ 
as catenoids in cylindrical Fermi coordinates at a singular point $p\in \Sigma$ of the corresponding LD solution $\varphi$, 
truncated at scale $\sim\tau_p^\alpha$ with fixed small $\alpha$ as in \ref{con:alpha}. 
The strength $\tau_p>0$ of the logarithmic singularity of $\varphi$ at $p$ 
determines the size of the catenoid and $\kappaunder_p = \kappaperp_p + \kappatilde_p$   
its elevation in the normal direction and the tilt of its axis relative to the normal (see \ref{Dinit}). 
The construction of the bridges is simpler than in \cite{kap},  
at the expense that now the bridges are only approximately minimal and their mean curvature has to be estimated and corrected. 
The catenoidal bridges are then smoothly attached to the graphs of $\pm( \varphi+\vunder_{\pm} )$ at scale $\sim\tau_p^\alpha$ to form $M[\varphi, \kappaunderbold]$.  

The estimation of the mean curvature on the bridges is done in two steps. 
First, we decompose the metric of $N$ in the vicinity of $p$ as $g = \gpeuc+ h$, 
where $\gpeuc$ is a Euclidean metric induced by Fermi coordinates and $\left. h \right|_p =0 $ (see \ref{dgeopolar}). 
$\cat[p, \tau_p, \kappaunder_p]$ is exactly minimal with respect to $\gpeuc$ and the mean curvature induced by $g$ 
can be expressed in terms of tensor fields induced by $h$.   
Second, using properties of cylindrical Fermi coordinates, 
we estimate these tensors on $\cat[p, \tau_p, \kappaunder_p]$ in terms of the background geometry near $p$. 

An important feature is that the dominant term in the mean curvature of $\cat[p, \tau_p, \kappaunder_p]$ is driven by the second fundamental form 
$\left. A^\Sigma \right|_p $,  
and without 
the observation that the projection of the mean curvature to the first harmonics satisfies better estimates 
(see \ref{LH} and \ref{LH2}),  
this term would be too large for our purposes when $\left. A^\Sigma \right|_p \ne 0 $.  
(Note that in \cite{kapouleas:clifford,Wiygul} there are no first harmonics because of the symmetries.) 
In the definition of the global H\"older norms (see \ref{D:norm}) we use a stronger weight on the graphical regions and 
for the first harmonics on the catenoidal regions.  
On the graphical regions this parallels \cite[4.12]{kap} and leads to stronger final estimates (see \ref{LglobalH}) than those in \cite{kapouleas:clifford}.     

The proof of the area expansion in Theorem \ref{TA}(v) requires a detailed understanding of the interplay between the geometries of the catenoidal bridges and graphical regions.  
In particular each summand $-\pi \tau^2_p$ in the dominant term in the expansion for $|\Mhat|-2|\Sigma|$ is smaller in magnitude than a term of order $\tau^2_p |\log \tau_p|$ 
appearing in the expansion of the area of the corresponding bridge (see the catenoid estimate in \cite{Ketover}), 
and it is necessary to observe a subtle cancellation (see \ref{LcatMarea} and \ref{Lextarea}) between these terms and opposing terms 
arising from the exterior graphical region in order to complete the expansion.

We now discuss the proof of Theorem \ref{TB} in Part II of this article. 
We assume given an $O(2)\times\Z_2$-symmetric background $(\Sigma,N,g)$, $\kcir$, and $\mbold$ as in Theorem \ref{TB},  
and we proceed to construct a family of LD solutions with parameters (see \ref{dParam}) $\zetabold^\top  = (\zeta_1, \bsigmaunder )$,  
and (when not all $|m_i|$'s are equal) more parameters $\zetabold^\perp$.  
The LD solutions with vanishing $\zetabold^\perp$ are \emph{maximally symmetric} (see \ref{dtau1}) 
with their logarithmic singularities equidistributed on $\kcir$ parallel circles we call \emph{singular}. 
The parameters $\zetabold^\perp$ are used to dislocate the maximally symmetric LD solutions in accordance with the geometric principle; 
in the cases we examine in this article there is exactly one $\zetabold^\perp$ parameter for each $m_i=-2m$ (see \ref{Rlp}). 

Our maximally symmetric LD solutions 
$\varphi = \varphi \llbracket \zetabold^\top; \kcir, \mbold \rrbracket := \, \tauo \Phi \llceil  \bsigmaunder: \kcir, \mbold \rrfloor$ 
are constructed in \ref{dtau1} and \ref{Lphiavg}  
so that their overall scale $\tauo$ is controlled by $\zeta_1$ and each $\Phi \llceil  \bsigmaunder: \kcir, \mbold \rrfloor$ is constructed from 
 $\phi\llceil  \bsigmaunder: \kcir, \mbold \rrfloor$,  
a \emph{rotationally invariant (averaged) linearized doubling (RLD) solution} which can be recovered 
from $\Phi \llceil  \bsigmaunder: \kcir, \mbold \rrfloor$ by averaging on parallel circles.

RLD solutions (defined in \ref{RL}) are easier to understand than LD solutions because the Jacobi equation reduces to an ODE. 
They have derivative jumps instead of logarithmic singularities at the singular circles.  
We construct them first 
and use the information they provide, for example the position of the singular circles, 
to construct the maximally symmetric LD solutions. 
Our constructions are facilitated by the observation that the classes of LD and RLD solutions are invariant under conformal changes of the intrinsic metric, 
allowing us to work on the flat cylinder instead of $\Sigma$.  

The main tools in studying existence and uniqueness for the RLD solutions is a scale invariant flux $F_{\pm}^\phi$ (see \ref{dF}), 
which amounts to the logarithmic derivative of the RLD solution $\phi$ 
on the cylinder, 
with its monotonicity properties stated in \ref{LFmono}.  
Balancing for an RLD solution $\phi$ amounts to \emph{horizontal balancing}, requiring equality of the two one-sided fluxes at a singular circle,  
and \emph{vertical balancing}, requiring that the ratio of the fluxes at adjacent singular circles equals the ratio $|m_j/m_{j+1}|$ of the corresponding prescribed numbers of singularities 
(see \ref{dLbalanced}).  

The parameters $\bsigmaunder=(\bsigma, \xibold ) $ prescribe the RLD unbalancing with $\bsigma$ for vertical and $ \xibold $ for horizontal (see \ref{RLquant}, \ref{Pexist}, and \ref{Pexist2}). 
The effect of the parameters $\zetabold^\top  = (\zeta_1, \bsigma, \xibold)$ on the mismatch of the LD solutions is confirmed 
in \ref{LmatchingE} by using the equations in \ref{Lmatching} and the estimates in \ref{LPhip} for the LD solutions constructed.  
Note that $\bsigma$ prescribes (approximately) only differences of vertical mismatch, 
with (one) vertical mismatch prescribed (with less precision) by $\zeta_1$. 
Finally the estimates for the LD solutions in \ref{LPhip} are based on carefully decomposing each $\Phi = \Phi\llceil\bsigmaunder : \kcir, \mbold\rrfloor$ as 
$\Phi = \Ghat+ \Phat+\Phip$ (see \ref{ddecomp}), 
where $\Ghat$ captures the singular part, $\Phat$ is rotationally invariant, and $\Phip$ is the part we estimate 
($\Ghat$ and $\Phat$ being explicit).

In Part III the applications of the earlier results are fairly straightforward. 
For the catenoid and the critical catenoid we need some modifications to account for the noncompactness of the catenoid and the boundary of the critical catenoid. 
For the latter we follow more closely the methodology of \cite{kapwiygul} (see also \cite{kapli}) and we study the modified RLD's with an imposed Robin condition. 
Finally we remark that since the catenoid is conformally isomorphic to $\Spheq$, 
some of the families of RLD and LD solutions we use for the catenoid doubling were constructed and estimated already in \cite{kapmcg}.


\subsection*{General notation and conventions}
\label{sub:not}
\nopagebreak

\begin{notation}
\label{NT}
For $(N, g)$ a Riemannian manifold, $S\subset N$ a two-sided hypersurface equipped with a (smooth) unit normal $\nu$, and $\Omega\subset S$, 
we introduce the following notation where any of $N$, $g$, $S$ or $\Omega$ may be omitted when clear from context. 
\begin{enumerate}[label={(\roman*)}]
\item 
\label{isom}
We denote by $\Isom(N,g)$ the group of isometries of $(N,g)$. 
\item 
For $A\subset N$ we write $\dbold^{N, g}_A$ for the distance function from $A$ with respect to $g$ 
and we define the \emph{tubular neighborhood of $A$ of radius $\delta>0$} by
$  D^{N, g}_A(\delta):=\left \{p\in N:\dbold^{N, g}_A(p)<\delta\right\}. $  
If $A$ is finite we may just enumerate its points in both cases, for example if $A=\{q\}$ we write $\dbold_q(p)$. 
\item  We denote by $\exp^{N, g}$ the exponential map, 
by $\dom(\exp^{N,g}) \subset TN$ its maximal domain, and by $\inj^{N, g}$ the injectivity radius of $(N,g)$. 
Similarly by $\exp_p^{N, g}$, $\dom(\exp_p^{N,g})$ and $\inj_p^{N, g}$ the same at $p\in N$. 
\item 
If $h$ and $k$ are symmetric covariant two-tensors on $N$, 
we define a two-tensor $h *_{g,N} k$  by requesting  that in any local coordinates $(h *_{g,N} k)_{ij} = h_{ik} g^{kl} k_{lj}$.  
\item 
We denote the curvature endomorphism by $\Rend^{N,g}$,  
the curvature tensor by $\Rm^{N,g}$,  
and the Ricci tensor by $\Ric^{N,g}$, 
and we follow the convention 
$\Rend^{N,g} (X, Y)Z := [\nabla_X, \nabla_Y] Z {- \nabla_{[X, Y]} Z}$ for 
$X, Y,Z \in C^\infty(T S)$.   
We also define an endomorphism field $\Rend^{N,g}_Y : = \Rend(Y, \cdot)Y$ and a tensor $\Rm^{N,g}_Y: = \langle R(Y, \cdot)Y, \cdot\rangle$;  
note that then $\Ric(Y,Y) = - \tr_{g, N} \Rm_Y$. 
\item 
\label{N:A} 
We let $A^S$ and $B^S$ denote respectively the scalar-valued second fundamental form and Weingarten map of $S$, 
and $\Lcal_S$ the second variation of area or Jacobi operator 
(well known also to provide the linearization of the mean curvature change as in \ref{Lquad}),  
defined by 
($\forall X, Y\in C^\infty(T S)$)  
\begin{equation}
\label{E:Lcal}
\begin{gathered} 
A^S(X, Y) :=\langle \nabla_X  Y, \nu\rangle =  \langle B^S(X), Y\rangle, \quad
B^S(X) := -\nabla_X \nu,
\\ 
\Lcal_S:=\Delta_S + |A^S|^2+ \Ric(\nu, \nu). 
\end{gathered} 
\end{equation}

\item 
\label{Dpertimm}
Given also a map $X: \Sigma \rightarrow N$ and a vector field $V$ defined along $X$ satisfying $V_{X(p)} \in \dom(\exp^{N,g})$ for each $p \in \Sigma$, 
we define 
\begin{align*}
\pert_V X = 
\pert^{N, g}_V X : \Sigma \rightarrow N 
\quad \text{ by } \quad 
\pert_V X = 
\pert^{N, g}_V X := \exp^{N, g} \circ V \circ X.
\end{align*}

\item 
\label{dgraph}
Given also a function $f:S\to\R$ 
satisfying $|f|(p) < \inj_p^{N, g}$ $\forall p\in\Omega$, 
we use the notation 
$$ 
X^{N, g}_{\Omega,f} := \pert^{N, g}_{f\nu} I_\Omega^N,  
\qquad \quad 
\graph^{N,g}_\Omega (f) := X^{N, g}_{\Omega,f} (\Omega),  
$$ 
where $I_\Omega^N$ denotes the inclusion map of $\Omega$ in $N$. 
\hfill $\square$ 
\end{enumerate}
\end{notation}

\begin{notation} 
\label{NEuc} 
We denote by $g_{Euc}$ the standard Euclidean metric on $\R^n$ and by $g_{\Sph}$ the induced standard metric on $\Sph^n:=\{v\in\R^n: |v|=1\}$. 
By standard notation $O(n):= \Isom(\Sph^{n-1},g_\Sph)$ (recall \ref{NT}\ref{isom} ). 
\qed 
\end{notation} 

Our arguments require extensive use of cut-off functions and the following will 
be helpful. 
\begin{definition}
\label{DPsi} 
We fix a smooth function $\Psi:\R\to[0,1]$ with the following properties:
\begin{enumerate}[label=\emph{(\roman*)}]
\item $\Psi$ is nondecreasing.

\item $\Psi\equiv1$ on $[1,\infty)$ and $\Psi\equiv0$ on $(-\infty,-1]$.

\item $\Psi-\frac12$ is an odd function.
\end{enumerate}
\end{definition}

Given $a,b\in \R$ with $a\ne b$,
we define smooth functions
$\psicut[a,b]:\R\to[0,1]$
by
\begin{equation}
\label{Epsiab}
\psicut[a,b]:=\Psi\circ L_{a,b},
\end{equation}
where $L_{a,b}:\R\to\R$ is the linear function defined by the requirements $L_{a,b}(a)=-3$ and $L_{a,b}(b)=3$.

Clearly then $\psicut[a,b]$ has the following properties:
\begin{enumerate}[label={(\roman*)}]
\item $\psicut[a,b]$ is weakly monotone.

\item 
$\psicut[a,b]=1$ on a neighborhood of $b$ and 
$\psicut[a,b]=0$ on a neighborhood of $a$.

\item $\psicut[a,b]+\psicut[b,a]=1$ on $\R$.
\end{enumerate}

Suppose now we have two sections $f_0,f_1$ of some vector bundle over some domain $\Omega$.
(A special case is when the vector bundle is trivial and $f_0,f_1$ real-valued functions).
Suppose we also have some real-valued function $d$ defined on $\Omega$.
We define a new section 
\begin{equation}
\label{EPsibold}
\Psibold\left [a,b;d \, \right](f_0,f_1):=
\psicut[a,b\, ]\circ d \, f_1
+
\psicut[b,a]\circ  d \, f_0.
\end{equation}
Note that
$\Psibold[a,b;d\, ](f_0,f_1)$
is then a section which depends linearly on the pair $(f_0,f_1)$
and transits from $f_0$
on $\Omega_a$ to $f_1$ on $\Omega_b$,
where $\Omega_a$ and $\Omega_b$ are subsets of $\Omega$ which contain
$d^{-1}(a)$ and $d^{-1}(b)$ respectively,
and are defined by
$$
\Omega_a=d^{-1}\left((-\infty,a+\frac13(b-a))\right),
\qquad
\Omega_b=d^{-1}\left((b-\frac13(b-a),\infty)\right),
$$
when $a<b$, and 
$$
\Omega_a=d^{-1}\left((a-\frac13(a-b),\infty)\right),
\qquad
\Omega_b=d^{-1}\left((-\infty,b+\frac13(a-b))\right),
$$
when $b<a$.
Clearly if $f_0,f_1,$ and $d$ are smooth then
$\Psibold[a,b;d\, ](f_0,f_1)$
is also smooth.

In comparing equivalent norms or other quantities we will find the following notation useful. 
\begin{definition}
\label{Dsimc}
We write $a\Sim_c b$ to mean that 
$a,b\in\R\setminus\{0\}$, $c\in(1,\infty)$, and $\frac1c\le \frac ab \le c$. 
\end{definition}

We use the standard notation $\left\|u: C^{k,\beta}(\,\Omega,g\,)\,\right\|$ 
to denote the standard $C^{k,\beta}$-norm of a function or more generally
tensor field $u$ on a domain $\Omega$ equipped with a Riemannian metric $g$.
Actually the definition is completely standard only when $\beta=0$
because then we just use the covariant derivatives and take a supremum
norm when they are measured by $g$.
When $\beta\ne0$ we have to use parallel transport along geodesic segments 
connecting any two points of small enough distance
and this may be a complication if small enough geodesic balls are not convex.
In this article we take care to avoid situations where such a complication
may arise and so we will not discuss this issue further.

We adopt the following notation from \cite{kap} for weighted H\"{o}lder norms.  
\begin{definition}
\label{dwHolder}
Assuming that $\Omega$ is a domain inside a manifold,
$g$ is a Riemannian metric on the manifold, 
$k\in \N_0$, 
$\beta\in[0,1)$, $u\in C^{k,\beta}_{\mathrm{loc}}(\Omega)$ 
or more generally $u$ is a $C^{k,\beta}_{\mathrm{loc}}$ tensor field 
(section of a vector bundle) on $\Omega$, 
$\rho,f:\Omega\to(0,\infty)$ are given functions, 
and that the injectivity radius in the manifold around each point $x$ in the metric $\rho^{-2}(x)\,g$
is at least $1/10$,
we define
$$
\left\|u: C^{k,\beta} ( \Omega,\rho,g,f)\right\|:=
\sup_{x\in\Omega}\frac{\,\left\|u:C^{k,\beta}(\Omega\cap B_x, \rho^{-2}(x)\,g)\right\|\,}{f(x) },
$$
where $B_x$ is a geodesic ball centered at $x$ and of radius $1/100$ in the metric $\rho^{-2}(x)\,g$.
For simplicity we may omit any of $\beta$, $\rho$, or $f$, 
when $\beta=0$, $\rho\equiv1$, or $f\equiv1$, respectively.
\end{definition}

$f$ can be thought of as a ``weight'' function because $f(x)$ controls the size of $u$ in the vicinity of
the point $x$.
$\rho$ can be thought of as a function which determines the ``natural scale'' $\rho(x)$
at the vicinity of each point $x$.
Note that if $u$ scales nontrivially we can modify appropriately $f$ by multiplying by the appropriate 
power of $\rho$.  Observe from the definition the following multiplicative property: 
\begin{equation}
\label{E:norm:mult}
\left\| \, u_1 u_2 \, : C^{k,\beta}(\Omega,\rho,g,\, f_1 f_2 \, )\right\|
\le
C(k)\, 
\left\| \, u_1 \, : C^{k,\beta}(\Omega,\rho,g,\, f_1 \, )\right\|
\,\,
\left\| \, u_2 \, : C^{k,\beta}(\Omega,\rho,g,\, f_2 \, )\right\|.
\end{equation}

\begin{definition}[Tilting rotations $\RRR_\kappa$]
\label{dRk}
\label{dTiltop}
Let $\kappa : \Etwo \rightarrow (\Etwo)^\perp$ be a linear map, where $\Etwo$ is a two-di\-men\-sion\-al subspace of a three-dimensional Euclidean vector space $\Ethree$ 
and $(\Etwo)^\perp$ denotes the orthogonal complement of $\Etwo$ in $\Ethree$.  
By choosing a unit normal vector to $\Etwo$, we can identify $\kappa$ with an element of $(\Etwo)^*$. 
We define $\RRR_\kappa$ to be 
the rotation of $\Ethree$ characterized by 
$\RRR_\kappa (P) = \graph_P\kappa$ for $P\subset\Etwo$ a half-plane with $\partial P=\ker \kappa$ when $\kappa \neq 0$,  
or the identity $\text{\emph{Id}}_{\Ethree}$ when $\kappa = 0$.  

Given also a function $u:\Omega\to\R$ 
on $\Omega\subset \Etwo$ such that 
$\RRR_\kappa ( \graph^{\Ethree}_\Omega \!\! u)$ is graphical over $\Etwo$, 
we define $\tilt_\kappa( u) : \Omega' \to \R$, 
with $\Omega' \subset \Etwo$ a ``shift'' of $\Omega$, 
by requesting 
$ 
\RRR_\kappa ( \graph^{\Ethree}_\Omega \!\! u) = \graph^{\Ethree}_{\Omega'} \tilt_\kappa( u). 
$ 
\end{definition}

\section*{Part I: Generalizing the Linearized Doubling Approach}

\section[Tilted catenoids]{Tilted catenoids}
\label{S:cat}

\subsection*{Untilted catenoids in $T_p N$.}
\nopagebreak

\begin{convention}
\label{background}
In Parts I and II of this article we assume given a surface $\Sigma$ 
smoothly immersed in a Riemannian three-manifold $(N,g)$. 
To facilitate the discussion we will assume, 
unless stated otherwise,  
that $\Sigma$ is connected embedded minimal and two-sided with a chosen smooth unit normal $\nu_\Sigma$.  
Note however that most results can be modified to apply to situations where some or all of these assumptions do not apply. 
As in \ref{Ddoubling} we will call $\Sigma$ the \emph{base surface} and the data $(\Sigma,N,g)$ the \emph{background}, 
and we will not mention the dependence of constants on it. 
\end{convention}

\begin{definition}[Fermi coordinates about $\Sigma$]
\label{dgeopolar}
Given $p\in \Sigma$ 
we choose for $(T_pN, \left. g\right|_p)$ Cartesian coordinates $(\xxtilde,\yytilde,\zztilde) : T_pN\to \R^3$ 
satisfying $(\xxtilde,\yytilde,\zztilde)\circ \nu_\Sigma(p) =(0,0,1)$; 
clearly then $\left. g\right|_p  = d\xxtilde^2 + d\yytilde^2 + d\zztilde^2 $ on $T_pN$ 
and moreover $(\xxtilde,\yytilde)$ restricted to $T_p\Sigma$ are Cartesian coordinates on $T_p\Sigma\subset T_pN$.

Following \ref{dexp}, 
we define $U:= D^{\Sigma, N, g}_p ( \, \inj^{\Sigma, N, g}_p /2  \, ) \subset N $ and $U^\Sigma:= D^{\Sigma, g}_p ( \, \inj^{\Sigma, N, g}_p /2 \, ) = \Sigma \cap U $ 
to simplify the notation,    
and then we extend $\zz$ to a coordinate system $(\xx,\yy,\zz)$ on  
$U$ by requesting 
$(\xx,\yy,\zz ) = (\xxtilde,\yytilde,\zztilde) \circ ( \exp^{\Sigma, N, g}_p  )^{-1} $ on $U$.    
We define also a Riemannian metric $\gpeuc$ on $ U $ and symmetric two-tensor fields $h$ on $U$ 
and $h^\Sigma $ on $U^\Sigma $ by  
\[ \gpeuc:= ( \exp^{\Sigma, N, g}_p )_* \left. g\right|_p = d\xx^2 + d\yy^2 +d\zz^2 , \qquad 
 h: = g - \gpeuc, \qquad  
h^\Sigma := \left. h\right|_{U^\Sigma}.
 \]

Finally we define \emph{Fermi cylindrical coordinates} $(\rr,\theta,\zz)$ on 
$\breve{U}:=U\setminus \{ \xx = \yy =0 \}$ by requesting 
$\xx = \rr \cos \theta$ and $\yy  = \rr \sin \theta$; 
we have then 
$\gpeuc= d\rr^2 + \rr^2 d\theta^2 + d\zz^2$ on 
$\breve{U}$ 
and that 
$\evec_\rr:= \partial_\rr$, $\evec_\theta := \partial_\theta / \, |\partial_\theta |_{\gpeuc } $,  and $\evec_\zz:= \partial_\zz$ 
define an orthonormal frame 
$\{ \evec_\rr, \evec_\theta, \evec_\zz \}$ on 
$( \breve{U} , \gpeuc)$.  
\end{definition}

\begin{notation}
\label{Ecyl}
Let 
$\cyl : = \Sph^1\times \R \subset \R^2\times\R$ be the standard cylinder 
and $\chi$ the standard product metric on $\cyl$; 
we have then   
$\Isom(\cyl,\chi) = O(2)\times\Isom(\R,g_{Euc})$ (recall \ref{NEuc}). 
Let $(\vartheta,\sss)$ be the standard coordinates on $\cyl$ 
defined by considering the covering $\Thetacyl:\R^2\to\cyl$ given by 
$
\Thetacyl(\vartheta,\sss) := (\cos\vartheta,\sin\vartheta, \sss) 
$
so that 
$\chi = d\vartheta^2 + d\sss^2$.  
Finally, for $\sbar \in \R$, we define a \emph{parallel circle} $\cyl_{\sbar} :=  \{ \Thetacyl(\vartheta, \sbar) : \vartheta\in\R \}\subset \cyl$ 
and for $I\subset \R$, we define $\cyl_I := \cup_{\sbar\in I} \cyl_{\sbar}$. 
\hfill $\square$ 
\end{notation}

Given $p\in N$ and $\tau\in\R_+$, 
we define a catenoid $\tildecat[p, \tau] \subset T_pN\simeq \R^3 $ of size $\tau$ and its parametrization 
$X_{\tildecat} = X_{\tildecat}[p,\tau] : \cyl \rightarrow \tildecat[p, \tau]$ 
by taking 
(recall \ref{dgeopolar})   
\begin{equation}
\label{Ecatenoid}
\begin{gathered}
\rho(\sss):=\tau\cosh \sss, \qquad \zz(\sss):=\tau\, \sss, \quad \text{and} \\ 
(\xxtilde,\yytilde,\zztilde)\circ X_{\tildecat}\circ \Thetacyl(\vartheta, \sss) = (\rho(\sss) \cos \vartheta , \rho(\sss) \sin \vartheta, \zz(\sss)).  
\end{gathered}
\end{equation}
From now on we will use $X_{\tildecat}$ to identify $\tildecat[p, \tau]$ with $\cyl$; $\vartheta$ and $\sss$ can then be considered as coordinates on $\tildecat[p,\tau]$  
and by \eqref{Ecatenoid} and \ref{Ecyl} we clearly have 
\begin{align}
\label{Ecatmetric}
g_\tildecat := X_{\tildecat}^*  ( \left. g\right|_p) = \rho^2(\sss)  \left( d \vartheta^2 + d\sss^2 \right) = \rho^2\,  \chi.
\end{align}
Alternatively 
$(\xxtilde,\yytilde,\zztilde)^{-1} \{(\rr\cos\vartheta,\rr\sin\vartheta, \phicat(\rr) \,) \, : \, (\rr,\vartheta) \in [\tau,\infty) \times \R\,\} \subset T_pN $
is the part above the waist of $\tildecat[p, \tau] $, where 
the function $\phicat = \phicat[\tau]:[\tau,\infty)\to\R$ is defined by

\begin{equation}
\begin{aligned}
\label{Evarphicat}
\phicat[\tau](\rr):=
\tau\arccosh \frac \rr \tau &=
\tau\left(\log \rr-\log \tau+\log\left(1+\sqrt{1-{\tau^2}{\rr^{-2}}\,}\right)\right)
 \\
&=
\tau\left(\log \frac { 2 \rr } {\tau} 
+
\log\left(\frac12+\frac12\sqrt{1-\frac{\tau^2}{\rr^{2}}\,}\right)\right).
\end{aligned}
\end{equation}

By direct calculation or balancing considerations we have for future reference that
\begin{equation}
\label{Ecatder}
\frac{\partial\phicat}{\partial\rr_{\phantom{cat}}}(\rr) 
=
\frac\tau{\sqrt{\rr^2-\tau^2\,}}.
\end{equation}

\begin{lemma}[Area on untilted catenoidal bridges]
\label{Lcata0}
For any $\tau>0$ and any $r\geq \tau$, the area $|\tildecat(r)|$ of $\tildecat(r) : = \tildecat[p, \tau]\cap \Pi^{-1}_{T_p \Sigma} D^{T_p \Sigma}_{0}(r)$ satisfies 
\begin{align*} 
|\tildecat(r)| 
= 
\sqrt{1 - \frac{\tau^2}{r^2}} \left( 2|D^{T_p \Sigma}_{0}(r)| + \int_{\partial D^{T_p \Sigma}_{0}(r)} \phicat \frac{\partial \phicat}{\partial \eta}dl \right).
\end{align*}
\end{lemma}
\begin{proof}
Direct calculation using \eqref{Ecatder}.
\end{proof}

\subsection*{Tilted catenoids in $T_p N$.}
\nopagebreak

\begin{definition}[Spaces of affine functions] 
\label{DVcal}
Given $p\in\Sigma$ let $\val[p]\subset C^\infty(T_p\Sigma) $ be the space of \emph{affine functions on $T_p\Sigma$}.  
Given a function $v$ which is defined on a neighborhood of $p$ in $\Sigma$ and is differentiable at $p$ we define $\Ecalunder_p v:= v(p)+ d_pv\in\val[p]$. 
$\forall\kappaunder\in\val[p]$ let $\kappaunder=\kappaperp+\kappatilde$ be the unique decomposition with $\kappaperp\in\R$ and $\kappatilde\in T^*_p\Sigma$  
and let $|\kappaunder| := |\kappaperp| + |\kappatilde|$.  
We define for later use $\val[L] := \bigoplus_{p\in L} \val[p]$ for any finite $L\subset\Sigma$.  
\end{definition}

\begin{convention}
\label{con:alpha}
\label{Akappa} 
We fix now some $\alpha >0$ which we will assume as small in absolute terms as needed.  
In the rest of this section we assume that $\tau\in\R_+$ is as small as needed in terms of $\alpha$ only and that 
$\kappaunder\in\val[p]$ satisfies $|\kappaunder| < \tau^{1+ \alpha/6}$.
\end{convention}

\begin{definition}[Tilted catenoidal bridges]
\label{dphicattilt}
Given 
$\kappaunder\in \val[p]$ we  
define $\phicatpm[\tau, \kappaunder]: T_p \Sigma \setminus D_0^{T_p\Sigma} (9\tau) \rightarrow \R$ by 
$\phicatpm[\tau, \kappaunder] : = \text{\emph{Tilt}}_{\pm \kappatilde} (  \phicat[\tau] \circ \dbold_0^{T_p\Sigma} ) \pm  \kappaperp$  
in the notation of Definitions \ref{dTiltop} and \ref{DVcal}, where in \ref{dTiltop} we take $\Etwo = T_p \Sigma$, $\Ethree = T_p N$, and the normal vector to $\Etwo$ to be $\nu$.
\end{definition}

\begin{lemma}[Tilted catenoid asymptotics]
\label{Ltcest}
For $k\in\N$, $\tau\in\R_+$ and $\kappaunder\in\val[p]$ as in \ref{con:alpha} we have 
\begin{equation*} 
\left\| \phicatp[ \tau, \kappaunder] - \tau \log ( { 2 \rr } / {\tau} ) - \kappaunder : 
C^{k}\left( D_0^{T_p\Sigma} (8\tau^\alpha) \setminus D_0^{T_p\Sigma} (9 \tau), \rr, g, \rr^{-2}\right)\right\|  
\le C(k)(|\kappatilde|+\tau)^3. 
\end{equation*} 
\end{lemma}

\begin{proof}
If $\kappaunder$ vanishes it is enough to prove 
the following, which is true $\forall \tau\in\R_+$ 
by \eqref{Evarphicat} and \eqref{Ecatder} {\cite[Lemma 2.25]{kap}}. 
\begin{equation} 
\label{Lcatenoid}
\|\, \phicat[\tau] -
\tau \log ( { 2 \rr } / {\tau} )
\,
: C^{k}(\,
(9 \tau,\infty)\,,\,
\rr, d\rr^2,\rr^{-2}\,)\,\|
\le 
\, C(k) \, \tau^3. 
\end{equation} 

Clearly 
$
\|\, 
\tau \log ( { 2 \rr } / {\tau} )
\,
: C^{k}(\,
(9 \tau,8\tau^\alpha) \,,\,
\rr, d\rr^2\,)\,\|
\le 
\, C(k) \, \tau \, |\log\tau| . 
$
Combining with \eqref{Lcatenoid}, 
using \ref{dwHolder}, scaling, applying \ref{Ltiltgap}, and taking in this proof 
$\Omega:= D_0^{T_p\Sigma} (8\tau^\alpha)\setminus D_0^{T_p\Sigma} (9 \tau)$,  
we conclude  
\begin{align*}
\| \phicatp[ \tau, \kappatilde] - \phicat[\tau] - \kappatilde : C^{k}\left( \Omega, \rr, g\right)\| 
\le C(k) (\tau |\log \tau| + |\kappatilde|)^3. 
\end{align*}
Using that $\rr^2\le 8^2 \tau^{2\alpha}$ on $\Omega$, 
combining with \eqref{Lcatenoid}, 
and observing that $\kappaperp$ cancels out, we conclude 
\begin{equation*} 
\left\| \phicatp[ \tau, \kappaunder] - \tau \log ( { 2 \rr } / {\tau} ) - \kappaunder : 
C^{k}\left( \Omega, \rr, g, \rr^{-2}\right)\right\|  \le 
C(k) \, \left(\tau^3 + \tau^{2\alpha} (\tau |\log \tau| + |\kappatilde|)^3 \right) ,
\end{equation*} 
which implies the result by assuming $\tau$ small enough as in \ref{Akappa}. 
\end{proof}

\begin{definition}[Tilted catenoids in $T_p N$ and catenoidal bridges in $N$]
\label{DbM}
Given $p\in \Sigma$, $x\in[0,4]$ (where $x$ may be omitted when $x = 0$), $\tau>0$, and $\kappaunder = \kappaperp + \kappatilde   \in \val[p] $, we define 
an elevated and tilted by $\kappaunder$ \emph{model catenoid} in $T_pN$ of size $\tau$,  
a corresponding \emph{catenoidal bridge} in $N$ (slightly reduced if $x>0$), 
and its \emph{core} (slightly expanded if $x>0$), 
as follows (recall \ref{Ecyl} and \ref{dRk}),   
where $b$ is a large constant to be chosen later independently of the $\tau$ and $\kappaunder$ parameters.  
\begin{align*} 
\tildecat[p, \tau, \kappaunder] \, &:= \, X_{\tildecat}[p, \tau, \kappaunder] \, ( \cyl) \subset T_p N,  
\\ 
\cat_x[p, \tau, \kappaunder] \, &:= \, X_{\cat}[p, \tau, \kappaunder] 
\left(\cyl\left[ \tau, 2  \tau^{\alpha} / (1+x) \right] \right)  
\subset N,   
\\
\text{and}  \quad 
\Kcore_x[p, \tau, \kappaunder] \, &:= \, X_{\cat}[p, \tau, \kappaunder] 
\left(\cyl\left[ \tau, b(1+x) \tau \right] \right)  
\subset \cat_x[p, \tau, \kappaunder] 
\subset N,   
\\
\text{where} \quad 
X_{\tildecat}[p, \tau, \kappaunder] &:= \RRR_\kappa \circ X_{\tildecat}[p, \tau]+ \kappaperp \nu_\Sigma(p) : \cyl \rightarrow T_p N ,
\\
X_{\cat}[p, \tau, \kappaunder] &:= \exp^{\Sigma, N, g}_p \circ X_{\tildecat}[p, \tau, \kappaunder] : \cyl \rightarrow N ,   
\\
\text{and} \qquad 
\cyl[\tau,r] &:= \Thetacyl \left( \{( \vartheta , \sss ) \in \R^2 : \tau \cosh \sss < r \} \right) 
\qquad \forall r\in\R_+. 
\end{align*} 
Finally using the above maps we take the coordinates $(\vartheta, \sss)$ on the cylinder as in \ref{Ecyl} to be coordinates on $\tildecat[p, \tau, \kappaunder]$ and $\cat[p, \tau, \kappaunder]$ also,
where we also define $\rho(\sss) := \tau \cosh \sss$ as in \ref{Ecatenoid}.
\end{definition}

\begin{remark}
\label{R:comp} 
Note that Definitions \ref{dphicattilt} and \ref{DbM} are compatible in the sense that (recall also \ref{NT}\ref{dgraph})
$$ 
\graph_\Omega^{N}\left( \phicatp[\tau, \kappaunder;\Omega]\right) \cup 
\graph_\Omega^{N}\left(- \phicatm[\tau, \kappaunder;\Omega] \right)
\cup \cat[ p, \tau, \kappaunder] \subset N,  
$$ 
is a connected smooth surface with boundary; 
where (only) here we use 
$\Omega := D^\Sigma_p(8\tau^\alpha)\setminus D^\Sigma_p(9\tau)$  
and  
$
\phicatpm[\tau, \kappaunder;\Omega] := \phicatpm[\tau, \kappaunder] \circ (\exp^\Sigma_p)^{-1} :\Omega\to\R.
$
\qed
\end{remark}

\begin{lemma}[Area on tilted catenoids in $T_pN$]
\label{Lcata1}
Fix $\tau>0$ and $\rtop = \tau^{3/4}$.  Then the area $|\tildecat(\rtop)|$ of $\tildecat(\rtop) : = \tildecat[p, \tau, \kappaunder] \cap  \Pi^{-1}_{T_p \Sigma}D^{T_p \Sigma}_{0}(\rtop)$, satisfies 
\begin{equation*}
|\tildecat(\rtop) | 
= 
2|D^{T_p \Sigma}_{0}(\rtop)| - \pi \tau^2 + \frac{1}{2}\int_{\partial D^{T_p \Sigma}_{0}(\rtop)}\left( \phicat^+\frac{\partial\phicat^+}{\partial \eta}+ \phicat^- \frac{\partial \phicat^-}{\partial \eta}\right) dl 
+ 
O(\tau^{5/2}|\log \tau|),
\end{equation*}
where $\phicat^\pm = \phicat^\pm[\tau, \kappaunder]$ is as in \ref{dphicattilt}. 
\end{lemma}

\begin{proof}
In this proof, denote by $\tildecat[p, \tau](\rtop)$ for $\tildecat(\rtop)$ as in \ref{Lcata0} and by $\tildecat[p, \tau, \kappaunder](\rtop)$ for $\tildecat(\rtop)$ as in \ref{Lcata1}.  
We first compare the areas $|\tildecat[p, \tau](\rtop)|$ and $|\tildecat[p, \tau, \kappaunder](\rtop)|$.  
Using \ref{DbM}, we estimate the distance between any point on $\partial D^{T_p \Sigma}_0(\rtop)$ and its nearest point on 
$\Pi_{T_p \Sigma} ( \RRR_{-\kappa}(\partial \tildecat[p, \tau, \kappaunder](\rtop)))$ is bounded by 
$C( \rtop |\kappaunder|^2+ |\kappaunder| \tau |\log \tau|)$.  It is not difficult to see from this and the bound $|\kappaunder| < \tau^{1+ \alpha/6}$ from \ref{Akappa} that 
\begin{align}
\label{Etildecatcomp}
|\tildecat[p, \tau](\rtop)| =  |\tildecat[p, \tau, \kappaunder](\rtop)| + O(\tau^{11/4} | \log \tau| ).
\end{align}
Next, using \ref{Ltcest} and \ref{Lcatenoid} to expand $\phicat^\pm$ and \ref{Akappa} to estimate $|\kappaunder|$, it follows that
\begin{equation}
\label{Esflux}
\int_{\partial D^{T_p \Sigma}_0(\rtop)} \phicat \frac{\partial \phicat}{\partial \eta} dl =  
\frac{1}{2}\int_{\partial D^{T_p \Sigma}_0(\rtop)}\left( \phicat^+ \frac{\partial \phicat^+}{\partial \eta} + \phicat^-\frac{\partial \phicat^-}{\partial \eta}\right)dl
+ O(\tau^{2+3/4}).
\end{equation}
Finally, we have that $|D_0^{T_p \Sigma}| = 2\pi \rtop^2$ and $\sqrt{1- \tau^2/\rtop^2} = 1 - \frac{1}{2} \frac{\tau^2}{\rtop^2} + O(\frac{\tau^4}{\rtop^4})$.  The conclusion follows from combining this with \ref{Lcata0}, \eqref{Etildecatcomp} and \eqref{Esflux}. 
\end{proof}

\subsection*{Mean curvature on tilted catenoidal bridges in $N$}
\nopagebreak

The final goal of this section is to estimate the mean curvature of a tilted bridge $\cat[p, \tau, \kappaunder] \subset N$.  
We first introduce some convenient notation. 
\begin{notation} 
\label{N:H} 
We denote by $\Hcirc$ and $H$ the mean curvature of $\cat[p, \tau, \kappaunder] \subset N$ with respect to $\gpeuc$ and $g$ respectively. 
\qed 
\end{notation} 

Appendix \ref{S:A1} allows us to express $H$ in terms of $\Hcirc$ and  certain tensors defined on $\cat$;  
because the metric $\gpeuc$ is Euclidean, $\Hcirc= 0$ and the task is reduced to estimating the tensors defined on $\cat$.  
To motivate the discussion, we first consider the simplest situation in two model cases.

\begin{example}[$H$ on catenoidal bridges over $\Spheq \subset \Sph^3$] 
\label{Ex:sph} 
Let $\Spheq$ be the equatorial two-sphere in the round three-sphere $\Sph^3 \subset \R^4$.  
Given $p=(0,0, 1, 0)$, let $\cat: = \cat[p, \tau,0]$.  
From \ref{exSph} and by calculation we find that the metric $g^\cat$ and unit normal $\nu_\cat$  induced by $g$ on $\cat$ 
are given by 
\begin{gather*}
g^\cat = \rr^2( 1- \tanh^2 \sss \sin^2 \zz) d\sss^2 + \cos^2 \zz \sin^2 \rr \, d\theta^2,
\\ 
\nu_\cat  =( \tanh \sss\,  \partial_{\zz} - \sec^2\zz \sech \sss \, \partial_{\rr})/\sqrt{1+\tan^2 \zz \sech^2 \sss}, 
\end{gather*}
where $\zz = \tau \sss$ on $\cat$. 

We use the formula 
$A^\cat = \left( X^{k}_{, \alpha \beta} + \Gamma_{lm}^{k} X^l_{, \alpha} X^{m}_{, \beta}\right) g_{kn}\nu^n dx^\alpha dx^\beta$, 
where $X = X_\cat[p, \tau, 0]$ is as in \ref{DbM}, we have renamed the cylinder coordinates $(x^1, x^2): = (\sss, \theta)$, 
and Greek indices take the values $1$ and $2$ while Latin indices take the values $1,2,3$, corresponding to the coordinates $\rr, \theta, \zz$;  
and the Christoffel symbols in \ref{exSph}; to find
\begin{multline*}
\sqrt{ 1+ \tan^2 \zz \sech^2 \sss\, }  \, A^\cat = 
\left[ \tau^2 \tanh \sss \left( \tan \zz + \frac{1}{2}\sinh^2 \sss \sin 2\zz \right) - \tau\right] d\sss^2 
\\ +\frac{1}{2}(\sin 2\rr \sech \sss + \sin^2 \rr \sin 2\zz \tanh \sss) \, d\theta^2.   
\end{multline*}
Using that $ \sqrt{ 1+ \tan^2 \zz \sech^2 \sss} = 1+ O(\zz^2)$ and $\frac{1}{2}\sin 2\rr \sech \sss = \tau + O(\rr^3)$ we conclude  
$$ 
A^\cat = (1+ O(\zz^2)) \left( \tau( - d\sss^2 + d\theta^2) +  O(\rr^2 \zz)d\sss^2 + O(\rr^3+\rr^2 \zz) d\theta^2\right).
$$ 
Finally, using that $\rr^2 g^{\sss \sss} = 1+ O(\zz^2)$ and $\rr^2 g^{\theta \theta} = 1+O(\zz^2 +\rr^2)$ we estimate
\begin{align*}
\rr^2 H = O\left(\tau \zz^2 +\rr^2 |\zz| +  \tau \rr^2 \right).
\end{align*}
\end{example}

\begin{example}[$H$ on catenoidal bridges over $\T\subset\Sph^3$]
\label{Ex:tor} 
Let $\T$ be the Clifford torus in $\Sph^3 \subset\R^4 \simeq \C^2$.  
Given $p = (1/\sqrt{2}, 1/\sqrt{2}) \in \T$, let $\cat: = \cat[p, \tau, 0]$.  From \ref{exClifford}, the metric induced by $g$ on $\cat$ is
\begin{equation*}
 g^\cat = \rr^2 ( d\sss^2 + d\theta^2) 
+ \rr^2 \sin2 \zz \left( \tanh^2 \sss \cos 2\theta d\sss^2 - 2 \tanh \sss \sin 2\theta d\sss d\theta - \cos 2\theta d\theta^2\right),
\end{equation*}
where $\zz = \tau \sss$ and $\rr = \tau \cosh \sss$ on $\cat$.  
As in \cite[Lemma 3.18]{kapouleas:clifford} or \cite[Proposition 4.28]{Wiygul}, it follows that
\begin{align*}
\left\| \rr^2  H : C^{k}\left( \cat, \chi, \tau |\zz| + \rr^2 |\zz| + \tau^2 \right) \right\| \leq C.
\end{align*}
\end{example}

The preceding examples show that the mean curvature on a catenoidal bridge over $\Spheq \subset \Sph^3$ 
satisfies better estimates than the mean curvature on a bridge over $\T$.  
This is due to the fact that $\Spheq$ is totally geodesic while $\T$ is not.  
We will see more generally (see \ref{Ltensest}(iii) and (v) and \ref{LH}(i) below) that dominant terms in the mean curvature of a  
bridge $\cat[p, \tau, \kappaunder]$ are driven by the second fundamental form of $\Sigma$ when $\left. A^\Sigma\right|_p$ does not vanish.  
Unfortunately, the resulting estimates on $H$ will not be by themselves sufficient for our applications, 
and it will be essential to observe later in \ref{LH}(ii) and \ref{LH2}(ii) that the projection onto the first harmonics $\Hcal_1 H$ of the mean curvature $H$ of such a bridge 
(to be defined in \ref{dlowharm}) satisfies a better estimate.

For the rest of the section, fix $p\in \Sigma$ and let $(\rr, \theta, \zz)$ be cylindrical Fermi coordinates about $\Sigma$ centered at $p$ as in \ref{dgeopolar}.

\begin{lemma}
\label{Lh}
\begin{enumerate}[label=\emph{(\roman*)}]
\item $h= \Pi^*_\Sigma h^\Sigma  - 2 \zz \, \Pi^*_\Sigma A^\Sigma + \zz^2 \, \Pi^*_\Sigma(A^\Sigma*A^\Sigma + \Rm^\Sigma_{\nu}) + \zz^3 \, h^{\mathrm{err}}$, 
where $h^{\mathrm{err}}$ is a smooth symmetric two-tensor field on $U$.
\item  $\tr_{N, \gpeuc} h = \tr_{\Sigma, \gpeuc} h^\Sigma + \zz^2 \, \Pi^*_\Sigma \left(|A^\Sigma|^2- \Ric(\nu_{\Sigma}, \nu_\Sigma) \right) 
+ \zz^3 \, \tr_{N, \gpeuc} h^{\mathrm{err}}.$
\item $\| \nabla h : C^k(U, \gpeuc)\| \le C(k)$.
\item $\| h^\Sigma : C^k( U^\Sigma,\rr,  \gpeuc, \rr^2)\| \le C(k).$
\end{enumerate}
\end{lemma}
\begin{proof}
(i) follows from Lemma \ref{Lgauss}.  (ii) follows from taking the trace of (i), and (iii) follows from (i), using that $\gpeuc$ is Euclidean.  
For (iv), recall that $g^\Sigma = d\rr^2 + u(\rr, \theta)^2 d\theta^2$
where $u(\rr,\theta)$ solves the Gauss-Jacobi initial value problem
\begin{align*}
u_{\rr\rr} + K_\Sigma u = 0, \quad \lim_{\rr \searrow 0} u(\rr, \theta) = 0, \quad \lim_{\rr \searrow 0} u_\rr(\rr, \theta) = 1.
\end{align*}
It follows that $u(\rr, \theta) = \rr  - \frac{\left.K_\Sigma\right|_p}{3!} \rr^3 + O(\rr^5)$ and consequently
\begin{align}
h^\Sigma  = f(\rr, \theta) d\theta^2, \quad \text{where} \quad f(\rr,\theta) := - (\left. K_\Sigma\right|_p/3) \rr^4 +O(\rr^6). 
\end{align}
This completes the proof of (iv).
\end{proof}

For the remainder of this section we use notation from Appendices \ref{A:tilt} and \ref{S:A1} 
and we abbreviate by writing $\cat$ for $\cat[p, \tau, \kappaunder]$.   

\begin{lemma} For $X_{\tildecat} = X_{\tildecat}[p, \tau, \kappaunder]$ and $\tildecat = \tildecat[p, \tau, \kappaunder] \subset T_p N$, the following hold with 
\label{Lcatrimm}
\begin{equation} 
\label{Nw}
\wvec :=  \cos \vartheta \, \kvec + \sin \vartheta \cos \theta_\kappa \kvec^\perp, \qquad
\wvec' := - \sin \vartheta \kvec + \cos \vartheta \cos \theta_\kappa \kvec^\perp. 
\end{equation} 
\begin{enumerate}[label=\emph{(\roman*)}]
\item $X_{\tildecat} = \rho \cos \vartheta \, \kvec + ( \rho \sin \vartheta - \tau \sss \sin \theta_\kappa) \kvec^\perp +(\tau \sss \cos \theta_\kappa + \rho \sin \theta_\kappa) \partial_{\zz}$.
\item $\partial_\sss X_{\tildecat} = \tau \sinh  \sss \left( \wvec + \sin \vartheta \sin\theta_\kappa \partial_\zz\right) - \tau \sin \theta_\kappa \kvec + \tau \cos \theta_\kappa \partial_\zz$.
\item $\partial_\vartheta X_\tildecat = \rho \left( \wvec' + \cos \vartheta \sin \theta_{\kappa}\partial_\zz\right)$. 
\item $ {X_{\tildecat}}\!\!\!{}_{\phantom{K}}^{{ *_{{\phantom{p}_{\phantom{q}}}}  }} \:\!\!\!\!\!g = \rho^2\chi$.
 \item $\nu_{\tildecat } = \nu^{\parallel}+ \nu^\perp$, where $\nu^\perp : = (\tanh \sss \cos \theta_\kappa - \frac{\tau}{\rho} \sin \vartheta \sin \theta_\kappa ) \partial_{\zz}$  
and  $\nu^{\parallel} : = - \frac{\tau}{\rho} \wvec - \tanh \sss \sin \theta_\kappa \, \kvec^\perp$.
\end{enumerate}
\end{lemma}

\begin{proof}
Straightforward computation using \eqref{Ecatenoid}, \ref{Nw}, and  \ref{Lrot}(iii) implies
$\quad \RRR_\kappa (\kvec) = \kvec$, 
$\quad \RRR_\kappa (\kvec^\perp) = \cos \theta_\kappa \kvec^\perp + \sin \theta_\kappa \partial_\zz, \quad$ 
and 
$\quad \RRR_\kappa (\partial_\zz) = \cos \theta_\kappa \partial_\zz - \sin \theta_\kappa \kvec^\perp.$ 
\end{proof}

\begin{lemma}[cf. {\cite[Lemma 3.18]{kapouleas:clifford}}]
\label{Lzinit}
The following hold.
\begin{enumerate}[label=\emph{(\roman*)}]
\item $\| \rho^{\pm 1}: C^k(\cat, \chi, \rho^{\pm 1}) \| \le C(k)$.
\item $\| \zz : C^k(\cat, \chi,  |\zz| + \tau ) \| \le C(k) $.
\item \begin{enumerate}[label=\emph{(\alph*)}]
	\item $\| \nabla^{N,\gpeuc}_{\partial_\sss} h : C^k(\cat, \chi, \rho ( |\zz| + \rho) + \tau) \| \le C(k)$. 
	\item $\| \nabla^{N,\gpeuc}_{\partial_\vartheta} h : C^k(\cat, \chi, \rho|\zz| + \rho^2 ) \| \le C(k)$. 
	\item $\| \nabla^{N, \gpeuc}_{\nu} h: C^k(\cat, \chi  ) \| \le C(k)$.
	\end{enumerate}
\item $\| h\intprod \wvec: C^k(\cat, \chi, |\zz| + \tau  + \rho^2)\| \le C(k)$. 
\end{enumerate}
\end{lemma}

\begin{proof}
The estimate in (i) with $\rho$ is obvious, 
and the estimate in (i) with $\rho^{-1} = \tau^{-1} \sech \sss$ follows after observing that for each $k\geq 1$, 
$\partial^k_{\sss}( \sech \sss)$ is a polynomial expression in $\sech \sss$ and $\tanh \sss$, 
each term of which contains a factor of $\sech \sss$.  
From Lemma \ref{Lcatrimm} we have $\zz = \tau ( \cos \theta_\kappa\,  \sss + \sinh \sss \sin \theta_\kappa \sin \vartheta)$, which implies (ii). 
 Using Lemma \ref{Lh}, 
 \begin{align*}
 \nabla_{\partial_\zz} h &= - 2 \Pi^*_\Sigma A^\Sigma  + 2\zz \Pi^*_\Sigma(A^\Sigma*A^\Sigma + \Rm^\Sigma_{\nu}) + 3\zz^2 h^{\mathrm{err}} + \zz^3 \nabla_{\partial_\zz} h^{\mathrm{err}}, 
\\
 \nabla_{\vecu} h &= \nabla_\vecu h^\Sigma  - 2\zz \, \nabla_\vecu  \Pi^*_\Sigma A^\Sigma + \zz^2 \nabla_\vecu \Pi^*_\Sigma(A^\Sigma*A^\Sigma + \Rm^\Sigma_{\nu})+ \zz^3 \, \nabla_\vecu h^{\mathrm{err}},
 \end{align*}
 where $\vecu$ is either $\kvec$ or $\kvec^\perp$.  Using this in conjunction with (ii) and Lemma \ref{Lh}(iv), we conclude 
 \begin{equation}
 \label{Edh}
 \begin{aligned}
 \| \nabla_{\partial_\zz}^{N, \gpeuc} h : C^k(\cat, \chi ) \| &\le C(k),  \\ 
 \| \nabla_{\vecu}^{N,\gpeuc} h : C^k(\cat, \chi, |\zz| + \rho) \| &\le C(k).
 \end{aligned}
 \end{equation}
Recalling from \ref{Lcatrimm} that
\begin{align*}
 \partial_\sss &= ( \tau \sinh \sss \sin \vartheta \sin \theta_\kappa + \tau \cos \theta_\kappa) \partial_\zz 
\\ & \quad \qquad \qquad \qquad \qquad \qquad 
+ \tau \sinh\sss ( \cos \vartheta \, \kvec + \sin \vartheta \cos \theta_\kappa \kvec^\perp) - \tau \sin \theta_\kappa \kvec^\perp, \\
 \partial_\vartheta &=  \rho \left(  - \sin \vartheta\, \kvec + \cos \vartheta \cos \theta_\kappa \,  \kvec^\perp +  \cos \vartheta \sin \theta_\kappa\, \partial_\zz\right), \\
 \nu &= - {\tau}{\rho}^{-1} ( \cos \vartheta \, \kvec + \sin \vartheta \cos \theta_\kappa \kvec^\perp ) - \tanh \sss \sin \theta_\kappa \, \kvec^\perp 
\\ & \quad \qquad \qquad \qquad \qquad \qquad \qquad \qquad 
+ (\tanh \sss \cos \theta_{\kappa} - {\tau}{\rho}^{-1} \sin \vartheta \sin \theta_\kappa ) \partial_{\zz},
 \end{align*}
 we see that (iii) follows from the preceding.
 Finally, for (iv), we compute that
\begin{align*}
 h(\wvec, \kvec) &= \cos \vartheta\,  h(\kvec, \kvec) +  \sin \vartheta \cos \theta_\kappa h(\kvec, \kvec^\perp),\\ 
 h(\wvec, \kvec^\perp) &= \cos \vartheta\,  h(\kvec, \kvec^\perp) + \sin \vartheta \cos \theta_\kappa h(\kvec^\perp, \kvec^\perp).
\end{align*}
The estimate follows from this by combining the results of (ii) and (iii) above with Lemma \ref{Lh}(i).
\end{proof}

\begin{lemma}
\label{Ltensest}
The following hold.
\begin{enumerate}[label=\emph{(\roman*)}]
\item $\| \alpha : C^k(\cat, \chi, \rho^2 ( \rho^2 + |\zz| + \tau ))\| \le C(k)$. 
	\item $\| \alphatilde : C^k(\cat, \chi, \rho^2)\| \le C(k)$.
	\item $\| \tr_{\cat, \chi} \alphatilde - 2\tau^2 \tanh \sss \, (\Pi^*_\Sigma A^\Sigma) (\evec_\rr, \evec_\rr): C^k(\cat,  \chi,  \rho^2( |\zz| + \tau) ) \| \le C(k)$.
	\item $ \| \beta: C^k ( \cat, \chi, \tau( |\zz| + \rho^2 + \tau)  ) \| \le C(k)$.
	\item $\|  \ddiv_{\cat, \chi} \beta + 2\tau(2\zz - \zz \sech^2 \sss - \tau \tanh \sss) (\Pi^*_\Sigma A^\Sigma) (\evec_\rr, \evec_\rr) :$ 
$ C^k(\cat, \chi, \rho^2( \tau + |\zz|)) \| \le C(k)$.
\item $\| \sigma : C^k( \cat, \chi, |\zz| + \tau) \| \le C(k) $.
\end{enumerate}
\end{lemma}
\begin{proof}

 Using Lemma \ref{Lcatrimm} and Remark \ref{Nw} we compute
 \begin{equation*}
 \begin{gathered}
 \alpha_{\sss\sss} = (\rho^2 - \tau^2) h(\wvec , \wvec ) + \tau^2 \sin^2 \theta_\kappa h( \kvec^\perp, \kvec^\perp) 
 -2 \tau \sin \theta_\kappa \sqrt{\rho^2 - \tau^2} h(\wvec , \kvec^\perp), \\
 \alpha_{\vartheta \vartheta} = \rho^2 h( \wvec', \wvec') , \\
 \alpha_{\sss \vartheta} = - \tau \rho \sin \theta_\kappa h(\wvec' , \kvec^\perp)+ \rho \sqrt{\rho^2 - \tau^2} h( \wvec' ,\wvec).
 \end{gathered}
 \end{equation*}
(i) follows then from Lemma \ref{Lzinit}.
The proof of (ii) is straightforward so we proceed to (iii). 
Using \ref{Lcatrimm} we compute 
\begin{align*}
\alphatilde_{\sss \sss} &= 
\begin{aligned}[t]  
(\rho^2 - \tau^2) (\nabla_\nu h)(\wvec, \wvec)  
- 2\tau \sqrt{\rho^2 - \tau^2} & \sin \theta_\kappa (\nabla_\nu h)( \wvec, \kvec^\perp) 
\\ 
& 
+ \tau^2 \sin^2 \theta_\kappa (\nabla_\nu h)(\kvec^\perp, \kvec^\perp), 
\end{aligned} 
\\ 
\alphatilde_{\vartheta \vartheta} &= \rho^2(\nabla_\nu h)(\wvec', \wvec').
\end{align*}

 Because $\tr_{\cat, \chi} \alphatilde = \alphatilde_{\sss \sss} + \alphatilde_{\vartheta \vartheta}$, we have via \ref{Nw}  
\begin{multline}
\label{Etrat}
\tr_{\cat, \chi} \alphatilde = 
\rho^2 (\nabla_\nu h)(\kvec, \kvec) + \rho^2 (\nabla_\nu h)(\kvec^\perp, \kvec^\perp)
- \tau^2 (\nabla_\nu h)(\wvec, \wvec) 
\\ 
+ (\tau^2 - \rho^2) \sin^2 \theta_\kappa (\nabla_\nu h)(\kvec^\perp, \kvec^\perp)  
- 2 \tau \sqrt{ \rho^2 - \tau^2} \sin \theta_\kappa (\nabla_\nu h )(\wvec, \kvec^\perp).
\end{multline}

Noting that $(\nabla_\nu h)(\kvec, \kvec) + (\nabla_\nu h)(\kvec^\perp, \kvec^\perp ) = \tr_{N, \gpeuc} (\nabla_\nu h) = \nabla_\nu (\tr_{N, \gpeuc} h) = \nu (\tr_{N, \gpeuc} h)$, 
we have by  \ref{Lh} and \ref{Lcatrimm} that
\begin{align*}
\nu^\perp ( & \tr_{N, \gpeuc} h) = 
\begin{aligned}[t] 
2(\tanh \sss \cos \theta_\kappa - \sech \sss \sin \vartheta \sin\theta_\kappa) &  \left( |A^\Sigma|^2 - \Ric(\nu_\Sigma, \nu_\Sigma) \right) \zz  
\\ 
& \quad \qquad \qquad \qquad  + O(\zz^2),  
\end{aligned} 
\\
\nu^\parallel ( & \tr_{N, \gpeuc} h) =-  {\tau}{\rho}^{-1} \wvec \left( \tr_{\Sigma, \gpeuc} h^\Sigma + \left( |A^\Sigma|^2 - \Ric(\nu_\Sigma, \nu_\Sigma)\right) \zz^2 \right) 
\\ 
&\phantom{aaa} - \left(\tanh \sss \sin \theta_\kappa \kvec^\perp\right) \left( \tr_{\Sigma, \gpeuc} h^\Sigma + \left( |A^\Sigma|^2 - \Ric(\nu_\Sigma, \nu_\Sigma)\right) \zz^2\right) + O(\zz^3).
\end{align*}
This and \ref{Lzinit} imply that $\| \rho^2 \nu( \tr_{N,\gpeuc} h) : C^k(\cat, \chi, \rho^2( |\zz| + \tau  ) ) \| \le C(k)$.

We now estimate the remaining terms.  
Using \ref{Lh}(i) and \ref{Lcatrimm}(v) we have 
\[ 
\| \nabla_\nu h + 2 \tanh \sss\,  \Pi^*_\Sigma A^\Sigma : C^k(\cat, \chi, |\zz| + \tau + \rho^2)\| \le C(k).
\]

Estimating terms with $\sin\theta_\kappa$ in \eqref{Etrat} by Lemma \ref{Lzinit} we obtain 
\[ 
\| \tr_{\cat, \chi} \alphatilde - \rho^2 \nu(\tr_{N, \gpeuc} h) - 2\tau^2 \tanh \sss \, (\Pi^*_\Sigma A^\Sigma) (\evec_\rr, \evec_\rr) : C^k( \cat, \chi, \rho^2 \tau) \| 
\le C(k),
\] 
and (iii) with the estimate on $\tr_{\cat, \chi} \alphatilde$ follows.
To prove (iv) we use 
\begin{equation}
\label{Ebeta}
\beta(X) = h(X, \nu_\cat) = h( X, \nu^{\parallel}) 
= - {\tau}{\rho}^{-1} h(X, \wvec ) - \tanh \sss \sin \theta_\kappa \, h(X, \kvec^\perp),
\end{equation}
to compute using Lemma \ref{Lcatrimm} 
\begin{align*}
-\tau^{-1} \beta_\sss &=   
\begin{aligned}[t]  
\tanh \sss \: &\big[  h( \wvec,  \wvec) -  \sin^2 \theta_\kappa h(\kvec^\perp, \kvec^\perp)\big] 
\\
& \qquad \qquad \qquad + \sin \theta_\kappa h(\wvec, \kvec^\perp)( \sinh \sss \tanh \sss - { \tau}{ \rho}^{-1} ), 
\end{aligned} 
\\
-\beta_\vartheta &=  \tau h(\wvec', \wvec)  + \rho \tanh \sss \sin \theta_\kappa h( \wvec', \kvec^\perp).
\end{align*}
The estimate on $\beta$ in (iv) follows from Lemma \ref{Lzinit}.
We next compute 
$\ddiv_{\cat, \chi} \beta = \beta_{\sss, \sss} + \beta_{\vartheta, \vartheta}$. 
We have 
\begin{multline*}
-\tau^{-1} \beta_{\sss, \sss} = 
\sech^2 \sss \: \big[  h( \wvec,  \wvec) -  \sin^2 \theta_\kappa h(\kvec^\perp, \kvec^\perp)\big]  
\\
+ \sin \theta_\kappa h(\wvec, \kvec^\perp) \, ( \sinh \sss + 2 \sech \sss \tanh \sss) \qquad \qquad 
\\ 
+ \tanh \sss \big[ (\nabla_{\partial_\sss} h)(\wvec, \wvec) - \sin^2 \theta_\kappa (\nabla_{\partial_\sss} h )(\kvec^\perp, \kvec^\perp)\big] 
\\ 
+ \sin \theta_\kappa (\nabla_{\partial_\sss} h )(\wvec, \kvec^\perp)  ( \sinh \sss \tanh \sss - {\tau}{\rho}^{-1} ).
\end{multline*}
Using this with \ref{Lh}(i), \ref{Lcatrimm}(ii), \ref{Lzinit}, and \ref{Nw}, we estimate
\begin{equation*}
\big\| \, \beta_{\sss, \sss} - 2 \tau \zz \sech^2 \sss\, (\Pi^*_\Sigma A^\Sigma) (\evec_\rr, \evec_\rr) - 2\tau^2 \tanh\sss\, (\Pi^*_\Sigma A^\Sigma) (\evec_\rr, \evec_\rr) 
: C^k(\cat, \chi, \rho^2(|\zz| + \tau)) \, \big\| \lem C(k ). 
\end{equation*}

Next we compute 
\begin{equation*}
- \beta_{\vartheta, \vartheta} = - \tau h( \wvec, \wvec) + \tau h( \wvec', \wvec')  
+\rho \tanh \sss \sin \theta_\kappa \frac{\partial}{\partial \vartheta}\left(  h(  \wvec', \kvec^\perp)\right) + \tau (\nabla_{\partial_\vartheta} h)(\wvec', \wvec).
\end{equation*}
Using the minimality of $\Sigma$ and \ref{Lh}(i), we have
\begin{align*}
\| (\Pi^*_\Sigma A^\Sigma) (\evec_\rr, \evec_\rr)+ (\Pi^*_\Sigma A^\Sigma) (\evec_\theta, \evec_\theta) : 
C^k( \cat, \chi, \rho^2 + |\zz|) \| \leq C(k).
\end{align*}
Combining this with \ref{Lh}(i), \ref{Lcatrimm}(iii), \ref{Lzinit}, \ref{Nw} 
and the above we obtain 
\begin{align*}
\| \beta_{\vartheta, \vartheta} + 4 \tau \zz (\Pi^*_\Sigma A^\Sigma) (\evec_\rr, \evec_\rr): C^k(\cat, \chi, \rho^2(\tau +|\zz|) )\| \le C(k).
\end{align*}
Combining the preceding completes the estimate in (v) on $\ddiv_{\cat, \chi} \beta$.
Next 
\begin{equation*}
\sigma = h(\nu, \nu) = h(\nu^\parallel, \nu^\parallel)
= \frac{\tau^2}{\rho^2} h( \wvec, \wvec) - 2 \frac{\tau}{\rho} \tanh \sss \sin \theta_\kappa h(\wvec, \kvec) + \tanh^2 \sss \sin^2 \theta_\kappa h(\kvec, \kvec)
\end{equation*}
and the estimate on $\sigma$ follows from \ref{Lzinit}.
\end{proof}

\begin{definition}
\label{dlowharm}
Given a function $u$ defined on $\cat$, we define the \emph{projection $\Hcal_1( u) $ of $u$ onto first harmonics} by 
\begin{align*} 
\Hcal_1( u) = \frac{1}{\pi} \left( \int_0^{2\pi} u(\sss, \vartheta) \cos \vartheta d\vartheta\right) \cos \vartheta +  
\frac{1}{\pi} \left( \int_0^{2\pi} u(\sss, \vartheta) \sin \vartheta d\vartheta\right) \sin \vartheta.
\end{align*}
\end{definition}

\begin{lemma}
\label{LH}
The following hold.
\begin{enumerate}[label=\emph{(\roman*)}]
\item  $\| \rho^2 H : C^k(\cat, \chi,  (\tau+ \rho^2)(|\zz|+ \tau)  ) \| \le C(k)$. 
\item $\| \Hcal_1(\rho^2 H) :C^k\big(\cat, \chi, \rho^2(|\zz| + \tau) \big) \| \le C(k)$. 
\end{enumerate}
\end{lemma}
\begin{proof}
(i) follows by combining the estimates in \ref{Ltensest} with \ref{LHpertsmall}, 
where we note in particular that $\Hcirc =0$ because $\gpeuc$ is Euclidean 
and that $\ddiv_{\cat, \chi}\beta =  \rho^2 \ddiv_{\cat, \gpeuc} \beta$ and $\tr_{\cat, \chi} \alphatilde = \rho^2 \tr_{\cat, \gpeuc} \alphatilde$.   

To prove the estimate in (ii) we will need a more refined expansion for $\rho^2 H$: from \ref{Lhasmp}, and the estimates in \ref{Ltensest}, note first that
\begin{equation*}
\| \Hcal_1( \rho^2 H - \ddiv_{\cat, \chi} \beta + \textstyle{\frac{1}{2}} \tr_{\cat, \chi} \alphatilde +\rho^2 \textstyle{\frac{1}{2}} 
\langle \alphatilde, \alphahat\rangle_{\gpeuc} ) 
: C^k\big( \cat , \chi, \rho^2( |\zz| + \tau) \big) \| \le C(k), 
\end{equation*}
so it suffices to show that the estimate in (ii) holds when $\rho^2 H$ is replaced by $\ddiv_{\cat, \chi} \beta$, $\tr_{\cat, \chi} \alphatilde$, 
or $\rho^2\langle \alphatilde, \alphahat\rangle_{\gpeuc}$.  
The estimate for $\rho^2 \langle \alphatilde, \alphahat\rangle_{\gpeuc}$ follows by combining the estimates on $\alpha$ and $\alphatilde$ in \ref{Ltensest}(i) and (ii).  
The estimates on $\tr_{\cat, \chi} \alphatilde$ and $\ddiv_{\cat, \chi} \beta$ follow from \ref{Ltensest}(iii) and (v) 
using that $\Pi^*_\Sigma A^\Sigma(\evec_\rr, \evec_\rr)$ is orthogonal to first harmonics up to higher order terms involving $|\kappa|$ and $\rho$. 
\end{proof}

The following lemma relates estimates on  $H$, which will be crucial for our main applications, to estimates on $ \rho^2 H$, which are easy to compute due to the geometry of $\cat$. 
\begin{lemma}
\label{Erhow}
Given $f\in C^k(\cat)$ and $n\in \mathbb{Z}$, we have
$\phantom{k}$ \hfill $\| \rho^n f : C^k(\cat, \chi)\| \Sim_{C(k, n)} \| f: C^k(\cat, \chi, \rho^{-n})\|.$ \hfill $\phantom{k}$ 
\end{lemma}

\begin{proof}
Using \eqref{E:norm:mult} with $u_1 = \rho^n, u_2 = f, f_1 = \rho^n$, and $f_2 = \rho^{-n}$, we estimate 
\begin{align*} 
\| \rho^n f :C^k(\cat, \chi)\|& \le C(k)\| \rho^{n} : C^k (\cat, \chi, \rho^{n})\|  \| f: C^k( \cat, \chi, \rho^{-n}) \| \\
&\le C(k, n) \| \rho^{\frac{n}{|n|}} : C^k(\cat, \chi, \rho^{\frac{n}{|n|}})\|^{|n|}   \| f: C^k( \cat, \chi, \rho^{-n}) \| \\
&\le C(k, n)  \| f: C^k( \cat, \chi, \rho^{-n}) \| ,
\end{align*}
where in the second inequality we have used \eqref{E:norm:mult} iteratively and in the third we have used \ref{Lzinit}(i). 

Using \eqref{E:norm:mult} with $u_1 = \rho^{-n}, u_2 = \rho^n f, f_1= \rho^{-n}$, and $f_2 = 1$, we estimate in an analogous way
\begin{align*}
\| f: C^k (\cat, \chi, \rho^{-n})\| &\le C(k ) \| \rho^{-n}: C^k (\cat, \chi, \rho^{-n})\| \| \rho^n f: C^k (\tildecat, \chi)\| \\ 
&\le C(k, n)  \| \rho^n f: C^k (\cat, \chi)\| .
\end{align*}
Combining these estimates completes the proof.
\end{proof}

\begin{cor}
\label{LH2}
\begin{enumerate}[label=\emph{(\roman*)}]
\item $\| H : C^k(\cat, \chi ,\tau \rho^{-2} +1 )\| \le C(k)  \tau | \log \tau |$. 
\item $\| \Hcal_1 H : C^k(\cat, \chi)\| \le C(k) \tau | \log \tau |$. 
\end{enumerate}
\end{cor}
\begin{proof}
This follows from combining \ref{LH} with \ref{Erhow}  and using that $|\zz| \le C \tau | \log \tau |$ on $\cat$.
\end{proof}

\subsection*{Area of catenoidal bridges in $N$} 
\nopagebreak

\begin{lemma}[Area of truncated {{$\cat[p, \tau, \kappaunder] \subset N$}}]  
\label{Lcatarea}
Let $\cat = \cat[p, \tau, \kappaunder]$ be as in \ref{DbM}.  
Fix $\rtop = \tau^{3/4}$. The area
 $|\cat(\rtop)|$  of $\cat(\rtop): =  \cat \cap \Pi^{-1}_\Sigma (D^\Sigma_p(\rtop))$ 
 satisfies the following, where
 $\phicat^\pm = \phicat^\pm[\tau, \kappaunder]$ is as in \ref{dphicattilt}.
\begin{equation*}
|\cat(\rtop)| = 2 |D^\Sigma_p(\rtop)|- \pi \tau^2  
+ \frac{1}{2}\int_{\partial D^\Sigma_p(\rtop)}\left( \phicat^+ \frac{\partial \phicat^+}{\partial \eta} 
+ \phicat^-\frac{\partial \phicat^-}{\partial \eta}\right)dl
+
O( \tau^{5/2} | \log \tau|).
\end{equation*}
\end{lemma}

\begin{proof}
Since $\cat$ is $2$-dimensional, the determinant of the induced metric $g = \gcir +\alpha$ satisfies
\begin{align*}
\det g = \det \gcir ( 1 + \tr_{\gcir} \alpha+ \det \alpha^\sharp). 
\end{align*} 
Using that $\det \gcir = r(s)^4$ in the coordinates of \eqref{Ecatenoid}, that $\rtop = \tau^{3/4}$, and \ref{Ltensest} to estimate $\det \alpha^\sharp$ and $ \tr_{\gcir} \alpha$, it follows that
$\sqrt{\det g} = \sqrt{\det \gcir}(1+ O(\tau | \log \tau|))$ on $\cat(\rtop)$ and consequently that
\begin{align*}
|\cat(\rtop) |_g = |\cat(\rtop)|_{\gcir} +O(\tau^{5/2} |\log \tau|),
\end{align*}
where we have used (recall \ref{Lcata1}) that $|\cat(\rtop)|_{\gcir} = O(\rtop^2)$ to estimate the error term.

As a consequence of \ref{Lh}(iv), we have that $|D^\Sigma_p(\rtop)| = |D_0^{T_p \Sigma}(\rtop)| + O(\rtop^4)$, and that the length elements $dl_{g}$ and $dl_{\gcir}$ on $\partial D^{\Sigma}_p( \rtop)$ with respect to $g$ and $\gcir$ satisfy $dl_{g} = (1+ O(\rtop^3))dl_{\gcir}$.
The conclusion follows by combining \ref{Lcata1} with the preceding estimates.
\end{proof}


\section{LD solutions and initial surfaces}
\label{S:LD}

\subsection*{Green's functions and LD solutions}
\nopagebreak

\begin{definition}[Green's functions] 
\label{dggen}
Given a Riemannian surface $(\Sigma, g)$, $V\in C^\infty(\Sigma)$, and $p\in\Sigma$,  
we call 
$G_p$ 
a \emph{Green's function for $\Delta_g + V$ on $\Omega$ with singularity at $p$} 
if it satisfies the following.
\begin{enumerate}[label=\emph{(\roman*)}]
\item 
$G_p \in C^\infty \left( \Omega \setminus \{p\} \right)$   
and $\left( \Delta_g + V\right) G_p= 0$ on $\Omega\setminus\{p\}$.  
\item  
$G_p-\log \dbold^g_p$ is bounded on 
some deleted neighborhood of $p$ in $\Omega$. 
\end{enumerate}
\end{definition}

Clearly if $G_p$ is as in \ref{dggen} and $\Omega''\subset \Omega$ is also a neighborhood of $p$, then 
$\left. \phantom{^I} \!\!\!   \! G_p \right|_{\Omega''}$ is 
also a Green's function for $\Delta_g + V$ on $\Omega''$ with singularity at $p$. 

\begin{lemma}
\label{Lsing}
If 
$G_p \in C^\infty \left( \Omega \setminus \{p\} \right)$ and 
$\Gtilde_p \in C^\infty \left( \Omega \setminus \{p\} \right)$ 
are both Green's functions for $\Delta_g + V$ on $\Omega$ with singularity at $p$ as in Definition \ref{dggen},  
then $G_p - \Gtilde_p$ has a unique extension in $C^\infty(\Omega)$.
\end{lemma}

\begin{proof}
Clearly $G_p - \Gtilde_p$ is a smooth and bounded 
solution of the Partial Differential Equation 
on $\Omega\setminus \{p\} $ by the definitions. 
By standard regularity theory then the lemma follows \cite{bers}. 
\end{proof}

\begin{lemma}
\label{Lgreenlog}
\label{Ldsolve}
Given $(\Sigma, g)$, $V$, and $p\in\Sigma$ as in \ref{dggen}  
there exists $\deltaunder>0$ 
such that 
$\Delta_g +V$ on $D^\Sigma_p(\deltaunderp)$ satisfies the following $\forall\deltaunderp\in(0,\deltaunder]$  
where $ r: = \dbold^\Sigma_p$.
\begin{enumerate}[label=\emph{(\roman*)}]
\item 
There is a Green's function $G_p$ for $\Delta_g+V$ on $D^\Sigma_p(\deltaunderp)$ with singularity at $p$ satisfying 
\begin{align}
\label{Egreenlog} 
\big\| G_p - \log r : C^k\big( D^\Sigma_p (\deltaunderp) \setminus \{ p \}  , r, g, r^2| \log r| \big)\big\| \le C(k). 
\end{align}
\item 
For any given $u_\partial \in C^{2,\beta }( \partial D^\Sigma_p(\deltaunderp) \, )$ 
there is a unique solution $u\in C^{2,\beta }( D^\Sigma_p(\deltaunderp) \, )$ 
to the Dirichlet problem 
\\
$\phantom{a}$ 
\hfill 
$ 
(\Delta_g +V)u=0 \text{ on } D^\Sigma_p(\deltaunderp), 
$ 
\hfill 
$
\qquad 
u=u_\partial \text{ on } \partial D^\Sigma_p(\deltaunderp). 
$  
\hfill 
$\phantom{a}$ 
\end{enumerate} 
\end{lemma}

\begin{proof}
(i) is standard, see for example \cite{bers}. 
(ii) follows easily by scaling to unit size and treating $\Delta_g +V$ as a small perturbation of the flat Laplacian. 
\end{proof}

\begin{corollary}
\label{cor:green}
If $G_p$ and $\Gtilde_p$ are both Green's functions as in \ref{Lgreenlog}(i)  
satisfying \ref{Egreenlog} for some $\delta>0$, 
then the unique extension $G\in C^\infty(D^\Sigma_p (\delta) \, )$ of 
$G_p- \Gtilde_p$ (recall \ref{Lsing})  
satisfies $G(p) = 0$ and $d_pG=0$. 
\end{corollary}

\begin{proof} 
By subtracting the two versions of \eqref{Egreenlog} we conclude that 
$|G_p- \Gtilde_p|\le C r^2| \log r| $,  
which implies the result by \ref{Lsing}.  
\end{proof} 

\begin{definition}[LD solutions]
\label{dLD}
We call $\varphi$ a \emph{linearized doubling (LD) solution on $\Sigma$}     
when there exists 
a finite set $L\subset \Sigma$, 
called the \emph{singular set of $\varphi$}, 
and a function 
$\taubold: L \rightarrow \R\setminus\{0\}$,  
called the \emph{configuration of $\varphi$}, 
satisfying the following,   
where $\tau_p$ denotes the value of $\taubold$ at $p\in L$.  
\begin{enumerate}[label=\emph{(\roman*)}]
\item 
$
\varphi \in C^\infty (\, \Sigma \setminus L \, ) 
$ 
and       
$\Lcal_\Sigma\varphi=0$ on $\Sigma\setminus L$ (recall \ref{NT}\ref{N:A}).  
\item 
$\forall p \in L$ 
the function   
$\varphi - \tau_p \log \dbold^g_p$ is bounded on some deleted neighborhood of $p$ in $\Sigma$. 
\end{enumerate}
\end{definition}

In other words LD solutions are Green's functions for $\Lcal_\Sigma$ (recall \ref{dggen}) with multiple singularities of various strengths;  
we call them solutions because they satisfy the linearized equation as in \ref{dLD}(i). 

\begin{remark} 
\label{R:LD}
In some constructions we will need to modify the definition of LD solutions in \ref{dLD} either by imposing boundary or decay conditions or by relaxing the requirement 
$\Lcal_\Sigma\varphi=0$ on $\Sigma\setminus L$. 
Note that although we usually require $\forall p\in L$ $\tau_p>0$,  
in the definition we allow any $\tau_p \in \R\setminus\{0\}$ to ensure (by \ref{Lsing}) that the LD solutions form a vector space, 
and those with singular set a subset of a given finite set $L'\subset\Sigma$, a subspace. 
\qed
\end{remark} 

\subsection*{Mismatch and obstruction spaces} 
\nopagebreak

\begin{convention}[The constants $\delta_p$]
\label{con:L}
Given $L$ as in \ref{dLD} 
we assume that for each  $p\in L$ a constant $\delta_p>0$ has been chosen so that the following are satisfied. 
\begin{enumerate}[label={(\roman*)}]
\item 
$\forall p,p'\in L$ with $p\ne p'$ we have 
$D^\Sigma_p(9\delta_p) \cap D^\Sigma_{p'}(9\delta_{p'}) = \emptyset $. 
\item 
$\forall p\in L$ and $\forall\deltaunderp\in(0, 3 \delta_p]$, $\Lcal_\Sigma$ on $D^\Sigma_p(\deltaunderp)$ satisfies \ref{Lgreenlog}(i)-(ii). 
\item 
$\forall p\in L$, $\delta_p < \inj^{\Sigma,N,g}_p$ (recall \ref{dexp}). 
\end{enumerate} 
\end{convention}

\begin{lemma} 
\label{Rmismatch}
Given $\varphi$, $L$, and $\taubold$ as in \ref{dLD} and assuming \ref{con:L}, 
$\forall p\in L$ 
there exist $\varphihat_p\in C^\infty \! \left( D^\Sigma_p( 2 \delta_p) \right)$ and 
a Green's function $G_p$ for $\Lcal_\Sigma$ on $D^\Sigma_p( 2 \delta_p)$ with singularity at $p$ 
satisfying \ref{Egreenlog} with $ 2 \delta_p$ instead of $\deltaunderp$,  
such that the following hold (recall \ref{DVcal}). 
\begin{enumerate}[label=\emph{(\roman*)}]
\item 
$\varphi=\varphihat_p + \tau_p \, G_p \, $ on 
$D^\Sigma_p( 2 \delta_p) \setminus \{ p \} $.  
\item 
$\Ecalunder_p \varphihat_p : T_p\Sigma\to\R$ is independent of the choices of $\delta_p$ and $G_p$ and depends only on $\varphi$. 
\item 
$\varphi \circ \exp^\Sigma_p(v)= \tau_p \log |v| + \Ecalunder_p \varphihat_p (v) + O(|v|^2\log|v|)$ for small $v\in T_p\Sigma$.  
\end{enumerate} 
\end{lemma} 

\begin{proof}
The existence of $G_p$ follows from 
\ref{con:L}(ii) and (i) serves then as the definition of $\varphihat_p$. 
(ii) follows then from \ref{cor:green} and (iii) from a Taylor expansion of $\varphihat_p$ combined with \eqref{Egreenlog}. 
\end{proof} 

\begin{definition}[Mismatch of LD solutions {\cite[Definition 3.3]{kap}}] 
\label{Dmismatch}
Given $\varphi$, $L$, and $\taubold$ as in \ref{dLD} with $\tau_p>0$ $\forall p\in L$, 
we define the \emph{mismatch of $\varphi$}, $\Mcal_L \varphi \in \val[L]$ (recall \ref{DVcal}), 
by
$\Mcal_L \varphi : = \bigoplus_{p\in L} \Mcal_p \varphi$,  where 
$\Mcal_p \varphi \in \val[p]$ is defined (recall \ref{Rmismatch} and \ref{DVcal}) 
by requesting that for small $v\in T_p\Sigma$  
\begin{equation*} 
\begin{gathered}
\varphi \circ \exp^\Sigma_p(v) = \,\, \tau_p \log ( 2 |v| / \tau_p )  + \left( \Mcal_p \varphi \right) (v) + O(|v|^2\log|v|),  
\\
\text{or equivalently by \ref{Rmismatch}(iii)} \qquad  
\Mcal_p \varphi : = \,\, \Ecalunder_p \varphihat_p + \tau_p \log (\tau_p/2) . 
\end{gathered}
\end{equation*} 
\end{definition} 

\begin{assumption}[Obstruction spaces] 
\label{aK}
Given $L$ as in \ref{dLD} we assume we have chosen a subspace $\skernelv[L] =  \bigoplus_{p\in L}\skernelv[p] \subset C^\infty(\Sigma)$ satisfying the following, 
where the map 
$\Ecal_L : \skernelv[L] \rightarrow \val[L]$ (recall \ref{DVcal}) is defined by  
$\Ecal_L(v) := \bigoplus_{p\in L} \Ecalunder_p v$. 
\begin{enumerate}[label = {(\roman*)}]
\item The functions in $\skernelv[p]$ are supported on $D^\Sigma_p(4 \delta_p)$.
\item The functions in $\skernel[p]$, where $\skernel[p]: = \Lcal_\Sigma \skernelv[p]$, are supported on $D^\Sigma_p(4 \delta_p) \setminus D^{\Sigma}_p(\delta_p/4)$. 
\item 
$\Ecal_L : \skernelv[L] \rightarrow \val[L]$ is a linear isomorphism.  
\item  $\left\| \Ecal^{-1}_L \right\| \le C \delta_{\min}^{-2-\beta}$, 
where 
$\delta_{\min}:=\min_{p\in L}\delta_p$ 
and 
$\left\| \Ecal^{-1}_L \right\|$ is the operator norm of $\Ecal^{-1}_L  : \val[L] \rightarrow  \skernelv[L]$ 
with respect to the $C^{2, \beta}\left( \Sigma, g\right)$ norm  on the target and the maximum  norm on the domain subject to the metric $g$ on $\Sigma$. 
\item 
$\forall \kappaunderbold = (\kappaunder_p)_{p\in L} \in \val[L]$ we have for each $p\in L$
\[ 
\| \kappaunder_p \circ (\exp^\Sigma_p)^{-1} - \Ecal^{-1}_L\kappaunder_p : C^k( D^\Sigma_p(\delta_p), \dbold^\Sigma_p, g, (\dbold^\Sigma_p)^2)\| \le 
C(k) \, |\kappaunder_p| . 
\] 
\end{enumerate}
\end{assumption}

\begin{remark}
\label{R:p} 
Given $L$ as in \ref{dLD} and constants $\delta_p$ as in \ref{con:L}, 
a possible definition of 
spaces $\skernelv[p]$ satisfying \ref{aK} is by 
$$
\skernelv[p] := \operatorname{span}\left( \left\{ \Psibold[\delta_p, 2\delta_p; \dbold^\Sigma_p]( u_i, 0)\right\}_{i=1}^3 \right), 
$$
where $u_i, i=1, 2,3$ are solutions of the Dirichlet problem $\Lcal_\Sigma u_i = 0$ on $D^\Sigma_p (3\delta_p)$,  
with corresponding boundary data $u_1 = \sin \theta, u_2 = \cos \theta, u_3 = 1$ on $\partial D^\Sigma_p(3\delta_p)$, 
where $\theta$ is a local angular coordinate in geodesic polar coordinates for $D^\Sigma_p(\delta_p)$.  
In the constructions in this paper, we will use choices (see \ref{dkernelsym}) of $\skernelv[L]$ and $\skernel[L]$ adapted to symmetries of the problems.
\qed
\end{remark}

\subsection*{Mismatch and conformal change of metric}  
\nopagebreak

We prove two lemmas now which will be useful in Part II. 

\begin{lemma}[Distance expansion under conformal change of metric] 
\label{LGdiff}
Consider a metric $\ghat = e^{-2\conf}g$ on $\Sigma$, where $\conf \in C^\infty(\Sigma)$.  
For each $p\in \Sigma$ and 
 $q$ in some open neighborhood of $p$ in $\Sigma$, 
\begin{align*}
\left|\log \dbold^{\ghat}_p (q) - \log \dbold^{g}_p (q) + \conf(p)+ \frac{1}{2}d_p\conf(\, (\exp_p^g)^{-1} (q) \,)  \right| \leq 
C \, (\dbold^{g}_p(q))^2.
\end{align*}
\end{lemma}

\begin{proof}
In this proof, denote $r = \dbold^g_p(q)$ and $\rhat = \dbold^{\ghat}_p(q)$, where $q\in \Sigma$ is close to $p$.  
Let $\gamma$ and $\gammahat$ be respectively the $g$- and $\ghat$-geodesics joining $p$ to $q$.  
We have
\begin{equation*}
\begin{gathered}
\rhat \leq \int_{0}^{r} e^{-\conf(\gamma(t))}dt 
= e^{-\conf(p)}r \left(1-\frac{1}{2}d_p\conf(\gamma'(0))r + O(r^2)\right),  
\\ 
r \leq \int_{0}^{\rhat} e^{\conf(\gammahat(t))}dt
= e^{\conf(p)}\rhat\left(1+ \frac{1}{2} d_p \conf(\gammahat'(0))\rhat + O(\rhat^2)\right).
\end{gathered}
\end{equation*}
This implies that $\rhat/r = e^{-\conf(p)}+ O(r)$ and consequently that $|r \gamma'(0)-\rhat \, \gammahat'(0)|< C r^2$. 
We complete the proof by taking logarithms of both inequalities above and expanding. 
\end{proof}

\begin{lemma}[Mismatch expansion in a conformal metric]
\label{Lmm}
For given $\conf \in C^\infty(\Sigma)$ and $\varphi$ as in \ref{Dmismatch} we have for $\ghat := e^{-2\conf} g$ and $\forall p\in L$ and small $w\in T_p\Sigma$  
\begin{equation*}
\varphi \circ \exp^{\Sigma, \ghat}_p(w) = 
\tau_p \log (2 |w|_{\ghat}/\tau_p) + (\Mcal_p \varphi)(w) 
+ \tau_p \conf(p) + \tau_p d_p \conf (w)/2 + O(|w|^2_{\ghat} \log |w|_{\ghat}). 
\end{equation*}
\end{lemma}

\begin{proof}
By \ref{Dmismatch} we have for small $w\in T_p \Sigma$ that 
\begin{equation*}
\varphi \circ \exp^{\Sigma, \ghat}_p(w) 
= \tau_p \log (2 |w|_{\ghat}/ \tau_p) +\tau_p \log(|v|_g / |w|_{\ghat}) 
+ (\Mcal_p \varphi)(v) + O(|v|^2_g \log |v|_g),
\end{equation*}
where 
$v\in T_p\Sigma$ denotes the unique small vector satisfying 
$\exp^{\Sigma, g}_p(v) = \exp^{\Sigma, \ghat}_p(w)$, or equivalently 
$v = (\exp^{\Sigma, g}_p)^{-1} \circ \exp^{\Sigma, \ghat}_p(w)$.  
The proof is completed then by using \ref{LGdiff} and that $v = w+ O(|w|^2_{\ghat})$.
\end{proof}

\subsection*{The initial surfaces and their regions}
\nopagebreak

Each initial surface we construct depends not only on an LD solution $\varphi$ as in \cite{kap}, 
but also on additional parameters $\kappaunderbold \in \val[L]$ controlling the elevation and tilt of the catenoidal bridges in the vicinity of 
$\varphi$'s singular set $L$.
We list now the conditions imposed on these data. 

\begin{convention}[Uniformity of LD solutions] 
\label{con:one}
We assume given $\varphi$, $L$, and $\taubold$ as in \ref{dLD} with $\tau_p>0$ $\forall p\in L$, 
and $\delta_p$'s as in \ref{con:L}, 
satisfying the following with $\alpha$ as in \ref{con:alpha} and 
\begin{equation} 
\label{Ddeltaprime}
\begin{gathered} 
\tau_{\min}:=\min_{p\in L}\tau_p, 
\\ 
\delta_p':=\tau_p^\gammagl \quad (\forall p\in L),  
\end{gathered} 
\qquad  \qquad  
\begin{gathered} 
\tau_{\max}:=\max_{p\in L}\tau_p, 
\\ 
\delta_{\min}':=\min_{p\in L}\delta_p'=\tau_{\min}^\gammagl. 
\end{gathered} 
\end{equation} 
\begin{enumerate}[label=(\roman*)]
\item 
\label{con:one:i} 
\ref{con:L} holds and---in accordance with \ref{con:alpha}---$\tau_{\max}$ is as small as needed in terms of $\alpha$ only.  
\item
$\forall p\in L$ we have $9 \delta_p' = 9 \tau_p^\gammagl < \tau_p^{\alpha/100} < \delta_p \, $. 
\item $\tau_{\max}\le \tau_{\min}^{1-\gammagl/100}$.  
\item $\forall p\in L$ we have $(\delta_p)^{-2} \| \, \varphi : C^{2,\beta}(\, \partial D^\Sigma_p(\delta_p)    ,\, g\,)\,\| \le\tau_p^{1-\gammagl/9}$. 
\item $\| \varphi:C^{3,\beta} ( \, \Sigma  \setminus\disjun_{q\in L}D^\Sigma_q(\delta_q')    \, , \, g \, ) \, \|
\le
\tau_{\min}^{8/9} \, $.
\item
On $\Sigma\setminus\disjun_{q\in L}D^\Sigma_q(\delta_q') $ we have $\tau_{\max}^{1+\alpha/5} \le \varphi$.  
\end{enumerate}
\end{convention}

\begin{definition}[Initial surfaces] 
\label{Dinit}
Given 
$\varphi$, $L$, $\taubold$ and $\delta_p$'s as in \ref{con:one},  
and 
$\kappaunderbold = (\kappaunder_p)_{p\in L} \in \val[L]$ satisfying 
(in accordance with \ref{Akappa})   
\begin{equation} 
\label{dalpha} 
\forall p\in L 
\qquad  
|\kappaunder_p| < \tau_p^{1+ \alpha/6},   
\end{equation} 
we define the smooth initial surface 
(recall \ref{NT}\ref{dgraph}) 
$$
M = M[\varphi, \kappaunderbold]:= 
\graph^N_{\Omega}\big( \varphigl_+\, \big) \bigcup \graph^N_{\Omega}\big(-\varphigl_-\, \big) \bigcup 
\bigsqcup_{p\in L} \cat[p, \tau_p, \kappaunder_p ], 
$$ 
where 
$\Omega : =  \Sigma \setminus \disjun_{p\in L} D^\Sigma_p( 9 \tau_p)$  
and the functions 
$\varphigl_{\pm} = \varphigl_{\pm} [ \varphi, \kappaunderbold] : \Omega \rightarrow \R$ 
are defined as follows. 
\begin{enumerate}[label = \emph{(\roman*)}]
\item $\forall p\in L$ we have 
$\varphigl_{\pm} : = \Psibold [ 2 \delta'_p, 3 \delta'_p ; \dbold^\Sigma_p]\left( \phicatpm[\tau_p, \kappaunder_p] \circ 
( \, \exp^\Sigma_p\right)^{-1} , \, \varphi 
+\vunder_\pm  
)$ 
on $D^\Sigma_p(3 \delta'_p) \setminus D^\Sigma_p(9 \tau_p)$, 
where $\vunder_\pm := - \Ecal_L^{-1} \Mcal_L \varphi \pm \Ecal^{-1}_L \kappaunderbold \in \skernelv[L]$.  
\item On $\Sigma \setminus  \bigsqcup_{p\in L} D^\Sigma_p(3 \delta'_p)$ we have 
$\varphigl_{\pm} : =   \varphi +\vunder_\pm  $.
\end{enumerate}
\end{definition}

\begin{lemma}[The gluing region]
\label{Lgluingreg}
For $M = M[\varphi, \kappaunderbold]$ as in \ref{Dinit} and $\forall p \in L$ the following hold. 
\begin{enumerate}[label = \emph{(\roman*)}]
\item 
$\left\| \varphi^{gl}_{\pm} - \tau_p \log  \dbold^\Sigma_p  :  
C^{3, \beta}\left( D^\Sigma_p(4 \delta'_p)\setminus D^\Sigma_p (\delta'_p), (\delta'_p)^{-2} g\right) \right\| \le  \tau_p^{1+ \frac{15}{8}\alpha}$.
\item 
$\left\|  \varphi^{gl}_{\pm}  :  C^{3, \beta}\left( D^\Sigma_p(4 \delta'_p)\setminus D^\Sigma_p (\delta'_p), (\delta'_p)^{-2} g\right) \right\| 
\le C \tau_p  |\log \tau_p|$.
\item 
\label{itemH} 
$\left\| (\delta'_p)^2 H'_{\pm}: C^{0, \beta}\left( D_p^\Sigma(3\delta'_p) \setminus D_p^{\Sigma}(2 \delta'_p), (\delta'_p)^{-2} g\right) \right\| 
\le \tau_p^{1+\frac{15}{8}\alpha}$, where $H'_{\pm}$ denotes the pushforward of the mean curvature of  the graph of $\pm \varphigl_\pm$ to $\Sigma$ by $\PiSig$.  
\end{enumerate}
\end{lemma}

\begin{proof}
We have for each $p\in L$ 
on $\Omega_p := D_p^\Sigma(4 \delta'_p) \setminus D^\Sigma_p(\delta'_p)$, 
(recall \ref{Dinit})
\begin{equation} 
\label{ephiuu}
\begin{aligned}
\varphigl_{\pm} &= \tau_p G_p  - \tau_p \log \frac{\tau_p}{2} \Ecal^{-1}_L \delta^L_p \pm \Ecal^{-1}_L \kappaunderbold 
+ \Psibold [ 2 \delta'_p, 3 \delta'_p; \dbold^\Sigma_p ] ( \varphiunder_{\pm}, \varphiover_{\pm}), 
\\ 
\text{where} \quad 
\varphiunder_{\pm} &:= \phicatpm[\tau_p, \kappaunder_p]\circ \left( \exp^\Sigma_p\right)^{-1}- \tau_p G_p +   
\tau_p  \log \frac{\tau_p}{2}\Ecal^{-1}_L \delta^L_p \mp \Ecal^{-1}_L \kappaunderbold,   
\\
\varphiover_{\pm} &:= \varphi - \tau_p G_p+   \tau_p  \log \frac{\tau_p}{2}\Ecal^{-1}_L \delta^L_p  - \Ecal^{-1}_L \Mcal_L \varphi , 
\end{aligned}
\end{equation} 
where $\delta^L_p \in \val[L]$ is defined by $\delta^L_p := ( \delta_{pq} )_{q\in L}$ with $\delta_{pq}$ the Kronecker delta. 
By scaling the ambient metric to $\gtilde' : = (\delta'_p)^{-2} g$ and expanding in linear and higher order terms we have
\begin{align*}
(\delta'_p)^2 H'_\pm = (\delta'_p)^2 \Lcal_\Sigma \varphigl_\pm + \delta'_p \widetilde{Q}_{(\delta'_p)^{-1} \varphigl_\pm}.
\end{align*}
Note that on $\Omega_p$ we have
\begin{align*}
\varphigl_\pm - \phicatpm[\tau_p, \kappaunder_p ] &= \Psibold \left[ 2 \delta'_p, 3\delta'_p; \dbold^\Sigma_p\right]\big( 0, \varphiover_\pm - \varphiunder_\pm\big), 
\\ 
\Lcal_\Sigma \varphigl_\pm &= \Lcal_{\Sigma} \Psibold \left[ 2 \delta'_p, 3\delta'_p; \dbold^\Sigma_p\right]\big( \varphiunder_\pm, \varphiover_\pm\big).
\end{align*}
Using these, we have
\begin{equation*}
\begin{aligned}
\| \varphigl_{\pm}\| &\le  C\left( \tau_p | \log \tau_p | + \| \varphiunder_{\pm} \| + \| \varphiover_{\pm} \| \right),\\
\| \varphigl_{\pm} - \tau_p  \log  \dbold^\Sigma_p  \| &\le C\left( \| \varphiunder_\pm\| +  \|\varphiover_\pm \| \right),\\
\left\| (\delta'_p)^2 \Lcal_\Sigma \varphigl_\pm : C^{0, \beta}\left(\Omega_p, (\delta'_p)^{-2} g\right) \right\| &\le C \left( \| \varphiunder_\pm\| + \| \varphiover_\pm\|\right),\\
\left\| \delta'_p \widetilde{Q}_{ (\delta'_p )^{-1} \varphigl_\pm} : C^{0, \beta}\left( \Omega_p, (\delta'_p)^{-2} g\right) \right\|
&\le (\delta'_p)^{-1} \| \varphigl_\pm \|^2,
\end{aligned}
\end{equation*}
where in this proof we mean the $C^{3, \beta}\left( \Omega_p, (\delta'_p)^{-2}g \right)$ norm unless specified otherwise. 
We conclude that if $\|\varphigl_\pm\|\le\delta_p'$ (to control the quadratic terms), 
then we have
$$
\left\|
(\delta'_p)^{2}\, H'_\pm
:C^{0,\beta}(\, \Omega_p ,\, (\delta'_p)^{-2} g\,)\,\right\| 
\le 
\,C\, 
( \,
(\delta_p')^{-1}\tau_p^2|\log\tau_p|^2
+
\|\varphiunder_\pm\,\| 
+
\|\varphiover_\pm\,\| 
).
$$
Adding and subtracting $(\kappaunder_p+ \tau_p \log \frac{2\rr}{\tau_p})\circ (\exp^{\Sigma}_p)^{-1}$ 
in \eqref{ephiuu} we have $\varphiunder_\pm = (I)+(II)+(III)+(IV)$, where
\begin{equation*}
\begin{gathered}
(I)\circ \exp^\Sigma_p = \phicatpm[\tau_p, \kappaunder_p] - \tau_p \log  \frac{2 \rr}{\tau_p} \mp \kappaunder_p, \quad
(II) =  \tau_p( \log  \dbold^\Sigma_p - G_p) , \\
(III) =  - \tau_p \log \frac{\tau_p}{2} (1- \Ecal^{-1}_L \delta^L_p ), \quad
(IV) = \pm ( (\kappaunder_p)\circ (\exp^\Sigma_p)^{-1}- \Ecal^{-1}_L \kappaunderbold).
\end{gathered}
\end{equation*}
Using the triangle inequality and estimating (I)-(IV) using \ref{Ltcest}, \ref{Lgreenlog}(i), and \ref{aK}, we have
\begin{align*}
\| \varphiunder_\pm  \| &\le C(|\kappa_p|+ \tau_p)^3\tau_p^{-2\alpha} + C\tau_p^{1+2\alpha} |\log \tau_p|.
\end{align*}

Because $\Lcal_\Sigma \varphiover_{\pm}=0$ on $\Omega_p$ and $\varphiover_{\pm}$ has vanishing value and differential at $p$ 
(recall  \ref{Dmismatch}, \ref{aK} and \ref{ephiuu}), 
it follows from standard linear theory that
\begin{align*}
\left\|\varphiover_{\pm}  \right\|
&\le C( \delta'_p/\delta_p)^2   \left \|  \varphiover_{\pm}   : C^{2,\beta}\left( \partial D^\Sigma_p(\delta_p ), (\delta_p)^{-2} g\right)\right\|.
 \end{align*}
 Using \ref{Ddeltaprime}, \ref{con:one}(ii)  and \ref{con:one}(iv), \ref{con:alpha}, and \ref{aK} to estimate the right hand side, we conclude that
 \begin{align*} 
 \left\| \varphiover_\pm \right\| \le C  \left( \delta'_p\right)^2 \tau_p^{1-\frac{1}{9}\alpha}+ C\tau^{1+2\alpha}_p |\log \tau_p|  \le
 C\tau_p^{1+\frac{17}{9} \alpha}.
 \end{align*}
 Combining with the above we complete the proof.
\end{proof}

\begin{remark}[Smallness of mean curvature] 
\label{R:small} 
Note that the exponent in the right-hand side of \ref{Lgluingreg}\ref{itemH} is close to $1+2\alpha$ and hence $>1+\alpha$ as needed to ensure 
that the correction of the initial surface will be small compared to the size of the LD solution. 
\qed
\end{remark} 

\begin{lemma}
\label{LMemb}
$M$ defined in \ref{Dinit} (assuming \ref{con:one}) is embedded and moreover 
the following hold.
\begin{enumerate}[label = \emph{(\roman*)}]
\item On $\Sigma \setminus \disjun_{p\in L} D^\Sigma_p(\delta'_p )$ we have $\frac{8}{9} \tau_{\max}^{1+ \alpha/5} \leq \varphi_\pm^{gl}$.
\item $\left\| \varphi_{\pm}^{gl} : C^{3, \beta}\left( \Sigma \setminus \disjun_{p\in L} D^\Sigma_p(\delta'_p), g\right)\right\| \le \frac{9}{8} \tau_{\min}^{8/9}$.
\end{enumerate}
\end{lemma}

\begin{proof}
We first prove the estimates (i-ii): 
(i) on $\Sigma \setminus \disjun_{p\in L} D^\Sigma_p(3 \delta'_p)$ follows from \ref{con:one}(vi) and \ref{Dinit}, 
and on $D^\Sigma_p(4\delta'_p) \setminus D^\Sigma_p(\delta'_p)$ for $p\in L$ from \ref{Lgluingreg}(i) and \ref{con:one}(iii).  
(ii) on $\Sigma \setminus \disjun_{p\in L} D^\Sigma_p ( 3 \delta'_p)$ follows from \ref{con:one}(v) and \ref{Dinit}(i), 
and on $D^\Sigma_p(4\delta'_p) \setminus D^\Sigma_p(\delta'_p)$ for $p\in L$ from \ref{Lgluingreg}(ii) and \ref{con:one}(iii).  
Finally, the embeddedness of $M$ follows from (i) and by comparing the rest of $M$ with standard catenoids.
\end{proof}

\begin{definition}[Regions on the initial surfaces] 
\label{D:regions}
\label{DtauK}
We define the following for $L$ and $M$ as in \ref{Dinit} 
and $x\in[0,4]$, where $x$ may be omitted when $x = 0$ (recall \ref{DbM}). 
\begin{subequations}
\label{E:regions}
\begin{align} 
\label{EStildep}
\Stildep_x &:= \Sigma\setminus \disjun_{p\in L} D^\Sigma_p( b \tau_p (1+x) )\subset \Sigma, 
\\ 
\label{ShatL} 
\Shat_x[M] &:= \disjun_{p\in L}\Shat_x[p, \tau_p , \kappaunder_p ] \subset M, 
\\ 
\Kcore_x[M] &:= \disjun_{p\in L} \Kcore_x[p, \tau_p , \kappaunder_p ] \subset \Shat[M] \subset M. 
\end{align}
\end{subequations}
We also define 
$\tau_L   : \Shat [ M]   \to \R$ and $\PiY: \Shat[M] \rightarrow \KM         : = \bigsqcup_{p\in L} \tildecat[p, \tau_p, \kappaunder_p]$ 
by taking 
$\tau_L   :=\tau_p$ and $\PiY = ( \exp^{\Sigma, N, g}_p)^{-1}$ on each $\cat[p, \tau_p, \kappaunder_p]$.  
\end{definition}

Note that $M$ determines $L$ and so the above notation is legitimate. 
Moreover $\forall p\in L$ 
with $\kappa_p=0$ we have 
$\Shat_x[p] = M \cap \PiSig^{-1}( \, {D^\Sigma_p(2\delta'_p /(1+x) \, )} \,)$;  
when $\kappa_p\ne 0$ the two sides differ very little by the smallness of the tilt.

\begin{notation}
\label{Npm}
If $f^+$ and $f^-$ are functions supported on $\Stildep$ (recall \eqref{EStildep}), 
we define $J_M(f^+, f^-)$ to be the function on $M$ supported on $\left( \left. \Pi_\Sigma \right|_M\right)^{-1} \Stildep$ 
defined by $f^+\circ \Pi_\Sigma$ on the graph of $\varphi^{gl}_{+}$ and by $f^- \circ \Pi_\Sigma$ on the graph of $- \varphi^{gl}_-$.
\qed 
\end{notation}

\section[The linearized equation on the initial surfaces]{The linearized equation on the initial surfaces}
\label{S:linearized}

\subsection*{Global norms and the mean curvature on the initial surfaces}
\nopagebreak

In this section we state and prove Proposition \ref{Plinear} where 
we solve with estimates the linearized equation on an initial surface 
$M$ defined as in \ref{Dinit}.
We also provide in \ref{LglobalH} an estimate for the mean curvature in appropriate norm. 
In this subsection we discuss the global norms we use but first 
we introduce Assumption \ref{cLker} which simplifies the analysis and implies also Lemma \ref{Lldexistence}. 

\begin{assumption} 
\label{cLker}
In the rest of Part I of this article we assume \ref{background} holds and furthermore 
the base surface $\Sigma$ (recall \ref{background}) is closed and the kernel of $\Lcal_\Sigma$ is trivial. 
\end{assumption} 

\begin{definition}
\label{D:norm}
For $k\in\N$, $\betahat\in(0,1)$, 
$\gammahat\in\R   $,
and $\Omega$ a domain in $\Sigma$, 
$M$, or $\KM        $, we define
\[ \|u\|_{k,\betahat,\gammahat;\Omega}
:=
\|u:C^{k,\betahat}(\Omega ,\rr,g,\rr^\gammahat)\|,\]
where $\rr: = \dbold^\Sigma_L$ and $g$ is the standard metric on $\Sigma$ when $\Omega \subset \Sigma$, 
$\rr:= \dbold^\Sigma_L  \circ \Pi_\Sigma$  and $g$ is the  metric induced on $M$ by the standard metric on $N$ when $\Omega \subset M$,  
and $\rr  = \rho(\sss):=\tau_L\cosh \sss$ 
(recall \ref{Ecatenoid} and \ref{DtauK}) 
and $g$ is the metric induced by each Euclidean metric $\left. g\right|_p$ on $T_p N$ $\forall p\in L$ when $\Omega \subset \KM        $.
Given also $\gammahat' \in \R$ with $\gammahat - \gammahat' \in [1, 2)$ we define
$f_{\gammahat, \gammahat'} \in C^0( M )$ 
by $f_{\gammahat, \gammahat'}  : = \max( \rr^{\gammahat}, \tau_L^{(1-\alpha)/2} \rr^{\gammahat'}) = \rr^{\gammahat'} \max( \rr^{\gammahat-\gammahat'}, \tau_L^{(1-\alpha)/2} )$ 
(note that  $f_{\gammahat, \gammahat'}= \rr^{\gammahat}$ when $\rr^{\gammahat-\gammahat'} \ge \tau_L^{(1-\alpha)/2}$), 
and for $\Omega\subset M$ 
(recall \ref{dlowharm})  
$$
\| u\|_{k, \betahat, \gammahat, \gammahat'; \Omega} : =
\|u:C^{k,\betahat}(\Omega ,\rr,g, f_{\gammahat, \gammahat'} )\, \| + \| \Hcal_1 u\|_{k, \betahat, \gammahat; \Omega \cap \Shat[M]}.    
$$ 
\end{definition}

\begin{lemma}
\label{L:norms}
\begin{enumerate}[label=\emph{(\roman*)}]
\item If $\tau_{\max}$ is small enough in terms of given $\epsilon>0$, 
$\Omegatilde$ is a domain in $\PiY(\Shat[M])$, 
$\Omega:=\PiY^{-1}(\Omegatilde)\subset \Shat[M] \subset M$, 
$k=0,2$, $\gammahat\in\R$,
and $f\in C^{k,\beta}(\Omegatilde)$, 
then we have 
(recall \ref{Dsimc}):
\begin{align*}
\| \, f\circ \PiY \, \|_{k,\beta,\gammahat;\Omega}
\, \Sim_{1+\epsilon}
\|  f  \|_{k,\beta,\gammahat;\Omegatilde} \,.
\end{align*}

\item If $b$ is large enough in terms of given $\epsilon>0$, 
$\tau_{\max}$ is small enough in terms of $\epsilon$ and $b$,
$\Omega'$ is a domain in $\Stildep = \Sigma \setminus \disjun_{p\in L} D^\Sigma_p( b \tau_p)$ (recall \eqref{EStildep}), 
$\Omega:=\PiSig^{-1}(\Omega')\cap M$, 
$k=0,2$, $\gammahat\in\R$,
and $f\in C^{k,\beta}(\Omega')$, 
then 
\begin{align*}
\| \, f\circ\PiSig \, \|_{k,\beta,\gammahat;\Omega}
\, \Sim_{1+\epsilon} \,
\|f\|_{k,\beta,\gammahat;\Omega'} \,.  
\end{align*}
\end{enumerate}
\end{lemma}

\begin{proof}
To prove (i) it suffices to prove for each $p\in L$ and each $\cat = \cat[p, \tau_p, \kappaunder_p]$ that
\begin{align*}
\| f\circ \PiY  : C^{k, \beta}( \Omega \cap \cat , \rho, \gpeuc) \| \Sim_{1+\epsilon} 
\| f \circ \PiY : C^{k, \beta}(\Omega\cap \cat , \rho, g) \|,
\end{align*}
The induced metric from $g$ on $\cat$  is $g= \gpeuc+\alpha$, and so (i) follows from \ref{Lgcomp} and the estimate on $\alpha$ in \ref{Ltensest}(i) by taking $\tau_{\max}$ small enough. 
To prove (ii) let $q\in \Stildep$ and consider the metric 
$\gtilde_q:= \, (\,  \dbold^\Sigma_L(q)\, )^{-2} \, g $ on $N$, 
where $g$ is the standard metric on $N$.  
In this metric $M$ is locally the union of the graphs of $\pm \varphi^{\pm}_{:q}$ 
where $\varphi^{\pm}_{:q}:= \, (\,  \dbold^\Sigma_L(q)\, )^{-1} \, \varphigl_{\pm}$. 
First suppose that $\dbold^{\Sigma}_p (  q) \le 4 \delta'_p$ for some $p \in L$.  Note that 
\begin{align*}
\left \| \frac{  \log( 2\dbold^\Sigma_p(q)/\tau_p)}{\dbold^\Sigma_p(q)/\tau_p}- \frac{\kappaunder_p}{\dbold^\Sigma_p (q)} : C^{3, \beta}( B'_q, \gtilde_q) \right\| \le  C b^{-1} \log b ,
\end{align*}
where $B'_q = D^{\Sigma, \gtilde_q}_q(1/10)$.  
It follows by combining this with \ref{Ltcest} and \ref{Lgluingreg}, 
and assuming $b$ large enough, 
that
\begin{equation}
\label{Efw}
\|\,\varphi^{\pm}_{:q} \, : \, 
C^{3,\beta}(B'_q, \gtilde_q)\,\|  \:    \le  \:    C  \tau^3_p (\dbold^\Sigma_p(q))^{-3} + C b^{-1} \log b 
\le \:    C b^{-3} + C b^{-1} \log b 
\:    \le \:    C b^{-1} \log b. 
\end{equation} 

On the other hand, if $\dbold^\Sigma_L( q) > 4 \delta'_{\min} $, then by \ref{con:one}(v) 
we have
\begin{equation*}
\|\,\varphi^{\pm}_{:q} \, : \, 
C^{3,\beta}(B'_q, \gtilde_q )\,\|  \le C \tau_{\min}^{8/9}. 
\end{equation*}
By comparing the metrics and appealing to the definitions we complete the proof. 
\end{proof}

\begin{convention}
\label{con:b}
From now on we assume that $b$ (recall \ref{D:regions}) 
is as large as needed in absolute terms.
We also fix some $\beta\in(0,1)$,  $\gamma = \frac{3}{2}$, and $\gamma' = \gamma-1 = \frac{1}{2}$.  
Note that  $1-\frac\gamma2>2\alpha$ and $(1-\alpha)\,(\gamma-1)>2\alpha$.  
We will
suppress the dependence of various constants on $\beta$. 
\qed
\end{convention}

We estimate now the mean curvature in terms of the global norm defined in \ref{D:norm} and discussed in the introduction, 
by using the earlier estimates in \ref{Lgluingreg} and \ref{LH2}.  

\begin{lemma}
\label{LglobalH}
$\| H - J_M(w^+, w^-)\|_{0, \beta, \gamma-2, \gamma'-2; M} \le \tau_{\max}^{1+\alpha/3}$, 
where 
$w^\pm := \Lcal_\Sigma\Ecal^{-1}_L ( \, - \Mcal_L \varphi \pm  \kappaunderbold \,)$. 
\end{lemma}

\begin{proof}
Note that $J_M(w^+, w^-)=0$ on $\Shat[M]$ and by \ref{D:norm} we have 
\begin{multline*}
\| H - J_M(w^+, w^-)\|_{0, \beta, \gamma-2, \gamma'-2; M} = \| \Hcal_1 H \|_{0, \beta, \gamma-2; \Shat[M]}   
\\ 
+ \| H-J_M(w^+, w^-) : C^{0, \beta}( M , \rr, g ,f_{\gamma-2, \gamma'-2})\|. 
\end{multline*} 
By \ref{L:norms}(i) and \ref{LH2} we have
\begin{equation*}
\| \Hcal_1 H \|_{0, \beta, \gamma-2; \Shat[M]} 
\: \le \: C \max_{p\in L} \| \Hcal_1H: C^{0, \beta}( \cat_p, \chi, \rr^{\gamma-2})\| 
\le \:    C \max_{p\in L} \tau^{(2-\gamma)\alpha}_p \tau_p |\log \tau_p| 
\:    \le \:    \tau^{1+\alpha/3}_{\max},
\end{equation*} 
where here $\cat_p = \cat[p, \tau_p, \kappaunder_p]$ and we have used \ref{con:one}(iii). 
To estimate the weighted norm of $ H$, we use \ref{LH2}(i) in conjunction with the piecewise formula for $f_{\gamma-2, \gamma'-2}$ to see 
\begin{align*}
 \| H : C^{0, \beta}(\Shat[M], \rr, g, f_{\gamma-2, \gamma'-2})\| 
\:    \le \:    C\max_{p\in L} \tau^{1+\alpha/2}_p |\log \tau_p|  
\:    \le \:    \tau_{\max}^{1+\alpha/3}.
\end{align*}

Finally, we consider the estimate on the exterior of the gluing region.  
Let $q' \in M\cap \Pi^{-1}_\Sigma(\disjun_{p\in L} D^\Sigma_p(3 \delta'_p))$, define $q: = \Pi_\Sigma q' \in \Sigma\setminus \disjun_{p\in L} D^\Sigma(3\delta'_p)$ 
and consider the metric $\gtilde_q := (\dbold^\Sigma_L(q))^{-2} g$.  
In this metric $M$ is locally the union of the graphs of $\pm \varphi^{\pm}_{:	q}$, where $\varphi^{\pm}_{:q} = (\dbold^\Sigma_L(q))^{-1} \varphi^{gl}_{\pm}$.  
By expanding $H'_+$ and $H'_-$ in linear and higher order terms, we find (recall \ref{Dinit})
\begin{align*}
(\dbold^\Sigma_L(q))^2 H'_{\pm} 
&= (\dbold^\Sigma_L(q))^2 w^\pm + (\dbold^\Sigma_L(q)) \widetilde{Q}_{\varphi^{\pm}_{:q}}.
\end{align*}
We estimate then 
\begin{multline*}
\| (\dbold^\Sigma_L(q))^2 (H'_{\pm} - w^\pm) : C^{0, \beta}(B'_q,\gtilde_q, (\dbold^\Sigma_L(q))^\gamma) \| 
\\ 
\: \le \: \| (\dbold^\Sigma_L(q))  \widetilde{Q}_{\varphi^{\pm}_{:q}}: C^{2, \beta}(B'_q,\gtilde_q, (\dbold^\Sigma_L(q))^\gamma)\| 
\\ 
\: \le \: C \frac{1}{(\dbold^\Sigma_L(q))^{\gamma+1}} \| \varphigl_\pm : C^{3, \beta}(B'_q,\gtilde_q )\|^2 
\: \le \: \tau^{3/2}_{\max} ,
\end{multline*}
where $B'_q: = D^{\gtilde_q}_{q}(1/10)$, 
and we have used \ref{con:one}(v) and \ref{con:alpha}. 
Combining this estimate with  \ref{Lgluingreg}(iii), \ref{LH}, \ref{D:norm}, and \ref{L:norms}(ii) we complete the proof. 
\end{proof}

\begin{lemma}
\label{L:appr}
\begin{enumerate}[label=\emph{(\roman*)}]
\item
If $\gammahat\in\R$, $\tau_{\max}$ is small enough, and $u\in C^{2,\beta}( \, \PiY(\Shat[M]) \, )$, 
then we have 
\begin{equation*} 
\!\!\!\!\!\!\! \! \! \! \! \! \! 
\begin{aligned}
\| \,
\Lcal_M  
\, (\, u \,  \circ
\PiY \, )\, 
-
\, (\, \Lcal_\tildecat  
u \, )  
\circ \PiY
\, \|_{0,\beta,\gammahat-2; \, \Shat[M] } 
\, &\le \,
C\, \tau_{\max}^{2\alpha}\,
\|\,u 
\, \|_{2,\beta,\gammahat; \, \PiY( \Shat[M] ) } 
\, , \\
\| \,
\Lcal_M  
\, (\, u \,  \circ
\PiY \, )\, 
-
\, (\, \Lcal_\tildecat  
u \, )  
\circ \PiY
\, \|_{0,\beta,\gammahat-2; \, \Shat[M] } 
\, &\le \,
C\,
\|\,u 
\, \|_{2,\beta,\gammahat-1; \, \PiY( \Shat[M] ) } .
\,
\end{aligned}
\end{equation*} 

\item 
If $\gammahat\in\R$, $\tau_{\max}$ is small enough, and $u\in C^{2,\beta}(\Stildep\,)$,
then for $\epsilon_1\in[0,1/2]$
we have 
\begin{equation*}
\| \,
\Lcal_M
\,\{ \, u  \circ \PiSig \,\}\,
-
\{\, \Lcal_\Sigma \, u \,\}\, \circ \PiSig 
\, \|_{0,\beta,\gammahat-2 ; \, \Pi_\Sigma^{-1}(\Stildep\,)  } 
\lem 
C \,  
b^{\epsilon_1-1}\, \log b \,\, \tau_{\max}^{\epsilon_1} \, 
\|\, u 
\, \|_{ 2 , \beta , \gammahat + \epsilon_1 ; \,\Stildep  } 
\, . 
\end{equation*} 
\end{enumerate}
\end{lemma}

\begin{proof}
We first prove the first estimate of (i).  By \ref{D:norm}, \ref{L:norms}(i), and the definitions it suffices to prove that 
\begin{align*}
\|\rho^2( \Lcal_M- \Delta_{\gpeuc} -|\Acirc|^2_{\gpeuc}) u\circ \PiY\|_{0, \beta, \gammahat; \Shat[M]}
\le C \tau^{2\alpha}_{\max} \| u \|_{2, \beta, \gammahat; \Omegatilde },
\end{align*}
where $\rho$ is as in \eqref{Ecatenoid} on each $\cat = \cat[p, \tau_p, \kappaunder_p]$, $\Omegatilde =  \PiY ( \Shat[M])$, 
$\Lcal_M : = \Delta_{g} +|A|^2_{g} + \Ric(\nu_\cat, \nu_\cat) $,   
$A$ and $\Ric$ are the second fundamental form on $\cat$ and the Ricci tensor induced by $g$, 
and $\Acirc$ is the second fundamental form on $\cat$ induced by $\gpeuc$.  Recall from \eqref{Ecatmetric} that $\rho^{-2} \gpeuc$ is isometric to the flat metric $\chi$ on $\cyl$ from \ref{Ecyl}, and also that $\rho^2 \Delta_{\gpeuc}$ is the Laplacian with respect to the $\chi$ metric. 

Estimating the difference in the Laplacians using \ref{Llaplace}, we find
\begin{multline*}
\| \rho^2 ( \Delta_{g} - \Delta_{\gpeuc}) u\circ \PiY: C^{0, \beta}(\cat\cap \Shat[M], \chi, \rho^{\gammahat})\| 
\\ 
\lem 
C \| \rho^{-2} \alpha : C^{1, \beta}(\cat, \chi)\| \| u\circ \PiY : C^{2, \beta}(\cat\cap \Shat[M], \chi, \rho^{\gammahat})\| 
\\ 
\lem C \tau^{2\alpha}_{\max} \| u\|_{2, \beta, \gammahat; \Omegatilde},
\end{multline*}
where we have used \ref{Ltensest} to estimate $\alpha$.  
Next observe that 
\begin{multline*}
\| \rho^2 ( | A|^2_{g} - |\Acirc|^2_{\gpeuc} ) u\circ \PiY : C^{0, \beta}(\cat\cap \Shat[M], \chi, \rho^{\gammahat})\| 
\\
\lem \| \rho^2( | A|^2_{g} - |\Acirc|^2_{\gpeuc}) : C^{0, \beta}(\cat, \chi)\| \| u \|_{0, \beta, \gammahat;\Omegatilde} 
\\
\lem C\tau^{2\alpha}_{\max}  \| u \|_{0, \beta,\gammahat; \Omegatilde},
\end{multline*}
where we have estimated $ \rho^2( | A|^2_{g} - |\Acirc|^2_{\gpeuc})$ using \ref{Ljacdif}, 
estimated the tensors using \ref{Ltensest}, and used that  $\Acirc  = \tau_p ( - d\sss^2 + d \theta^2) $ and that $\rho^2|\Acirc|^2_{\gpeuc} = 2\sech^2 \sss$.

Finally, we have the trivial estimate $\| \rho^2 {\Ric}(\nu_\cat, \nu_\cat) u\circ \PiY \|_{0, \beta, \gammahat; \Shat[M]} \le C \tau^{2\alpha}_{\max} \| u \|_{2, \beta, \gammahat; \Omegatilde}$; combined with the preceding and the definitions, this concludes the proof of the first estimate in (i).  The proof of the second estimate is similar, so we omit it. 

We now prove (ii).  In this case we apply the notation and observations in the proof of \ref{L:norms}(ii) 
including \eqref{Efw}. 
We have then using scaling for the left hand side that for $q\in \Stildep$ 
\begin{multline*}
(\,  \dbold^{\Sigma}_L(q)\, )^{2} \, 
\| \,
\Lcal_M
\,\{ \, u  \circ \PiSig \,\}\,
-
\{\, \Lcal_\Sigma \, u \,\}\, \circ \PiSig 
\, : C^{0,\beta} ( \Pi^{-1}_{\Sigma} (B'_q), \, \gtilde_q \, ) \, \| 
\\ 
\, \le 
\, 
C \,  
f_{\mathrm{weight}}(q) \, 
\|\, u 
\, : C^{2,\beta} ( B'_q, \, \gtilde_q \, ) \, \| 
\, ,
\end{multline*}
where here $f_{\mathrm{weight}}(q) = \frac{\log (\dbold^\Sigma_p(q)/ \tau_p)}{\dbold^\Sigma_p(q)/ \tau_p}$ 
if $q \in D^\Sigma_p(3 \delta'_p)$ for some $p\in L$ and $f_{\mathrm{weight}}(q) = \tau^{8/9}_{\min}$ otherwise.  
By the definitions it is enough then to check that 
$\forall q\in \Stildep$ 
we have 
$$
f_{\mathrm{weight}}(q) \, 
(\,  \dbold^\Sigma_L(q)\, )^{\epsilon_1} \, 
\le \, 
C \, b^{\epsilon_1-1}\, \log b \,\, \tau_{\max}^{\epsilon_1}. 
$$
This follows from the definition of $f_{\mathrm{weight}}$ and the observation that $x^{\epsilon_1 -1} \log x$ is decreasing in $x$ for $x\geq b$.  This completes the proof. 
\end{proof}

\subsection*{The definition of $\RMa$}
\nopagebreak

We consider now  the linearized equation modulo $\skernel[L]$ (recall \ref{NT}\ref{N:A}, \ref{aK}, and \ref{Npm}),
\begin{equation}
\label{ELcal}
\Lcal_M u=E+ J_M(w^+_E, w^-_E), 
\end{equation}
with $E\in C^{0,\beta}(M)$ given and $u\in C^{2,\beta}(M)$ and $w^\pm_E \in \skernel[L]$ the unknowns. 
We will construct a linear map
\begin{equation}
\label{ERcalMappr}
\RMa           : 
C^{0,\beta}(M) 
\to 
C^{2,\beta}(M) 
\oplus
\skernel[L] 
\oplus
\skernel[L]
\oplus
C^{0,\beta}(M), 
\end{equation}
where using the notation 
\begin{equation}
\label{EEone}
E_1 := \Lcal_M u_1 - E - J_M(w^{+}_{E, 1}, w^-_{E, 1}), 
\quad 
\text{and}\quad 
\RMa           E = (u_1,w^+_{E,1},w^-_{E,1},E_1), 
\quad \forall E\in C^{0,\beta}(M),  
\end{equation}
we will have that 
$(u_1,w^+_{E,1},w^-_{E,1})$ is an approximate solution of \eqref{ELcal} with approximation error $E_1$ in the sense that the norm of $E_1$ is small 
compared to the norm of $E$. 
The approximate solution will be constructed by combining semi-local approximate solutions.
Before we proceed with the construction we define some cut-off functions we will need.

\begin{definition}
\label{DpsiM}
We define 
$\psi'\in C^{\infty}(\Sigma)$ 
and
$ \psihat \in C^{\infty}(M)$ 
by requesting the following.
\begin{enumerate}[label = \emph{(\roman*)}]
\item
$\psihat$ is supported on
$\Shat[M] \subset M$
and $\psi'$ on 
$\Stildep \subset \Sigma$ (recall \ref{D:regions}).

\item $\psi'=1$ on $\Stildep_1$
and for each $p\in L$ we have 
\begin{equation*}
\begin{aligned}
\psi'=&\, 
\Psibold\left[b \tau_p  , 2 b \tau_p ; \dbold^\Sigma_p  \right]
(0,1)
\quad\text{on}\quad
D^\Sigma_p(\,2b\,\tau_p\,),
\\
\psihat=&\, 
\Psibold\left[2\delta'_p , \delta'_p ; \, \dbold^\Sigma_p \circ \PiSig \,   \right]
(0,1)
\quad\text{on}\quad
\Shat[p].
\end{aligned}
\end{equation*}
\end{enumerate}
\end{definition}

Given 
$E\in C^{0,\beta}(M)$, 
we define 
$E'_{\pm}\in C^{0,\beta}(\Sigma)$ 
by requiring that they are supported on $\Stildep$
and that 
\begin{equation}
\label{Edecom}
J_M(E'_+ , E'_-) = (\psi' \circ \Pi_\Sigma)\,  E .
\end{equation}
Because of \ref{cLker} and \ref{aK}, there are unique $u'_{\pm} \in C^{2, \beta}(\Sigma)$ and $w^{\pm}_{E, 1} \in \skernel[L]$ such that 
\begin{align}
\label{Eup}
\Lcal_\Sigma u'_{\pm} = E'_{\pm} + w^{\pm}_{E,1} \quad \text{on} \quad \Sigma
\quad
\text{and} 
\quad
\forall p\in L \quad \Ecalunder_p u'_\pm =0.
\end{align}
Note that $\Lcal_\Sigma \left( (1- \psi') u'_\pm\right) = [\psi', \Lcal_\Sigma] u'_\pm + (1-\psi')E'_\pm$ is supported on 
$\Kcore_1[M] \setminus \Kcore[M] \subset \Shat[M] \subset M $.  
We define now $\widetilde{E}\in C^{0,\beta}(\KM        )$, 
by requesting that it is supported on 
$\PiY(\Kcore_1[M] ) $ 
and that on $\Kcore_1[M]$ we have 
\begin{equation}
\label{EEpp}
\widetilde{E} \circ \PiY = 
(1- \psi' \circ \Pi_\Sigma)E+
J_M\left( \Lcal_\Sigma \left( (1-\psi') u'_{+}\right) , \Lcal_\Sigma \left( (1-\psi') u'_{-}\right)\right).  
\end{equation}

For $k\in \{0, 2\}$, we introduce a decomposition 
$C^{k, \beta}(\KM        ) = C^{k,\beta}_{\mathrm{low}}(\KM        ) \oplus 
C^{k,\beta}_{\mathrm{high}}(\KM        )$
into subspaces of functions which satisfy the condition that their
restrictions to a parallel circle of a $\tildecat[p, \tau_p, \kappaunder_p]$  
belong or are orthogonal respectively 
to the the span of the constants and the first harmonics on the circle.
We then have
\begin{equation}
\label{EEtildedecom}
\widetilde{E}
=
\widetilde{E}_{\mathrm{low}}
+
\widetilde{E}_{\mathrm{high}},
\end{equation}
with 
$ \widetilde{E}_{\mathrm{low}} \in C^{0,\beta}_{\mathrm{low}}(\KM        )$ 
and
$ \widetilde{E}_{\mathrm{high}} \in C^{0,\beta}_{\mathrm{high}} (\KM        )$ 
supported on $\PiY( \Kcore_1[M] ) \subset \KM         $. 

Let 
$\Lcal_\tildecat $ denote the linearized operator on $\KM        $
(defined in \ref{NT}\ref{N:A}),   
and let 
$ \widetilde{u}_{\mathrm{low}} \in C^{2,\beta}_{\mathrm{low}}(\KM        )$ 
and
$ \widetilde{u}_{\mathrm{high}} \in C^{2,\beta}_{\mathrm{high}} (\KM        )$ 
be solutions of (recall \ref{DtauK})
\begin{equation}
\label{Etildeueq}
\Lcal_\tildecat  \,  \widetilde{u}_{\mathrm{low}}  =  \widetilde{E}_{\mathrm{low}},
\qquad 
\Lcal_\tildecat  \,  \widetilde{u}_{\mathrm{high}}  = \widetilde{E}_{\mathrm{high}},
\end{equation}
determined uniquely as follows.
By separating variables the first equation amounts to 
uncoupled ODE equations which are solved uniquely
by assuming vanishing initial data on the waist of the catenoids.  For the second equation we can as usual change the metric conformally
to $h=\frac12|A|^2g=\nu^*g_{\Sph^2}$, 
and then we can solve uniquely because the inhomogeneous term is clearly
orthogonal to the kernel.  We conclude now the definition of $\RMa          $:

\begin{definition}
\label{DRMappr}
We define $\RMa          $
as in \ref{ERcalMappr} and \ref{EEone} 
by taking 
$\RMa           E =(u_1,w^+_{E, 1},  w^-_{E,1},E_1)$,
where $w^{\pm}_{E,1}$ were defined in \eqref{Eup},
$E_1$ was defined in \eqref{EEone}, 
and
\begin{align*}
u_1
:=
\psihat\, \widetilde{u} \circ
\PiY
+ 
J_M(\psi' u'_+, \psi' u'_-) ,
\quad
\text{where}
\quad
\widetilde{u}  :=  \widetilde{u}_{\mathrm{low}} + \widetilde{u}_{\mathrm{high}}  
\in C^{2,\beta}(\KM        ).
\end{align*}
\end{definition}

\subsection*{The main Proposition}
\nopagebreak

\begin{prop}[cf. {\cite[Proposition 4.17]{kap}}] 
\label{Plinear}
Recall that $M$ is as in \ref{Dinit} and we assume that \ref{aK}, \ref{con:one}, \ref{cLker}, and \ref{con:b} hold. 
A linear map 
$$
\Rcal_M: C^{0,\beta}(M) \to C^{2,\beta}(M) \times \skernel[L] \times \skernel[L] 
$$ 
can be defined then by 
$$
\Rcal_M E
:=
(u,w^+_E,  w^-_E) 
:=
\sum_{n=1}^\infty(u_n, w^+_{E,n}, w^-_{E,n})
\in C^{2,\beta}(M) \times \skernel[L]\times\skernel[L]
$$
for 
$E\in C^{0,\beta}(M)$, 
where the sequence 
$\{(u_n,w^+_{E,n}, w^-_{E,n},E_n)\}_{n\in \N}$
is defined inductively for $n\in \N$ by 
$$
(u_n,w^+_{E,n}, w^-_{E,n},E_n) := - \RMa           E_{n-1}, 
\qquad\quad
E_0:=-E.
$$
Moreover the following hold.
\begin{enumerate}[label=\emph{(\roman*)}]
\item $\Lcal_M u = E + J_M(w^+_E, w^-_E)$.
\item $ \| u \|_{2, \beta, \gamma; M} \le C(b) \delta^{-4-2\beta}_{\min}  \| E\|_{0, \beta, \gamma-2; M}$.
\item $\| w^{\pm}_E : C^{0, \beta}(\Sigma, g)\| \le C \delta^{\gamma-4-2\beta}_{\min} \| E\|_{0, \beta, \gamma-2; M}$.
\end{enumerate}
\end{prop}

\begin{proof}
The proof is similar to the proof of \cite[Proposition 4.17]{kap} but we include it to keep the article self-contained. 
We subdivide the proof into five steps:\newline
\textit{Step 1: Estimates on $u'_\pm$ and $w^{\pm}_{E,1}$:} We start by decomposing $E'_+$  and $u'_+$ (defined as in \eqref{Edecom} and \eqref{Eup}) 
into various parts which will be estimated separately; $E'_-$  and $u'_-$ are decomposed in analogous fashion.
We clearly have by the definitions and the equivalence of the norms 
as in \ref{L:norms} that 
$$
\|E'_+\|_{0,\beta,\gamma-2;\Sigma}
\le \, C \,
\|E\|_{0,\beta,\gamma-2;M}.
$$
For each $p\in L$, we use \ref{Lwe} to define $u'_{p, +}\in C^{2,\beta}( D^\Sigma_p(2\delta_p ))$ satisfying $\Ecalunder_p u'_{p,+}=0$ and 
\begin{align*}
\|u'_{p,+}\|_{2,\beta,\gamma;D^\Sigma_p(2\delta_p)}
\le \, C \,
\|E'_+\|_{0,\beta,\gamma-2;\Sigma}.
\end{align*}

We define now 
$u''_+\in C^{2,\beta}(\Sigma)$ 
supported on $\disjun_{p\in L} D^\Sigma_p( 2\delta_p )$
by requesting for each $p\in L$ that 
$$
u''_+=\,
\Psibold\left[2 \delta_p , \delta_p ; \dbold^\Sigma_p  \right]
(0,u'_{p,+})
\quad\text{on}\quad
D^\Sigma_p( 2\delta_p).
$$
We clearly have
\begin{align*}
\| u''_+ \| _{2, \beta, \gamma; \Sigma} \le C\| E \|_{0, \beta, \gamma-2; M}.
\end{align*}

Now
$E'_+-\Lcal_\Sigma u''_+$ 
vanishes on 
$\disjun_{p\in L} D^\Sigma_p( \delta_p )$
and therefore it is supported on 
$\Sp_1 := \Sigma\setminus \disjun_{p\in L} D^\Sigma_p( \delta_p ) $.
Moreover it satisfies 
$$
\|E'_+-\Lcal_\Sigma u''_+\|_{0,\beta,\gamma-2;\Sigma}
\le \, C \,
\|E\|_{0,\beta,\gamma-2;M}.
$$
Using the definition of the norms and 
the restricted support $\Sp_1$ 
we conclude that 
$$
\|E'_+-\Lcal_\Sigma u''_+: C^{0,\beta}(\Sigma,g)\|
\le \, C \, 
\delta_{\min}^{\gamma-2-\beta} \,
\|E'_+-\Lcal_\Sigma u''_+\|_{0,\beta,\gamma-2;\Sigma}.
$$
The last two estimates and standard linear theory (recall \ref{cLker}) imply that 
the unique solution 
$u'''_+\in C^{2,\beta}(\Sigma)$ to $\Lcal_\Sigma u'''_+=E'_+-\Lcal_\Sigma u''_+$ 
satisfies
$$
\|u'''_+ : C^{2,\beta}(\Sigma,g)\|
\le \, C \, 
\delta_{\min}^{\gamma-2-\beta} \,
\|E\|_{0,\beta,\gamma-2;M}. 
$$
By \ref{aK} there is a unique 
$v_+\in \skernelv[L]$ 
(recall \ref{aK})
such that $\Ecalunder_p(u'''_++v_+)=0$ for each $p\in L$. 
Moreover by the last estimate and \ref{aK},
$v_+$ satisfies the estimate 
\begin{equation*}
\|v_+ : C^{2,\beta}(\Sigma,g)\|
+
\|\Lcal_\Sigma v_+  : C^{0,\beta}(\Sigma,g)\|
\le  \, C \, 
\delta_{\min}^{\gamma-4-2\beta}  \,
\|E\|_{0,\beta,\gamma-2;M}. 
\end{equation*}
Combining now the definitions of $u'''_+$ and $v_+$ 
we conclude that $\Lcal_\Sigma(u''_++u'''_++v_+) = E'_++ \Lcal_\Sigma v_+$.
By the definitions of $u''_+$ and $v_+$ we clearly have 
that $\Ecalunder_p(u''_++u'''_++v_+)=0$ $\forall p\in L$ 
and hence 
$$
u'_+=u''_++u'''_++v_+ \qquad \text{and}\qquad w^+_{E,1} = \Lcal_\Sigma v_+.
$$

Note now that 
$\Lcal_\Sigma u'''_+=E'-\Lcal_\Sigma u''_+$ 
vanishes on 
$\disjun_{p\in L} D^\Sigma_p( \delta_p/4 )$
and 
by \ref{aK} 
so does $\Lcal_\Sigma v_+\in\skernel[L]$. 
We conclude that for each $p\in L$ 
we have $\Lcal_\Sigma (u'''_++v_+)=0$ on $D^\Sigma_p( \delta_p /4)$, 
and since we know already that $\Ecalunder_p (u'''_++v_+)=0$, 
it follows that
\begin{align*}
\| u'''_++v_+ \|_{2, \beta, \widetilde{\gamma}; \Sigma} \le C \delta^{-\widetilde{\gamma}}_{\min}
 \| u'''_++v_+ : C^{2, \beta}( \Sigma, g)\|, 
\end{align*}
where $\widetilde{\gamma}:=\frac{\gamma+2}2\in(\gamma,2)$.  Combining with the earlier estimates for $u'''_+$ and $v_+$ 
we conclude that
$$
\|u'''_++v_+ \|_{2,\beta,\widetilde{\gamma};\Sigma}
\le  \, C \, 
\delta_{\min}^{\gamma-\widetilde{\gamma}-4-2\beta} \,
\| \, E \, \|_{0,\beta,\gamma-2;M}, 
$$
We need the stronger decay for estimating $E_1$ later.
A similar estimate holds with $\gamma$ instead of $\widetilde{\gamma}$. 
Combining with the earlier estimate for $u''_+$ and arguing similarly for $u'_-$ we finally conclude that 
$$
\|u'_{\pm}\|_{2,\beta,\gamma;\Sigma}
\le  \, C \, 
\delta_{\min}^{-4-2\beta} \,
\| \, E \, \|_{0,\beta,\gamma-2;M}. 
$$

\textit{Step 2: Estimates on $\widetilde{u}$:}
By the definitions and \ref{L:norms} (with $\epsilon=1$) we have that
\begin{equation*}
\| \,  \widetilde{E} \, \|_{0,\beta,\gamma-2;\KM        }
\, \le \,
C\, 
(\, \|\, {E} \, \|_{0,\beta,\gamma-2;M}
\, + \,
\|\, u'_+ \, \|_{2,\beta,\gamma;\Sigma}
+ \,
\|\, u'_- \, \|_{2,\beta,\gamma;\Sigma}
\,)
\le  \, C \, 
\delta_{\min}^{-4-2\beta} \,
\| \, E \, \|_{0,\beta,\gamma-2;M}, 
\end{equation*}
where the second inequality follows from the previous estimate. 
By scaling \ref{Etildeueq}, the definitions, and standard theory, we conclude that $\forall p\in L$ 
\begin{equation*}
\|  \, \widetilde{u}_{\mathrm{low}} \,: \, C^{2,\beta} (\KKcore_1[p], \tau_p^{-2} \left. g\right|_p ) \,  \|
\, \le \,
C(b) \, 
\|  \, \tau_p^2 \widetilde{E} \,: \, C^{0,\beta} (\KKcore_1[p], \tau_p^{-2} \left. g\right|_p ) \,  \| 
\, \le \,
C(b) \, \tau_p^{\gamma} \, 
\|  \, \widetilde{E} \, \|_{0,\beta,\gamma-2;\KM        }, 
\end{equation*}
where $\KKcore_1[p]:= \PiY( \Kcore_1[p] ) \subset \tildecat[p, \tau_p, \kappaunder_p]$.  
Using the fact that the ODE solutions of the Jacobi equation corresponding to constants grow at most logarithmically in $\rho$, 
and the ones corresponding to first harmonics at most linearly in $\rho$,  
and that $\widetilde{E}$ is supported on $\bigsqcup_{p\in L} \KKcore_1[p]$, we conclude by comparing weights 
and using that $\rho\ge\tau_L$ on $\KM        $ and $\gamma>1$,  
that   
\begin{align*}
\|  \, \widetilde{u}_{\mathrm{low}} \, \|_{2,\beta, \gamma ;\KM        }
\, \le \,
\|  \, \tau_{L}^{1-\gamma}\widetilde{u}_{\mathrm{low}} \, \|_{2,\beta,1;\KM        }
\, \le \,
C(b) \, 
\|  \, \widetilde{E} \, \|_{0,\beta,\gamma-2;\KM        }.
\end{align*}
(Actually $\widetilde{u}_{\mathrm{low}}$ can be expressed explicitly in terms of $\widetilde{E}_{\mathrm{low}} $ 
by using variation of parameters and the ODE solutions corresponding to constants $\phie : = 1- \sss \tanh \sss$ and $\phio : = \tanh \sss$, 
and the ODE solutions corresponding to first harmonics $\widetilde{u}_{e} : = \sech \sss$ and $\widetilde{u}_o: = \sinh \sss + \sss \sech \sss$.)     

By scaling \ref{Etildeueq}, standard linear theory, 
that $\widetilde{E}$ is supported on $\bigsqcup_{p\in L} \KKcore_1[p]$, 
and that on each $\KKcore_1[p]$ the conformal metrics $\tau_L^{-2} \left. g\right|_p $ and $h=\frac12|A|^2g=\nu^*g_{\Sph^2}$ are uniformly equivalent, 
we conclude that $\forall p\in L$ 
\begin{equation*}
\| \, \widetilde{u}_{\mathrm{high}} \, :C^0(\tildecat [p, \tau_p, \kappaunder_p]  ) \,\| 
\, \le \,
C(b) \, 
\| \, \tau_p^{2} \, \widetilde{E}_{\mathrm{high}} \, :C^0( \KKcore_1[p] ) \,\| 
\, \le \,
C(b) \, \tau_p^{\gamma} \, 
\|  \, \widetilde{E} \, \|_{0,\beta,\gamma-2;\KM        }. 
\end{equation*}
Similarly we obtain the second inequality below; the first inequality follows by comparing weights as done earlier. 
\begin{align*}
\| \, \widetilde{u}_{\mathrm{high}} \, \|_{2,\beta,\gamma ;\KM        }
\, \le \,
\| \, \tau_L^{-\gamma} \, \widetilde{u}_{\mathrm{high}} \, \|_{2,\beta,0;\KM        }
\, \le \,
C(b) \, 
\|  \, \widetilde{E} \, \|_{0,\beta,\gamma-2;\KM        }.
\end{align*}
Combining the above we obtain 
\begin{align*}
\|  \, \widetilde{u} \, \|_{2,\beta,\gamma;\KM        }
\, \le \,
C(b) \, 
\delta_{\min}^{-4-2\beta} \| \, E \, \|_{0,\beta,\gamma-2;M}. 
\end{align*}

\textit{Step 3: A decomposition of $E_1$:} 
Using \eqref{EEone} and \ref{DRMappr}, \eqref{Edecom}, \eqref{EEpp}, \eqref{EEtildedecom}, and \eqref{Etildeueq}, 
we obtain 
\begin{equation}
\label{EEoneF}
E_1=E_{1,I}+E_{1,II}+E_{1,III},
\end{equation}
where 
$E_{1,I} ,E_{1, II}, E_{1,III}\in 
C^{0,\beta}(M)$ 
are defined by 
(recall \ref{DpsiM}),
\begin{equation}
\label{EEoneI}
\begin{aligned}
E_{1,I} \, :=& \,
[\Lcal_M,\psihat] 
\,( \, \widetilde{u} \circ \PiY \,) \, ,
\\
E_{1,II}\, : = & 
\begin{aligned}[t]  
\, 
\psihat\, 
&
\left(
\Lcal_M  
\, (\, \widetilde{u} \,  \circ
\PiY \, )\, 
-
\, (\, \Lcal_\tildecat  
\widetilde{u} \, )  
\circ \PiY 
\right) \, 
\\ 
& \qquad \qquad \qquad = \: \widehat{\psi} \Lcal_M \, (\, \widetilde{u}\, \circ \PiY\, )\, - \widetilde{E}\circ \PiY ,  
\end{aligned} 
\\
E_{1,III} \, :=& \,
\Lcal_M
\,\{ \, J_M(\psi' u'_+, \psi'u'_-)\}\,
-
J_M \left( \Lcal_\Sigma (\psi' u'_+), \Lcal_\Sigma (\psi' u'_-)\right), 
\end{aligned}
\end{equation}
on 
$\Shat[M]\setminus\Shat_1[M]$,
$\Shat[M]$, and 
$\Stildep$ respectively, and to vanish elsewhere,  
and we have used that 
$ \Lcal_\Sigma u'_{\pm}= E'_{\pm}+w^{\pm}_{1,E} $ which follows from \eqref{Eup}.

\textit{Step 4: Estimates on $u_1$ and $E_1$:}
Using the definitions, \ref{L:norms} with $\epsilon=1$, 
and the estimates for $u'_{\pm}$ and $\widetilde{u}$ above 
we conclude that 
\begin{align*}
\| \, u_1 \, \|_{2,\beta,\gamma;M} 
\, \le \,
C(b) \, 
\delta_{\min}^{-4-2\beta} \,  
\| \, E \, \|_{0,\beta,\gamma-2;M}. 
\end{align*}

By \ref{L:norms} we have 
$\|\, 
\,  \widetilde{u} \, \circ \PiY \, \|_{2,\beta,2;\,
\Shat[M]\setminus\Shat_1[M]}
\Sim_2
\|  \, \widetilde{u} \, \|_{2,\beta,2;\, \PiY \,(\,
\Shat[M]\setminus\Shat_1[M]\,)}
$.
Using definitions \ref{D:norm} and \ref{D:regions} we conclude that 
\begin{align*}
\|\, 
\, \widetilde{u} \circ \PiY \, \|_{2,\beta,\gamma;\,
\Shat[M]\setminus\Shat_1[M]}
\, \le \,
C\, 
\tau_{\max}^{\alpha(1-\gamma)} 
\, \,
\|  \, \widetilde{u} \, \|_{2,\beta,1;\KM        }, 
\end{align*}
and therefore we have by the definition of $E_{1,I}$ and the preceding estimates
\begin{multline*}
\| \, E_{1,I} \, \|_{0,\beta,\gamma-2;\,M}
\lem 
C\, 
\tau_{\max}^{\alpha(1-\gamma)} 
\, \,
\|  \, \widetilde{u}_{\mathrm{low}} \, \|_{2,\beta,1;\KM        } + \|\, \widetilde{u}_{\mathrm{high}}\, \|_{2, \beta, \gamma; \PiY(\Shat[M]\setminus\Shat_1[M])} 
\\ 
\lem C(b) \tau_{\max}^{(1-\alpha)(\gamma-1)} \| \, \widetilde{E}\, \|_{0, \beta, \gamma-2; \KM        }.
\end{multline*}
Applying now \ref{L:appr}(i) with $f=\widetilde{u}$ and $\gammahat=\gamma$ 
and using the definition of $\psihat$ we conclude that  
\begin{align*}
\| \, E_{1,II} \, \|_{0,\beta,\gamma-2;\,M}
\, \le \,
C\, \tau_{\max}^{2\alpha}\,
\, \,
\|  \, \widetilde{u} \, \|_{2,\beta,\gamma;\KM        }. 
\end{align*}

We decompose now $E_{1,III}= E''_{1,III}+E'''_{1,III}$ 
where $E''_{1,III}$ and $E'''_{1,III}$ are defined 
the same way as $E_{1,III}$ but with $u'_{\pm}$ replaced by 
$u''_{\pm}$ and $u'''_{\pm}+v_{\pm}$ respectively. 
Applying \ref{L:appr}(ii) with $\epsilon_1=0$, $f=u''_{\pm}$, and $\gammahat=\gamma$, 
we conclude that 
$$
\| \, E''_{1,III} \, \|_{0,\beta,\gamma-2;\,M}
\, \le \,
C\, b^{-1}\, \log b \,\, \left( 
\|u''_+\|_{2,\beta,\gamma;\Sigma}+\|u''_-\|_{2,\beta,\gamma;\Sigma}\right).
$$
Applying \ref{L:appr}(ii) with $\epsilon_1=\widetilde{\gamma}-\gamma$, $f=u'''_\pm + v_\pm$, and $\gammahat=\gamma$ we obtain 
\begin{equation*}
\| \, E'''_{1,III} \, \|_{0,\beta,\gamma-2;\,M}
\, \le \,
C\, b^{\widetilde{\gamma}-\gamma-1}\, \log b \,\, \tau_{\max}^{\widetilde{\gamma}-\gamma} \,
\left(\|u'''_{+}+v_{+} \|_{2,\beta,\widetilde{\gamma};\Sigma}+ 
\|u'''_{-}+v_{-} \|_{2,\beta,\widetilde{\gamma};\Sigma}\right).
\end{equation*} 
Combining the above with the earlier estimates
and using \ref{con:one}(ii)  and \ref{con:b} we conclude 
that 
\begin{equation*}
\| \, E_{1} \, \|_{0,\beta,\gamma-2;\,M}
\, \le \,
(\, C(b)\, \tau_{\max}^{\alpha/2} 
\, + \,
C\, b^{-1}\, \log b \,\, 
\, + \,
C\, b^{-1/2}\, \log b \,\, \tau_{\max}^{\widetilde{\gamma}-\gamma-\alpha} \,
\,)\,
\| \, E \, \|_{0,\beta,\gamma-2;\,M}.
\end{equation*} 

\textit{Step 5: The final iteration:}
By assuming $b$ large enough 
and $\tau_{\max}$ small enough in terms of $b$ we conclude using $\widetilde{\gamma}-\gamma-\alpha>0$ and induction 
that
\begin{align*}
\| \, E_{n} \, \|_{0,\beta,\gamma-2;\,M}
\, \le \,
2^{-n}
\, \| \, E \, \|_{0,\beta,\gamma-2;\,M}.
\end{align*}
The proof is then completed by using the earlier estimates.
\end{proof}

\begin{cor}
\label{Plinear2}
Recall that we assume that 
 \ref{aK}, \ref{con:one}, \ref{cLker}, and \ref{con:b} hold. 
A linear map 
$
\Rcal'_M: C^{0,\beta}(M) \to C^{2,\beta}(M) \times \skernel[L] \times \skernel[L] 
$ 
can be defined such that given $E\in C^{0, \beta}(M)$, the following hold.
\begin{enumerate}[label=\emph{(\roman*)}]
\item $\Lcal_M u = E + J_M(w^+_E, w^-_E)$ where $\Rcal'_M (E)=(u,w^+_E, w^-_E)$. 
\item $ \| u \|_{2, \beta, \gamma, \gamma'; M} + \| w^{\pm}_E : C^{0, \beta}(\Sigma, g)\| 
\le C(b) \delta^{-4-2\beta}_{\min} \| E\|_{0, \beta, \gamma-2, \gamma'-2; M}$.
\end{enumerate}
\end{cor}

\begin{proof}
We first define $\Ehat\in C^{0, \beta}(M)$, supported on $\Shat[M]$, by
\begin{align}
\label{Eehat}
\widehat{E} := \psihat E - \Hcal_1(\psihat E). 
\end{align}
We then solve $\Lcal_{\tildecat} \uhat = \Etilde := \Ehat\circ \PiY^{-1}$ on $\KM        $ and then estimate $\uhat$ by modifying Step 2 in the proof of Proposition \ref{Plinear}. 
The modifications are necessary because the support of the inhomogeneous term is much larger and the decay is different. 
We decompose as in \ref{EEtildedecom} and \ref{Etildeueq}. 
To estimate $\uhatlow$ we argue as before, 
utilizing the fact that there are no first Fourier modes, 
and so the growth is only logarithmic; 
this is slower than the $\gamma'$ growth rate allowed. 
To estimate $\uhathigh$ we first solve 
with Dirichlet boundary conditions 
the equation $\Lcal_{\tildecat} \uhatann  = \Etildehigh$ 
on the annuli (conformal to punctured discs) of 
$\KM \setminus \PiY( \Kcore[M] ) $  
(cf. \ref{D:regions}).  
By arguments which are standard by now we estimate then $\uhatann$, and by cutting it off and subtracting from $\uhathigh$, 
we reduce without loss of generality to estimating $\uhathigh$ in the case $\Etildehigh$ is supported on $\bigsqcup_{p\in L} \KKcore_1[p]$ as in Step 2 of \ref{Plinear}; 
this can be handled then as in \ref{Plinear}. 
We conclude that 
\begin{align}
\label{Euhatest}
\| \uhat \circ \PiY \|_{2, \beta, \gamma, \gamma'; \cat_L} \le C \| \Ehat \|_{0, \beta, \gamma-2, \gamma'-2; M}.
\end{align}

We then define $\Ehat'\in C^{0, \beta}(M)$ by
\begin{align}
\label{Eehat'} 
\Ehat': = 
E - \psihat \Ehat - [\Lcal_M, \psihat] ( \uhat\circ \PiY) - \psihat \left(\Lcal_M( \uhat \circ \PiY) - (\Lcal_\tildecat \uhat ) \circ \PiY\right),
\end{align}
$(\uhat', w^+_{E}, w^-_{E}): = \Rcal_M \Ehat' \in  C^{2,\beta}(M) \times \skernel[L] \times \skernel[L] $,  
and 
$
\Rcal'_M E := (\psihat\,  \uhat\circ \PiY + \uhat', w^+_E, w^-_E). 
$
It is straightforward to check that (i) holds by using \ref{Eehat}, \ref{Eehat'} and \ref{Plinear}(i).  

Using \eqref{Eehat'} for the first inequality, 
\ref{L:appr}(i) and the definitions for the second, and  
\eqref{Euhatest} and Definition \ref{D:norm} for the third, 
we have 
\begin{multline*}
\| \Ehat'\|_{0, \beta, \gamma-2; M} \lem 
C { \Big( }\|E- \psihat \Ehat \|_{0, \beta, \gamma-2; M} \: + 
\\ 
 \| [\Lcal_M, \psihat] ( \uhat\circ \PiY) \|_{0, \beta, \gamma-2 ; \Shat[M]\setminus \Shat_1[M] } 
+ \| \Lcal_M( \uhat \circ \PiY) - (\Lcal_\tildecat \uhat ) \circ \PiY \|_{0, \beta, \gamma-2 ; \cat_L} \Big)
\\ 
\lem C \big( \|E\|_{0, \beta, \gamma-2, \gamma'-2; M} 
+ \| \uhat \circ \PiY \|_{2, \beta, \gamma ; \Shat[M]\setminus \Shat_1[M]} 
+ \| \uhat \circ \PiY \|_{2, \beta, \gamma-1; \cat_L} \big) 
\\ 
\lem  C \| E\|_{0, \beta, \gamma-2, \gamma'-2; M}.
\end{multline*}
Combining with \ref{Plinear}, \eqref{Euhatest}, and the definition of $\Rcal'_M E$, we complete the proof. 
\end{proof}

\section[Constructing minimal doublings from families of LD solutions]{Constructing minimal doublings from families of LD solutions}
\label{S:mainI}

\subsection*{The nonlinear terms} 
\nopagebreak

In this section we state and prove Theorem \ref{Ttheory} which ``automates'' the construction of a minimal doubling given a suitable family of LD solutions. 
Continuing the discussion of the initial surfaces from the previous sections, 
we first state and prove Lemma \ref{Lquad}, 
where we discuss their perturbations and the nonlinear terms in the resulting mean curvature.  

\begin{lemma}
\label{Lquad}
If $M$ is an initial surface as in \ref{Plinear} and 
$\phi\in C^{2,\beta}(M)$ 
satisfies $\|\phi\|_{2,\beta,\gamma, \gamma';M} \, \le \, \tau_{\max}^{1+\alpha/4} $ (recall \ref{D:norm}),  
then (recalling \ref{NT}\ref{dgraph})  
$M_\phi: = \graph^{N, g}_M ( \phi)$ is a well-defined embedded surface. 
Moreover if $H_\phi$ denotes the mean curvature of $M_\phi$ pulled back to $M$ by $ X^{N, g}_{M,\phi} :M \to M_\phi$ 
and $H$ the mean curvature of $M$, then we have 
$$
\|\, H_\phi-\, H - \Lcal_M \phi \, \|_{0,\beta,\gamma-2, \gamma'-2;M}
\, \le \, C \, \tau_{\min}^{-\alpha/2}
\|\, \phi\, \|_{2,\beta,\gamma, \gamma';M}^2.
$$
\end{lemma}

\begin{proof}
Following the notation in the proof of \ref{L:norms} and by \ref{Efw} we have that for $q\in \Stilde'$, 
the graph $B''_q$ of $\varphi_{:q}^+$ (or $-\varphi_{:q}^-$)            
over $B_q'$ in $(\Sigma, \gtilde_q)$ can be described by an immersion $X_{:q}: B'_q \rightarrow B''_q = X_{:q}(B'_q)$, 
such that there are coordinates on $B'_q$ and a neighborhood in $N$ of $B''_q$, 
which has uniformly bounded $C^{3, \beta}$ norms, 
the standard Euclidean metric on the domain is bounded by $CX^*_{:q}\gtilde_q$, 
and the coefficients of $\gtilde_q$ in the target coordinates have uniformly bounded $C^{3, \beta}$ norms.  
By the definition of the norm and since $\| \phi \|_{2, \beta, \gamma,\gamma'; M}\le \tau^{1+\alpha/4}_{\max}$, 
we have that the restriction of $\phi$ on $B''_{q}$ satisfies
\begin{equation*}
\| (\dbold_L(q))^{-1} \phi :C^{2, \beta}(B''_{q}, \gtilde_q)\|  
\lem C (\dbold_L(q))^{-1} \max \big( (\dbold(q))^{\gamma}, \tau_{\max}^{(1-\alpha)/2} (\dbold_L(q))^{\gamma'}\big) \, \|\phi \|_{2, \beta, \gamma, \gamma'; M}. 
\end{equation*}

Since the right hand side is small in absolute terms we can conclude that $\graph^{M, g}_{B'_q} \phi$ is well defined and embedded. 
Using scaling for the left hand side we further conclude that 
\begin{multline*}
\| \dbold_L(q)( H_\phi - H - \Lcal_M \phi) : C^{0, \beta}(B''_q, \gtilde_q)\| \lem 
\\
 C  (\dbold_L(q))^{-2}  \max \big( (\dbold_L(q))^{\gamma}, \tau_{\max}^{(1-\alpha)/2} (\dbold_L(q))^{\gamma'}\big)^2  \, \| \phi \|^2_{2, \beta, \gamma, \gamma'; M}.
\end{multline*}
Rearranging this, we conclude that (recalling from \ref{con:b} that $\gamma' = \gamma-1$)
\begin{multline*}
\| H_\phi - H - \Lcal_M  : C^{0, \beta}(B''_q, \gtilde_q)\| 
\\ 
\lem C (\dbold_L(q))^{\gamma-1} \max \big( (\dbold_L(q))^{\gamma-2}, \tau_{\max}^{1-\alpha} (\dbold_L(q))^{\gamma'-3}\big)\, \| \phi \|^2_{2, \beta, \gamma, \gamma'; M} 
\\ 
\lem C \tau_{\min}^{-\alpha/2} \max \big( (\dbold_L(q))^{\gamma-2}, \tau_{\max}^{(1-\alpha)/2} (\tau/\dbold_L(q))^{1/2} (\dbold_L(q))^{\gamma'-2} \big) \, \| \phi \|^2_{2, \beta, \gamma, \gamma'; M}.
\end{multline*} 
where we have used that $\gamma' = 1/2$.  
Finally, note that that the components of $\Kcore[M]$ appropriately scaled are small perturbations of a fixed compact region of the standard catenoid, 
which allows us to repeat the arguments above in this case.  
By combining with the earlier estimates and using the definitions, we conclude the estimate in the statement of the lemma. 
\end{proof}

\subsection*{The fixed point theorem}
\nopagebreak

\begin{assumption}[Families of LD solutions] 
\label{A:FLD} 
\label{Azetabold}
\label{Adiffeo}
We assume \ref{cLker} holds and that we are given continuous families of the following parametrized by $\zetabold\in \domzb \subset \Pcal$, 
where $\Pcal$ is a finite dimensional vector space and $\domzb \subset \Pcal$ a convex compact subset containing the origin $\zerobold$. 
\begin{enumerate}[label=(\roman*)]
\item 
\label{Idiffeo} 
Diffeomorphisms $\Fcal^\Sigma_\zetabold : \Sigma \rightarrow \Sigma$ with $\Fcal^\Sigma_\zerobold$ the identity on $\Sigma$.  
\item 
Finite sets $L=L\llbracket \zetabold\rrbracket=\Fcal^\Sigma_\zetabold L\llbracket\zerobold\rrbracket \subset\Sigma$ of cardinality    
$|L|=|L\llbracket \zetabold\rrbracket|= |L\llbracket\zerobold\rrbracket|$.     
\item 
Configurations $\taubold=\taubold\llbracket\zetabold\rrbracket: L\llbracket\zetabold\rrbracket\to \R_+$.  
\item 
LD solutions $\varphi =\varphi \llbracket \zetabold \rrbracket$ as in \ref{dLD}  
of singular set 
$L=L\llbracket \zetabold\rrbracket$     
and configuration 
$\taubold=\taubold\llbracket\zetabold\rrbracket$.   
\item 
For each $L=L\llbracket \zetabold\rrbracket$  
constants $\delta_p=\delta_p\llbracket\zetabold\rrbracket$ as in \ref{con:L}.  
\item 
For each $L=L\llbracket \zetabold\rrbracket$  
a space $\skernelv \llbracket \zetabold \rrbracket = \skernelv[ \, L\llbracket \zetabold \rrbracket \, ]$ as in \ref{aK}. 
\item 
Linear isomorphisms $Z_{\zetabold}: \val\llbracket \zetabold\rrbracket \rightarrow \Pcal$ where 
$\val\llbracket \zetabold\rrbracket := \val [\, L \llbracket \zetabold \rrbracket \, ]$ (recall \ref{DVcal}).  
\end{enumerate} 
Moreover we assume the following are satisfied   
$\forall \zetabold\in \domzb$.      
\begin{enumerate}[label=\emph{(\alph*)}]
\item 
\label{Aa} 
$\| \Fcal^\Sigma_\zetabold : C^4 \| \le C$ where the norm is defined with respect to some atlas of $\Sigma$ and the constant $C$ depends only on the background $(\Sigma,N,g)$. 
\item 
\label{Ab} 
$\forall p\in L\llbracket\zerobold\rrbracket$ we have $\Fcal^\Sigma_\zetabold( D^\Sigma_p(3\delta_p)) = D^\Sigma_{q}(3\delta_q)$ with $q:=\Fcal^\Sigma_\zetabold(p)$. 
\item 
\label{Ac} 
$\varphi =\varphi \llbracket \zetabold \rrbracket$,  
$L=L\llbracket \zetabold\rrbracket$, and 
$\taubold=\taubold\llbracket\zetabold\rrbracket$ 
satisfy \ref{con:one}, including the smallness of $\tau_{\max}$ in \ref{con:one}\ref{con:one:i}
which is now in terms of the constant $C$ in \ref{Aa} as well.   
\item 
\label{Atau}
$\forall p\in L\llbracket\zerobold\rrbracket$ we have the uniformity condition $\tau^2_q \leq \tau_p \leq \tau_q^{1/2}$, 
where here $\tau_p$ denotes the value of $\taubold\llbracket\zerobold\rrbracket$ at $p$ 
and $\tau_q$ the value of $\taubold\llbracket \zetabold\rrbracket$ at $q = \Fcal^\Sigma_\zetabold(p)\in L\llbracket \zetabold \rrbracket$. 
\item 
\label{AZ}
$\zetabold - Z_{\zetabold} ( \Mcal_{L\llbracket\zetabold\rrbracket} \varphi\llbracket \zetabold\rrbracket) \in \frac{1}{2} \domzb$ (prescribed unbalancing). 
\end{enumerate} 
\end{assumption}

\begin{definition}
\label{dfpf}
Assuming \ref{Azetabold} we define a scaled push-forward map 
$\Fcal^\val_{\zetabold} : \val\llbracket \zerobold\rrbracket \rightarrow \val\llbracket \zetabold \rrbracket$ by 
\begin{align*}
\Fcal^\val_{\zetabold} (\widetilde{\kappa}^\perp_p+ \widetilde{\kappa}_p)_{p\in L\llbracket\zerobold\rrbracket} 
= \left( \tau_{q}^{1+\alpha/5} \left( \widetilde{\kappa}^\perp_{(\Fcal^\Sigma_\zetabold)^{-1}(q)}+
(\Fcal^\Sigma_\zetabold)_* \widetilde{\kappa}_{(\Fcal^\Sigma_\zetabold)^{-1}(q)} \right) \right)_{q\in L\llbracket\zetabold\rrbracket}, 
\end{align*}
where 
$(\Fcal^\Sigma_\zetabold)_* \widetilde{\kappa}_{(\Fcal^\Sigma_\zetabold)^{-1}(q)} = 
\widetilde{\kappa}_{(\Fcal^\Sigma_\zetabold)^{-1}(q)} \circ (\, d_{(\Fcal^\Sigma_\zetabold)^{-1}(q)} \Fcal^\Sigma_\zetabold \,)^{-1} \, \in \, T^*_q \Sigma$ 
and $\tau_q$ denotes the value of $\taubold\llbracket\zetabold\rrbracket$ at $q\in L\llbracket\zetabold\rrbracket$. 
\end{definition}

\begin{definition}[Families of initial surfaces]
\label{dkappatilde}
Assuming \ref{Azetabold} and following \ref{Dinit} 
we write 
\begin{equation*} 
\begin{gathered} 
M\llbracket \zetaboldunder \rrbracket : = M[ \, \varphi\llbracket \zetabold \rrbracket \, , \, \Fcal^\val_{\zetabold} \kappaunderbold \, ]  
\quad \text{ for } \quad 
\zetaboldunder = (\zetabold, \kappaunderbold) \in \domzb \times \domku  \quad \text{ where}
\\ 
\domku : = \left\{ \, (\widetilde{\kappa}^\perp_p+ \widetilde{\kappa}_p)_{p\in L} \, : \, 
\forall p\in L \, , \, \widetilde{\kappa}^\perp_p \in [-1, 1]  \, , \, |\widetilde{\kappa}_p |_g \le 1 \right\} 
\, \subset \, \val \llbracket \zerobold\rrbracket . 
\end{gathered} 
\end{equation*} 
\end{definition}

\begin{lemma}[Diffeomorphisms $\Fcal_{\zetaboldunder}$]
\label{Lrefdiff}
Assuming \ref{Azetabold}, there exists a family of diffeomorphisms 
$\Fcal_{\zetaboldunder} : M\llbracket\zerobold \rrbracket \rightarrow M\llbracket\zetaboldunder\rrbracket$
satisfying the following, 
where here $\zerobold$ denotes the zero element of $\Pcal \times \val\llbracket \zerobold\rrbracket$
and $\zetaboldunder \in \domzb \times \domku$ is as in \ref{dkappatilde}.
\begin{enumerate}[label=\emph{(\roman*)}]
\item $\Fcal_{\zetaboldunder}$ depends continuously on $\zetaboldunder$. 

\item  For any $u\in C^{2, \beta}(M\llbracket\zetaboldunder\rrbracket)$ and $E\in C^{0, \beta}(M\llbracket\zetaboldunder\rrbracket)$ 
we have the following equivalence of norms: 
\begin{align*}
\| \, u\circ \Fcal_{\zetaboldunder} \, \|_{2,\beta,\gamma, \gamma';M\llbracket\zerobold\rrbracket }
\, & \Sim_4 \, 
\| \, u\, \|_{2,\beta,\gamma, \gamma';M\llbracket\zetaboldunder\rrbracket } ,
\\ 
\| \, E\circ \Fcal_{\zetaboldunder} \, \|_{0,\beta,\gamma-2, \gamma'-2;M\llbracket\zerobold\rrbracket }
\, & \Sim_4 \, 
\| \, E\, \|_{0,\beta,\gamma-2, \gamma'-2;M\llbracket\zetaboldunder\rrbracket} .
\end{align*}
\end{enumerate}
\end{lemma}

\begin{proof}
$\forall p\in L \llbracket \zerobold \rrbracket$, 
we first define $\widehat{\Fcal}^p_{\zetabold} : \cyl_{[-\sss_{\zerobold, p}, \sss_{\zerobold, p}]} \rightarrow \cyl_{[-\sss_{\zetabold, q}, \sss_{\zetabold, q}]}$ by
 $\widehat{\Fcal}^p_{\zetabold} \circ \Thetacyl(\theta, \sss) = \Thetacyl(\theta, \frac{ \sss_{\zetabold, q} }{\sss_{\zerobold, p}}\sss)$, 
where $q: = \Fcal^\Sigma_{\zetabold}(p)$ and $\sss_{\zerobold, p}$ and $\sss_{\zetabold, q}$ are defined by the equations 
$\tau_p \cosh \sss_{\zerobold, p} = \tau_p^{\alpha}/2$ and 
$\tau_q \cosh \sss_{\zetabold, q} = \tau^\alpha_q/2$.
  
We then define $\Fcal_{\zetaboldunder}$ to map each $\Lambda^p_{\zerobold}$ onto $\Lambda^q_{\zetabold}$, where
\begin{align*}
\Lambda^p_\zerobold: =& X_\cat[ p, \tau_p, \kappaunder_p] ( \cyl_{[-\sss_{\zerobold, p}, \sss_{\zerobold, p}]})\subset M\llbracket \zerobold\rrbracket, \quad
\\
\Lambda^q_\zetabold: =&  X_\cat[ q, \tau_q, \kappaunder_q](\cyl_{[-\sss_{\zetabold, q}, \sss_{\zetabold, q}]}) \subset M\llbracket \zetaboldunder \rrbracket,
\end{align*}
by requesting that
\begin{align}
\Fcal_{\zetaboldunder} \circ X_\cat[p, \tau_p, \kappaunder_p] = X_\cat[q, \tau_q, \kappaunder_q] \circ \widehat{\Fcal}^p_{\zetabold}.
\end{align}

We next define the restriction of $\Fcal_{\zetaboldunder}$ on $M\llbracket \zerobold \rrbracket \setminus \Shat[M\llbracket \zerobold\rrbracket]$ 
to be a map onto $M\llbracket \zetaboldunder\rrbracket \setminus \Shat[M\llbracket \zetabold\rrbracket]$ which preserves signs of the $\zz$ coordinate (recall \ref{dgeopolar}) and satisfies
\begin{align*}
\Pi_\Sigma \circ \Fcal_{\zetaboldunder} = \Fcal^\Sigma_{\zetabold} \circ \Pi_\Sigma.
\end{align*}

On the region $\Shat[M\llbracket \zerobold\rrbracket] \setminus \cup_{p\in L\llbracket \zerobold \rrbracket} \Lambda^p_{\zerobold}$ we apply the same definition as in the above paragraph 
but with $\Fcal^\Sigma_{\zetabold}$ appropriately modified by using cut-off functions so that the final definition provides an interpolation between the two definitions above and satisfies (i).

Using \ref{Azetabold}\ref{Atau} and \ref{Ecatenoid}, it is not difficult to check that for each $p\in L\llbracket\zerobold\rrbracket$, 
\begin{align*}
\sss_{\zerobold, p} \,\, \Sim_{1+ C/|\log \tau_p|} \,\, \sss_{\zetabold, q}.
\end{align*}
Using this and arguing as in the proof of \ref{L:norms}, we conclude (ii). 
\end{proof}

\begin{theorem}[Theorem \ref{TA}] 
\label{Ttheory}
Assuming that 
\ref{Azetabold} holds,  
there exist $\zetaboldhatunder = (\zetaboldhat, \breve{\kappaunderbold}) \in \domzb\times \domku$ 
and $\upphihat \in C^\infty(M\llbracket \zetaboldhatunder \rrbracket)$ (recall \ref{dkappatilde}) 
satisfying
(recall \ref{D:norm}) 
$$ 
\| \upphihat\|_{2, \beta, \gamma, \gamma'; M\llbracket \zetaboldhatunder \rrbracket } \le 
\tau_{\max}^{1+\alpha/4},  
$$
such that the normal graph $\Mhat:= (M\llbracket \zetaboldhatunder\rrbracket)_\upphihat$ is an embedded smooth closed minimal surface doubling $\Sigma$ in $N$ 
as in \ref{Ddoubling}, satisfying   
\begin{equation} 
\begin{gathered} 
\label{EMguO}
\Mhat = \graph^{N, g}_{\Omegahat}( \ubreve^+) \cup \graph^{N, g}_{\Omegahat} (-\ubreve^-),  
\\ 
\quad \text{where} \quad 
\Omegahat = \Pi_\Sigma(\Mhat) = \Sigma \setminus \textstyle{ \bigsqcup_{p \in L} \Dbreve_{p} }, 
\quad 
L = L\llbracket \zetaboldhat\rrbracket, 
\\
\text{ and } \quad 
D^\Sigma_p(\tau_p(1-\tau_p^{8/9})) \subset \Dbreve_p \subset D^\Sigma_p(\tau_p(1+\tau^{8/9}_p)) \quad \forall p\in L. 
\end{gathered} 
\end{equation} 
Moreover 
$\Mhat $ has genus $2g_\Sigma-1+|L|$ (where $g_\Sigma$ is the genus of $\Sigma$) and its area $|\Mhat|$ satisfies 
\begin{align}
\label{E:Marea} 
| \Mhat | = 2 |\Sigma| - \pi \sum_{p \in L } \tau^2_p\left(1+ O(\, \tau_p^{1/2}|\log \tau_p|\, )\right). 
\end{align}
\end{theorem}

\begin{proof}
We first define 
$B \subset C^{2,\beta}(M\llbracket\zerobold\rrbracket)\times \Pcal \times \val\llbracket \zerobold\rrbracket$ 
by 
\begin{align}
\label{E:proofA} 
B := 
\big\{ 
\, v\in C^{2,\beta}(M\llbracket \zerobold \rrbracket):\|v\|_{2,\beta,\gamma, \gamma';M\llbracket \zerobold \rrbracket} \, \le \, \tau_{\max}\llbracket\zerobold\rrbracket^{1+\alpha}\, 
\big\} 
\times \domzb \times \domku.
\end{align}
We next define a map 
$\Jcal : B \rightarrow C^{2, \beta}(M\llbracket\zerobold\rrbracket)\times \Pcal \times \val\llbracket \zerobold\rrbracket$ 
as follows; note that the proof is based on finding a fixed point for $\Jcal$. 
Suppose $(v, \zetaboldunder) \in B$.  
Use \ref{Plinear2} to define $(u, w^+_H, w^-_H):= - \Rcal'_{M\llbracket \zetaboldunder \rrbracket}\left( H - J_M(w^+, w^-)\right)$, 
where 
$w^\pm := \Lcal_\Sigma\Ecal^{-1}_L ( \, - \Mcal_L \varphi \pm  \kappaunderbold \,)$ as in \ref{LglobalH}. 
Define also $\phi \in C^{2, \beta}\left(M\llbracket \zetaboldunder \rrbracket\right)$ by $\phi : = v\circ \Fcal^{-1}_{\zetaboldunder} + u$.  
We then have: 
\begin{enumerate}
\item $\Lcal_M u + H = J_M(w^++w^+_H, w^-+w^-_H).$
\item By \ref{LglobalH}, \ref{Plinear2}, and the size of $v$ in \eqref{E:proofA},  
	\begin{align*}
	\left\| w^\pm_H: C^{0, \beta}( \Sigma, g)\right\| +\left \| \phi \right \|_{2, \beta, \gamma, \gamma'; M\llbracket \zetaboldunder \rrbracket} \leq \tau_{\max}^{1+\alpha/4}.
	\end{align*}
\item 
Using \ref{Plinear2} again we define $(u_Q, w^+_Q, w^-_Q): = - \Rcal'_{M\llbracket \zetaboldunder \rrbracket}(H_\phi- H - \Lcal_M \phi)$ and we have  
$\Lcal_M u_Q+H_\phi = H + \Lcal_M \phi +J_M(w^+_Q, w^-_Q)$.   
\item 
Moreover by \ref{Lquad}, 
$\left\| w^\pm_Q : C^{0, \beta}( \Sigma, g)\right\| + \left\| u_Q\right\|_{2, \beta, \gamma, \gamma'; M\llbracket \zetaboldunder \rrbracket} 
\leq \tau_{\max}^{2-\alpha/4}$. 
\item 
Combining the above we conclude 
$$ 
\Lcal_M( u_Q - v\circ \Fcal^{-1}_{\zetaboldunder}) + H_\phi = J_M(\Lcal_\Sigma \Ecal^{-1}_L \mubold^+, \Lcal_\Sigma \Ecal^{-1}_L \mubold^- ), 
$$ 
where $\mubold^\pm : = - \Mcal_L \varphi \pm \kappaunderbold + \mubold^{\pm}_{H, Q}$ and $\mubold^{\pm}_{H,Q}$ 
are defined by requesting that $\Lcal_{\Sigma}\Ecal^{-1}_L \mubold^\pm_{H, Q} = w^\pm_H + w^\pm_Q$. 
\end{enumerate}
We then define $\mubold_{\sym}$, $\mubold_{\asym}$, $\mubold^{\sym}_{H, Q}$, $\mubold^{\asym}_{H, Q}$, and $\Jcal$ by  
\begin{equation} 
\label{Tfp}
\begin{gathered}
2 \mubold_{\sym} := \mubold^+ +\mubold^-, \qquad
2\mubold_{\asym} := \mubold^+ - \mubold^-, \\
2\mubold^{\sym}_{H, Q}: = \mubold^+_{H, Q} + \mubold^{-}_{H, Q},  \qquad
2\mubold^{\asym}_{H, Q}: = \mubold^+_{H, Q} - \mubold^{-}_{H, Q}, 
\\
\Jcal(v, \zetaboldunder) := \big( u_Q \circ \Fcal_{\zetaboldunder} \, , \, \zetabold +  Z_{\zetabold}(\mubold_{\sym}) \, , \, 
(\Fcal^\val_{\zetabold} )^{-1}( \kappaunderbold - \mubold_{\asym})\big)   \\
= \big( u_Q \circ \Fcal_{\zetaboldunder} \, ,\, \zetabold - Z_{\zetabold}( \Mcal_L \varphi )\, , \, 0\big)+ 
\big( 0\, , \, Z_{\zetabold}( \mubold^{\sym}_{H, Q})\, ,\, - (\Fcal^\val_{\zetabold} )^{-1}( \mubold^{\asym}_{H, Q}) \big).
\end{gathered} 
\end{equation} 

We are now ready for the fixed-point argument.  
Clearly $B$ is convex.  Let $\beta' \in (0, \beta)$.  
The inclusion $B\hookrightarrow C^{2, \beta'}( M\llbracket \zeroboldunder\rrbracket)\times \domzb\times \domku$ is compact by the Ascoli-Arzela theorem. 
By inspecting the proofs of \ref{Plinear} and \ref{Plinear2}, 
it is easy to see that $\Rcal'_{M\llbracket \zetaboldunder \rrbracket}$ depends continuously on $\zetabold$, 
and from this, \ref{Azetabold}\ref{AZ} and \ref{Lrefdiff}(i), that $\Jcal$ is continuous in the induced topology.  

We next check that $\Jcal(B) \subset B$ by analyzing the action of $\Jcal$ on each factor of $B$.  By  (4) above and \ref{Lrefdiff} it follows that $\Jcal$ maps the first factor of $B$ to itself.  We see that $\Jcal$ maps the second and third factors of $B$ into themselves using Schauder estimates, \ref{Azetabold}\ref{AZ}, \eqref{Tfp}, (2) and (4) above, and by \ref{con:one}.

The Schauder fixed point theorem  \cite[Theorem 11.1]{gilbarg} now implies there is a fixed point $(\breve{v}, \zetaboldhatunder)$ of $\Jcal$.  
Using \eqref{Tfp} and the fixed point property in conjunction with \ref{Azetabold}\ref{AZ}, \ref{aK}, and \ref{cLker}, 
we see that $u_Q = \vbreve \circ \Fcal^{-1}_{\zetaboldhatunder}$ and $\breve{\mubold}^{\pm} =  0$, 
where we use ``$\breve{\phantom{a}}$'' to denote the various quantities for  $(v, \zetaboldunder) = (\vbreve, \breve{\zetaboldunder})$.  
By (5), we conclude the minimality of $\Mhat$.
The smoothness follows from standard regularity theory, and the embeddedness follows from \ref{Lquad}, (2), and (4).  

We now prove the existence of $\ubreve^\pm$ and $\Omegahat = \Sigma \setminus \bigsqcup_{p\in L}\Dbreve_p$ as in \eqref{EMguO} which satisfy the claimed properties.  
Consider the smooth function $f: \Mhat \rightarrow \R$ defined by $f = \langle \nu_{\Mhat} , \partial_{z}\rangle$, 
where $\partial_z$ is the gradient of the signed distance to $\Sigma$ as in \ref{dexp}.  
The set of points $p$ where $\Mhat$ fails to be graphical over $\Sigma$ in a neighborhood of $p$ is the level set $f^{-1}(0)$.  
From the smallness of $\upphihat$ and the geometry of $M\llbracket \zetaboldhatunder \rrbracket$, it is clear that $f^{-1}(0) \subset \bigsqcup_{p\in L} D^N_L(2\tau_p)$.  
Moreover, because the second fundamental form $A^{\Mhat}$ is nondegenerate in this neighborhood, 
it follows from the implicit function theorem that $f^{-1}(0)$ is a union of smooth curves whose projections under $\Pi_\Sigma$ bound smooth discs $\Dbreve_p \subset \Sigma$, 
$p\in L$ which are perturbations of $D^{\Sigma}_{p}(\tau_p)$.  
This discussion implies the existence of $\ubreve^\pm : \Omegahat \rightarrow \R$ satisfying \eqref{EMguO}.  
The claimed smoothness of $\ubreve^\pm$ follows from standard regularity theory, since $\Mhat$ is a minimal surface, 
and it follows from the embeddedness of $\Mhat$ that $\ubreve^+ + \ubreve^->0$ on $\Omegahat\setminus \partial \Omegahat$.   
Finally, the containment $D^\Sigma_p(\tau_p(1-\tau_p^{8/9})) \subset \Dbreve_p \subset D^\Sigma_p(\tau_p(1+\tau^{8/9}_p))$  follows from the estimate 
$\| \upphihat\|_{2, \beta, \gamma, \gamma'; M\llbracket \zetaboldhatunder \rrbracket } \le \tau_{\max}^{1+\alpha/4}$ along with \ref{D:norm}.  

It only remains to prove the area bound \eqref{E:Marea}; its proof will be broken up into several steps estimating the areas of different portions of $\Mhat$.  
We first estimate the area of a truncated catenoidal region.

\begin{lemma}
\label{LcatMarea}
For any $p \in  L$ and $\rtop : = \tau^{3/4}_p$, we have 
\begin{equation*}
|\Mhat \cap \Pi^{-1}_\Sigma ( D^\Sigma_p(\rtop))| = 
2 |D^\Sigma_p(\rtop)|- \pi \tau_p^2 
+ \frac{1}{2}\int_{\partial D^\Sigma_p(\rtop)}\left( \phicat^+ \frac{\partial \phicat^+}{\partial \eta} + \phicat^-\frac{\partial \phicat^-}{\partial \eta}\right)dl
+O(\tau^{5/2}_{p}|\log \tau_{p}|).
\end{equation*}
\end{lemma}

\begin{proof}
There is a domain $\cat_{\rtop} \subset \cat =  \cat[p, \tau_p, \kappaunder_p] \subset M\llbracket \zetaboldhatunder\rrbracket$ defined by requesting that 
$\Mhat \cap \Pi^{-1}_\Sigma ( D^\Sigma_p(\rtop)) = \graph^{N,g}_{\cat_{\rtop}}(\upphihat)$.
Using \ref{Lda}, it follows that
\begin{equation*}
|\Mhat \cap \Pi^{-1}_\Sigma ( D^\Sigma_p(\rtop))| = |\cat_{\rtop}| 
+ \frac{1}{2} \int_{\cat_{\rtop}} \left( |\nabla \upphihat |^2 - 2 \upphihat H - \upphihat^2 ( | A|^2 + \Ric (\nu, \nu)- H^2) \right) d\sigma +
 O(\tau^{3(1+\frac{\alpha}{4})}_{\max}),
\end{equation*}
where $A$ and $H$ are the second fundamental form and mean curvature of $\cat$ and we have used that $\| \upphihat\|_{2, \beta, \gamma, \gamma'; M\llbracket \zetaboldhatunder \rrbracket } \le 
\tau_{\max}^{1+\alpha/4}$ to estimate the error term.  From this last estimate, the estimate for $H$ in \ref{LglobalH}, and the definition of the global weighted norms in \ref{D:norm}, it follows that
\begin{equation*}
\begin{gathered}
\int_{\cat_{\rtop}} \frac{1}{2} |\nabla \upphihat |^2 d\sigma = O(\rtop \tau^{3-\frac{\alpha}{2}}_{\max}),
\quad
\int_{\cat_{\rtop}}  \upphihat H d\sigma = O(\rtop\tau^{3 - \frac{5}{12}\alpha}_{\max}), \\
 \int_{\cat_{\rtop}} \upphihat^2 ( | A^\cat|^2 + \Ric (\nu, \nu)- H^2) d\sigma 
 = O(\rtop^3\tau_{\max}^{3- \frac{\alpha}{2}}).
 \end{gathered}
\end{equation*}
The conclusion now follows from combining these estimates with the estimate on the area of $\cat(\rtop) = \cat \cap \Pi^{-1}_\Sigma(D^\Sigma_p(\rtop))$ from \ref{Lcatarea}, using the closeness of $\cat(\rtop)$ to $\cat_{\rtop}$, and using \ref{con:one}(iii). 
\end{proof}

\begin{lemma}
\label{Lgruest}
The functions $\ubreve^\pm$ satisfy the following. 
\begin{enumerate}[label=\emph{(\roman*)}]
\item $\big\| \ubreve^\pm : C^{2, \beta}(\Sigma \setminus \bigsqcup_{p \in L} D^{\Sigma}_p (\delta'_p))\big\|$ 
$\lem C \tau_{\max}|\log \tau_{\max}| + 
 C\| \varphi : C^{3, \beta}(\Sigma \setminus  \bigsqcup_{p \in L} D^{\Sigma}_p (\delta'_p))\|
 \leq C\tau^{8/9}_{\max}$.
\item  For each $p \in L$, $\| \ubreve^\pm - \phicat^\pm[\tau_p, \kappaunder_p]\|_{2, \beta, \gamma, \gamma'; D^\Sigma_p(\delta'_p)\setminus D^\Sigma_p(\tau^{7/8}_p)}\lem \tau_{\max}^{1+\alpha/4}$.
\end{enumerate}
\end{lemma}

\begin{proof}
Recall that 
$\Mhat = \graph^{N,g}_{M\llbracket \zetaboldhatunder\rrbracket }(\upphihat)$, that 
$\| \upphihat\|_{2, \beta, \gamma, \gamma'; M\llbracket \zetaboldhatunder \rrbracket } \lem \tau_{\max}^{1+\alpha/4} $, 
and with 
$\Omega = \Sigma \setminus \bigsqcup_{p\in L} D^\Sigma_p(\tau^{1/3}_p)$
we have from \ref{Dinit} that 
\begin{align*}
M\llbracket \zetaboldhatunder\rrbracket \cap \Pi^{-1}_\Sigma(\Omega)
=\graph^{N,g}_{\Omega}(\varphigl_+)
\cup
\graph^{N,g}_{\Omega}(-\varphigl_-). 
\end{align*}
Recall also from Lemma \ref{LMemb}(ii) and its proof that the estimate 
 \begin{equation*}
 \| \varphi_{\pm}^{gl} : C^{3, \beta}(\Sigma \setminus \bigsqcup_{p\in L} D^\Sigma_p(\delta'_p) , g)\| 
 \lem 
 C \tau_{\max}|\log \tau_{\max}| + 
 C\| \varphi : C^{3, \beta}(\Sigma \setminus  \bigsqcup_{p \in L} D^{\Sigma}_p (\delta'_p))\|
  \lem C \tau_{\max}^{8/9}
 \end{equation*}
  holds.  Using these estimates and the smallness assumptions on $\tau_{\max}$, we can apply \ref{Cwgraph} to conclude (i).  
Items (ii) also follows from \ref{Cwgraph}, combined with the definition of the norms in \ref{D:norm}, 
the fact that $\varphigl_\pm = \phicat^\pm[\tau_p, \kappaunder_p]$ on $D^\Sigma_p(\delta'_p)\setminus D^\Sigma_p(\tau^{7/8}_p)$, and \ref{Ltcest}.
\end{proof}

\begin{lemma}
\label{Lextarea}
The following estimate holds. 
\begin{multline*}
| \Mhat \cap \Pi^{-1}_\Sigma(\Sigma \setminus \bigsqcup_{p\in L} D^\Sigma_{p}(\tau^{3/4}_p))| 
= 2 |\Sigma|
 - 2\sum_{p\in L} | D_p^\Sigma( \tau_p^{3/4})| 
 \\
 -\frac{1}{2}\sum_{p\in  L}\int_{\partial D^\Sigma_p(\tau^{3/4}_p)}\left( \phicat^+ \frac{\partial \phicat^+}{\partial \eta} 
 + \phicat^-\frac{\partial \phicat^-}{\partial \eta}\right)ds 
 +\sum_{p\in L} O( \tau_p^{5/2}|\log \tau_p|). 
\end{multline*}
\end{lemma}
\begin{proof}
By applying \ref{LAomega} with $u = \ubreve^+$ on $\Omega_1 :  = \Sigma \setminus \bigsqcup_{p\in L} D^{\Sigma}_p(\delta'_p) $, we have 
\begin{align}
\label{Egraphuu1}
| \graph^{N, g}_{\Omega_1}( \ubreve^+)| = |\Omega_1 |
- \frac{1}{2} \int_{\Omega_1} \ubreve^+ \Lcal_\Sigma u^+ d \sigma
+ \int_{\partial \Omega_1} \ubreve^+ \frac{\partial u^+}{\partial \eta}ds
+O (\tau_{\max}^{8/3}),
\end{align}
where we have used the minimality of $\Sigma$ and \ref{Lgruest}(i) to estimate the error terms.  
By applying \ref{LAomega} with $u= \ubreve^+$ on $\Omega_2 : = \bigsqcup_{p\in L} ( D^\Sigma_p(\delta'_p)\setminus D^\Sigma_p(\tau^{3/4}_p))$, we have
\begin{multline}
\label{Egraphuu2}
| \graph^{N, g}_{\Omega_2}( \ubreve^+)| 
= |\Omega_2 |
- \frac{1}{2} \int_{\Omega_2} \ubreve^+ \Lcal_\Sigma \ubreve^+ d \sigma 
\\
+ \frac{1}{2} \sum_{p\in L} \int_{\partial D^{\Sigma}_p(\tau^{3/4}_p)} \ubreve^+ \frac{\partial \ubreve^+}{\partial \eta}ds+ 
- \int_{\partial \Omega_1} \ubreve^+ \frac{\partial u^+}{\partial \eta}ds+ 
 O(|L| \tau^3_{\max}|\log \tau_{\max}|^2), 
\end{multline}
where we have used \ref{Lgruest}(ii) to estimate the error term.

We now estimate the integrals of $\ubreve^+ \Lcal_\Sigma \ubreve^+$. 
From the minimality of $\Sigma$ and $\Mhat$, it follows that
on $\Sigma \setminus \bigsqcup_{p \in L} D^\Sigma_p(\tau^{3/4}_p)$
\begin{align}
\label{ELuquad}
|\Lcal_\Sigma \ubreve^\pm | \lem C|\ubreve^\pm|^2 + C|\nabla \ubreve^\pm|^2+ C|\nabla^2 \ubreve^\pm| \, \big( |\ubreve^\pm| + |\nabla \ubreve^\pm|^2\big)
\end{align}
(notice there are no $|\nabla^2 \ubreve^\pm|^2$ or $|\nabla^2 \ubreve^\pm| |\nabla \ubreve^\pm|$ terms; see e.g. \cite[Lemma C.2]{kapouleas:annals} or \cite[Appendix A]{mantoulidis:annals}).
Working this into \eqref{Egraphuu1} and estimating using \ref{Lgruest}(i) reveals that
\begin{align}
\label{Egraphuu3}
| \graph^{N, g}_{\Omega_1}( \ubreve^+)| = |\Omega_1 |
+ \int_{\partial \Omega_1} \ubreve^+ \frac{\partial \ubreve^+}{\partial \eta}ds
+O (\tau_{\max}^{8/3}).
\end{align}
A similar estimate of $\int_{\Omega_2} \ubreve^+ \Lcal_\Sigma \ubreve^+ d\sigma$, using \ref{Lgruest}(ii) to estimate, reveals that
\begin{equation}
\label{Egraphuu4}
| \graph^{N, g}_{\Omega_2}( \ubreve^+)| 
= |\Omega_2 |
+ \frac{1}{2} \sum_{p\in L} \int_{\partial D^{\Sigma}_p(\tau^{3/4}_p)} \ubreve^+ \frac{\partial \ubreve^+}{\partial \eta}ds 
- \int_{\partial \Omega_1} \ubreve^+ \frac{\partial \ubreve^+}{\partial \eta}ds+ 
\sum_{p\in L} O(\tau^{5/2}_p |\log \tau_p|). 
\end{equation}
Since $\Omega: = \Sigma \setminus \bigsqcup_{p\in L} D^\Sigma_p(\tau^{3/4}_p)$ is the disjoint union of $\Omega_1$ and $\Omega_2$, adding the estimates \eqref{Egraphuu3} and \eqref{Egraphuu4} implies that
\begin{equation}
\label{Euhal}
\big| \graph^{N, g}_{\Omega}( \ubreve^+)\big| 
=
|\Sigma| 
- \sum_{p\in L} \big| D_p^\Sigma( \tau_p^{3/4})\big| 
-\frac{1}{2} \sum_{p\in L} \int_{\partial D^\Sigma_p(\tau^{3/4}_p)} \ubreve^+ \frac{\partial \ubreve^+}{\partial \eta} ds 
+\sum_{p\in L} O\big(\tau^{5/2}_p | \log \tau_p|\big).
\end{equation}
Next, using \ref{Lgruest}(ii) and \ref{Ltcest}, it follows for any $p \in L$ that
\begin{align}
\label{Ebuph}
\int_{\partial D^{\Sigma}_p(\tau^{3/4}_p)} \ubreve^+ \frac{\partial \ubreve^+}{\partial \eta} ds 
-\int_{\partial D^{\Sigma}_p(\tau^{3/4}_p)} \phicat^+ \frac{\partial \phicat^+}{\partial \eta} ds 
=O(\tau^{2+7/8-\alpha/4}_{\max}|\log \tau_{\max}|). 
\end{align}
The conclusion follows by combining \eqref{Ebuph} with \eqref{Euhal} and the completely analogous estimates for $| \graph^{N, g}_{\Omega}( -\ubreve^-)|$, and using \ref{con:one}(iii) to estimate the error terms involving $\tau_{\max}$ in terms of $\tau_p$.
\end{proof}

The proof of \eqref{E:Marea} follows now by adding the estimates provided by \ref{LcatMarea} for each $p\in L$ on 
$|\Mhat \cap \Pi_\Sigma^{-1}(D^\Sigma_p(\tau^{3/4}_p))|$,  
to the estimate in \ref{Lextarea}, and noting in particular that the boundary terms cancel.  This completes the proof of Theorem \ref{Ttheory}.
\end{proof}

The following observation which follows from \ref{cLker} will be useful in constructing and studying LD solutions. 

\begin{lemma}[Existence and uniqueness for LD solutions {\cite[Lemma 3.10]{kap}}] 
\label{Lldexistence}
Given finite 
${L\subset \Sigma}$ and a function $\taubold : L\rightarrow \R$, 
there exists a unique LD solution $\varphi = \varphi[\taubold]$ of singular set 
$L':=\{p\in L: \tau_p\ne0\}$ and configuration $\left. \taubold\right|_{L'}$.  Moreover, $\varphi$ depends linearly on $\taubold$. 
\end{lemma}

\begin{proof}
We define $\varphi_1 \in C^\infty(\Sigma \setminus L')$ by requesting that it is supported on $\bigsqcup_{p\in L'} (D^\Sigma_p(2\delta_p))$ 
and $\varphi_1 = \Psibold[ \delta_p, 2\delta_p; \dbold^\Sigma_p]( \tau_p G_p, 0)$ on $D^\Sigma_p(2\delta_p)$ for each $p\in L'$.  
Note that $\Lcal_\Sigma \varphi_1 \in C^\infty(\Sigma)$ (by assigning $0$ values on $L'$) and it is supported on 
$\bigsqcup_{p\in L'} (D^\Sigma_p(2\delta_p)\setminus D^\Sigma_p(\delta_p))$.  
Using \ref{cLker}, there is a unique $\varphi_2 \in C^\infty(\Sigma)$ such that $\Lcal_\Sigma \varphi_2 = - \Lcal_\Sigma \varphi_1$.  
We then define $\varphi = \varphi_1+\varphi_2$ and the conclusion follows. 
\end{proof}

\begin{remark}[Index, nullity, and characterizations of minimal doublings] 
\label{R:uniqueness} 
It is interesting that currently no characterizations of doublings are known, even under strong assumptions, 
for example given bridge positions and any further information. 
This means that even a small modification in the construction process would lead in principle to different minimal surfaces, 
even though the new ones would strongly resemble the previous ones. 
The only known such characterizations for minimal surfaces in the round three-sphere are for Lawson surfaces \cite{KW:Luniqueness}. 

Since it seems very likely that surfaces strongly resembling each other by their constructions are actually congruent, 
it is customary in the literature to discuss them as if they were known to be. 
In this article we also adhere to this and we consider the doublings in Remark \ref{R:oldT} the same with 
the ones in \cite{Wiygul} and also (for square lattices) with the ones constructed in \cite{kapouleas:clifford} or \cite{Ketover};  
similarly surfaces constructed as in Remark \ref{RSph2} in the case all $m_i=m$ with surfaces constructed in \cite{kapmcg}.
Proving however that such surfaces are congruent remains at the moment an interesting open problem.  
We hope that eventually index, nullity and characterization results will be provided for the surfaces constructed in Theorem \ref{Ttheory}, 
similarly to the results in \cite{k8,KW:Luniqueness}, 
and with the same generality as in the area estimate \eqref{E:Marea}. 
\hfill $\square$ 
\end{remark}

\section[New minimal surfaces via doubling the Clifford Torus]{New minimal surfaces via doubling the Clifford Torus} 
\label{S:clifford}
\nopagebreak

\subsection*{Symmetries and LD solutions}
\nopagebreak

Let $\T := \left\{ (z_1, z_2) \in \C^2: |z_1| = |z_2| = 1/\sqrt{2}\right\} \subset \Sph^3 \subset \C^2$ be the Clifford torus in the unit three-sphere $(\Sph^3,g)$.  
We recall that doublings of the Clifford torus with catenoidal bridges centered at the points of a square $m\times m$ (large $m\in\N$) lattice $L\subset \T$ were 
first constructed in \cite{kapouleas:clifford};  
this was extended in \cite{Wiygul} to rectangular lattices $k\times m$ (large $k,m\in \N$ and a priori bounded $m/k$).  
These results can easily be reproduced by constructing the required LD solutions and applying Theorem \ref{Ttheory} (see Remark \ref{R:oldT}).  
Our main focus in this section however is to construct new doublings in the following cases: 
first, when the necks are centered at the points of a lattice with $m/k$ \emph{not} constrained 
(see \ref{Amkcliff}) 
and second, less symmetric doublings where there are three different bridges up to symmetries.
These new constructions are only indicative of the possibilities and many more are carried out in \cite{douT} with other symmetry groups or more necks per fundamental domain. 

We briefly recall now some notation from \cite{kapouleas:clifford, Wiygul}.    
Given an oriented circle $C$ in $\Sph^3$, write $\Rcap^\theta_C$ for the rotation by $\theta$ about $C$.  
Define the circles $C: = \{z_2 = 0\}$ and $C^\perp: = \{ z_1 = 0\}$.  
We have
\begin{align*}
\Rcap^\theta_{C^\perp}(z_1, z_2) = (e^{i\theta} z_1, z_2), \quad \text{and} \quad
\Rcap^\theta_{C}(z_1, z_2) = (z_1, e^{i\theta} z_2).
\end{align*}
Define the following symmetries of $\C^2$ and the domain of the coordinates $(\xx, \yy, \zz)$ defined in  \ref{exClifford}:
\begin{equation}
\begin{gathered} 
\xbartilde(z_1, z_2) = (\overline{z_1}, z_2), \quad
\ybartilde(z_1, z_2) = (z_1, \overline{z_2}), \quad
\zbartilde(z_1, z_2) = (z_2, z_1),\\
\xbar^h(\xx, \yy, \zz)=(\xx+h , \yy, \zz), \quad
\ybar^h(\xx, \yy, \zz) = (\xx, \yy+h, \zz) \quad \forall h \in \R, \\
\widehat{\xbartilde}(\xx, \yy, \zz) = (-\xx, \yy, \zz), \quad
\widehat{\ybartilde}(\xx, \yy, \zz) = (\xx, -\yy, \zz), \quad 
\widehat{\zbartilde}(\xx, \yy, \zz) = (\yy, \xx, -\zz). 
\end{gathered}
\end{equation}
With $\tildeE$ the parametrization map in \ref{exClifford},  
these satisfy the relations 
\begin{equation}
\begin{gathered} 
\tildeE \circ \widehat{\xbartilde} = \xbartilde \circ \tildeE , \quad
\tildeE \circ \widehat{\ybartilde} = \ybartilde \circ \tildeE, \quad
\tildeE \circ \widehat{\zbartilde} = \zbartilde \circ \tildeE, \\
\tildeE \circ \xbar^h= \Rcap^{\sqrt{2} h}_{C^\perp} \circ \tildeE 
\quad \text{and} \quad
\tildeE \circ  \ybar^h= \Rcap^{\sqrt{2} h}_{C} \circ \tildeE  
\quad \forall h \in \R.
\end{gathered}
\end{equation}

\begin{assumption}
\label{Amkcliff}
We fix $k, m\in \N$ with $k\geq 3$, $m\geq k$, and  assume $m$ is as large as needed in absolute terms.
\end{assumption}

We define the symmetry group $\group$, a point $p_0\in \Tor$, a lattice $L$, and set of parallel circles $\Lpar$ by 
\begin{equation}
\label{Egsymcl}
\begin{gathered}
\group = \group[k, m] := \langle \Rcap^\frac{2\pi}{k}_{C^\perp}, \Rcap^{\frac{2\pi}{m}}_{C}, \xbartilde, \ybartilde \rangle , \quad
p_0 : = \frac{1}{\sqrt{2}}(1, 1) = \tildeE(0, 0, 0), \\ 
L = L[k, m]: =  \group p_0,
 \quad \text{and} \quad \Lpar : = \{ \Rcap^{i\frac{2\pi}{k}}_{C^\perp} \Rcap^\theta_C p_0: \theta\in \R, i \in \mathbb{Z} \}.
 \end{gathered}
\end{equation}
If $X$ is a function space consisting of functions defined on a domain
$\Omega\subset \Tor$ and $\Omega$ is invariant under the action of $\group$,
we use a subscript ``$\sym$'' to denote the subspace $X_\sym\subset X$
consisting of those functions  $f\in X$ which are invariant under the action of $\group$.

The linearized operator is $\Lcal_{\T} = \Delta_{\T} + 4$, and it is easy to see that $(\ker \Lcal_{\T})_\sym$ is trivial.  
By Lemma \ref{Lldexistence} there is therefore a unique $\group$-symmetric LD solution $\Phi = \Phi[k,m]$ with singular set $L$ and satisfying $\tau_p = 1$ $\forall p \in L$. 
For convenience, we define the scaled metric,  scaled linear operator, and scaled coordinates $(\xxtilde,\yytilde)$ on $\T$ by 
\begin{align}
\label{Egtildecliff}
\gtilde := m^2 g, \qquad
\Lcaltilde_{\T} := \frac{1}{m^2} \Lcal_{\T} = \Delta_{\gtilde}+ \frac{4}{m^2} , \qquad
(\xxtilde, \yytilde) := m(\xx, \yy).
\end{align}
We define $\delta = 1/(100m )$ and for $p\in L$ define $\delta_p = \delta$.  

\begin{remark}[Applying the LD approach in the cases of \cite{kapouleas:clifford, Wiygul}]
\label{R:oldT} 
We first sketch the construction of the required LD solutions $\Phi$. 
Integrating $\Lcal_{\T} \Phi = 0$ over $\T$ and integrating by parts, we find $\frac{km}{4\pi}= \frac{1}{|\T|} \int_\T \Phi$. 
Define $\Ghat \in C^\infty_\sym(\T\setminus L)$ by requesting that $\Ghat$ is supported on $D^\T_L(3\delta)$ and  satisfies there 
$\Ghat = \Psibold[2 \delta, 3\delta; \dbold^\T_p]( G_p - \log \delta, 0)$ (where $G_p$ is a Green's function for $\Lcal_{\T}$ as in \ref{dggen}) and define $\Phip \in C^\infty_\sym(\T)$ 
by requesting  $\Phi = \Ghat + \frac{km}{4\pi} + \Phip$.  
From this decomposition, estimates on the average and oscillatory parts of $\Ghat$, and the uniform boundedness of 
$m/k$, we conclude that $\Lcaltilde_{\T}$ has no small eigenvalues when restricted to functions that have average zero, hence $\| \Phip : C^j_\sym(\T, \gtilde) \| \le C(j)$.

For some $\cunder$ fixed independently of $m$, define now $\domzb : = [ -\cunder, \cunder] \subset \Pcal : = \R$, 
LD solutions $\varphi = \varphi\llbracket \zeta\rrbracket : = \tau \Phi : =  \frac{1}{m} e^\zeta e^{- \frac{km}{4\pi}} \Phi$ for $\zeta \in \domzb$, 
and $\val_\sym[L]$ the subspace of $\val[L]$ consisting of the $\group$-invariant elements.  
Clearly $\val_\sym[L]$ is one dimensional and may be identified with $\R$.  
Using the definition of $\tau$ and the estimate on $\Phip$, 
it follows that the map $Z_\zeta : \val_\sym[L] \rightarrow \Pcal$ defined by $Z_\zeta(\mu) = \frac{1}{\tau}\mu $ satisfies
$|\zeta - Z_\zeta(\Mcal_L\varphi)| \leq C$ for a constant $C$ independent of $\cunder$.  
After restricting to spaces adapted to the symmetries and choosing $\cunder$ to be large enough in terms of $C$, we can then apply Theorem \ref{Ttheory} because the remaining assumptions are easy to check. 
\qed
\end{remark}

\begin{definition}
\label{davgcliff}
Given a function $\varphi$ on some domain $\Omega\subset \T$, we define
a rotationally invariant function $\varphi_\ave$ on 
the union $\Omega'$ of the orbit circles of $\{ \Rcap^\theta_C: \theta \in \R\}$ on which $\varphi$ is integrable
(whether contained in $\Omega$ or not),
by requesting that on each such circle $C'$, 
\[\left. \varphi_\ave \right|_{C'}
:=\avg_{C'}\varphi.\]
We also define $\varphi_\osc$ on $\Omega\cap\Omega'$ by $\varphi_\osc:=\varphi-\varphi_\ave$.
\end{definition}

\begin{lemma}
\label{LVBcliff}
$\Phi_\ave = \frac{m}{2 \sqrt{2} \sin(\sqrt{2}\frac{\pi}{k})} \cos\big( \sqrt{2} \frac{\pi}{k} - 2\dbold^{\T}_{\Lpar}\big)$ and $\left. \Phi_\ave\right|_{\Lpar} = \frac{m}{2F}$, where 
$F: = \sqrt{2} \tan \left( \sqrt{2} \frac{\pi}{k}\right)$.
\end{lemma}
\begin{proof}
Since $\Lcal_{\T} \Phi_\ave = 0$ on $\T\setminus \Lpar$ and the distance between neighboring circles of $\Lpar$ is $\sqrt{2}\pi/k$, the symmetries imply that $\Phi_\ave = C  \cos\big( \sqrt{2} \frac{\pi}{k} - 2\dbold^{\T}_{\Lpar}\big)$ for some $C\neq 0$.  By integrating $\Lcal_{\T} \Phi =0$ on $\Omega_{\epsilon_1, \epsilon_2} := D^{\T}_{\Lpar}(\epsilon_2) \setminus D^{\T}_{L}(\epsilon_1)$, where $0< \epsilon_1<< \epsilon_2$ and integrating by parts, we obtain
\begin{align*}
\int_{\partial \Omega_{\epsilon_1, \epsilon_2}} \frac{\partial}{\partial \eta} \Phi + \int_{\Omega_{\epsilon_1, \epsilon_2}} 4\Phi  = 0,
\end{align*}
where $\eta$ is the unit outward conormal field along $\partial \Omega_{\epsilon_1, \epsilon_2}$.  By taking the limit as $\epsilon_1 \searrow 0$ first and then as $\epsilon_2\searrow 0$, we obtain by  using the logarithmic behavior near $L$ and the preceding that
\begin{align*}
2\pi m =
4 \sqrt{2}\pi C \sin\left(\sqrt{2}\frac{\pi}{k}\right),
\end{align*}
which implies the conclusion.
\end{proof}

We introduce now the following decomposition.
Note that we could assume (but is not necessary) that $G_p$ is rotationally invariant, 
in which case it is uniquely determined and can be expressed in terms of Bessel functions. 

\begin{definition}
\label{Dcliffdecomp2}
Define $\Ghat \in C^\infty_\sym(\T \setminus L)$ and $\Phat, \Phip, E' \in C^\infty_\sym(\T)$ by requesting that 
\begin{equation*}
\begin{gathered}
\Ghat = \Psibold[2 \delta, 3\delta; \dbold^{\T}_p]( G_p - \log \delta \cos (2 \dbold^{\T}_{\Lpar}), 0) \quad \text{on} \quad D^{\T}_L(3 \delta), \\
\Phat =\Phi_\ave -  \Psibold \bigg[ \frac{2}{m}, \frac{3}{m}; \dbold^{\T}_{\Lpar} \bigg]\left(\frac{m}{2\sqrt{2}}\sin (2 \dbold^\T_{\Lpar}) , 0\right) \quad \text{on} \quad D^{\T}_{\Lpar}(3/m),
\end{gathered}
\end{equation*}
 that $\Ghat = 0$ on $\T \setminus D^{\T}_L(3\delta)$,  $\Phat = \Phi_\ave$ on $\T \setminus D^{\T}_{\Lpar}(3/m)$, and
 \begin{align}
\label{EPhdcliff2}
\Phi = \Ghat + \Phat + \Phip, \quad
E' = -\Lcaltilde_{\T}( \Ghat + \Phat).
\end{align}
\end{definition}

\begin{remark}
Note that from Lemma \ref{LVBcliff} and the fact that
\[
\cos \left( \sqrt{2} \frac{\pi}{k} - 2 \dbold^\T_\Lpar\right) = \cos \left( \sqrt{2} \frac{\pi}{k}\right) \cos \left( 2 \dbold^\T_{\Lpar}\right) + \sin \left( \sqrt{2} \frac{\pi}{k}\right) \sin \left( 2 \dbold^\T_\Lpar\right)
\] 
that $\Phat$ as defined in \ref{Dcliffdecomp2} is indeed smooth across $\Lpar$.
\qed 
\end{remark}

We estimate the average and oscillatory parts of $\Phi$ separately.

\begin{lemma}
\label{LPhipcliff2}
$E'$ vanishes on $D^\T_{L}(2\delta)$ and $E'_{\osc}$ is supported on $D^\T_{\Lpar}(3\delta)$.  Moreover: 
\begin{enumerate}[label=\emph{(\roman*)}]
\item $\| \Ghat : C^j_\sym(\T \setminus D^\T_L(\delta), \gtilde)\| \leq C(j)$. 
\item $\| E' : C^j_\sym(\T , \gtilde ) \| \le C(j)$. 
\item $\| \Phip : C^j_\sym(\T, \gtilde) \| \le C(j)$. 
\end{enumerate}
In (ii), the same estimate holds if $E'$ is replaced with either $E'_{\ave}$ or $E'_{\osc}$. 
\end{lemma}
\begin{proof}
Because $\Ghat$ is supported on $D^\T_{L}(3\delta)\setminus L$, (i) follows using \eqref{Egreenlog}
and Definition \ref{Dcliffdecomp2}.
The statements on the support of $E'$ and $E'_\osc$ follow from Definition \ref{Dcliffdecomp2}, from which we also see that 
$E' = \Lcaltilde_{\T} \Psibold[2, 3; \dbold^{\T, \gtilde}_{\Lpar}]( \frac{m}{2\sqrt{2}} \sin(2 \dbold^\T_{\Lpar}), 0)$
on $D^\T_{\Lpar}(3/m)\setminus D^\T_{\Lpar}(2/m)$.  
Thus, when restricted to this set, the bound in (ii) follows from the uniform bounds on the cutoff in the $\gtilde$ metric.  It follows from \ref{Dcliffdecomp2} that $E'$ vanishes on $D^{\T}_{\Lpar}(2/m)\setminus D^\T_L(3 \delta)$.  On $D^\T_{\Lpar}(3\delta)$, note that $\Lcaltilde_{\T} \Phat =0$.  Since $\Lcaltilde_{\T} \Ghat = 0$ on $D^\T_L(2\delta)$, when restricted to $D^\T_{\Lpar}(3\delta)$ the required bound in (ii) follows from (i).  Finally, we can replace $E'$ by $E'_\ave$ or $E'_\osc$ in (ii) by taking averages and subtracting. 

To prove (iii) it suffices to prove that the estimate holds when $\Phip= \Phip_\ave + \Phip_\osc$ is replaced by either $\Phip_\ave$ or $\Phip_\osc$.  We first prove the estimate for $\Phip_\ave$.  Note by \ref{Dcliffdecomp2} that on $D^\T_{\Lpar}(2/m)$, 
\begin{align*}
\Phip_\ave = \frac{m}{2\sqrt{2}} \sin ( 2 \dbold^\T_{\Lpar}) - \Ghat_\ave.
\end{align*}
Note that the left hand side is smooth and the discontinuities on the right hand side cancel.  Using that  $\Lcaltilde_{\T} \Phip_\ave = E'_\ave$, on $D^\T_{\Lpar}(3/m)$ we have 
\begin{align}
\label{Eavcliff}
\partial^2_{\widetilde{\xx}} \Phip_\ave + \frac{4}{m^2} \Phip_\ave = E'_\ave,
\end{align}
where $\widetilde{\xx} : = m \xx$.  On a neighborhood of $\partial D^\T_{\Lpar}(2/m)$, we have that $\Ghat_\ave =0$ from Definition \ref{Dcliffdecomp2}.  It follows that $\left|\Phip_\ave\right|<C$ and $\left| \partial_{\, \widetilde{\xx}} \, \Phi'_\ave\right|<C$ on $\partial D^\T_{\Lpar}(2/m)$.   
Using this as initial data for the ODE and bounds of the inhomogeneous term from (ii) yields the $C^2$ bounds on $\Phip_\ave$ in (iii).  Higher derivative estimates follow inductively from differentiating \eqref{Eavcliff} and again using (ii).

This establishes the bound on $D^\T_{\Lpar}(2/m)$, and the proof of the estimate on $D^\T_{\Lpar}(3/m) \setminus D^\T_{\Lpar}(2/m)$ is even easier since $\Ghat_\ave = 0$ there, so we omit the details. 

We now estimate $\Phip_\osc$.  For $n\in \N$ and $l \in \{0, 1, \dots\}$, we define $\phi_{\ell, n} \in C^\infty_{\sym}(\T)$ by 
\begin{align*}
\phi_{\ell, n}(\xx, \yy) : = \cos \bigg( \frac{\ell}{\sqrt{2}} \frac{k}{m} \xxtilde \bigg) \cos \bigg( \frac{n}{\sqrt{2}} \yytilde \bigg).
\end{align*} 
Clearly $\{ \varphi_{\ell, n} : n\in \N, \ell\geq 0\}$ is a complete orthogonal set for the subspace of $(L^2(\T))_\sym$ consisting of functions with zero average; moreover,  
\begin{align*}
\Lcaltilde_{\T} \phi_{\ell, n} =  \lambda_{\ell, n} \phi_{\ell, n}, 
\quad \text{where} \quad
\lambda_{\ell, n}: = - \frac{1}{2} \left( \frac{k^2}{m^2}\ell^2 + n^2 \right) + \frac{4}{m^2}.
\end{align*}
For appropriate coefficients $E'^{, \ell,n}_{\osc}$, we have then
\begin{align*}
E'_\osc = \sum_{\ell\geq 0, n\in \N} E'^{, \ell, n}_{\osc} \phi_{\ell, n}, \quad
\text{and} \quad
\Phip_\osc =  \sum_{\ell, n}  \lambda^{-1}_{\ell, n} E'^{, \ell, n}_\osc \phi_{\ell,n}
\end{align*}
since $\Lcaltilde_{\T} \Phip_\osc = E'_\osc$.
Since $n \geq 1$, $\lambda^{-1}_{\ell,n}$ is bounded by a constant independent of $\ell, n, m$, and $k$, provided $m$ is large enough. The required bound on $\Phip_\osc$ now follows from the bound in (ii).
\end{proof}

\subsection*{Configurations with a single singularity modulo symmetries}
\nopagebreak 

\begin{definition}[Obstruction spaces] 
\label{dkercliff}
Let $\skernelv_{\sym}[L], \val_\sym[L]$ be the subspaces of $\skernelv[L], \val[L]$ consisting of the $\group$-invariant elements, where $\skernelv[L] = \bigoplus_{p\in L} \skernelv[p]$, and $\forall p\in L$, $\skernelv[p]$ is defined as in \ref{R:p}. 
\end{definition}
Since $\group$ is generated by reflections, $\val_\sym[L]$ is one-dimensional and may be identified with $\R$. 

 For some $\cunder>0$ fixed independently of $m$, define $ \domzb : = [ -\cunder, \cunder] \subset \Pcal : = \R$ and LD solutions
\begin{align}
\label{Etaucliff2}
 \varphi : = \varphi \llbracket \zeta \rrbracket = \tau \Phi: =  \frac{1}{m} e^{\zeta} e^{-\frac{m}{2F}} \Phi ,
\quad
\zeta \in \domzb. 
\end{align}

\begin{prop}
\label{Pclifford1}
There is an absolute constant $C$ (independent of $\cunder$) such that for $m$ large enough (depending on $\cunder$), the map $Z_\zeta : \val_\sym\llbracket\zeta\rrbracket \rightarrow \Pcal$ defined by $Z_\zeta(\mu) = \frac{1}{\tau}\mu $ satisfies
$|\zeta - Z_\zeta(\Mcal_L\varphi)| \leq C$.
\end{prop}
\begin{proof}
For any $p\in L$, expanding $\frac{1}{\tau} \Mcal_p \varphi$ (recall \ref{Dmismatch}) using \ref{Rmismatch} and \ref{EPhdcliff2}, we find
\begin{align*}
 \frac{1}{\tau} \Mcal_p \varphi= \frac{m}{2F} + \log\left( \frac{\tau}{2\delta}\right) + \Phip(p) = \zeta + \Phip(p)+ \log(50),
\end{align*}
where the second equality uses \eqref{Etaucliff2}.  The conclusion follows from using \ref{LPhipcliff2}(iii) to estimate $\Phip(p)$. 
\end{proof}

\begin{theorem}
\label{Tcmain1}
There exists an absolute constant $\cunder>0$ such that for all $(k, m)\in \N^2$ satisfying \ref{Amkcliff} and $m$ large enough in terms of $\cunder$, there exists a genus $mk+1$, $\group[k, m]$-invariant doubling of $\T$ by applying Theorem \ref{Ttheory}.
\end{theorem}
\begin{proof}
After the obvious trivial modifications to Theorem \ref{Ttheory} and its proof to restrict to $\group$-symmetric data, 
we need only to check that 
\ref{Azetabold} holds. 

It was noted above that \ref{cLker} holds in the space of $\group$-symmetric functions.  Define diffeomorphisms  $\Fcal^\Tor_\zeta : \T \rightarrow \T$ as in \ref{Azetabold}\ref{Idiffeo} by $\Fcal^\Tor_\zeta = \text{Id}_\T$.  $L\llbracket \zeta\rrbracket$, $\taubold\llbracket \zeta\rrbracket$, and $\varphi \llbracket \zeta\rrbracket$ as in \ref{Azetabold}(ii)-(iv) were defined in \ref{Egsymcl} and \eqref{Etaucliff2}.  Next, $\delta_p\llbracket \zeta\rrbracket = 1/(100m)$ as in \ref{Azetabold}(v) was defined earlier, and the spaces $\skernel_\sym[L], \widehat{\skernel}_\sym[L]$, and $\val_\sym[L]$ defined in \ref{dkercliff} satisfy \ref{aK}, verifying \ref{Azetabold}(vi).  Finally, isomorphisms $Z_\zeta$ as in \ref{Azetabold}(vii) were defined in \ref{Pclifford1}.

We now check \ref{Azetabold}\ref{Aa}-\ref{AZ}: \ref{Aa}-\ref{Ab} hold trivially.  
For \ref{Azetabold}\ref{Ac} we must verify that \ref{con:one} holds: Convention \ref{con:L} clearly holds for all large enough $m$. 
Because $k\geq 3$, $F = \sqrt{2}\tan\left( \sqrt{2} \frac{\pi}{k}\right)> 0$, 
and consequently $\tau$ in \eqref{Etaucliff2} can be made as small as needed by taking $m$ large.
Then \ref{con:one}(ii)-(iii) follow immediately using that $\forall p\in L$, $\tau_p = \tau$ and $\delta_p =  1/(100 m)$, where $\tau$.
Because $k\geq 3$, we have $\frac{\sqrt{2}\pi}{k}< \frac{\pi}{2}$ and consequently from Lemma \ref{LVBcliff} that $\Phi_\ave>0$.  In particular, it follows from \ref{LVBcliff} that $\Phi_{\ave}>cmk$ for some $c> 0$.  The estimates in \ref{con:one}(iv)-(vi) now follow easily using that $\varphi = \tau \Phi$, the decomposition of $\Phi$ in  \ref{Dcliffdecomp2}, and the estimates on $\Ghat$ and $\Phip$ in \ref{LPhipcliff2}.  This completes the verification of \ref{Azetabold}\ref{Ac}. 

Next, \ref{Azetabold}\ref{Atau} holds trivially since $\tau_p =\tau$  $\forall p \in L$, where $\tau$ is as in \eqref{Etaucliff2}.
\ref{Azetabold}\ref{AZ} holds by \ref{Pclifford1} by taking $\cunder$ large enough.  This completes the proof.
\end{proof}

\begin{remark}[The cases where $k=1$ and $k=2$]
\label{R:k12} 
In the proof of \ref{Tcmain1} we used that $k\geq 3$ (recall \ref{Amkcliff})---which implies that $\Phi_\ave>0$ to verify that \ref{con:one}(vi) holds.  While  \ref{con:one}(vi) is necessary in Theorem \ref{Ttheory} to ensure the embeddedness of the resulting surfaces, a modified version of \ref{Ttheory} holds---without requiring \ref{con:one}(vi)---which produces immersed doublings.  This modified theorem produces immersed doublings when $k=1$: to see this, note that although when $k=1$ we no longer have $\Phi_\ave>0$ and consequently \ref{con:one}(vi) does not hold, the rest of \ref{Azetabold} holds.  In particular, $\tau$ can still be made arbitrarily small by taking $m$ large since $F = \sqrt{2} \tan (\sqrt{2}\pi)> 0$.  On the other hand, the construction fails when $k=2$ because $F  = \sqrt{2} \tan (\sqrt{2}/2 \pi)< 0$ and $\tau$ cannot be made as small as needed. 
\qed 
\end{remark}

\subsection*{Configurations with three singularities modulo symmetries}
\nopagebreak 

In this subsection we construct and estimate $\group$-symmetric LD solutions on $\Tor$   
(recall \eqref{Egsymcl}) which have three singularities on each fundamental domain, 
and apply Theorem \ref{Ttheory} to construct corresponding minimal surfaces.  
To simplify the estimates, we assume in this subsection that $m/k < C_1$ for a fixed constant $C_1>0$.  
To begin, for $p_0$ as defined in \eqref{Egsymcl}, define
\begin{align*}
 \begin{gathered}
p_1:= \Rcap^{\frac{\pi}k}_{C^\perp} p_0, \quad
p_2: = \Rcap^{\frac{\pi}{m}}_{C}p_0, \quad 
L = L[k,m]:   = \bigcup_{i=0}^2 L_i : = \bigcup_{i=0}^2 \group p_i
\end{gathered}
\end{align*}
and define for $i=0, 1, 2$ the $\group$-invariant LD solution $\Phi_i = \Phi_i[k,m]$ satisfying $\tau_p = 1$ $\forall p\in L_i$. 

\begin{figure}[h]
	\centering
	\begin{tikzpicture}
	\draw[-, dashed, thick, black] (-3, 0) to (3, 0);
	\draw[-, dashed, thick, black] (0, -1) to (0, 1);
	\draw[-, dashed, thick, black] (3, -1) to (3, 1);
	\draw[-, dashed, thick, black] (-3, -1) to (-3, 1);
	\draw[-, dashed, thick, black] (-3, 1) to (3, 1);
	\draw[-, dashed, thick, black] (-3, -1) to (3, -1);
	\filldraw[color=black](0, 0) circle (3pt);
	\filldraw[color=black](0, -1) circle (3pt);
	\filldraw[color=black](0, 1) circle (3pt);
	\filldraw[color=black](3, 0) circle (3pt);
	\filldraw[color=black](-3, 0) circle (3pt);
	\draw (0, 0) node[circle, below left]{$p_0$};
	\draw (0, 1) node[circle, below left]{$p_2$};
	\draw (0, -1) node[circle]{};
	\draw (3, 0) node[circle, below left]{$p_1$};
	\draw (-3, 0) node[circle]{};
	\draw (-3, 1.5) -- (-.4, 1.5);
	\draw (.4, 1.5)-- (3, 1.5);
	\node  at (0, 1.5)[align = center] {$\frac{\sqrt{2}\pi}{k}$};
	\draw (3.5, -1) -- (3.5, -.4);
	\draw (3.5, .4)--(3.5, 1);
	\node  at (3.5, 0)[align = center] {$\frac{\sqrt{2}\pi}{m}$};
	\draw (-3, 1.3)--(-3,1.7);
	\draw (3, 1.3)--(3,1.7);
	\draw (3.3, -1)--(3.7, -1);
	\draw (3.3, 1)--(3.7, 1);
	\end{tikzpicture}
	\caption{A fundamental domain (for the group generated by rotations) with three singularities $p_0, p_1$, and $p_2$.  Dotted lines indicate reflectional symmetries.}
\end{figure}

For $\cunder> 0$ to be determined later, we define $\domzb : = [-\cunder,\cunder] \times \left[ \frac{ \cunder}{km} , \frac{\cunder}{km} \right]^2$,  
and for $\forall \zetabold = (\zeta, \sigma_1, \sigma_2) \in \domzb$ an LD solution 
\begin{equation}
\label{Dvarphi3}
\begin{gathered}
\varphi = \varphi\llbracket\zetabold\rrbracket := \sum_{i=0}^2 e^{\sigma_i} \tb  \Phi_i,  \quad \text{where} \quad
\tb = \tb \llbracket \zetabold\rrbracket: = \frac{1}{m}e^{\zeta} e^{-3\frac{km}{4\pi}},
\end{gathered}
\end{equation}
and by convention we define $\sigma_0: = - \sigma_1 - \sigma_2$.

Since each of $p_0, p_1, p_2$ and their $\group$-orbits are fixed by a pair of orthogonal reflections in $\group$,  $\val_{\sym}\llbracket \zetabold\rrbracket: = \val_\sym[L]$ is three-dimensional and may be identified with $\R^3$. 
\begin{prop}
\label{Pclifford2}
There is an absolute constant $C$ (independent of $\cunder$) such that for $k,m$ as in \ref{Amkcliff}, $mk$ large enough (depending on $\cunder$), and $m/k<C_1$ the map $Z_\zetabold : \val_\sym\llbracket\zetabold\rrbracket \rightarrow \Pcal$ defined by
\begin{align}
Z_\zetabold(\muboldtilde) = \frac{1}{3}\left( \sum_{i=0}^2 \mutilde_i, \frac{4\pi}{3km}(\mutilde_0+ \mutilde_2-2\mutilde_1), \frac{4\pi}{3km}(\mutilde_0 + \mutilde_1 -2 \mutilde_2)\right) , 
\end{align}
where here $\muboldtilde = \tb (\mutilde_0, \mutilde_1, \mutilde_2)$
satisfies 
$\zetabold - Z_\zetabold(\Mcal_L\varphi) \in [-C, C] \times \left[ \frac{-C}{km} , \frac{C}{km}\right]^2.$
\end{prop}

\begin{proof}
Using \ref{Dcliffdecomp2} and \eqref{Dvarphi3}, for each $i\in \{0, 1, 2\}$, 
$\mu_i : = \frac{1}{\tb} \Mcal_{p_i} \varphi $ satisfies
\begin{equation}
\label{Emuc}
\begin{aligned}
 \mu_i &= \frac{km}{4\pi}\sum_{j=0}^2 e^{\sigma_j} +  \bigg(\sum_{j=0}^2 e^{\sigma_j} \Phip_j \bigg)\bigg|_{p_i} + e^{\sigma_i}  \log \left( \frac{e^{\sigma_i} \tb}{2\delta}\right)\\
 &= 3\frac{km}{4\pi} + O\left( \frac{\cunder^2}{km}\right)+ O(C) + (1+ \sigma_i) \log \frac{\tb}{2\delta},
\end{aligned}
\end{equation}
where we have expanded the exponentials and used that $\sum_{i=0}^2 \sigma_i =0$.  
The\-refore,
\begin{equation}
\label{Evbc1}
\frac{1}{3}\sum_{i=0}^2 \mu_i = 3 \frac{km}{4\pi} + \log \frac{\tb}{2\delta} + O\left( C + \frac{\cunder^2}{km}\right) = \zeta + O\left( C + \frac{\cunder^2}{km}\right).
\end{equation}
Using  \eqref{Emuc}, that $\sum_{i=0}^2 \sigma_i =0$, and \eqref{Dvarphi3}, we calculate
\begin{equation}
\label{Evbc2}
\begin{aligned}
\frac{1}{3}(2\mu_2 - \mu_1 - \mu_0) &= -3 \frac{km}{4\pi} \sigma_2 + O\left( C + \frac{\cunder^2}{km}\right), \\
\frac{1}{3}(2\mu_1 - \mu_2 - \mu_0) &= - 3\frac{km}{4\pi} \sigma_1 + O\left( C + \frac{\cunder^2}{km}\right).
\end{aligned}
\end{equation}
The proof is concluded by combining \eqref{Evbc1} and \eqref{Evbc2}.
\end{proof}

\begin{theorem}
\label{Tcmain2}
Given $C_1> 0$, there exists an absolute constant $\cunder>0$ such that for all $(k, m)\in \N^2$ satisfying \ref{Amkcliff}, $m$ large enough in terms of $\cunder$, and $m/k<C_1$,
there exists a genus $3mk+1$, $\group[k, m]$-invariant doubling of $\T$ by applying Theorem \ref{Ttheory}.
\end{theorem}

\begin{proof}
The proof consists of checking the hypotheses of Theorem \ref{Ttheory} and is very similar to the proof of \ref{Tcmain1}, so we only give a sketch pointing out some of the differences.  Although now $\taubold$ takes three distinct values, Assumption \ref{Azetabold}\ref{Atau} still holds because of \ref{Dvarphi3}.  
The map $Z_\zetabold$ defined in \ref{Pclifford2} is clearly a linear isomorphism for each $\zetabold \in \domzb$, and so by \ref{Pclifford2} Assumption \ref{Azetabold}\ref{AZ} holds. 
\end{proof}


\section*{Part II: Construction of LD solutions on $O(2)\times \Z_2$ symmetric backgrounds}

\section{RLD Solutions}
\label{S:RLD}

\subsection*{Symmetries} 
\label{sub:sym}
\nopagebreak

\begin{definition}[Symmetries on $\cyl$] 
\label{Drotcyl} 
\label{Dgroup}
We define the group $O(2)\times \Z_2$, where $O(2)$ was defined in \ref{NEuc} and $\Z_2: = \{ \mathrm{Id}, \Sbar\}$. 
By convention, we identify each element of $O(2)$ or $\Z_2$  with its image under the inclusion  $O(2) \hookrightarrow O(2)\times \Z_2$ or $\Z_2 \hookrightarrow O(2) \times \Z_2$. 

Fix an orientation on $\Sph^1$ and define for $c\in \R$ the rotation $\thetasf_c \in O(2)$ of $\Sph^1$ by angle $c$ in accordance with the given orientation and the reflection $\thetabar_c \in O(2)$ determined by requesting that $\thetabar_c$ reverses orientation and has fixed-point set $\{ \pm (\cos c, \sin c)\}\subset \Sph^1$. 

Let $O(2)$ act on $\Sph^1$ by the usual isometric action, $\Z_2$ act on $\R$ by requesting that $\Sbar \, \sss = - \sss$ $\forall \sss\in \R$, and $O(2)\times \Z_2$ act on $\cyl: = \Sph^1\times \R$ (recall \ref{Ecyl}) by the product action with respect to the actions of $O(2)$ on $\Sph^1$ and $\Z_2$ on $\R$ just defined.

Finally, for $c\in \R$ we define the reflection $\Sbar_c \in \Isom ( \cyl, \chi)$ by $\Sbar_c(p, \sss):= (p, 2c - \sss)$.
\end{definition} 

\begin{assumption}[Assumptions on the background]
\label{Aimm}
In Part II 
we assume the following:
\begin{enumerate}[label={(\roman*)}]
\item Convention \ref{background} holds and $\Sigma$ is orientable and closed. 
\item The embedding of $\Sigma$ in $N$ is equivariant with respect to effective, isometric actions of $O(2)\times \Z_2$ on $\Sigma$ and on $N$.  
Moreover, the action of $\Sbar$ on $\Sigma$ is orientation reversing. 
\item $|A|^2+ \Ric(\nu, \nu)> 0$ on $\Sigma$.
\item $\ker \Lcal_\Sigma$ is trivial modulo the $O(2)\times \Z_2$ action on $\Sigma$.
\end{enumerate}
\end{assumption}

\begin{definition}[Parallel circles and equatorial circles] 
\label{Dpar}
We call the nontrivial orbits of the action of $O(2)$ on $\Sigma$ \emph{parallel circles} 
and those fixed by $\Sbar$ \emph{equatorial}.  
\end{definition}

\begin{lemma}
\label{LAconf}
\ref{Aimm} implies that the following hold. 
\begin{enumerate}[label=\emph{(\roman*)}]
\item $\Sigma$ is diffeomorphic to a sphere or to a torus.
\item There is an $O(2)\times \Z_2$-equivariant (with respect to the actions in \ref{Aimm} and in \ref{Dgroup}) map $X_\Sigma : \cyl_I \rightarrow \Sigma$, 
for some $I = (-l, l), 0 < l \leq \infty$ (recall \ref{Ecyl}), which is a conformal diffeomorphism onto its image and satisfies the following.  
\begin{enumerate}[label=\emph{(\alph*)}]
\item If $\Sigma$ is a torus then $\Sigma$ contains exactly two equatorial circles and $l<\infty$. 
Moreover $X_\Sigma$ extends to a covering map $\Xtilde_{\Sigma} : \cyl \rightarrow \Sigma$
satisfying $\Xtilde_\Sigma \circ \Sbar_l = \Xtilde_\Sigma$ and the equatorial circles are 
$\Xtilde_\Sigma(\cyl_0)$ and $\Xtilde_\Sigma(\cyl_l)$. 
Furthermore the image of $X_\Sigma$ is $\Sigma\setminus \Xtilde_\Sigma(\cyl_l)$. 
\item If $\Sigma$ is a sphere then $I = \R$ and the image of $X_\Sigma$ is $\Sigma$ minus two points 
with the ends of $\cyl$ mapped to deleted neighborhoods of the points removed.
Moreover $X_\Sigma(\cyl_0)$ is the unique equatorial circle of $\Sigma$.
\end{enumerate}
\end{enumerate}
\end{lemma}

\begin{proof}
Since $\Sigma$ admits an effective circle action, a result of Kobayashi \cite[Corollary 4]{Kobayashi} implies that 
$\Sigma$ has nonnegative Euler characteristic.  
Item (i) follows from this since $\Sigma$ is orientable. 

Next, we claim that $\Sigma$ has no exceptional orbits under the action of $SO(2)$.  
To see this, consider a $SO(2)$-orbit circle $S\subset \Sigma$ and choose (since $\Sigma$ is orientable) a unit normal field $\nu$ on $S$ in $\Sigma$.  Since $\exp^\Sigma$ commutes with $SO(2)$, $S_{z}: = \{ \exp^\Sigma_p( z \nu(p)) : p\in S\}$ is a $SO(2)$-orbit for all $z\in \R$, and smoothness implies the $S_z$ are of the same orbit type for all $z\in (-\epsilon, \epsilon)$ and $\epsilon>0$ small enough.  Since the principal orbits are dense, $S$ is a principal orbit and (since $SO(2)$ acts effectively) is in particular covered exactly once by $SO(2)$. 
 
The quotient of the principal orbits of $\Sigma$ by $SO(2)$ is diffeomorphic to $\R$ or $\Sph^1$, corresponding to the cases where $\Sigma$ is respectively a sphere or torus.  
Observe that the $\Z_2$ action descends to the quotient and (by \ref{Aimm}(ii)) $\Sbar$ reverses orientation.  
The corresponding fixed point set is then either a single point (when $\Sigma$ is a sphere) or a pair of points (when $\Sigma$ is a torus).
 The existence of a conformal map $X_\Sigma$ as in (ii) and (a) and (b) now follows easily. 
\end{proof}

\begin{remark}
\label{Rcylid}
Occasionally, we will use the diffeomorphism $X_\Sigma : \cyl_I \rightarrow X_\Sigma(\cyl_I)$ in \ref{LAconf}(ii) 
to use the standard coordinates $(\sss, \vartheta)$ on $\cyl$ as a coordinate system on $X_\Sigma( \cyl) \subset \Sigma$.
To simplify notation, later we will also occasionally identify $\cyl_I$ with $X_{\Sigma}(\cyl_I) \subset \Sigma$; 
for example, in Section \ref{S:LDs} we will identify configurations $L[\sbold; \mbold]$ defined in \ref{dL} 
with their images $X_\Sigma(L[\sbold; \mbold]) \subset \Sigma$ in order to define an appropriate class of LD solutions (see \ref{LsymLD} and \ref{Lphiavg}) on $\Sigma$. 
\qed 
\end{remark}

\begin{notation}
\label{Nconf}
We denote $\conf \in C^\infty(\cyl_I)$ the function satisfying  $X^*_\Sigma g = e^{2\conf} \chi$ (recall \ref{LAconf}(ii)).
\qed 
\end{notation}

\begin{remark}
\label{R:cc}
In Sections \ref{S:cat} and \ref{S:ccat} \ref{Aimm}(i) does not apply since the catenoid is not compact and the critical catenoid is a compact annulus with boundary. 
The theory in this section then has to be modified accordingly (see for example Lemma \ref{Lccb}).   
\qed 
\end{remark}

%

We call a function defined on an $O(2)$-invariant domain of $\Sigma$ or of $\cyl$ which is constant on each $O(2)$ orbit  a \emph{rotationally invariant function}.  The following notation will simplify the presentation. 

\begin{notation}
\label{Nsym}
Consider a function space $X$ consisting of functions defined on a domain
$\Omega$, where $\Omega \subset \cyl$ or $\Omega \subset \Sigma$. 
If $\Omega$ is a union of $O(2)$ orbits, we use a subscript ``$\sss$'' to denote 
the subspace of functions $X_{\sss}$ consisting of 
rotationally invariant functions, which are therefore constant on each $O(2)$ orbit.
If moreover $\Omega$ is invariant under $\grouprotcyl$,
we use a subscript ``$|\sss|$'' to denote 
the subspace of $O(2)\times \Z_2$-invariant functions. 
\qed 
\end{notation}

\begin{notation}
\label{Npartial}
If $\Omega \subset \cyl$ or $\Omega \subset X_\Sigma(\cyl_I)$ is a domain and $u \in C^0_\sss(\Omega)$ has one-sided partial derivatives at $\sss=\sbar$, then we denote these partial derivatives by using the notation
$$
\partial_{+\,}u(\sbar) :=\left. \frac{\partial u  }{\partial \sss}\right|_{\sss=\sbar+},
\qquad\qquad
\partial_{-\,} u(\sbar) :=-\left. \frac{\partial u  }{\partial \sss}\right|_{\sss=\sbar-}.
$$
If $u$ is $C^1$, we use the abbreviation $\partialx u := \frac{\partial u}{\partial \sss}$.  In that case, $\partial u = \partial_{+}u = -\partial_{-}u$.  
\qed 
\end{notation}

\begin{definition}[Symmetry groups] 
\label{dHcyl}
We define $\grot_{\munder} :=  \big\langle  \thetasf_{2\pi/\munder}, \Sbar \big \rangle$ and  
$\gcyl_{\munder} :=  \big\langle  \thetabar_0, \thetabar_{\pi/\munder} \big\rangle \times \big\langle  \Sbar \big\rangle$ 
for $\munder \in \N$ 
(recall \ref{Dgroup}). 
Clearly $\grot_{\munder}$ is an index $2$ subgroup of  $\gcyl_{\munder}< \grouprotcyl$ and 
$\big\langle  \thetabar_0, \thetabar_{\pi/\munder} \big\rangle < O(2)$ is a dihedral group 
of order $2\munder$. 
\end{definition}

\begin{definition}[$(\sbold,\mbold)$-symmetric sets and configurations] 
\label{Dlpar}
\label{dL} 
\label{dLmbold} 
Given $\sbar \in [0, \infty)$ or $\sbold := (\sss_1, \dots, \sss_k)\in \R^k$ 
such that 
$0 \leq\sss_1<\cdots<\sss_k < l$, we define 
\begin{align*}
 \Lpar[\sbar] := \cyl_{\{\pm \sbar\}}, \qquad 
  \Lpar[\sbold] := \bigcup_{i=1}^k \Lpar[\sss_i],
\end{align*}
and we denote the number of connected components (circles) of $\Lpar[\sbold]$ by $\kcir[\sbold]$. 
For $\munder \in\N$ we define 
\begin{equation*}
\begin{gathered}
L[\sbar; \pm \munder]  : =\Lmer[\pm \munder]\cap \Lpar[\sbar], 
\quad \text{where} \\
\Lmer[\munder]:= \gcyl_{\munder} \Lmer[1], 
\qquad 
\Lmer[1] := \left\{  (1,0) \right\} \times \R ,
\\    
\Lmer[-\munder] := \thetasf_{\pi/\munder} \Lmer[\munder].  
\end{gathered}
\end{equation*}
Given then $\mbold := (m_1, \dots, m_k)\in (\Z\setminus \{0\})^k$ 
we call a set $L\subset \Sigma$ or a configuration $\taubold : L \rightarrow \R_+ $  
\emph{$(\sbold,\mbold)$-rotational} 
if $L := \bigcup_{i=1}^k L_i$ with each $L_i\subset \Lpar[\sss_i]$ containing $|m_i|$ points in each component of $\Lpar[\sss_i]$ ($i=1,...,k$);  
we denote then the average value of the restriction $\left. \taubold \right|_{L_i}$ by $\tau_i$. 
We call such an $L$ or $\taubold$ \emph{$(\sbold,\mbold)$-symmetric} 
if we moreover have 
$L_i = L[\sss_i ; m_i ]$ 
and that the restriction $\left. \taubold \right|_{L_i}$ is $\gcyl_{|m_i|}$-invariant---hence $\taubold(L_i)= \{\tau_i\}$---for each $i=1,...,k$.  
An $(\sbold,\mbold)$-symmetric set $L$ is then uniquely determined and will be denoted by $L[\sbold; \mbold]$.   
Finally we denote by $m\in \N$ the greatest common divisor of $|m_1|,\dots, |m_k|$, 
so that the stabilizer in $O(2)\times \Z_2$ of an $(\sbold,\mbold)$-symmetric $L$ or $\taubold$ is $\groupmcyl$.
\end{definition}

\begin{remark}
\label{RLcard}
It is worth noting the following: 
\begin{enumerate}[label={(\roman*)}]
\item $\kcir[\sbold] = 2k$ if $\sss_1 > 0$ and $\kcir[\sbold] = 2k-1$ if $\sss_1 =0$.
\item In the case where $\Sigma$ is a torus, we could allow $\sss_k = l$ in the constructions later in Part II.  In order to simplify the presentation, however, 
we do not discuss this case. 
\item $L[\sbar; \pm \munder]$ 
are the only subsets of $\Lpar[\sbar]$ which are invariant under $\gcyl_{\munder}$  and contain exactly $\munder$ points equidistributed on each circle of $\Lpar[\sbar]$
($2\munder$ in total if $\sbar\ne0$). 
The sign of $\pm \munder$ encodes the choice between these two subsets in $L[\sbold; \mbold]$. 
\item An $(\sbold,\mbold)$-rotational set $L$ as in \ref{dLmbold} has cardinality $|L|= |m_1|+ 2\sum_{i=2}^k|m_i|$ if $\sss_1=0$ 
and $|L|= 2\sum_{i=1}^k |m_i|$ points if $\sss_1 >0$. 
\item 
Note that an $(\sbold,\mbold)$-symmetric configuration $\taubold$ is uniquely determined by $\{\tau_i\}_{i=1}^k$. 
\qed 
\end{enumerate}
\end{remark}

\subsection*{Basic facts and definitions}

We will estimate our LD solutions by comparing them with corresponding rotationally invariant solutions.  We therefore need to define the appropriate class of rotationally invariant solutions of the linearized equation.  We begin with some notation from \cite{kapmcg}.

\begin{definition}
\label{dsigma}
Let $\R^\N := \left\{ (a_i)_{i\in \N}: a_i \in \R\right\}$.  
For any $k\in \N$, we identify $\R^{k}$ with a subspace of~$\R^\N$ by the map
$(a_1, \dots, a_{k}) \mapsto (a_1, \dots, a_{k}, 0, 0, \dots)$.
We consider the normed space $\left(\ell^1(\R^{\N}), |\cdot |_{\ell^1}\right)$, where
\begin{align*}
\ell^1(\R^\N) := \bigg\{ \abold = (a_i)_{i\in \N} \in \R^\N : \sum_{i=1}^{\infty} |a_i|< \infty \bigg \}, \quad
|\abold |_{\ell^1} := \sum_{i=1}^{\infty} |a_i|.
\end{align*}

\end{definition} 
\begin{remark}
\label{rsigbij}
If $\bsigma = (\sigma_i)_{i\in \N}\in \ell^{1}\left( \R^{\N}\right)$, $\xibold = (\xi_i)_{i\in \N} \in \ell^{\infty}\left( \R^{\N} \right)$ and some positive numbers $F_{i\pm}$, $i\in \N$, satisfy
\begin{align*}
e^{\sigma_i} = \frac{F_{i+1+} + F_{i+1-}}{F_{i+} + F_{i-}}, \quad \xi_i = \frac{F_{i+} - F_{i-}}{F_{i+}+F_{i-}}\quad i\in \N,
\end{align*}
then note that
$\left| \xibold \right|_{\ell^\infty} < 1$ and 
 for any $1\leq j \leq i <\infty$ that
\begin{align*}
F_{i+} = \frac{1  + \xi_i }{1   + \xi_j}\big(e^{\sum_{l=j}^{i-1}\sigma_l}\big)F_{j+} 
= \frac{1  + \xi_i }{1   - \xi_j}\big(e^{\sum_{l=j}^{i-1}\sigma_l}\big)F_{j-},
\\ 
F_{i-} = \frac{1  - \xi_i }{1   + \xi_j}\big(e^{\sum_{l=j}^{i-1}\sigma_l}\big)F_{j+}
= \frac{1  - \xi_i }{1   - \xi_j}\big(e^{\sum_{l=j}^{i-1}\sigma_l}\big)F_{j-},
\end{align*}
and therefore
$\sup\{ F_{i\pm} \}_{i\in \N}  \Sim_{\Eunder} \inf\{ F_{i\pm} \}_{i\in \N}$ with $\Eunder : = \frac{1+|\xibold|_{\ell^\infty}}{1- |\xibold|_{\ell^\infty}}\big(e^{|\bsigma|_{\ell^1}}\big)$. 
\qed
\end{remark}

\begin{definition}[Scale invariant flux]
\label{dF}
 If $\phi \in C^0_\sss( \Omega)$, where either $\Omega = \cyl_{(a, b)}$ or $\Omega = X_{\Sigma}( \cyl_{(a, b)})$ (recall \ref{Ecyl} and \ref{Nsym}), $(a, b) \subset I$, and  $\phi$ is piecewise smooth and nonzero on $\Omega$, we define $F^\phi_{\pm}:(a, b)\rightarrow \R$ by
\[ F^\phi_{\pm}(\sss) =  \frac{\partial_\pm \phi(\sss)}{\phi(\sss)} = \partial_\pm \log |\phi|(\sss).\] \end{definition} 

\begin{remark}
\label{rFlux} Note that $F^\phi_\pm = F^{c\phi}_\pm$  $\forall c\in \R \setminus \{0\}$.
\qed 
\end{remark}

\begin{definition}[Subdivisions of cylindrical domains] 
\label{Dlpar2}
Given $\sbold$ as in \ref{Dlpar} and a domain $\Omega$, where $\Omega \subset \cyl$ (or $\Omega \subset \Sigma$), we will denote by $\Omega^{\sbold}$ 
the \emph{subdivision of $\Omega$} by $\Lpar[ \sbold]$ (or of $X_{\Sigma}(\Lpar[\sbold])$): 
$\Omega^{\sbold}$ is the abstract surface which is the disjoint union of the $\Omega \cap A$'s, 
where $A$ is the closure of any connected component (a disk or an annulus) of $\cyl\setminus \Lpar[\sbold]$ (or of $\Sigma \setminus X_{\Sigma}(\Lpar[\sbold])$).  
Clearly functions on $\Omega$ can be thought of as functions on $\Omega^{\sbold}$ as well.
\end{definition}

Note for example that a function defined on $\Omega$ which is in $C^\infty( \Omega^{\sbold})$ is also in $C^0(\Omega)$ but not necessarily in $C^1(\Omega)$; 
it is ``piecewise smooth" on $\Omega$.

\begin{definition}
\label{dLchi}
We define an operator $\Lchi$ on $\cyl_I$ by
\begin{align}
\label{EopL}
\Lchi : = \Delta_\chi + V = e^{2\conf} \Lcal_{\Sigma},
\quad
\text{ where }  
\quad
\Delta_\chi := \frac{\partial^2}{\partial \sss^2} +  \frac{\partial^2}{\partial \theta^2},
\end{align}
$V\in C^\infty_{|\sss|}(\cyl_I)$ is defined by $V = e^{2\conf } X^*_{\Sigma}(|A|^2+ \Ric(\nu, \nu))$, and $\conf$ as in \ref{Nconf}.
 \end{definition}
 When $\phi$ is rotationally invariant, note also that the equation $\Lchi \phi = 0$ amounts to the ODE 
\begin{align}
\label{ELchirot}
\frac{d^2\phi}{d\sss^2} + V(\sss) \, \phi = 0.
\end{align}

\begin{definition}[RLD solutions, cf. {\cite[3.5]{kapmcg}}]
\label{RL}
Given $\Omega = \Sigma$ or $\Omega = X_{\Sigma}(\cyl_I) \subset \Sigma$, we say  $\phi\in C^0_{|\sss|} \big( \Omega \big)$ is a \emph{rotationally invariant (averaged) linearized doubling (RLD) solution} on $\Omega$ if 
\begin{enumerate}[label=\emph{(\roman*)}]
\item  
\label{RL+} 
$\phi >0$. 
\item  There is $k\in \N$ and $\sbold^\phi \in [0, l)^k$ as in \ref{Dlpar}, such that 
$\phi \in C^\infty_{|\sss|} \big( \Omega^{\sbold^\phi}\big)$ and 
$\Lcal_\Sigma \phi =0$ on $\Omega^{\sbold^\phi}$. 
\item  
\label{RLF} 
For $i=1, \dots, k$, $F^\phi_-(\sss^\phi_i)> 0$ and $F^\phi_+(\sss^\phi_i)> 0$. 
\end{enumerate}
We call $\sbold^\phi$ the \emph{singular} or \emph{(derivative) jump latitudes}, and the circles contained in $\Lpar[\sbold^\phi]$ \emph{singular circles}, of $\phi$.
If $\phi(0) = 1$, we say $\phi$ is a \emph{unit RLD solution}. 
We say that an RLD solution on $X_{\Sigma}(\cyl_I)$ is \emph{smooth at the ends} if it can be extended smoothly to $\Sigma$.  
\end{definition}

\begin{remark}
\label{R:RLD}
Note that to allow the construction of immersed doublings we can simply relax \ref{RL}\ref{RL+} to requiring 
$\phi(\sss^\phi_i)\ne 0$ for $i=1, \dots, k$. 
If $\phi>0$ fails then, the constructed doublings will not be embedded, as for $k=1$ in \ref{R:k12}.  
On the other hand \ref{RL}\ref{RLF} is always necessary to ensure positive size for the catenoidal bridges, 
so its violation makes the construction impossible, as for $k=2$ in \ref{R:k12} 
or $\kcir=1$ in \ref{rphibd}.   
Finally note that by Lemma \ref{LAconf}, $X_\Sigma(\cyl_I)$ is either $\Sigma$ with two points (when $\Sigma$ is a sphere) or a circle (when $\Sigma$ is a torus) removed.  
We will first study RLD solutions on $X_\Sigma(\cyl_I)$, instead of on $\Sigma$, 
in order to facilitate the parametrization of the families of RLD solutions. 
\qed 
\end{remark}

\begin{definition}
\label{Dkmin}
Define $\kcirmin = \min \kcir[ \sbold^\phi]$, where the minimum is over all RLD solutions $\phi$. 
\end{definition}

\begin{definition}[Quantities associated to RLD solutions, cf. {\cite[3.6]{kapmcg}}]
\label{RLquant}
Given $\phi$ as in \ref{RL}, define
\begin{align*}
\Fboldunder^\phi =& \big( F^\phi_{i-}, F^\phi_{i+} \big)_{i=1}^{k} \in \R^{2k}_+, 
& \quad
\Fbold^\phi =& ( F^\phi_i)_{i=1}^k  \in \R^{k}_+,
\\
\quad \bsigma^\phi =& (\sigma^\phi_j)_{j=1}^{k-1} \in \R^{k-1}, 
&
\xibold^\phi =& \big( \xi^\phi_i\big)_{i=1}^k \in \R^k,  
\end{align*}
where for $i=1, \dots, k$  and $j=1, \dots, k-1$,
\begin{equation} 
\label{Exi}
\begin{aligned}
F^\phi_{i\pm} :=& F^\phi_\pm (\sss^\phi_i), 
& \qquad 
2F^\phi_i :=& F^\phi_{i+} + F^\phi_{i-} , 
\\ 
e^{\sigma^\phi_j} :=& \frac{F^\phi_{j+1}}{F^\phi_j}, 
& \qquad 
\xi^\phi_i :=& \frac{F^\phi_{i+} - F^\phi_{i-}}{F^\phi_{i+} + F^\phi_{i-}} .
\end{aligned}
\end{equation} 
We define $\bsigmaunder^\phi: = (\bsigma^\phi, \xibold^\phi) \in \R^{k-1}\times \R^k$ and call the entries of $\bsigmaunder^\phi$ the \emph{flux ratios} of $\phi$.   
\end{definition}

\begin{remark}
\label{Rs0}
By the $\Sbar$ symmetry $\xi^\phi_1 = 0$ when $\sss^\phi_1 = 0$. 
\qed 
\end{remark}

\begin{remark}
\label{Rru}
Using \eqref{Exi} (see also Remark \ref{rsigbij}), we recover $\Fboldunder^\phi$ from $F^\phi_{1}$ and $\bsigmaunder^\phi$ 
by 
\begin{align}
\label{EFxi}
F^\phi_{1\pm} = (1\pm \xi^\phi_1)F^\phi_{1} , 
\qquad F^\phi_{i\pm} =  (1\pm \xi^\phi_i) \big( e^{\sum_{l=1}^{i-1} \sigma^\phi_l} \big)  F^\phi_{1}, 
 \quad i>1.
\end{align}
In Proposition \ref{Pexist} we construct RLD solutions $\phi$ by prescribing $F^\phi_{1-}$ and $\bsigmaunder^\phi$.
\qed 
\end{remark}

In our applications later in Part II, we will primarily be interested (see \ref{Lphiavg}(iii) and \ref{Lmatching}) in RLD solutions which are close to being ``balanced" in the sense of the following definition.

\begin{definition}[Balanced RLD solutions]
\label{dLbalanced} 
Given 
$\mbold := (m_1, \dots, m_k)\in (\Z\setminus \{ 0\} )^k$ we define  
$\bsigmaslash = \bsigmaslash[\mbold] = (\sigmaslash_j)_{j=1}^{k-1} \in \R^{k-1}$     
and 
$\bsigmaunderslash = \bsigmaunderslash[\mbold] := (\bsigmaslash, \zerobold) \in \R^{k-1}\times \R^k$ by 
$e^{\sigmaslash_j} :=  \left| {m_{j+1}} / {m_j} \right| $ for $j=1, \dots, k-1$.  
We call an RLD solution $\phi$ \emph{balanced} with respect to $\mbold$ if it satisfies
$\bsigmaunder^\phi = \bsigmaunderslash[\mbold]$ (recall \ref{RLquant}).  
\end{definition}
The corresponding definition of balanced RLD solutions in \cite[Definition 3.5]{kapmcg} asserted instead that $\bsigmaunder^\phi = \zerobold$, 
which occurs in the context of \ref{dLbalanced} when all the $m_i$'s have the same absolute value.  
This difference is explained by the fact that in the constructions of \cite{kapmcg}, 
the intersection of an LD solution's singular set with $\Lpar[\sss_i]$ consisted of some number $m \in \N$ points for each $i \in \{1, \dots, k\}$, 
whereas the singular sets of the LD solutions we study later in Part II more generally have $|m_i| \in \N$ points on each component of $\Lpar[\sss_i]$ 
(see \ref{dLmbold} and \ref{Lphiavg})  for $i \in \{1, \dots, k\}$, and the $|m_i|$'s need not all be equal. 

\subsection*{Existence and uniqueness of RLD solutions} 

 \begin{lemma}[Existence and properties of $\phiend$] 
\label{AV1}
\label{AV2}
The following hold. 
\begin{enumerate}[label=\emph{(\roman*)}]
\item 
$V$ as defined in \ref{dLchi} satisfies $V> 0$.  
There exists $C>0$ and for each $j\in \N$, $C(j)>0$ such that for all $\sss\in \cyl_I$, $V(\sss) \Sim_{C} e^{-2 |\sss|}$ and $\frac{\partial^jV}{\partial \sss^j}(\sss) \leq C(j) e^{-2|\sss|}$.
\item 
There exists a unique $\phiend \in C^\infty_{\sss}(\cyl_{I})$ satisfying
\[ 
\Lchi \phiend = 0, \quad
\lim_{\sss \nearrow l} \phiend(\sss) = 1, \quad
 \text{and} \quad 
 \lim_{\sss \nearrow l} F^{\phiend}_+(\sss) = 0.
 \]
\item 
For all large enough $\sss$, $F_+^{\phiend}(\sss) < C e^{-2\sss}$ for some $C>0$.
\item 
An RLD solution $\phi$ as in \ref{RL} is smooth at the ends if and only if $F^\phi_{k+} = F^{\phiend}_+(\sss^\phi_k)$.
\end{enumerate}
\end{lemma}

\begin{proof}
Items (i)-(iii) follow easily from \ref{Aimm}, \ref{LAconf}, and \ref{dLchi} and are trivial in the case $\Sigma$ is a torus.  
In the case $\Sigma$ is a sphere, items (i)-(iii) follow easily from the fact that the conformal map $X_\Sigma$ maps the ends of $\cyl$ to punctured disks on $\Sigma$.  
Finally, (iv) follows from \ref{RL} , item (ii) above,  and uniqueness for ODE solutions.
\end{proof}

\begin{lemma}[Flux monotonicity]  
\label{LFmono}
Suppose $\phi \in C^\infty_{\sss}\big(\cyl_{[a, b]}\big)$, $\phi> 0$, and $\Lchi \phi = 0$.
\begin{enumerate}[label=\emph{(\roman*)}]
	\item For $\sss\in (a, b)$, $\frac{d F^{\phi}_-}{d\sss}(\sss) = V(\sss) + \big(F^\phi_-(\sss)\big)^2>0$.
	\item $F^\phi_-(b) + F^\phi_+(a) =  \int_{a}^b V(\sss) + \big( F^\phi_-(\sss)\big)^2d\sss$.
\end{enumerate}
\end{lemma}
\begin{proof}
The equalities are calculations using \eqref{ELchirot} and \ref{dF}.  The inequality in (i) follows from \ref{AV1}.
\end{proof}

\begin{definition} 
\label{dHflux}
Given $F\in \R$ and $\sbar \in I$, we define $H = H[ F; \sbar] \in C^\infty_\sss (\cyl_I)$
by requesting that it satisfies the equation $\Lchi H = 0$ with initial data $H(\sbar) = 1$ and $F^H_+(\sbar) = F$.  We also define $\phie: = H[0;0]$. 
\end{definition}

\begin{lemma}
\label{LHmono}
\begin{enumerate}[label=\emph{(\roman*)}]
\item $\displaystyle{\frac{ \partial F_+^{H[F; \sbar]}}{\partial \sbar}(\sss) = \frac{V(\sbar)+F^2}{(H[F; \sbar](\sss))^2} >0}$.
\item $\displaystyle{\frac{ \partial F_+^{H[F; \sbar]}}{\partial F}(\sss) = \frac{1}{(H[F; \sbar](\sss))^2} >0}$.
\end{enumerate}
\end{lemma}
\begin{proof}
By direct calculation, switching the order of differentiation, and using \ref{LFmono}, we find
\begin{align*}
\frac{\partial}{\partial \sss} \left( \frac{\partial F^H_+}{\partial u} H^2 \right) = 0, 
\end{align*}
where $H = H[F; \sbar]$ and $u$ is either $\sbar$ or $F$.  It follows that
\begin{align*}
\frac{\partial F^H_+}{\partial u}(\sss) = \frac{\partial F^H_+}{\partial u}(\sbar) \left( \frac{H(\sbar)}{H(\sss)}\right)^2. 
\end{align*}
Differentiating both sides of the equation $F^{H[F; \sbar]}_+(\sbar) = F$ with respect to $\sbar$ yields 
the first equality in 
\[ \frac{\partial F^H_+}{\partial \sbar}(\sbar) = - \frac{\partial F^H_+}{\partial \sss}(\sbar) = V( \sbar)+ F^2,\]
where the second equality follows from \ref{LFmono}.  Observing also that $\frac{\partial F^H_+}{\partial F}(\sbar) = 1$ and combining the above completes the proof.
\end{proof}

\begin{definition}
\label{Fmax}
Define $\Fmax: = \lim_{\sss \nearrow l} F^{\phie}_-(\sss)$ if $\phie>0$ on $\cyl_I$ and $\Fmax: =\infty$ otherwise.
\end{definition}

\begin{lemma}
If $\Sigma$ is a sphere, $\Fmax = \infty$. 
\end{lemma}
\begin{proof}
If $\Sigma$ is a sphere, then $I=\R$, and then it follows from \eqref{EopL}, that $V>0$, and that $\partial \phie(0)=0$ that $\phie$ has a root $\sss^{\phie}_{\mathrm{root}} \in (0, \infty)$. 
\end{proof}

We are now ready to parametrize families of RLD solutions by their flux ratios and $F^\phi_{1-}$. The notation differs slightly depending on whether the total number of circles $\kcir$ (recall \ref{Dlpar}) is even (Proposition \ref{Pexist}) or odd (Proposition \ref{Pexist2}).

\begin{prop}[Existence and uniqueness of RLD solutions, $\kcir$ even] 
\label{Pexist}
Given $F \in (0,  \Fmax)$ and 
 \begin{align*}
 \bsigmaunder = (\bsigma, \xibold) = \left( \, (\sigma_i)_{i=1}^\infty , (\xi_{j})_{j=1}^\infty\, \right)
  \in \ell^1\left(\R^\N\right)\oplus \ell^{\infty}\left( \R^\N\right)
\end{align*}
satisfying $|\xibold |_{\ell^\infty} < 1$,  there is a unique $k= k[F; \bsigmaunder]\in \N$ and a unique unit RLD solution $\phat = \phat[F; \bsigmaunder]$ on $X_\Sigma(\cyl_I)$ satisfying the following.
\begin{enumerate}[label=\emph{(\alph*)}]
\item $F^{\phat}_{1-} = F$ and $\sss_1^\phat >0$. 
\item $\bsigmaunder^\phat = \left. \bsigmaunder \right|_k$ where $k = k[ F; \bsigmaunder] \in \N$ 
is the number of jump latitudes of $\phat$ (recall \ref{RL}) and $\left. \bsigmaunder \right|_k: =  \left( \, (\sigma_i)_{i=1}^{k-1} , (\xi_{j})_{j=1}^{k}\, \right) \in \R^{k-1}\times \R^k$.
\end{enumerate}
Moreover the following hold. 
\begin{enumerate}[label=\emph{(\roman*)}]
\item  $\sss_1^\phat, \dots, \sss_k^\phat$ are increasing smooth functions of $F$ for fixed $\bsigmaunder$.
\item $k[F; \bsigmaunder]$ is a nonincreasing function of $F$.  Further, there exists $\kminev\in \N$ and a decreasing sequence $\{a_{\kminev-1, \bsigmaunder}, a_{\kminev, \bsigmaunder}, \dots\}$ such that $k[F; \bsigmaunder] = k$ if and only if $F \in [a_{k, \bsigmaunder}, a_{k-1, \bsigmaunder})$.
\item The restriction of $\phat[F; \bsigmaunder]$ on any compact subset of $\cyl_I$ depends continuously on $F$ and $\bsigmaunder$.
\end{enumerate}
\end{prop}
\begin{proof}
Suppose $\phat$ is a unit RLD solution satisfying (a) and (b).  Because $\sss^\phat_1>0$, the symmetries imply that $\phat = \phie$ on a neighborhood of $\cyl_0$.  But then \ref{RL}(i)-(ii), the flux monotonicity (Lemma \ref{LFmono}), and Remark \ref{rsigbij} inductively determine $\sbold^\phat$ and $\phat$ uniquely on $\cyl_I$.  This concludes the uniqueness part. 

We next construct a family of RLD solutions $\phat[F; \bsigmaunder]$ satisfying (a) and (b).  By the hypotheses, \ref{Fmax}, and \ref{LFmono}, there is a unique $\sss_1 \in (0, l)$ such that $\phie$ is positive on $\cyl_{(-\sss_1, \sss_1)}$ and $F = F^{\phie}_-(\sss_1)$.  By Remark \ref{rsigbij}, there is a unique extension $\phat[F; \bsigmaunder]$ of $\left.\phie\right|_{\cyl_{(-\sss_1, \sss_1)}}$ to a maximal domain $\cyl_{(-a, a)}$ and  a unique (a priori possibly infinite) sequence $\sbold = (\sss_1, \sss_2, \dots )$ such that $\phat >0$, $\Lchi \phat = 0$ on $\cyl^{\sbold}_{(-a, a)}$, and
$F^\phat_{\pm}(\sss_i) = F_{i\pm}$, where $F_{i-}:= F$ and all other $F_{i\pm}$ are defined by requesting the identities in \ref{rsigbij} hold.  To show that $\phat$ is an RLD solution, we need only show that $a = l$ and $\sbold$ is a finite sequence. 
By Remark \ref{rsigbij} and Lemma \ref{LFmono}, 
\begin{align*}
 2F & \Sim_{\Eunder} \left( F_{i+1-} + F_{i+}\right)=  \int_{\sss_i}^{\sss_{i+1}} V(\sss) + \big( F^{\phat}_-(\sss)\big)^2d\sss, 
\end{align*}
where $\Eunder$ is as in \ref{rsigbij}.  
This implies a uniform in $i$ lower bound on $\sss_{i+1} - \sss_i$.  
Therefore $a = l$.  
In the case $l = \infty$ we show there are finitely many jump latitudes by estimating an upper bound for $\sss_k$ in terms of $F$ and $\bsigmaunder$: 
specifically, we claim that if $\phat$ has a jump at $\sss_{j+1}$ and $\sss_j$ is large enough that $\phiend>0$ on $(\sss_j, l)$ (recall \ref{AV2}), 
then $F^\phat_+(\sss_j) \leq F^{\phiend}_+(\sss_j)$, 
which using the comparability of all the fluxes to $F$ with the fact from \ref{AV2} that $\lim_{\sss \rightarrow \infty} F^{\phi_{\mathrm{end}}}_+(\sss) = 0$, 
implies that $\sss_j$ cannot be arbitrarily large.  

The claim follows by observing that $F^\phat_+ - F^{\phiend}_+$ cannot change sign on $(\sss_j, \sss_{j+1})$ as $F^\phat_+$ and $F^{\phiend}_+$ 
satisfy the same first order equation \ref{LFmono}(i), while $F^\phat_+$ changes sign and $F^{\phiend}_+$ remains positive on $(\sss_j, \sss_{j+1})$.  
This concludes the proof of the existence and uniqueness of $\phat[F; \bsigma]$ satisfying (a)-(b).

Since as above $F = F^{\phie}_-(\sss^\phat_1)$, \ref{LFmono} implies $\sss^\phat_1$ is increasing as a function of $F$.  By \ref{rsigbij}, $F^\phat_{2-} = \frac{1-\xi_2}{1-\xi_1}e^{\sigma_1} F$.  By combining this with both parts of Lemma \ref{LHmono}, it follows that $\sss^{\phat}_2$ is increasing as a function of $F$.  Using this and arguing inductively shows that $\sss^\phat_j$ is a strictly increasing function of $F$ for $2<j \leq k$.

 That $k[F; \bsigmaunder]$ is nonincreasing in $F$ and the existence the sequence follows easily from the monotonicity of $\sss^\phat_k$ in (ii).  To complete the proof of (ii), we must show that $\kminev$ is well defined, that is independent of $\bsigmaunder$.  If $\Fmax <\infty$, by the flux monotonicity \ref{LFmono} it is easy to see that this is true and moreover that $\kminev = 1$.  Suppose then that $\Fmax = \infty$, and consider RLD solutions $\phat = \phat[F, \bsigmaunder], \phat' = \phat[F', \bsigmaunder']$ for $\bsigmaunder, \bsigmaunder' \in \ell^1\left(\R^\N\right)\oplus \ell^{\infty}\left( \R^\N\right)$ fixed and variable $F, F' \in (0, \Fmax)$.  By choosing $F'$ large enough in terms of $F$ and $\bsigmaunder$, we may ensure that $F^{\phat'}_{\pm i}> F^\phat_{\pm i}$ for all $i$ such that both of the preceding are defined.  By \ref{LFmono} and \ref{LHmono}, this implies that $k[F; \bsigmaunder] \geq k[F'; \bsigmaunder']$.  On the other hand, by choosing $F'$ small enough in terms of $F$ and $\bsigmaunder$, it follows analogously that $k[F; \bsigmaunder] \leq k[F'; \bsigmaunder']$, and together these inequalities prove that $\kminev$ is well defined.  

Finally, (iii) follows from \eqref{EFxi} and smooth dependence of ODE solutions on initial conditions.
\end{proof}

\begin{prop}
[Existence and uniqueness of RLD solutions, $\kcir$ odd]
\label{Pexist2}
Given $F \in (0, \infty)$ and
 \begin{align*}
 \bsigmaunder = (\bsigma, \xibold) = \left( \, (\sigma_i)_{i=1}^\infty , (\xi_{j})_{j=1}^\infty\, \right)
  \in \ell^1\left(\R^\N\right)\oplus \ell^{\infty}\left( \R^\N\right)
\end{align*}
satisfying $|\xibold |_{\ell^\infty} < 1$ and $\xi_1 =0$, there is a unique $k= k[F; \bsigmaunder]\in \N$ and a unique unit RLD solution $\phicheck = \phicheck[F; \bsigmaunder]$ on $X_\Sigma(\cyl_I)$ satisfying the following.
\begin{enumerate}[label=\emph{(\alph*)}]
\item $F^{\phicheck}_{1} = F^{\phicheck}_{1\pm} = F$ and $\sss_1^\phicheck=0$. 
\item $\bsigmaunder^\phicheck = \left. \bsigmaunder \right|_k$ where $k = k[ F; \bsigmaunder] \in \N$ 
is the number of jump latitudes of $\phicheck$ (recall \ref{RL}) and $\left. \bsigmaunder \right|_k: =  \left( \, (\sigma_i)_{i=1}^{k-1} , (\xi_{j})_{j=1}^{k}\, \right) \in \R^{k-1}\times \R^k$.
\end{enumerate}
Moreover the following hold. 
\begin{enumerate}[label=\emph{(\roman*)}]
\item  $\sss_2^\phicheck, \dots, \sss_k^\phicheck$ are increasing smooth functions of $F$ for fixed $\bsigmaunder$.
\item $k[F; \bsigmaunder]$ is a nonincreasing function of $F$. Further, there exists $\kminodd\in \N$ and a decreasing sequence $\{b_{\kminodd-1, \bsigmaunder}, b_{\kminodd, \bsigmaunder}, \dots\}$ such that $k[F; \bsigmaunder] = k$ if and only if $F \in [b_{k, \bsigmaunder}, b_{k-1, \bsigmaunder})$.
\item The restriction of $\phicheck[F; \bsigmaunder]$ on any compact subset of $\cyl_I$ depends continuously on $F$ and $\bsigmaunder$.
\end{enumerate}
\end{prop}
\begin{proof}
We omit the details of the proof, which are very similar to the proof of \ref{Pexist}.  Note however that the assumption $\xi_1 = 0$ is necessary (recall \ref{Rs0}) due to the symmetry about $\sss^\phicheck_1 = 0$.
\end{proof}

\begin{remark}
It is clear that any RLD solution is a constant multiple of a $\phat[F; \bsigmaunder]$ as in \ref{Pexist} or a $\phicheck[F; \bsigmaunder]$ as in \ref{Pexist2}.
\qed 
\end{remark}

\begin{remark}
\label{Rphicheck}
RLD solutions with $\kcir$ odd were constructed on $\Sph^2$ in \cite[Lemma 7.22]{kapmcg}, where they were called $\phat_{\mathrm{eq}}[F; \bsigmaunder]$.
\qed 
\end{remark}

\subsection*{Estimates on RLD solutions}

We proceed to estimate the families of RLD solutions just constructed.  To avoid unnecessary notational difficulties, we state and prove the next results for the families of RLD solutions $\phat$ with $\kcir$ even and leave the trivial modifications for the families of solutions $\phicheck$ with $\kcir$ odd to the reader. 

\begin{definition}
\label{Sk}
We define for $k\in\N$, $k\geq \kminev$ 
the domain $S_k\subset \R\times\R^{k-1}\times\R^k$ by 
\begin{align*}
S_k:=
\left \{\, \left ( \, F \, ,
\, (\sigma_i )_{i=1}^{k-1} \, , \, (\xi_j)_{j=1}^k \, \right) \, : \, 
F \in (0, a_{k-1, \bsigmaunder})
 \text{ and } |\xibold|_{\ell^\infty} < 1\right\}, 
\end{align*}
where $a_{k-1, \bsigmaunder}$ is as in \ref{Pexist}.
By \ref{dsigma} and \ref{Pexist}, $S_{k+1}\subset  S_k \subset \R\times \R^\N\times\R^\N$. 
\end{definition}

\begin{lemma}[Recursive formulas for the derivatives of $\sss_k$]
\label{Lsderiv}
An RLD solution $\phat = \phat[F; \bsigmaunder]$ as in \ref{Pexist} has $k \geq \kminev$ jumps if and only if $F \in (0, a_{k-1, \bsigmaunder})$
 or equivalently by \ref{Sk}, $(F, \left. \bsigmaunder\right|_k) \in S_k$. 
The $k$th jump latitude $\sss_k$ depends only on $F$ and $\left. \bsigmaunder\right|_k$ and can  be considered as a smooth function defined on $S_k$.  
Alternatively, we can consider each $F$ as a smooth function of $F_1 = F^\phat_1$ and $\left. \bsigmaunder\right|_k$, and then we have for $k=1$
\begin{equation} 
\label{Msderiv0}
\begin{aligned} 
\big(V(\sss_1) + \big( F^\phat_{1-}\big)^2\big)\frac{\partial \sss_1 }{ \partial F_1} =1 - \xi_1, 
\qquad 
\big(V(\sss_1) + \big( F^\phat_{1-}\big)^2\big)\frac{\partial \sss_1 }{ \partial \xi_1} = - F_1, 
\end{aligned} 
\end{equation} 
and for $k>1$ the recursive formulas (note $S_k\subset S_{k-1}$) 
\begin{multline}
\label{Msderiv}
\big( V(\sss_k) + \big(F^\phat_{k-}\big)^2\big)\frac{ \partial \sss_{k}}{ \partial F_1}
= 
\big( V(\sss_{k-1}) + \big(F^\phat_{k-1+}\big)^2\big)\frac{ \partial \sss_{k-1}}{ \partial F_1}\bigg( \frac{\phat(\sss_{k-1})}{\phat(\sss_{k})}\bigg)^2+\\
+ (1+ \xi_{k-1}) \big( e^{\sum_{l=1}^{k-2} \sigma_l} \big)\bigg( \frac{\phat(\sss_{k-1})}{\phat(\sss_{k})}\bigg)^2   + 
(1- \xi_{k}) \big( e^{\sum_{l=1}^{k-1} \sigma_l} \big) , 
\end{multline}
\begin{multline}
\label{Msderiv3}
\big( V(\sss_k) + \big(F^\phat_{k-}\big)^2\big)\frac{ \partial \sss_{k}}{ \partial \sigma_j}
= 
\big( V(\sss_{k-1}) + \big(F^\phat_{k-1+}\big)^2\big)\frac{ \partial \sss_{k-1}}{ \partial \sigma_j}\bigg( \frac{\phat(\sss_{k-1})}{\phat(\sss_{k})}\bigg)^2+\\
+F_1(1+\xi_{k-1})\frac{\partial}{\partial \sigma_j} \big( e^{\sum_{l=1}^{k-2} \sigma_l} \big) \bigg( \frac{\phat(\sss_{k-1})}{\phat(\sss_{k})}\bigg)^2
- F_1(1- \xi_{k}) \frac{\partial}{\partial \sigma_j}\big( e^{\sum_{l=1}^{k-1} \sigma_l} \big) ,
\end{multline}
\begin{multline}
\label{Msderiv2}
\big( V(\sss_k) + \big(F^\phat_{k-}\big)^2\big)\frac{ \partial \sss_{k}}{ \partial \xi_j}
= 
\big( V(\sss_{k-1}) + \big(F^\phat_{k-1+}\big)^2\big)\frac{ \partial \sss_{k-1}}{ \partial \xi_j}\bigg( \frac{\phat(\sss_{k-1})}{\phat(\sss_{k})}\bigg)^2+\\
+\delta_{j(k-1)} \big( e^{\sum_{l=1}^{k-2} \sigma_l} \big)  F_1\bigg( \frac{\phat(\sss_{k-1})}{\phat(\sss_{k})}\bigg)^2- 
\delta_{jk} \big( e^{\sum_{l=1}^{k-1} \sigma_l} \big) F_1.
\end{multline}
\end{lemma}

\begin{proof}
Below, we compute partial derivatives of $\sss_k$ with respect to $F_1, \sss_{k-1}$, and the entries of $\left.\bsigmaunder\right|_{k}$, 
from which the smoothness claimed follows immediately.  
To this end, we recall from \eqref{EFxi} and \ref{dHflux} that on $\cyl_{[\sss_{k-1}, \sss_{k}]}$ 
\begin{align}
\label{EHbump}
\phat = \phat(\sss_{k-1}) H\left[ (1+ \xi_{k-1})  \big( e^{\sum_{l=1}^{k-2}\sigma_l} \big)  F_1; \sss_{k-1} \right]. 
\end{align} 
For $H$ as in \eqref{EHbump} and using \eqref{EFxi}, we find
\begin{align}
\label{EFconstraint}
F^{H}_-(\sss_{k}) = (1-\xi_{k})  \big( e^{\sum_{l=1}^{k-1}\sigma_l} \big)  F_1.
\end{align}
Items \eqref{Msderiv}-\eqref{Msderiv2} then follow by using the chain rule to differentiate \eqref{EFconstraint} and Lemma \ref{LFmono} and both parts of \ref{LHmono} to calculate the partial derivatives of $F^H_-$.
\end{proof}

\begin{lemma}[Estimates for large $k$] 
\label{Lrldest}
For all $(F, \bsigmaunder ) \in S_{k+1}\setminus S_{k+2} $ with $|\bsigma|_{\ell^1}$ bounded and $|\xibold|_{\ell^\infty}< 1/10$, 
the RLD solution $\phat = \phat[F; \bsigmaunder]$ 
satisfies the following, where $C$ denotes constants depending only on an upper bound of $|\bsigma|_{\ell^1}$. 
\begin{enumerate}[label=\emph{(\roman*)}]
\item $F_1 \Sim_{C} \frac{1}{k}$ and $\sss_k < \frac{1}{2} \log Ck$.
\item 
For $2\leq i \leq k$ 
we have $\frac{1}{C k } < \sss_{i}-\sss_{i-1} <C$ 
and $\left| \log \frac{\phat(\sss_i)}{\phat(\sss_{i-1})}\right| < \frac{C}{k}$.  
\item 
$\big\| 1- \phat(\sss) : C^0\big( \cyl_{[-\sss_k, \sss_k]} \big) \big\| < \frac{C}{k}\log k$.  
\end{enumerate}
\end{lemma}

\begin{proof}  
By Lemma \ref{LFmono} we conclude that the maximum of $|F^{\phat}_-|$ is achieved at the jump latitudes. 
Using also \ref{rsigbij} we conclude 
(where $\Eunder$ in \ref{rsigbij} depends only on an upper bound of $|\bsigma|_{\ell^1}$ by the hypotheses) that 
\begin{align} 
\label{EFcomp} 
\max_{\sss\in[0,l)} |F^{\phat}_-(\sss)| 
 =  
\max_{i=1}^{k+1} {F^\phat_{i\pm}}
 \Sim_{ \Eunder }  
\min_{i=1}^{k+1}{ F^\phat_{i\pm}}.
\end{align}
Since $\phat$ has $k+1$ jumps, we may argue as in the proof of \ref{Pexist}, and by using also 
\eqref{EFcomp} and \ref{AV2} 
\begin{align}
\label{EfVest}
 \frac{1}{\Eunder} F_1 <  F^\phat_{k+} <  F^{\phi_{\mathrm{end}}}_+(\sss_k) < C e^{-2\sss_k  }.
\end{align}
By using Lemma \ref{LFmono} on $\cyl_{[ 0, \sss_1]}, \dots, \cyl_{ [\sss_{k-1}, \sss_k]}$ and summing, we find
\begin{align}
\label{Emostfluxes}
F^\phat_{1-}+ F^\phat_{1+}+ \cdots + F^\phat_{k-} =  \int_{0}^{\sss_k} V(\sss) + \big( F^\phat_-(\sss)\big)^2 \, d\sss.
\end{align}
Next using \ref{RLquant}, \eqref{EFcomp}, and \eqref{EfVest} to estimate \eqref{Emostfluxes}, we find
\begin{align}
\label{Efavg}
\frac{1}{\Eunder} \int_0^{\sss_k} V(\sss) d\sss <  (2k-1)F_1 < \Eunder \big(\| V\|_{L^1(\cyl_I)}+ \Eunder\sss_k (F^{\phi_{\mathrm{end}}}_+(\sss_k))^2\big)
\end{align}
from which we conclude using \ref{AV2} the first part of (i) and then by \eqref{EfVest} and \ref{AV2} the rest of (i).
For (ii),  \ref{LFmono}(ii) and the mean value theorem imply that for some $\sss' \in (\sss_{i-1}, \sss_i)$,
\begin{align}
 \sss_{i}- \sss_{i-1} =  \frac{F^\phat_{i-1+} + F^{\phat}_{i-}}{V(\sss') + (F^\phat(\sss'))^2}.
\end{align}
Estimating a trivial upper bound for $V$, using (i) and \eqref{EFcomp} gives the first inequality in the first part of (ii), and using that $\frac{1}{V(\sss')} \le \frac{C}{V(\sss_k)}$ and using \eqref{EfVest} and \ref{AV2}, we conclude $\sss_i - \sss_{i-1} \le C$.  Using then Definition \ref{dF} and part (i), we have
\begin{align*}
\left| \log \frac{\phat(\sss_i)}{\phat(\sss_{i-1})}\right| \leq \int_{\sss_{i-1}}^{\sss_i} \big|F^\phat_+(\sss)\big| d\sss 
\leq \frac{C}{k}\int_{\sss_{i-1}}^{\sss_i} d\sss \leq \frac{C}{k}, 
\end{align*}
which completes the proof of (ii), and (iii) follows similarly.
\end{proof}

\begin{corollary}[Estimates for the derivatives of $\sss_k$] 
\label{Csderiv}
If
$\phat = \phat[F; \bsigmaunder]$ is an RLD solution as in \ref{Pexist}, 
$\sbold = \sbold^{\phat[F;\bsigmaunder]}$, 
 $(F, \left. \bsigmaunder\right|_{k+1}) \in S_{k+1}$ (recall \ref{Sk}),
 and
$|\xibold|_{\ell^\infty}<\min(\frac{1}{10}, \frac{1}{k})$, 
the following estimates hold, where $C$ depends only on an upper bound of $|\bsigma|_{\ell^1}$. 
\begin{enumerate}[label=\emph{(\roman*)}]
\item 
$  \big( V( \sss_{k}) + \big(F^\phat_{k-}\big)^2\big)\frac{ \partial \sss_{k}}{ \partial F_1}
\Sim_C k$.
\
\item $
	\left| \big( V( \sss_{k} )+ \big(F^\phat_{k-}\big)^2\big)\frac{ \partial \sss_{k}}{ \partial \sigma_i }\right| < C
	$, $i=1, \dots, k-1$.
\item
	$
	\left|\big( V( \sss_{k}) + \big(F^\phat_{k-}\big)^2\big)\frac{ \partial \sss_{k}}{ \partial 
	\xi_j}	\right| < \frac{C}{k}
	$, $j=1, \dots, k$.
\end{enumerate}
\end{corollary}

\begin{proof}
We first prove (i).  
To simplify notation in this proof, we denote
\begin{equation*}
\begin{gathered}
P_k := \big( V( \sss_k )+ \big( F^\phat_{k-}\big)^2\big) \frac{ \partial \sss_k}{\partial F_1}, 
\quad
Q_{k-1} := \bigg( \frac{\phat(\sss_{k-1})}{\phat(\sss_k)}\bigg)^2, 
\\
R_k := \frac{ V( \sss_k) + \big( F^\phat_{k+}\big)^2}{ V( \sss_k) + \big( F^\phat_{k-}\big)^2}, 
\\
T_{k-1} :=   Q_{k-1} (1+ \xi_{k-1}) \big( e^{\sum_{l=1}^{k-2} \sigma_l} \big)  + 
(1- \xi_{k}) \big( e^{\sum_{l=1}^{k-1} \sigma_l} \big). 
\end{gathered}
\end{equation*}
In this notation, \eqref{Msderiv} from Lemma \ref{Lsderiv} is equivalent to the equation
\begin{align}
\label{Esderivrec}
P_k =  R_{k-1} Q_{k-1}P_{k-1} + T_{k-1}
\end{align}
from which we conclude by applying \eqref{Msderiv} recursively
\begin{align}
\label{Epk}
P_k = P_1 \prod_{i=1}^{k-1} Q_i R_i + \sum_{i=1}^{k-1} \bigg( T_i \prod_{j=i+1}^{k-1} Q_j R_j\bigg).
\end{align}
From \eqref{Msderiv0} it follows that $P_1 = 1 - \xi_1$.  
By \ref{rsigbij},  \ref{Lrldest}, and the assumptions, the following estimates hold:  
\begin{align}
\label{Eqrest}
Q_i \Sim_{1+ C/k} 1, \quad R_i \Sim_{1+C/k} 1, \quad T_i \Sim_C 1, \qquad i=1, \dots, k-1.
\end{align}
Combining \eqref{Eqrest} with \eqref{Epk} completes the proof of (i).  
Proofs of (ii) and (iii) are similar and use respectively \eqref{Msderiv2} and \eqref{Msderiv3} in place of \eqref{Msderiv}, so we omit the details. 
\end{proof}

\subsection*{Smooth at the ends RLD solutions}

We concentrate now on smooth at the ends RLD solutions and 
we introduce a unified notation in terms of their flux ratios and total number $\kcir$ of parallel circles:  

\begin{lemma}[RLD solutions {$\phat[{{{\bsigmaunder: \kcir}}}]$}]
\label{Nphik}
Given $\kcir\in \N$ with $\kcir \geq \kcirmin$ (recall \ref{Dkmin}) and $\bsigmaunder = (\bsigma, \xibold) \in \R^{k-1}\times \R^k$, 
$k := \lceil \kcir/2\rceil$, satisfying $|\xibold|_{\ell^\infty} < 1$ and $\xi_1 = 0$ if $\kcir$ is odd, 
there is a unique, smooth at the ends, unit RLD solution (recall \ref{RL}) 
$\phat = \phat[ {{{\bsigmaunder : \kcir}}}]$ satisfying $\kcir [ \sbold^\phat] = \kcir$ 
and $\bsigmaunder^\phat = \bsigmaunder$. 
\end{lemma}

\begin{proof}
Recall from Proposition \ref{Pexist} that $\phat[F; \bsigmaunder]$ has $k\geq \kminev$ jump latitudes precisely when $F \in [a_{k, \bsigmaunder}, a_{k-1, \bsigmaunder})$ 
and from \ref{Pexist2} that $\phicheck[F; \bsigmaunder]$ has $k \geq \kminodd$ jump latitudes when $F \in [b_{k, \bsigmaunder}, b_{k-1, \bsigmaunder})$.  
By \ref{AV1} and the flux monotonicity in \ref{Pexist}(i) and \ref{Pexist2}(i), 
there exist unique $\widetilde{a}_{k, \bsigmaunder} \in [a_{k, \bsigmaunder}, a_{k-1, \bsigmaunder})$ for all $k\geq \kminev$ such that 
$\phat[ \widetilde{a}_{k, \bsigmaunder} ; \bsigmaunder]$ is smooth at the ends,
and unique $\widetilde{b}_{k, \bsigmaunder}\in [b_{k, \bsigmaunder}, b_{k-1, \bsigmaunder})$ for all $k\geq \kminodd$ such that 
$\phat[\widetilde{b}_{k, \bsigmaunder}; \bsigmaunder]$ is smooth at the ends.  
To complete the proof, we need only show that $\{2\kminev, 2\kminodd-1\}$ is a set of consecutive natural numbers. 
 
For this, first consider $\bsigmaunder = (\zerobold, \zerobold)$ and the smooth-at-the ends RLD solution $\phicheck = \phicheck[\widetilde{b}_{\kminodd, \bsigmaunder}; \bsigmaunder]$.  
For $\xi$ close to $1$, define $\bsigmaunder' = (\zerobold, (\xi, 0, \dots))$ 
and the RLD solution $\phat = \phat[ (1-\xi) \widetilde{b}_{\kminodd, \bsigmaunder}; \bsigmaunder']$, 
which satisfies $F^\phat_1 = \widetilde{b}_{\kminodd, \bsigmaunder} = F^{\phicheck}_1$ and $F^\phat_{i\pm} = F^{\phicheck}_{i\pm}$ for all $i>1$ 
such that both of the preceding are defined.  
On the other hand, by choosing $\xi$ close enough to $1$, we can ensure that $F^{\phat}_{1-} = (1-\xi)\widetilde{b}_{\kminodd, \bsigmaunder}$ is as small as desired.  
By \ref{LFmono} and \ref{LHmono}, it follows that for $\xi$ close enough to $1$, $\kcir[\sbold^\phat] = 2\kminodd$.  
An analogous argument shows that we can find an RLD solution $\phicheck$ with $\kcir[\sbold^{\phicheck}] = 2\kminev+1$, and these two assertions complete the proof.
 \end{proof}

\begin{remark}
In the case that $I = \R$, it may be the case that $\widetilde{a}_{k, \bsigmaunder} = a_{k, \bsigmaunder}$, as was the case for example in \cite[Prop. 3.14]{kapmcg}.  
If $I$ is a finite interval however, compactness guarantees that $\widetilde{a}_{k, \bsigmaunder} > a_{k, \bsigmaunder}$.
\qed 
\end{remark}

\begin{figure}[h]
\centering
\begin{subfigure}[]{.5\textwidth}
\begin{tikzpicture}[scale=.8,
  declare function={
    func(\s)= 
      (\s <-1.377)*(-1.084*tanh(\s))+
   and(\s >=-1.377, \s < -.744)*(.430*(1-\s*tanh(\s)) - 1.187*tanh(\s))+
   and(\s >=-.744, \s < -.399)*( .746*(1-\s*tanh(\s)) - .922*tanh(\s))+
         and(\s >=-.399, \s < -.127)*( .936*(1-\s*tanh(\s)) - .495*tanh(\s))+
      and(\s >=-.127, \s < .127)*( 1-\s*tanh(\s))+
      and(\s >=.127, \s < .399)*( .936*(1-\s*tanh(\s)) + .495*tanh(\s))+
      and(\s >=.399, \s < .744)*( .746*(1-\s*tanh(\s)) + .922*tanh(\s))+
      and(\s >=.744, \s < 1.377)*(.430*(1-\s*tanh(\s)) + 1.187*tanh(\s))+
    (\s >= 1.377)*(1.084*tanh(\s))
 ;
   func2(\s)=    
   	(\s<-1.30746)*(-1.13*tanh(\s))+
      and(\s >=-1.30746, \s < -.661002)*( .49433*(1-\s*tanh(\s))-1.20*tanh(\s))+
      and(\s >=-.661002, \s < -.29855)*( .830575*(1-\s*tanh(\s))-.8322*tanh(\s))+
               and(\s >=-.29855, \s <=0)*( 1-\s*tanh(\s)-.29427*tanh(\s))+
         and(\s >=0, \s <=.29855)*( 1-\s*tanh(\s)+.29427*tanh(\s))+
   and(\s >=.29855, \s < .661002)*( .830575*(1-\s*tanh(\s))+.8322*tanh(\s))+
   and(\s >=.661002, \s < 1.30746)*( .49433*(1-\s*tanh(\s))+1.20*tanh(\s))+
 (\s >= 1.30746)*(1.13*tanh(\s))
    ;
 }
]
\begin{axis}[
legend pos= north east,
  axis x line=middle, axis y line=left,
  ymin=.9, ymax=1.15, ytick={.9, 1, 1.1}, ylabel=$y$,
  xmin=-2, xmax=2, 
  xtick={-2, -1, 0,1, 2},
  xticklabels={$-2$, $-1$, $0$, $1$, $2$},
   xlabel=$\sss$,
   ylabel=,
  samples=630
]

\addplot[black, thick, domain=-2:2, 
]{func(x)};
\addlegendentry{$\phat[\bsigmaunder:8]$}
\addplot[black, thick, dashed, domain=-2:2, 
]{func2(x)};
\addlegendentry{$\phat[\bsigmaunder: 7]$}
\end{axis}
\end{tikzpicture} 
\end{subfigure}
\begin{subfigure}{.25\textwidth}
\end{subfigure}
\begin{subfigure}[]{.3\textwidth}
\begin{tikzpicture}[scale=.4,
  declare function={
    func(\s)= 
      (\s <-1.377)*(-1.084*tanh(\s))+
   and(\s >=-1.377, \s < -.744)*(.430*(1-\s*tanh(\s)) - 1.187*tanh(\s))+
   and(\s >=-.744, \s < -.399)*( .746*(1-\s*tanh(\s)) - .922*tanh(\s))+
         and(\s >=-.399, \s < -.127)*( .936*(1-\s*tanh(\s)) - .495*tanh(\s))+
      and(\s >=-.127, \s < .127)*( 1-\s*tanh(\s))+
      and(\s >=.127, \s < .399)*( .936*(1-\s*tanh(\s)) + .495*tanh(\s))+
      and(\s >=.399, \s < .744)*( .746*(1-\s*tanh(\s)) + .922*tanh(\s))+
      and(\s >=.744, \s < 1.377)*(.430*(1-\s*tanh(\s)) + 1.187*tanh(\s))+
    (\s >= 1.377)*(1.084*tanh(\s))
 ;
   func2(\s)=    
   	(\s<-1.30746)*(-1.13*tanh(\s))+
      and(\s >=-1.30746, \s < -.661002)*( .49433*(1-\s*tanh(\s))-1.20*tanh(\s))+
      and(\s >=-.661002, \s < -.29855)*( .830575*(1-\s*tanh(\s))-.8322*tanh(\s))+
               and(\s >=-.29855, \s <=0)*( 1-\s*tanh(\s)-.29427*tanh(\s))+
         and(\s >=0, \s <=.29855)*( 1-\s*tanh(\s)+.29427*tanh(\s))+
   and(\s >=.29855, \s < .661002)*( .830575*(1-\s*tanh(\s))+.8322*tanh(\s))+
   and(\s >=.661002, \s < 1.30746)*( .49433*(1-\s*tanh(\s))+1.20*tanh(\s))+
 (\s >= 1.30746)*(1.13*tanh(\s))
    ;
 }
]
\begin{axis}[
legend pos=outer north east,
  axis x line=middle, axis y line=left,
  ymin=.9, ymax=1.15, ytick={.9, 1, 1.1}, ylabel=$y$,
  xmin=-4, xmax=4, 
  xtick={-2, -1, 0,1, 2},
  xticklabels={$-2$, $-1$, $0$, $1$, $2$},
   xlabel=$\sss$,
   ylabel=,
  samples=630
]

\addplot[black, thick, domain=-4:4, 
]{func(x)};
\addlegendentry{$\phat[\bsigmaunder:8]$}
\addplot[black, thick, dashed, domain=-4:4, 
]{func2(x)};
\addlegendentry{$\phat[\bsigmaunder: 7]$}
\end{axis}
\end{tikzpicture} 
\end{subfigure}
\caption{Profiles of RLD solutions $\phat[\bsigmaunder:7]$ and $\phat[\bsigmaunder:8]$ with respectively $7$ and $8$ singular circles.  In each case $\bsigmaunder = \zerobold$ and $V = 2\sech^2 \sss$, corresponding to the case where $\Sigma$ is an equatorial $\Sph^2$ in $\Sph^3$.  The right image depicts the same profiles over a wider domain, to emphasize the smooth at the ends behavior.}
\end{figure}

\begin{convention}
\label{conkcir}
Hereafter, we assume $\kcir \in \N$ with $\kcir\geq \kcirmin$ is given, and define $k\in\N$ by $k = \lceil \kcir/2 \rceil$, i.e. $\kcir = 2k$ if $\kcir$ is even and $\kcir = 2k-1$ if $\kcir$ is odd. 
\end{convention}

\begin{lemma}[Characterization of low $\kcirmin$]
\label{Rk1}
$\phantom{ab}$
\begin{enumerate}[label=\emph{(\roman*)}]
\item $\kcirmin = 1$ if and only if $\phiend>0$ on $\cyl_{[0, l)}$.
\item $\kcirmin = 2$ if and only if  for some $\sss^\phiend_{\mathrm{root}}\geq 0$, $\phiend > 0$ on $(\sss^\phiend_{\mathrm{root}}, l)$, $\phiend(\sss^\phiend_{\mathrm{root}}) =0$ and $\phie>0$ on $(0, \sss^\phiend_{\mathrm{root}}]$.
\item $\kcirmin \geq 3$ if and only if for some $0< \sss^{\phie}_{\mathrm{root}} < \sss^\phiend_{\mathrm{root}}$, $\phiend>0$ on $(\sss^\phiend_{\mathrm{root}}, l)$, $\phiend(\sss^\phiend_{\mathrm{root}}) = 0$, $\phie> 0$ on $(0, \sss^\phie_{\mathrm{root}})$, and $\phie( \sss^\phie_{\mathrm{root}}) =0$.  
\end{enumerate}
\end{lemma}
\begin{proof}
We first prove (i).  If $\phiend> 0$ on $\cyl_{[0, l)}$, then clearly the function on $C^\infty_{|\sss|}(\cyl^{(0)})$ defined by $\sss \mapsto \phiend(|\sss|)$ is an RLD solution with $\kcir = 1$, so $\kcirmin = 1$.  Conversely, if $\kcirmin = 1$, then in the notation of \ref{Pexist2}(ii), $\phicheck[F; \bsigmaunder]$ has only one jump latitude for $F \in [b_{1, \bsigmaunder}, b_{0, \bsigmaunder})$.  By the flux monotonicity \ref{LFmono} there is a unique $F$ in this interval such that $\phicheck[F; \bsigmaunder](\sss) = \phiend(|\sss|)$.  By \ref{RL}(i), it follows that $\phiend >0$ on $\cyl_{[0, l)}$.  The proofs of (ii)-(iii) are straightforward from \ref{LFmono}, so we omit the details. 
\end{proof}

\begin{remark}
\label{Rtorrld}
If $\Sigma$ is a torus, recall from \ref{LAconf} that $\Sbar$ fixes the two equatorial circles $\widetilde{X}_\Sigma(\cyl_0)$ and $\widetilde{X}_\Sigma(\cyl_l)$.  
Up to redefining $\widetilde{X}_\Sigma$, we could interchange the role of which equatorial circle $\cyl_0$ maps to, which would interchange the roles of $\phie$ and $\phiend$.  
This would lead to slightly different classes of RLD solutions; in particular if either $\phie>0$ or $\phiend> 0$ on $\cyl_{[0,l)}$, 
we would have $\kcirmin = 1$ with respect to at least one of these choices (recall \ref{Rk1}(i)).  
It would also be possible (recall Remark \ref{RLcard}(ii)) to consider an expanded class of RLD solutions which have jumps on both equators, 
although we have not considered this case for the sake of simplicity.
\qed 
\end{remark} 

\begin{prop}
\label{PODEest}
Suppose $\phat = \phat[ \bsigmaunder: \kcir]$ and $\phat' = \phat[ \bsigmaunder': \kcir]$ are as in \ref{Nphik}, 
where $\bsigmaunder = (\bsigma, \xibold)\in \R^{k-1}\times \R^k , \bsigmaunder'= (\bsigma', \xibold') \in \R^{k-1}\times \R^k$ satisfy 
$ |\xibold|_{\ell^\infty}< \min(\frac{1}{10}, \frac{1}{k})$,  $|\xibold'|_{\ell^\infty}< \min(\frac{1}{10}, \frac{1}{k})$.  
There is a constant $C>0$ depending only on an upper bound of $|\bsigma|_{\ell^1}$ and $|\bsigma'|_{\ell^1}$ such that
\[ \left| \Fboldunder^{\phat'} - \Fboldunder^\phat \right|_{\ell^\infty} \le 
\frac{C}{k}\left(  | \bsigma' - \bsigma|_{\ell^1}+ |\xibold' - \xibold|_{\ell^\infty}\right).\] 
\end{prop}

\begin{proof}
Observe that the conclusion follows from the estimate
\begin{align}
\label{Efluxdiff1}
\left| F^{\phat'}_1 - F^{\phat}_1 \right| \le \frac{C}{k}\left(  |\bsigma' - \bsigma |_{\ell^1} + |\xibold' - \xibold|_{\ell^\infty}\right), 
\end{align}
since for any $i\in \{1, \dots, k\}$ we have (taking $+$ or $-$ in every instance of $\pm$ below)
\begin{multline*}
\left| F^{\phat'}_{i\pm} - F^\phat_{i\pm}\right| = 
\left| (1\pm \xi'_i ) \big( e^{\sum_{l=1}^{i-1}\sigma'_l} \big)  F^{\phat'}_1 - (1\pm \xi_i)  \big( e^{\sum_{l=1}^{i-1}\sigma_l} \big)  F^\phat_{1}\right| 
\lem \, 
\\
\left|  (1 \pm \xi'_i)  \big( e^{\sum_{l=1}^{i-1}\sigma'_l} \big)  -  (1 \pm \xi_i)  \big( e^{\sum_{l=1}^{i-1}\sigma_l} \big) \right|  \, F^{\phat'}_1+ 
(1 \pm \xi_i)  \big( e^{\sum_{l=1}^{i-1}\sigma_l} \big) \left| F^{\phat'}_1 - F^{\phat}_1\right| .
\end{multline*}

We now prove \eqref{Efluxdiff1}.  Fix $k\in \N$ and 
consider the map 
defined by
\begin{equation}
\label{EFcal}
\begin{aligned}
\Fcal ( F_1, \bsigmaunder ) := & 
 F^{\phat[ F; \bsigmaunder]}_+ \big(\sss_k^{\phat[ F; \bsigmaunder]} \big) - F^{\phiend}_+ \big(\sss_k^{\phat[ F; \bsigmaunder]} \big) 
\\
= & 
(1+ \xi_k) \big( e^{\sum_{l=1}^{k-1} \sigma_l} \big) F_1  
- F^{\phiend}_+ \big(\sss_k^{\phat[ F; \bsigmaunder]} \big) , 
\end{aligned}
\end{equation} 
where $\phat[ F; \bsigmaunder]$ is as in \ref{Pexist}
and the second equality uses \eqref{EFxi}.  
Clearly, $\Fcal ( F_1, \bsigmaunder) = 0$ if and only if $\phat[ F; \bsigmaunder] $ is smooth at the ends.  
Now let $\pointp \in \Fcal^{-1}\left(\{0\}\right)$ be arbitrary.

It follows from Lemma \ref{LFmono} and \ref{LHmono} that $\Fcal$ is $C^1$; below we estimate the partial derivatives of $\Fcal$ at $\pointp$.
Differentiating \eqref{EFcal} with respect to $F_1$ and using \ref{LFmono}, we compute
\begin{align}
\label{EFF1}
\left. \frac{ \partial \Fcal }{ \partial F_1}\right |_{\pointp} &= 
(1+ \xi_k) \big( e^{\sum_{l=1}^{k-1} \sigma_l} \big)  
+ \big( V( \sss_k) + \big(F^\phat_{k+}\big)^2\big) \left.\frac{ \partial \sss_k}{\partial F_1}\right|_{\pointp}, 
\end{align}
and similarly the derivatives with respect to $\sigma_i$ and $\xi_j$. 
By combining with Corollary \ref{Csderiv}, we estimate that for $j\in \{1, \dots, k\}$ and $i\in \{1, \dots, k-1\}$,
\begin{align}
\label{EF1deriv}
\left.\frac{\partial \Fcal}{\partial F_1}\right|_{\pointp} \Sim_C k, \qquad 
\left| \left.\frac{\partial \Fcal}{\partial \sigma_i}\right|_{\pointp} \right| \le C, \qquad
\left|\left. \frac{\partial \Fcal}{\partial \xi_j}\right|_{\pointp} \right| \le \frac{C}{k},
\end{align}
where $C>0$ is a constant independent of $k$.

By the implicit function theorem, $\Fcal^{-1}\left(\{ 0 \} \right)$ is a graph of a function of $\bsigmaunder$ in the vicinity of any given  $(F_1, \bsigmaunder) \in \Fcal^{-1}\left( \{0\}\right)$,  and moreover (abusing notation slightly), for $i\in \{1, \dots, k-1\}$ and $j\in \{1, \dots, k\}$,
\begin{equation} 
\label{Eift}
\begin{aligned}
\left.\frac{\partial F_1}{\partial \sigma_i}\right|_{\bsigmaunder}  =&  
- \bigg(\left. \frac{ \partial \Fcal}{ \partial F_1}\right|_{\pointp}\bigg)^{-1} \left.\frac{ \partial \Fcal }{ \partial \sigma_i}\right|_{\pointp}, 
\\   
\left.\frac{\partial F_1}{\partial \xi_j} \right|_{\bsigmaunder}  =&  
- \bigg( \left.\frac{ \partial \Fcal}{ \partial F_1}\right|_{\pointp}\bigg)^{-1}\left. \frac{ \partial \Fcal }{ \partial \xi_j}\right|_{\pointp}.
\end{aligned}
\end{equation} 
The conclusion follows by combining this with the estimates \eqref{EF1deriv}.  
\end{proof}

\section{LD solutions from RLD solutions}
\label{S:LDs}

\subsection*{Basic facts} 
\nopagebreak

Given $\munder \in \Z\setminus \{ 0\} $, we define a scaled metric $\chitilde = \chitilde[\munder]$ on $\cyl$ and scaled coordinates
  $(\, \stilde[\munder], \thetatilde[\munder]\, )$
defined by
\begin{align}
\label{Echitilde}
\chitilde:= \munder^2 \chi, 
\qquad 
\stilde =| \munder |\sss, 
\qquad 
\thetatilde =| \munder |\theta.
\end{align}
We also define corresponding a scaled linear operator
\begin{equation} 
\begin{aligned}
\Lcal_{\chitilde[\munder]}: = \Delta_{\chitilde} + \munder^{-2} V = \munder^{-2} \Lchi.
\end{aligned}
\end{equation}

\begin{definition}
\label{dchi}
Given $\sbar \in \R_+$ and $\munder \in \Z\setminus \{0\}$, we define a shifted coordinate $\shat = \shat\, [\sssunder, \munder]$ by 
\[\shat := \stilde - |\munder|\,  \sbar = |\munder| \, (\sss-\sbar) .\] 
\end{definition}

\begin{definition}
\label{dOmega}
Given $\munder \in \Z\setminus \{0\}$, we define
\begin{align}
\label{Edelta}
\delta[\munder] := 1/(9|\munder|).
 \end{align}
 Given $\sbar$, $\sbold$, and $\mbold$ as in \ref{dLmbold}, for $i=1, \dots, k$ we define $\delta_i := \delta[m_i]$ and define nested open sets $D^\chi_{L[\sbar;\munder]}(3\delta[\munder]) \subset \Omega'[\sbar;\munder]\subset \Omega[\sbar;\munder]$, where
\begin{align*}
\Omega[\sbar; \munder] : = D^{\chi[\munder]}_{\Lpar[\sbar]}\left(3/|\munder
|\right) = D^{\chitilde[\munder]}_{\Lpar[\sbar]}\left(3\right),\\   
 \Omega'[\sbar; \munder] := D^{\chi[\munder]}_{\Lpar[\sbar]}\left(2/|\munder|\right) = D^{\chitilde[\munder]}_{\Lpar[\sbar]}\left(2\right). 
\end{align*}
We also define $\Omega[\sbold; \mbold] : = \bigcup_{i=1}^k \Omega[\sss_i; m_i]$ and $\Omega'[\sbold; \mbold]: = \bigcup_{i=1}^k \Omega'[\sss_i; m_i]$. 
\end{definition}

\begin{figure}[h]
	\centering
	\begin{tikzpicture}[scale=.65]
\draw[-, dashed, black] (-7, -3) to (7, -3);
\draw[-, dashed, black] (-7, -2) to (7, -2);
\draw[-, dashed, black] (-7, -1/2) to (7, -1/2);
\draw[-, dashed, black] (-7, 0) to (7, 0);
\draw[-, dashed, black] (-7, 1/2) to (7, 1/2);
\draw[-, dashed, black] (-7, 2) to (7, 2);
\draw[-, dashed, black] (-7, 3) to (7, 3);
\draw[-, thick] (-4, -1.25) node[]{$\Lpar[\sbar]$};
\draw[-, thick] (0, 1.25) node[]{$D^\chi_{\Lpar[\sbar]}(3\delta)$};
\draw[-, thick] (8, 0) node[]{$\Omega[\sbar; \munder]$}; 
\draw[-, thick] (4, 1.25) node[]{$\Omega'[\sbar; \munder]$};
\draw[->, thick] (-4, -1)  to  (-4, 0);
\draw[-, thick] (0, .7)  to (0, 0) to (1/2, 0) to (1/2, .4);
\draw[-, thick] (.3, .4) to (.7, .4);
\draw[-, thick] (1/2, 0) to (1/2, -.4) to (.3, -.4) to (.7, -.4);
\draw[-, thick] (8, -3) to (8, -1/2);
\draw[-, thick] (8, .5) to (8, 3); 
\draw[-, thick] (7.8, 3) to (8.2, 3); 
\draw[-, thick] (7.8, -3) to (8.2, -3); 
\draw[-, thick] (4, .8)  to (4, -1.9); 
\draw[-, thick] (4, 1.9) to (4, 1.5);
\draw[-, thick] (3.8, -1.9) to (4.2, -1.9); 
\draw[-, thick] (3.8, 1.9) to (4.2, 1.9); 
\end{tikzpicture}
\caption{A schematic of connected components of the neighborhoods of $\Lpar[\sbar]$ (defined in \ref{dOmega}) near latitude $\sbar$.}
\end{figure}

\begin{definition}[Antisymmetry operators] 
\label{Dantisym} 
Given a domain $\Omega \subset \cyl$ satisfying $\Sbar_{\sbar}(\Omega) = \Omega$ for some $\sbar \in \R$ (recall \ref{Dgroup}), we define operators $\Rcal_{\sbar}$ and $\Acal_{\sbar}$, each acting on real-valued functions defined on $\Omega$, by 
$\Rcal_{\sbar} u = u \circ \Sbar_{\sbar}$ and 
$\Acal_{\sbar} u  = u - \Rcal_{\sbar} u$. 
\end{definition}

\begin{lemma}
\label{Rprod}
Let $\sbar \in (\frac{3}{\munder}, \infty)$ and $\munder \in \Z \setminus \{ 0\}$.  
The following hold:
\begin{enumerate}[label=\emph{(\roman*)}]
\item  For all $u, v \in C^0(\Omega[\sbar; \munder])$,  
$ 
\Acal_{\sbar}\left( u v\right) = u \, \Acal_{\sbar}  v  + \Acal_{\sbar}  u  \, \Rcal_{\sbar}  v . 
$  

\item  For all $u \in C^2(\Omega[\sbar;\munder])$, $[\Acal_{\sbar}, \Lcal_\chitilde ] u =   \munder^{-2} \Acal_{\sbar} V \, \Rcal_{\sbar} u $.

\item 
$\left\| V: C^j\left( \Omega , \chi\right) \right\| \le C(j) \left \| V : C^0\left(\Omega, \chi \right)\right\|$  
holds for any domain $\Omega \subset \cyl_I$.   

\item 
$\left\| \Acal_{\sbar} V: C^j\left( \Omega[\sbar; \munder], \chi \, \right)\right\| \le  
\frac{C(j)}{|\munder|} V(\sbar) $.  
\end{enumerate}
\end{lemma}
\begin{proof}
(i) follows from a straightforward computation, and (ii) follows from (i) and a similar computation, using the fact that $\Delta$ commutes with $\Acal_{\sbar}$.  (iii) follows from \ref{AV1}.  (iv) is a discrete version of (iii) which follows from the mean value theorem and (iii). 
\end{proof}

\begin{lemma}[Green's functions {\cite[Lemma 2.28]{kapmcg}}]
\label{Lgreen}
There exists $\epsilon>0$ depending only on $V$ such that for any $p\in \cyl$, there exists a Green's function $G^\chi_p$ for $\Lchi$ on $D^\chi_p(\epsilon)$ satisfying: 
 \begin{align*}
 \left\| G^{\chi}_p - \log r : C^j( D^\chi_p(\epsilon) \setminus \{p\}, r, \chi, r^2|\log r|)\right| \leq C(j), 
 \quad{where} \quad
  r = \dbold^{\chi}_p.
  \end{align*}
\end{lemma}
\begin{proof}
This follows from \ref{Lgreenlog} and \ref{AV1}.
\end{proof}

\begin{definition}[{\cite[Definition 2.21]{kapmcg}}]
\label{dauxode}
Given $a,b,c\in\R$ and $\sbar \in \R_+$, we define (recall \ref{Dlpar})
\begin{equation*}
\begin{aligned}
\phiunder=\phiunder [a,b;\sbar]\in \,\,
C^\infty_{|\sss|}\big( \, \cyl^{(0)} \,\big ),
\quad 
\junder=\junder\big[c\, ;\sbar\big]\in \,\,
C^\infty_{|\sss|}\big( \cyl^{(0, \sbar)} \,\big)
\end{aligned}
\end{equation*}
by requesting they satisfy the initial data 
$$
\phiunder(\sbar)=a,
\qquad
 \partial \phiunder (\sbar) =  b ,
\qquad
\junder(\sbar)=0,
\qquad
\partial_{+}\junder(\sbar)=\partial_{-}\junder(\sbar)=c,
$$
and the ODEs $\Lchi \phiunder=0$ on
$\cyl^{(0)}$, 
and $\Lchi \junder=0$ on $\cyl^{(0, \sbar)}$.
\end{definition}

\begin{remark}
\label{Rauxode}
Note that $\phiunder$ depends linearly on the pair $(a, b)\in \R^2$ and $\junder$ depends linearly on $c \in \R$.  
\qed
\end{remark}

\begin{lemma}[{\cite[Lemma 2.23]{kapmcg}}]
\label{Lode}
For all $\munder\in \N$ large enough and $\sbar \in (\frac{3}{\munder}, \infty)$, the following estimates hold (recall \ref{Nsym}).  
\begin{enumerate}[label=\emph{(\roman*)}]
\item
$\left\|  \, \phiunder[1, 0 ;\sbar] -1 \, : \, C_{|\sss|}^{j}( \, \Omega[\sbar; \munder]\,,\chitilde[\munder] \, )\, \right\|\, \le \, C(j)/\munder^2$.

\item
$\left\|  \, \junder[\munder;\sbar] - |\, \shatbar \, |  \, : \, C_{|\sss|}^{j}( \, \Omega[\sbar; \munder] \setminus \Lpar[\sbar] \,,\chitilde[\munder] \, ) \, \right\| \, \le \,  C(j)/\munder^2\, $.

\item
$\left\| \, \Acal_{\sbar} \,   \phiunder[1, 0 ;\sbar] \, : C_{|\sss|}^{j}( \, \Omega[\sbar; \munder]\, , \chitilde[\munder] \, ) \, \right\|\le \, C(j)/\munder^3$.

\item
$\left\|  \, \Acal_{\sbar} \,   \junder[\munder;\sbar] \, : \, C_{|\sss|}^{j}( \, \Omega[\sbar; \munder] \setminus \Lpar[\sbar]  \,,\chitilde[\munder] \, ) \, \right\| \, \le  \,  C(j)/\munder^3 $.
\end{enumerate}
\end{lemma}
\begin{proof}
The proof is an ODE comparison argument which is only superficially different than the proof of \cite[Lemma 2.23]{kapmcg}---here we use properties of $V$ established in \ref{Rprod} 
instead of properties of $2\sech^2 \sss$ as in \cite{kapmcg}---so we omit the details. 
\end{proof}

\subsection*{Maximally symmetric LD solutions from RLD solutions}
\nopagebreak

For convenience and uniformity, we now identify $\cyl_I$ with $X_\Sigma(\cyl_I) \subset \Sigma$ (recall Remark \ref{Rcylid}), 
which allows us to study LD solutions on the cylinder instead of on $\Sigma$.

\begin{definition}
\label{davg}
Given a function $\varphi$ on some domain $\Omega\subset \cyl$ or $\Omega \subset \Sigma$, we define
a rotationally invariant function $\varphi_\ave$ on 
the union $\Omega'$ of the $O(2)$ orbits on which $\varphi$ is integrable
(whether contained in $\Omega$ or not),
by requesting that on each such orbit $C$, 
\[\left. \varphi_\ave \right|_C
:=\avg_C\varphi.\]
We also define $\varphi_\osc$ on $\Omega\cap\Omega'$ by $\varphi_\osc:=\varphi-\varphi_\ave$.
\end{definition}

\begin{lemma}[$\groupcyl$-Symmetric LD solutions, cf. {\cite[Lemma 3.10]{kap}}]
\label{LsymLD}
For $m\in\N$ large enough (depending only on $\Sigma$) the following hold. 
\begin{enumerate}[label=\emph{(\roman*)}]
\item $\ker \Lcal_\Sigma$ is trivial modulo the $\groupcyl$ action on $\Sigma$.
\item 
Given a $\groupcyl$-invariant invariant configuration $\taubold : L\rightarrow \R$ 
there exists a unique $\groupcyl$-invariant LD solution $\varphi = \varphi[\taubold]$ of configuration $\taubold$.  
\end{enumerate}
\end{lemma}

\begin{proof}
Item (i) follows from the triviality of $\ker \Lcal_\Sigma$ modulo $\grouprotcyl$ in \ref{Aimm}, by taking $m$ large enough.  Item (ii) follows from (i) and applying Lemma \ref{Lldexistence}.
\end{proof}

\begin{lemma}[Vertical balancing, {\cite[Lemma 4.5]{kapmcg}}]
\label{Lvbal}
Suppose $\varphi$ is an LD solution whose configuration $\taubold$ and singular set $L$ are $(\sbold, \mbold)$-rotational as in \ref{dL}.  Then the following hold. 
\begin{enumerate}[label=\emph{(\roman*)}]
\item 
$\varphi_\ave \in C^\infty( \cyl_I^{\sbold})$, where $\Lpar = \Lpar[\sbold]$.  
\item 
On $\cyl^{\sbold}_I$, $\varphi_\ave$ satisfies the ODE $\Lchi \varphi_\ave = 0$. 
\item 
\label{Evbal} 
$|m_i|\tau_i = \partial_+ \varphi_\ave(\sss_i) + \partial_- \varphi_\ave (\sss_i),
\quad i=1, \dots, k$. 
\end{enumerate}
\end{lemma}

\begin{proof}
To prove (i) and (ii), we need to check that $\varphi$ is integrable on each circle contained in $\Lpar$ and that $\varphi_\ave$ is continuous there also.  But these follow easily from the logarithmic behavior of $\varphi$ (recall \ref{dLD}).  We now prove item (iii).
Fix $i\in \{1, \dots, k\}$.  For $0<\epsilon_1 << \epsilon_2$ we consider the domain $\Omega_{\epsilon_1, \epsilon_2}: = D^\chi_{\Lpar[\sss_i]}(\epsilon_2)\setminus D^\chi_{L}(\epsilon_1)$.  By integrating $\Lchi \varphi = 0$ on $\Omega_{\epsilon_1, \epsilon_2}$ and integrating by parts we obtain
\begin{align*}
\int_{\partial \Omega_{\epsilon_1, \epsilon_2}} \frac{\partial}{\partial \eta}  \varphi + \int_{\Omega_{\epsilon_1, \epsilon_2}} V \varphi  = 0,
\end{align*}
where $\eta$ is the unit outward conormal field along $\partial \Omega_{\epsilon_1, \epsilon_2}$.  By taking the limit as $\epsilon_1 \searrow 0$ first and then as $\epsilon_2 \searrow 0$, we obtain \ref{Evbal} by using the logarithmic behavior near $L$.
\end{proof}

\begin{lemma}[Normalized maximally symmetric LD solutions] 
\label{Lphiavg}
Given $\kcir \geq \kcirmin$ and $k$ as in \ref{conkcir}, 
$\mbold \in (\Z \setminus \{ 0 \})^k$ with $m$ (as in \ref{dL}) large enough as in \ref{LsymLD}, 
and  $\bsigmaunder = (\bsigma, \xibold) \in \R^{k-1}\times \R^k$ with $|\xibold|_{\ell^\infty}<1$, 
there is a unique $\groupcyl$-invariant LD solution 
$\Phi = \Phi \llceil  \bsigmaunder: \kcir, \mbold \rrfloor $  
characterized by the following requirements where 
$\bsigmaunderslash=\bsigmaunderslash[\mbold]$ is as in \ref{dLbalanced}. 

 \begin{enumerate}[label=\emph{(\alph*)}]
 \item $\phi = \phi\llceil  \bsigmaunder: \kcir, \mbold \rrfloor :=\Phi_\ave$ is a multiple of $\phat[\bsigmaunderslash+ \bsigmaunder: \kcir]$  (recall \ref{Nphik}).
 \item The singular set of $\Phi$ is $L = L \llceil  \bsigmaunder: \kcir, \mbold \rrfloor := L[\sbold[\bsigmaunderslash+\bsigmaunder:\kcir]; \mbold]$ (recall \ref{dL}).
 \item The configuration $\taubold' := \taubold' \llceil  \bsigmaunder: \kcir, \mbold \rrfloor $  of $\Phi$ 
is a $\big( \sbold[\bsigmaunderslash+\bsigmaunder:\kcir] , \mbold \big)$-symmetric configuration as in \ref{dL} 
satisfying 
$\tau'_1 = 1$ (normalizing condition). 
 \end{enumerate}
 
Moreover,  the following hold.
 \begin{enumerate}[label=\emph{(\roman*)}]
\item  For $i\in \{1, \dots, k\}$  we have 
$\tau_i' = \displaystyle{  \frac{\phi(\sss_i)}{|m_i|}2F^\phi_i}$. 
Moreover $\tau_i'$ is independent of $m$ and satisfies 
$\tau_i' = \tau_i' \llceil  \bsigmaunder: \kcir, \mbold \rrfloor  
:=\displaystyle{\frac{\phat[\bsigmaunderslash+\bsigmaunder: \kcir](\sss_i)}{\phat[\bsigmaunderslash+\bsigmaunder: \kcir ](\sss_1)}  \big( e^{\sum_{l=1}^{i-1}\sigma_l} \big)}$.
\item  $\displaystyle{\phi \llceil  \bsigmaunder: \kcir, \mbold \rrfloor  
\, = \, 
\frac{|m_1|}{\, \phat[\bsigmaunderslash+\bsigmaunder:\kcir](\sss_1) \,\, 2F^{\phat[\bsigmaunderslash+\bsigmaunder:\kcir]}_1\,}\, \phat[\bsigmaunderslash+\bsigmaunder: \kcir ]}$.
\item  
\label{item:ji} 
On $\Omega[\sss_i;m_i]$, 
$\phi = \phiunder_i + \junder_i$, 
where $\phiunder_i : = \tau'_i \phiunder \big[\textstyle{\frac{|m_1|}{2F^\phi_1}}(e^{-\sum_{l=1}^{i-1} \sigma_l}), \textstyle{\frac{|m_i|}{ 2}}\xi_i; \sss_i\big]$ 
and $\junder_i :=  \, \, \junder \left [\textstyle{\frac{|m_i| }{2}}\tau'_i; \sss_i\right]$.
\end{enumerate}
\end{lemma}

\begin{proof}
Let $m$ be as in \ref{LsymLD} and suppose $\Phi$ is a $\groupcyl$-invariant LD solution satisfying (a)-(c).  
Let $c$ be such that $\phi=c\phat$ and  $i\in \{1, \dots, k\}$.
Using Lemma \ref{Lvbal} to solve for $\tau'_i$, we immediately conclude $\tau'_i = \phi(\sss_i) 2F^\phi_i /|m_i|$; 
furthermore using Lemma \ref{Lvbal}, (a)-(c) above, \eqref{Exi}, and the definition of $\bsigmaunderslash$ in \ref{dLbalanced}, we compute
\begin{equation*} 
\begin{aligned}
\tau'_i &= \frac{\tau'_i}{\tau'_1} 
= \frac{|m_1|}{|m_i|} \frac{\phi(\sss_i)}{\phi(\sss_1)} \frac{F^\phi_i}{F^\phi_1}
= \frac{|m_1|}{|m_i|} \frac{\phat(\sss_i)}{\phat(\sss_1)} \left( e^{\sum_{l=1}^{i-1} \sigmaslash_l + \sigma_l}\right) 
=  \frac{\phat(\sss_i)}{\phat(\sss_1)} \left( e^{\sum_{l=1}^{i-1} \sigma_l}\right), \\
1&= \tau'_1 = \frac{\phi(\sss_1)}{|m_1|} 2F^\phi_1 = \frac{c \phat(\sss_1)}{|m_1|}2 F^\phat_1.
\end{aligned}
\end{equation*}
We conclude from these equations that (a)-(c) imply (i) and (ii).  In particular, the second equation in (i) determines $\taubold'$ and hence uniqueness follows from Lemma \ref{LsymLD}.  

To prove existence we define $L$ by (b) and $\taubold'$ by the second equation in (i). 
Using \ref{LsymLD} we then define $\Phi : = \varphi[\taubold']$ and we verify that $\Phi_\ave = c\phat$, where $c$ is defined by $c \phat(\sss_1)  2F^{\phat}_1 = |m_1|$: 
Let $i\in\{1, \dots, k\}$. By Lemma \ref{Lvbal}, it follows that $|m_i| \tau'_i = 2\Phi_\ave(\sss_i) F^{\Phi_\ave}_i$.  
By the definitions of $\tau'_i$ and $c$, we have $|m_i| \tau'_i = c \phat(\sss_i ) 2F^\phat_i$.  
Since $F^\phat_i = F^{c\phat}_i$, by equating the right hand sides of the preceding equations, 
we conclude that the function $f := c\phat - \Phi_{\ave}$ satisfies 
\begin{align}
\partial_+f(\sss_i)+\partial_-f(\sss_i) = 0, \quad i=1, \dots, k.
\end{align}
This amounts to the vanishing of the derivative jumps of $f$ at each $\sss_i$.  
Clearly $f$ is smooth at the ends and satisfies $\Lcal_\chi f = 0$ in between the $\sss_i$.  
Hence we conclude $f\in C^\infty_{|\sss|}(\Sigma)$ and satisfies $\Lcal_\Sigma f = 0$ everywhere.  
By \ref{Aimm}(iv), we conclude $f = 0$.

It remains to check (iii).  
By \ref{Lvbal}\ref{Evbal},
$m_i \tau'_{i} =  \partial_+ \phi(\sss_i) + \partial_-\phi(\sss_i)$,
so from the definition of $\junder$ in \ref{dauxode},
\[ \partial_+\junder_i (\sss_i)  = \partial_-\junder_i (\sss_i)= \frac{\partial_+\phi+\partial_-\phi}{2}(\sss_i). \] 
Therefore, $\phi - \junder_i$ satisfies
\begin{align*}
\partial_+ (\phi - \junder_i)(\sss_i) = \frac{ \partial_+ \phi - \partial_- \phi}{2}(\sss_i) = - \partial_-(\phi - \junder_i)(\sss_i).
\end{align*}
Hence, $\phi - \junder_i\in  C^1_{\sss}\left(\Omega[\sss_i;m_i]\right)$ and $\Lcal_\chi(\phi-\junder_i)  = 0$.  
By uniqueness of ODE solutions,  $\phi - \junder_i = \phiunder_i$. 
Finally, the expressions for $\phiunder_i$ and $\junder_i$ follow from this, (i) above, \eqref{Exi}, \ref{Rauxode}, and \ref{dLbalanced}.
\end{proof}

\subsection*{Decomposition and estimates of LD solutions $\Phi = \Phi\llceil\bsigmaunder : \kcir, \mbold\rrfloor$} 
\nopagebreak

We now decompose and estimate a $\Phi =  \Phi\llceil\bsigmaunder : \kcir, \mbold\rrfloor$ as in \ref{Lphiavg}.  In order to get good estimates, we assume the following. 
\begin{assumption}
\label{Amk}
We assume that $m$ may be taken as large as needed in terms of $k$.  
 \end{assumption}

\begin{notation}
Consider a function space $X$ consisting of functions defined on a domain $\Omega \subset \cyl$.  
If  $\Omega$ is invariant under the action of $\gcyl_{|\munder|}$ for some $\munder \in \Z\setminus \{0\}$ (recall \ref{dHcyl}),
we use a subscript ``$\sym[\munder]$'' to denote the subspace $X_{\sym[\munder]}\subset X$
consisting of those functions  $f\in X$ which are invariant under the action of $\gcyl_{|\munder|}$.
\qed 
\end{notation}

\begin{definition}[Decompositions of $\Phi = \Phi \llceil  \bsigmaunder: \kcir, \mbold \rrfloor $]   
\label{ddecomp}
\label{ELW}
We first define a decomposition $\Phi = \sum_{i=1}^k \Phi_i$ by applying \ref{LsymLD} and requesting that $\Phi_i$ is an LD solution with singular set $L_i=L[\sss_i;m_i]$ 
(recall \ref{dL}). 
We define then the following $\forall i\in \{1, \dots, k\}$.  

$\Pp_i \in C^{\infty}_{\sym[m_i]} \left(\cyl_I \setminus L_i\right)$ by requesting that is it is supported on 
$D^\chi_{L_i}(3\delta_i)$ and 
$\Pp_i := \tau'_i \Psibold[2\delta_i, 3\delta_i; \dbold^\chi_{p}] \big( G^\chi_{p} -   \phiunder[ \log \delta_i,0; \sss_i],  0 \big)$  
on $D^\chi_{p}(3\delta_i)$ ($\forall p \in L_i$).  

$\Phat_i \in C^{\infty}_{|\sss|}(\cyl_I )$ 
by 
$\Phat_i := 
\Phi_{i,\ave} - \Psibold\Big[ \textstyle{\frac{2}{|m_i|}}, \textstyle{\frac{3}{|m_i|}}; \dbold^{\chi}_{\Lpar[\sss_i]}\Big] \big(\junder_i, 0\big)$  
on  $\Omega[\sss_i;m_i]$  
and 
$\Phat_i := \Phi_{i,\ave}$ on $\cyl_I \setminus \Omega[\sss_i;m_i]$   
(recall \ref{davg} and \ref{Lphiavg}\ref{item:ji}).  

$\Phi'_i, E'_i \in C^\infty_{\sym[m_i]}(\cyl_I)$ by requesting 
that $\Phi_i=\Pp_i+\Phat_i+\Phip_i$ on  $\cyl_I \setminus L_i$ 
and 
$E'_i:= \Lcal_{\chitilde[m_i]} \Phi'_i$ on  $\cyl_I$ (clearly supported on $ \Omega[\sss_i ;m_i]$).  

We then define $\Ghat \in C^\infty(\cyl_{I}\setminus L)$,  
$\Phat \in C^{\infty}_{|\sss|}(\cyl_I )$, 
and $\Phi',E' \in C^\infty_{\sym[m]}(\cyl_I)$ 
by 
$\Pp = \sum_{i=1}^k \Pp_i$,  
$\Phat = \sum_{i=1}^k \Phat_i$,  
$\Phi' = \sum_{i=1}^k \Phi'_i$  
and $E' := \sum_{i=1}^k E'_i$.
Clearly then 
$\Phi=\Pp+\Phat+\Phip$ on  $\cyl_I \setminus L$,  
$\Phat := \phi$ on $\cyl_I \setminus \Omega[\sbold; \mbold]$,  
and 
$\Phat = 
\Psibold\left[ \textstyle{\frac{2}{|m_i|}}, \textstyle{\frac{3}{|m_i|}}; \dbold^{\chi}_{\Lpar[\sss_i]}\right] \big( \phiunder_i, \phi\big) 
= \phi - \Psibold\left[ \textstyle{\frac{2}{|m_i|}}, \textstyle{\frac{3}{|m_i|}}; \dbold^{\chi}_{\Lpar[\sss_i]}\right] \big(\junder_i, 0\big)$ 
on  $\Omega[\sss_i;m_i]$  
($\forall i\in \{1, \dots, k\}$), 
with $\phi, \phiunder_i, \junder_i$ as in \ref{Lphiavg}.  
\end{definition}

We estimate the average and oscillatory parts of $\Phi$ separately.  
The decomposition $\Phi = \Pp+ \Phat + \Phip$ is designed so that $\Phip$ is small in comparison to $\Phat$ (cf. Proposition \ref{LPhip} below).  
We have the following characterization of $\Phip_{\ave}$.

\begin{lemma}
\label{LPhipave}
$\Phip_{\ave} = \sum_{i=1}^k \Phip_{i, \ave}$ where $\Phip_{i, \ave}:= (\Phip_i)_\ave$ is supported on $\Omega[\sss_i;m_i]$ and satisfies
$\Lcal_{\chitilde[m_i]} \Phip_{i, \ave}= E'_{i, \ave}: = (E'_i)_{\ave}$ and 
\begin{align}
\label{Ephipavg}
\Phip_{i, \ave} = \begin{cases}
\Psibold\left[ \frac{2}{|m_i|}, \frac{3}{|m_i|}; \dbold^\chi_{\Lpar[\sss_i]}\right] \big( \junder_i, 0 \big) \quad \text{on } \Omega[\sss_i;m_i]\setminus \Omega'[\sss_i;m_i], 
\\
\junder_i -(\Pp_i)_{\ave}, \quad  \qquad \qquad \qquad \qquad \text{on } \Omega'[\sss_i;m_i].
			\end{cases}
\end{align}
\end{lemma}

\begin{proof}
Taking averages of $\Phi = \Pp + \Phat + \Phip$ and rearranging establishes
$ \Phip_\ave = \phi - \Phat - \Pp_\ave$.  Applying $\Lcal_{\chitilde[m_i]}$ to both sides of this decomposition and using the definition of $E'_i$ in \ref{ddecomp} establishes $\Lcal_{\chitilde[m_i]} \Phip_{i, \ave}= E'_{i, \ave}$.  Finally, 
\eqref{Ephipavg} follows from the decomposition $\Phip_\ave = \phi - \Phat - \Pp_\ave$ by substituting the expression for $\Phat$ from \ref{ddecomp} 
and using that $\Pp_i = 0$ on $\Omega[\sss_i;m_i]\setminus \Omega'[\sss_i;m_i]$.
\end{proof}

In order to estimate $\Phip_{\osc}$ we will need the following lemma. 

\begin{lemma}[cf. {\cite[Lemma 5.23]{kap}}]
\label{Lsol}
Given 
$E\in C^{0,\beta}_{\sym[m_i]}(\cyl_I)$ with $E_\ave\equiv 0$ and $E$ supported on $D^\chi_{\Lpar[\sss_i]}(3\delta_i)$ for some $i\in \{1, \dots, k\}$, 
there is a unique $u \in C^{2,\beta}_{\sym[m_i]}(\cyl)$ solving $ \Lcal_{\chitilde[m_i]} u = E $ and satisfying the following.
\begin{enumerate}[label=\emph{(\roman*)}]
\item  $u_\ave=0$.
\item  $ \left \|u: C^{2, \beta}_{\sym[m_i]}\big(\cyl, \chitilde[m_i], e^{-\frac{|m_i|}{2}\left| |\sss| - \sss_i \right|} \big)\right\| $   
$ \lem C \big \|  E: C^{0, \beta}_{\sym[m_i]}\big(D^\chi_{\Lpar[\sss_i]}(3\delta_i), \chitilde[m_i]\,\big)\big\|.$ 
\item $\big\|\Acalsssi      u:C^{2,\beta}_{\sym[m_i]}(D^\chi_{\Lpar[\sss_i]}(3\delta_i)\, ,\chitilde[m_i] \, )\big\| $ 
$\lem {C}\big\| E : C_{\sym[m_i]}^{0,\beta }(D^\chi_{\Lpar[\sss_i]}(3\delta_i), \chitilde[m_i] \, )\big \| \big/ {m_i^2} $ 
\hfill $\phantom1$ 
\\ $\phantom1$ \hfill 
$+ \: C\big\| \Acalsssi E: C_{\sym[m_i]}^{0,\beta }(D^\chi_{\Lpar[\sss_i]}(3\delta)\, ,\chitilde[m_i] \, )\big\|.$
\end{enumerate}
\end{lemma}

\begin{proof}
The existence and uniqueness of $u$ is clear, and (i) follows since $\Lcal_{\chitilde[m_i]} u= E_\ave = 0$.  For (ii), let $u_1$ be the solution on $\cyl$ of $\Delta_{\chitilde[m_i]} u_1 = E$,
subject to the condition that $u_1 \rightarrow 0 $ as $\sss \rightarrow \pm \infty$.  By standard theory and separation of variables, we have
\begin{equation*}
\big\| u_1 : C^{2, \beta}_{\sym{[m_i]}}(\cyl, \chitilde[m_i], e^{-\frac{|m_i|}{2} | |\sss| - \sss_i |}\big)\big \| 
\lem  C \big\| E: C^{0, \beta}_{\sym[m_i]}\big(D^{\chi}_{\Lpar[\sss_i]}(3\delta_i), \chitilde[m_i]\big)\big\|.
\end{equation*}
Note that $(u_1)_\ave$ clearly vanishes. 
Define now inductively a sequence $(u_j)_{j\in \N}$, $u_j\in C^{2, \beta}_{\sym[m_i]}(\cyl)$ by requesting that for each $j\geq 2$,  $\Delta_{\chitilde[m_i]} u_j = - m_i^{-2} V u_{j-1}$ and $u_j\rightarrow 0$ as $\sss \rightarrow \pm \infty$.  Estimating $u_j$ in the same fashion used to estimate $u_1$, we have for $i\geq 2$
\begin{equation*}
\big\| u_j : C^{2, \beta}_{\sym[m_i]}(\cyl, \chitilde[m_i], e^{-\frac{|m_i|}{2} | |\sss| - \sss_i |}\big)\big \|
\lem 
C m_i^{-2} \big\| u_{j-1}: C^{0, \beta}_{\sym[m_i]}\big(D^{\chi}_{\Lpar[\sss_i]}(3\delta_i), \chitilde[m_i]\big)\big\|.
\end{equation*}
Note now that $u = \sum_{j=1}^\infty u_j$, and the estimates above imply (i).

Applying $\Acalsssi$ to both sides of the equation $\Lcal_{\chitilde[m_i]}u = E$ and using Lemma \ref{Rprod}(ii), we obtain
\begin{align*}
\Lcal_{\chitilde[m_i]}  \, \Acalsssi \, u =  \Acalsssi E - m_i^{-2} \Acalsssi V \, \Rcalsssi u .
\end{align*}
Although $\Acalsssi \, E-  m_i^{-2} \Acalsssi \, V \, \Rcalsssi u$ is not supported on $D^\chi_{\Lpar[\sss_i]}(3\delta_i)$, it has average zero, so a straightforward modification of the argument proving (ii) by replacing the assumption that the inhomogeneous term is compactly supported  with the assumption (from (ii) above) that the right hand side has exponential decay away from $D^\chi_{\Lpar[\sss_i]}(3\delta)$, we conclude that
\begin{multline*}
\big\| \Acalsssi\,  u: C^{2, \beta}_{\sym[m_i]} \big(D^\chi_{\Lpar[\sss_i]}(3\delta_i) , \chitilde[m_i]\, \big)\big\|
\\ 
\lem C\big\| \Acalsssi\,  E: C^{0, \beta}_{\sym[m_i]} \big(D^\chi_{\Lpar[\sss_i]}(3\delta_i) , \chitilde[m_i]\, \big)\big \| + 
\\
+ C\big\|m_i^{-2} \Acalsssi \, V \Rcalsssi u: C^{0, \beta}_{\sym[m_i]} \big(D^\chi_{\Lpar[\sss_i]}(3\delta_i) , \chitilde[m_i]\, \big)\big \|.
\end{multline*}
(iii) follows after using \eqref{E:norm:mult}, Lemma \ref{Rprod}(iv), and part (i) above to estimate the last term.
\end{proof}

We are now ready to begin estimating $\Phip$.  
We will estimate $\Phip_{\ave}$ and $\Phip_{\osc}$ separately, by estimating each $\Phip_{i, \ave}$ and $\Phip_{i, \osc}$ 
in the decompositions 
$\Phip_{\ave} = \sum_{i=1}^k \Phip_{i, \ave}$, $\Phip_{\ave} = \sum_{i=1}^k \Phip_{i, \osc}$ (recall  \ref{ELW} and \ref{LPhipave}).  
To estimate $\Phip_{i, \ave}$ and $\Phip_{i, \osc}$ we will use that they satisfy the equations $\Lcal_{\chitilde[m_i]} \Phip_{i, \ave} = E'_{i, \ave}$ 
and $\Lcal_{\chitilde[m_i]} \Phip_{i, \osc} = E'_{i, \osc}$. 
We first establish relevant estimates for $E'_i, E'_{i, \ave}$ and $E'_{i, \osc}$ and then estimate $\Phip_{i, \ave}$ and $\Phip_{i, \osc}$.

\begin{lemma}
\label{LEest}
For each $i=1, \dots, k$, $E'_i$ vanishes on $D^\chi_{L_i}(2\delta_i)$ and $E'_{i, \osc}$ is supported on $D^\chi_{\Lpar[\sss_i]}(3\delta_i)$.  The following hold. 
\begin{enumerate}[label=\emph{(\roman*)}]
\item  $\big \|\Pp_i: C_{\sym[m_i]}^{j}( \, D^{\chi}_{L_i}(3 \delta_i) \setminus D^\chi_{L_i}(\delta_i)\,,\chitilde[m_i] \, )\big\|\le C(j)\,$ 
\item  $\big\| \Acalsssi \Pp_i : C_{\sym[m_i]}^j(   D^\chi_{L_i}(3\delta_i)\setminus D^\chi_{L_i}(2\delta_i) \, , \chitilde[m_i] \, ) \big\| \le C(j)/|m_i|$.
\item  $\left\|E'_i:C_{\sym[m_i]}^{j}( \Omega[\sss_i;m_i] ,\chitilde[m_i] \, )\right\|\le C(j)\,$
\item $\big\|\Acalsssi \, E'_i :C_{\sym[m_i]}^{j}(  D^\chi_{\Lpar[\sss_i]}(3\delta_i) , \chitilde[m_i] \, )\big\|\le C(j)/|m_i|\,$, for $i\in \{1, \dots, k\}$.
\end{enumerate}
In either (iii) or (iv), the same estimate holds if $E'_i$ is replaced with either $E'_{i,\ave}$ or $E'_{i,\osc}$. 
\end{lemma}

\begin{proof}  
(i) follows from Lemma \ref{Lgreen}, Definition \ref{ddecomp}, 
and uniform bounds on the $\tau'_i$'s which follow from \ref{Lphiavg}(i) and \ref{Lrldest}(ii).  
For (ii), it suffices to prove for any $i=1, \dots, k$ and any $p \in L[\sss_i; m_i]$ the estimate
\begin{align*}
\label{Egasymmetry}
\big\| \Acalsssi \Pp_i : C^j( D^\chi_{p}(3\delta_i)\setminus D^\chi_{p}(2\delta_i), \chitilde[m_i] \, )\big\| \le C(j)/|m_i|.
\end{align*}
By Definition \ref{ddecomp}, 
we have $\Acalsssi \Pp_i = (I)-(II)$ 
on $D^\chi_{p}(3\delta_i)\setminus D^\chi_{p}(2\delta_i)$, 
where 
\begin{equation}
\label{EAG}
\begin{aligned}
(I) := & \Psibold[2\delta_i, 3\delta_i; \dbold^\chi_{p}](\tau'_i \Acalsssi \, G^\chi_{p}, 0) ,  
\\
(II) := & \Psibold[2\delta_i, 3\delta_i; \dbold^\chi_{p}]( \tau'_i \Acalsssi\,  \phiunder[ \log \delta_i   , 0; \sss_i], 0).  
 \end{aligned}
 \end{equation}
From Lemma \ref{Lgreen} and the uniform bounds on the cutoff $\Psibold$, we conclude that 
$\| (I):  C^j( D^\chi_{p}(3\delta_i)\setminus D^\chi_{p}(2\delta_i) \, , \chitilde[m_i] \, ) \| \le C(j)/|m_i|$,  
and by Lemma \ref{Lode}(iii), that $\| (II):  C^j( D^\chi_{p}(3\delta_i)\setminus D^\chi_{p}(2\delta_i) \, , \chitilde[m_i] \, ) \| \le C(j) |m_i|^{-3}\log |m_i|$. 
These estimates complete the proof of (ii).

The statements on the support of $E'_i$ and $E'_{i,\osc}$ follow from Definition \ref{ddecomp}, 
from which we also see that 
$E'_i= \Lcal_{\chitilde[m_i]}  \Psibold\big[ 2, 3; \dbold^{\chitilde[m_i]}_{\Lpar[\sss_i]}\big] \big( \junder_i, 0 \big)$ 
on $\Omega[\sss_i; m_i]\setminus \Omega'[\sss_i;m_i]$. 
Thus, when restricted to $\Omega[\sss_i; m_i]\setminus \Omega'[\sss_i;m_i]$, the bound in (iii)
follows from \ref{dauxode}, \ref{Lphiavg}(iii), and the uniform bounds of the cutoff.  It remains to estimate $E'_i$ on $ \Omega'[\sss_i;m_i]$.  
By \ref{ddecomp}, $E'_i$ vanishes on $\Omega'[\sss_i;m_i]\setminus D^{\chi}_{\Lpar[\sss_i]}(3\delta_i)$.  Note that $\Lcal_{\chitilde[m_i]} \Phat=0$ on $D^{\chi}_{\Lpar[\sss_i]}(3\delta_i)$. 
Since $\Lcal_{\chi} \Pp_i = 0$ on  $D^{\chi}_{L}(2\delta_i)$, when restricted to $D^{\chi}_{\Lpar[\sss_i]}(3\delta_i)$, the required bound in (iii) follows from (i).  With the preceding, this completes the proof of (iii). 

For (iv), we have using \ref{ddecomp} that  $\Acalsssi\,  E'_i = -\Acalsssi \, \Lcal_{\chitilde[m_i]}\,  \Pp_i$ on $D^\chi_{\Lpar[\sss_i]}(3\delta_i)$.  
Since $E'_i$ vanishes on 
$D^\chi_{\Lpar[\sss_i]}(2\delta_i)$, 
it is only necessary to prove the estimate on the set difference.  
Using \ref{Rprod}(ii) to switch the order of $\Lcal_{\chitilde[m_i]}$ and $\Acalsssi$, 
we find that
\begin{align*}
\Acalsssi\,  E'_i = - \Lcal_{\chitilde[m_i]}\,  \Acalsssi \, \Pp_i - m_i^{-2} \Acalsssi \, V \, \Rcalsssi \, E'_i \quad \text{on} \quad 
D^\chi_{\Lpar[\sss_i]}(3 \delta_i) \setminus D^\chi_{\Lpar[\sss_i]}(2\delta_i).
\end{align*}
Using (ii) to estimate the first term on the right and Lemma \ref{Rprod}(iv) and (iii) to estimate the second 
term,  we obtain (ii).  
Finally, we can replace $E'_i$ by $E'_{i,\ave}$ or $E'_{i,\osc}$ in either (iii) or (iv) by taking averages and subtracting. 
\end{proof}

\begin{lemma}[Estimates for $\Phip_{i, \ave}$]     
\label{LavePhi}
$\forall i \in \{1, \dots, k\}$ the following hold. 
\begin{enumerate}[label=\emph{(\roman*)}]
\item $\| \Phip_{i,\ave}: C^{j}_{|\sss|}\left( \Omega[\sss_i; m_i], \chitilde[m_i]\,\right) \| \le C(j)$. 
\item  For $i\in \{ 1, \dots, k\}$, $\big\|\Acalsssi \Phip_{i,\ave} : C^{j}_{|\sss|}\big( D^\chi_{\Lpar[\sss_i]}(3\delta_i), \chitilde[m_i]\, \big)\big \| \le C(j)/|m_i|$.
\end{enumerate}
\end{lemma}

\begin{proof}
Fix $i\in \{1, \dots, k\}$.  We first establish the estimate on $\Omega'[\sss_i;m_i]$.  
By \eqref{Ephipavg},
\begin{align*}
 \Phip_{i, \ave} = \junder \left [\textstyle{\frac{|m_i|}{2}\tau'_i}; \sss_i\right]  -(\Pp_i)_{\ave}
\quad \text{on} \quad
\Omega'[\sss_i;m_i].
\end{align*} 
Note that the left hand side is smooth and the discontinuities on the right hand side cancel.  
Using that $\Lcal_{\chitilde[m_i]} \Phip_{i,\ave} = E'_{i,\ave}$ from \eqref{ELW}, on  $\Omega[\sss_i;m]$ we have (where $\shat = \shat\, [ \sss_i, m_i]$ is as in \ref{dchi})
\begin{align}
\label{Ephipode} 
\partial^2_{\,\shat} \Phip_{i,\ave} + \frac{1}{|m_i|^2}V\left( \frac{\shat}{|m_i|}+ \sss_i\right)  \Phip_{i,\ave} =  E'_{i,\ave}.
\end{align}
On a neighborhood of $\partial \Omega'[\sss_i;m_i]$, 
we have that $\Pp_\ave = 0$ from Definition \ref{ddecomp}.  
This combined with estimates on $\junder$ from Lemma \ref{Lode} implies that 
$\left|\Phip_{i,\ave}\right|<C$ and $\left| \partial_{\, \shat} \, \Phi'_{i,\ave}\right|<C$ on $\partial \Omega'[\sss_i;m_i]$.  
Using this as initial data for the ODE and bounds of the inhomogeneous term from Lemma \ref{LEest} yields the $C^2$ bounds in (i).  
Higher derivative estimates follow inductively from differentiating \eqref{Ephipode} and again using Lemma \ref{LEest}.  
This establishes (i) on $\Omega'[\sss_i;m_i]$.  
The proof of the estimate (i)  on $\Omega[\sss_i;m_i]\setminus \Omega'[\sss_i;m_i]$ follows in a similar way using \eqref{Ephipavg} 
but is even easier since there $(\Pp_i)_{\ave} = 0$, so we omit the details.

By \eqref{Ephipode} and Lemma \ref{Rprod}(ii), $\Acalsssi \Phip_{i,\ave}$ satisfies
(recall \ref{dchi} for the relation of $\sss$ and $\shat$) 
\begin{align}
\label{Eu}
\Lcal_{\chitilde[m_i]} \, \Acalsssi \Phip_{i,\ave} +  
\frac{1}{m_i^2}\, \Acalsssi \, V \, \Rcalsssi \Phip_{i,\ave} = \Acalsssi  \,  E'_{i,\ave}. 
\end{align}
The $C^2$ bounds in (ii) follow in a similar way by using Lemma \ref{Lode}(iii)-(iv) 
to estimate the initial data on $\partial D^\chi_{\Lpar[\sss_i]}(3\delta_i)$, 
estimates on $\Acalsssi  E'_{i,\ave}$ from Lemma \ref{LEest}(iv), 
and estimates on $\Acalsssi\, V$ and $\Phip_{i,\ave}$ from Lemma \ref{Rprod}(iv) and (i) above.  
Higher derivative bounds follow inductively from  differentiating \eqref{Eu} and using Lemma \ref{Rprod}(iv) and Lemma \ref{LEest}(iv).
\end{proof}

\begin{lemma}[Estimates for $\Phip_{i, \osc}$]    
\label{LPhipest}
$\forall i \in \{1, \dots, k\}$ the following hold. 
\begin{enumerate}[label=\emph{(\roman*)}] 
\item \begin{enumerate}[label=\emph{(\alph*)}]
	\item $\big\| \Phip_{i,\osc} : C^{j}_{\sym[m_i]}\big( \cyl, \chitilde, e^{-m\left| |\sss| - \sss_i\right|}\big)\big\|\le C(j)$.
	\item 
	 $\big\| \Acalsssi \Phip_{i,\osc} : C^{j}_{\sym[m_i]}\big(D^\chi_{\Lpar[\sss_i]}(3\delta_i), \chitilde\, \big)\big\| \le C(j)/|m_i|$.
	 \end{enumerate}
\item  $\left\| \Phip_{\osc} : C^{j}_{\sym[m]}\big( \cyl_I, \chitilde\, \big)\right \| \le C(j).$

\item $\big\| \Acalsssi \Phip_\osc : C^{j}_{\sym[m]}\big(D^\chi_{\Lpar[\sss_i]}(3\delta_i), \chitilde[m_i]\, \big)\big\| \le C(j)/|m_i|$. 
\end{enumerate} 
\end{lemma}

\begin{proof}
We first complete the proof in the case where $\cyl_I = \cyl$ and leave the modifications for the case where $l<\infty$ (when $\Sigma$ is a torus) to the end.
(i) follows directly from applying Lemma \ref{Lsol} to $E'_{i,\osc}$, using Lemma \ref{LEest} and Schauder regularity for the higher derivative estimates.
For small $k$, (ii) follows from (i).  On the other hand, Lemma \ref{Lrldest}(ii) implies that for all $i, j \in \{1, \dots, k\}$, $|\sss_j - \sss_i| > \frac{|j-i|}{Ck}$.  Using this with part (i) above, we estimate
\begin{align*}
\left\| \Phip_{\osc} : C^{j}_\sym\left( \cyl, \chitilde\,\right)\right \| \le  C(j) \sup_{\sss\in \R} \sum_{i=1}^k e^{-m\left| |\sss| - \sss_i\right|} 
\leq C(j) \sum_{l=0}^{k-1} e^{- \frac{m}{Ck} l }   
\leq C(j),
\end{align*}
where we have used Assumption \ref{Amk}.  This completes the proof of (ii).
Now fix some $i\in \{1, \dots, k\}.$
As before, for $i, j \in \{1, \dots, k\}$, $|\sss_j - \sss_i| > \frac{|j-i|}{Ck}$.  Using the definitions and (i) above, 
\begin{multline*}
\big\| \Acalsssi \Phip_\osc : C^{r}_\sym\big(D^\chi_{\Lpar[\sss_i]}(3\delta), \chitilde\,\big)\big\| 
\lem 
\begin{aligned}[t] 
&\big\| { \Acalsssi \Phip_{i,\osc}: C^{r}_\sym \big( D^\chi_{\Lpar[\sss_i]}(3\delta), \chitilde\, \big) } \big\| 
\\
&\!\!\!\!\!\!\!\!\!\!\!\!\!+\: \sum_{ j\neq i}\big \| \Acalsssi \Phip_{j,\osc} : C^{r}_\sym\big(D^\chi_{\Lpar[\sss_i]}(3\delta), \chitilde\,\big)\big\| 
\end{aligned} 
\\
\lem \frac{C(r)}{m} + C(r)\sum_{ j\neq i}\big \|  \Phip_{j,\osc} : C^{r}_\sym\big(D^\chi_{\Lpar[\sss_i]}(3\delta), \chitilde\,\big)\big\| 
\\
\lem \frac{C(r)}{m} + C(r) \sum_{j\neq i} e^{-m\left| \sss_j - \sss_i\right|}
\\
\lem \frac{C(r)}{m}+ C(r) \sum_{l=1}^k e^{-\frac{m}{Ck}l } \qquad \leq \qquad \frac{C(r)}{m},
\end{multline*}
where we have used Assumption \ref{Amk}.

We now address the case when $\Sigma$ is a torus, that is when $\cyl_I$ is a proper subset of $\cyl$.  In this case we must lift all of our functions to functions on $\cyl$ which are invariant under the translation $\SSS_{2l} : \cyl \rightarrow \cyl$ defined by $\sss \mapsto \sss+ 2l$.  More precisely, we define for each $i\in \{1,\dots, k\}$ and $j \in \Z$
$\widetilde{E}'_{i, \osc}, \widetilde{E}'_{i,j,\osc}  \in C^\infty_{\sym[m_i]}(\cyl)$ by $\widetilde{E}'_{i, j, \osc} = E'_{i, \osc} \circ \SSS^j_{2l}$, $ \widetilde{E}_{i, \osc} = \sum_{j\in \Z} \widetilde{E}'_{i, j, \osc}$, $\widetilde{\Phi}'_{i,j, \osc} \in C^\infty_{\sym[m_i]}(\cyl)$ by using \ref{Lsol} and requesting that
$\Lcal_{\chitilde[m_i]} \widetilde{\Phi}'_{i,j, \osc}= \widetilde{E}'_{i, j,\osc}$, and $\widetilde{\Phi}'_{i, \osc} = \sum_{j \in \Z} \widetilde{\Phi}'_{i,j,\osc}$.

Finally, we use \ref{Lsol} to establish exponential decay for each $\widetilde{\Phi}'_{\osc, i, j}$ away from $\Lpar[\sss_i+2lj]$ and complete the proof; 
we omit the details because they are similar to those from the proof when $\cyl_I = \cyl$. 
\end{proof}

\begin{prop}[cf. {\cite[Proposition 4.18]{kapmcg}}]
\label{LPhip}
The following hold. 
\begin{enumerate}[label=\emph{(\roman*)}]
\item $\| \Phip: C^{j}_{\sym[m]}\left( \cyl_I, \chitilde\,\right) \| \le C(j)$. 
\item  
$\big\|\Acalsssi \Phip : C^{j}_{\sym[m]}\big( D^\chi_{\Lpar[\sss_i]}(3\delta), \chitilde\, \big)\big \| \le C(j)/\max_i|m_i|$, 
($i\in \{ 1, \dots, k\}$). 
\end{enumerate}
\end{prop}

\begin{proof}
Because of the estimates on $\Phip_\osc$ established in Proposition \ref{LPhipest}(ii), 
it is enough to prove the estimate (i) for  $\Phi'_\ave$, 
and this follows from \ref{LPhipave} and \ref{LavePhi}.
\end{proof}

\section{Families of LD solutions on $O(2)\times \Z_2$ symmetric backgrounds} 
\label{S:LDfamilies}

\subsection*{The family of LD solutions}
\nopagebreak

For the applications at hand, it will be convenient to use definitions of $\skernel[L]$ and $\skernelv[L]$ (recall \ref{aK}) which exploit the symmetries of the problem.  When $\mmax>m$, we will also need to consider LD solutions with more general singular sets than the ones studied in Section \ref{S:LDs}. 

\begin{definition}[The space {$\val_{\sym}[\Ltilde]$}]
\label{dVal}
Given $\Ltilde = \cup_{i=1}^k \Ltilde_i$ which is a small $\gcyl_m$-symmetric perturbation of an $L = L[\sbold;\mbold]$ as in \ref{dLmbold},  
define the subspace $\val_{\sym}[\Ltilde]$ of $\val[\Ltilde]$ (recall \ref{DVcal}) consisting of elements equivariant under the obvious action of $\gcyl_m$ on $\val[\Ltilde]$, 
an inner product $\langle \cdot, \cdot\rangle_{\val[\Ltilde]}$ by $\langle \cdot, \cdot\rangle_{\val[\Ltilde]}: = \sum_{p\in \Ltilde}\langle \cdot, \cdot\rangle_{\val[p]}$, 
where $\langle a_0 + a_1 d\sss + a_2 d\theta, a'_0 + a'_1 d\sss + a'_2 d\theta \rangle_{\val[p]} := a_0a'_0 + a_1a'_1 + a_2a'_2 $ and a decomposition
\begin{equation*}
\begin{gathered}
 \val_{\sym}[\Ltilde] = \bigoplus_{i=1}^k \val_{\sym}[\Ltilde_i] = \bigoplus_{i=1}^k \valtop[\Ltilde_i]\oplus \valperp[\Ltilde_i], \quad \text{where} \\
 \valtop[\Ltilde_i] : = \{  (a+ b  \, d\sss)_{p\in \Ltilde_i} \in \val_{\sym}[\Ltilde_i] :  a, b \in \R \}
 \end{gathered}
\end{equation*}
and $\valperp[\Ltilde_i]$ is the orthogonal complement of $\valtop[\Ltilde_i]$ in $\val_{\sym}[\Ltilde_i]$. 
\end{definition}

We will need to convert estimates on the cylinder---particularly those established for $\Phi$ in Section 
\ref{S:LDs}---into estimates on $\Sigma$ with the $g$ metric.  Before doing this we need the following lemma, which compares the geometry induced by the metrics $\chi$ and $g$. 

\begin{lemma}
\label{Lconf}
Suppose $u\in C^j(\Omega)$ for a domain $\Omega \subset \cyl$.  Then 
\begin{align*}
\| u : C^{j}(\Omega, e^{2\conf} \chi ) \| \le C(j) \| u : C^{j}(\Omega, \chi)\| ( 1+ \sup_{x\in \Omega} e^{-\omega})^j (1+ \| \conf : C^k(\Omega, \chi)\|)^j.
\end{align*}
\end{lemma}

\begin{proof}
In this proof we denote objects computed with respect to $g = e^{2\conf} \chi$ by a hat, 
so for example the Levi-Civita connection of $g$ is denoted by $\widehat{\nabla}$.  
Taking covariant derivatives of $\nablahat^j u$ with respect to $g$, we find
\begin{align*}
(\nablahat^j u)_{i_1\cdots i_j \hat{;} m} = (\nablahat^j u)_{i_1\cdots i_j , m} - \sum_{s=1}^j (\nablahat^j u )_{i_1 \cdots p\cdots i_j} \widehat{\Gamma}^p_{m i_s}.
\end{align*}
Recall the formula for the Christoffel symbols computed with respect to  the conformal metric $g = e^{2\conf} \chi$:
\begin{align*}
\widehat{\Gamma}_{m i_s}^p = \Gamma_{m i_s}^p + \delta^p_m \conf_{, i_s} + \delta^p_{i_s}\conf_{, m} - g_{m i_s} g^{p l} \conf_{, l}.
\end{align*}
Combining the preceding, we find
\begin{equation*}
(\nablahat^j u)_{i_1\cdots i_j \hat{;} m} = (\nablahat^j u)_{i_1\cdots i_j ; m} - \sum_{i=1}^k (\nablahat^j u)_{i_1\cdots m \cdots i_j} \conf_{, i_s} 
- (\nablahat^j u)_{i_1 \cdots i_j} \conf_{, m} + \sum_{s=1}^j (\nablahat^j u)_{i_1\cdots p \cdots i_j} g_{mi_s} (\nabla \conf)^p .
\end{equation*}
The conclusion now follows by a straightforward inductive argument. 
\end{proof}

\begin{definition}[The constants $\delta_p$, cf. \ref{con:L}] 
\label{ddeltai}
For each $p\in L$ we define a constant $\delta_p>0$ by requesting that the set of $\delta_p$'s is invariant under the action of $\groupcyl$ on $L$ 
and that for $i=1, \dots, k$ we have $\delta_{p} = e^{-2\conf(\sss(p))} \delta = \frac{1}{9|m_i|}e^{-2\conf(\sss(p))} $.
\end{definition}

\begin{definition}[The space of parameters]
\label{dParam}
We define $\Pcal := \Pcal^\top \oplus \Pcal^\perp$, where 
\begin{align*}
\Pcal^\top: = \R^{2k}, \quad 
\Pcal^\perp := \bigoplus_{i=1}^k \Pcal_i^\perp : =\bigoplus_{i=1}^k \R^{\dim \valperp[L_i]}.
\end{align*}
The continuous parameters of the LD solutions are 
\begin{equation*} 
\begin{gathered}
\zetabold = (\zetabold^\top, \zetabold^\perp) \in \Btilde_{\Pcal} := \cunder \Btilde^1_{\Pcal} := \cunder ( \Btilde^1_{\Pcal^\top} \times \Btilde^1_{\Pcal^\perp}), 
\\
\text{where} \quad 
\zetabold^\top : = (\zeta_1, \bsigmaunder) = (\zeta_1, \bsigma, \xibold) \in \cunder \Btilde^1_{\Pcal^\top}, \quad 
\zeta_1\in\R, \quad \bsigma\in\R^{k-1}, \quad \xibold \in \R^k, 
\\
\Btilde^1_{\Pcal^\top}: = [ - 1, 1] \times \left[ -\frac{1}{m}, \frac{1}{m}\right]^{2k-1}, \qquad \qquad 
\Btilde^1_{\Pcal^\perp} : = \Times_{i=1}^k \Btilde^1_{\Pcal^\perp_i}, 
\\ 
\Btilde^1_{\Pcal^\perp_i} :=  \big[ - e^{-{m}/{\const}}, e^{-{m}/{\const}}\big]^{\dim \valperp[L_i]}, 
\end{gathered}
\end{equation*} 
and $\const>0$ is a constant which can be taken as large as needed in terms of $k$ but is independent of $m$. 
\end{definition}

We now define a family $\varphi\llbracket \zetabold^\top\rrbracket$ of LD solutions parametrized by $\zetabold^\top\in \Btilde_{\Pcal^\top}$ 
and choose the overall scale so that we have approximate matching.  
The singular sets of the $\varphi\llbracket \zetabold^\top\rrbracket$ are $(\sbold, \mbold)$-symmetric singular sets (recall \ref{dLmbold}) 
and $(\varphi\llbracket \zetabold^\top\rrbracket)_{\ave}$ is close 
to being balanced in the sense of  \ref{dLbalanced}.  

\begin{definition}[Maximally symmetric LD solutions $\varphi\llbracket \zetabold^\top\rrbracket$]
\label{dtau1}
Given $\zetabold^\top\in \Btilde_{\Pcal^\top}$ as in \ref{dParam}, we define using \ref{Lphiavg} the LD solution 
\begin{equation}
\label{Etau1}
\begin{aligned}
 \varphi = \varphi \llbracket \zetabold^\top\rrbracket= \varphi \llbracket \zetabold^\top; \kcir, \mbold \rrbracket :=& \, \tauo \Phi \llceil  \bsigmaunder: \kcir, \mbold \rrfloor, 
\\
\text{where} 
\quad 
\tauo = \tauo\llbracket \zetabold^\top\rrbracket =  \tauo\llbracket\zetabold^\top; \kcir,  \mbold\rrbracket :=& \, 2\delta_1 e^{\zeta_1} e^{-\phi(\sss_1)} 
\\ 
= & \, {2} e^{\zeta_1} e^{- {|m_1|} / {2F^\phi_1} } / {9|m_1|} ,
\end{aligned}
\end{equation}
and we denote the singular set and configuration of $\varphi$ by 
\begin{equation*}
\begin{aligned}
L =L \llbracket \zetabold^\top\rrbracket= L \llbracket \zetabold^\top;\kcir,  \mbold\rrbracket = & L[\sbold[\bsigmaunderslash+\bsigmaunder:\kcir]; \mbold] 
\subset 
\Lpar [\sbold[\bsigmaunderslash+\bsigmaunder:\kcir] ],   
\\
\taubold =\taubold\llbracket \zetabold^\top\rrbracket= \taubold\llbracket \zetabold^\top;\kcir,  \mbold\rrbracket = & \tauo \taubold' \llceil  \bsigmaunder: \kcir, \mbold \rrfloor .  
\end{aligned} 
\end{equation*} 
Note that $\varphi$, $L$ and $\taubold$ are $\groupcyl$-invariant by \ref{Lphiavg}.  
We call the jump circles of 
$\phi\llceil  \bsigmaunder: \kcir, \mbold \rrfloor$ (whose union is $\Lpar [\sbold[\bsigmaunderslash+\bsigmaunder:\kcir] ]$) 
the \emph{singular circles} of $\varphi$. 
 \end{definition}

\begin{lemma}[Matching equations for $\varphi\llbracket \zetabold^\top \rrbracket$]
\label{Lmatching}
Given $\zetabold^\top \in \Btilde_{\Pcal^\top}$ we have for $\varphi = \varphi \llbracket \zetabold^\top\rrbracket$ as in \ref{dtau1} that 
$\Mcal_L \varphi = \lambdabold+ \lambdabold'$ with 
$\lambdabold = ( \tau_p ( \mu_p+|m_i|\mu'_p d\sss))_{p\in L}\in \valtop[L]$ 
and $\lambdabold' = (\tau_p(\muhat_p + |m_i|\muhat'_p d\sss + m\muhat^\circ_p d\theta))_{p\in L} \in \val_{\sym}[L]$ 
where $\forall i \in \{1, \dots, k\}$ and $\forall p \in L_i$ 
we define 
(with $\Phip_\ave$ and $\Phip_{j, \osc} = (\Phip_j)_\osc$ defined in \ref{davg} and \ref{ELW})  
$\mu_p := \mu_i$, $\mu'_p := \mu'_i$, 
$\Phip_{\ne i, \osc}:= \sum_{j\neq i} \Phip_{j, \osc}$,  
\begin{equation*} 
\begin{gathered} 
\begin{aligned}
\mu_i &:= \frac{|m_1|}{ 2F^\phi_{1}} \left( e^{-\sum_{l=1}^{i-1}\sigma_l} - 1\right)+  \frac{\Phip_\ave(p)}{\tau'_i} + \frac{\Phip_{i, \osc}(p)}{\tau'_i}+\zeta_1 + \log \tau'_i - \conf(p) , 
\\
\mu'_i &:=  \frac{1}{2} \xi_i + \frac{1}{|m_i|\tau'_i}\frac{\partial \Phip_\ave}{\partial \sss}(p) 
+ \frac{1}{|m_i|\tau'_i}\frac{\partial \Phip_{i, \osc}}{\partial \sss}(p) - \frac{1}{2|m_i|} \frac{\partial \conf}{\partial \sss}(p),
\end{aligned}
\\ 
\muhat_p := \frac{1}{\tau'_i} \Phip_{\ne i, \osc}(p), \ \, 
\muhat'_p := \frac{1}{|m_i| \tau'_i} \frac{\partial \Phip_{\ne i, \osc}}{\partial \sss}( p), \ \,      
\muhat^\circ_p := \frac{1}{m \tau'_i} \frac{\partial \Phip_{\ne i, \osc}}{\partial \theta}( p). 
\end{gathered} 
\end{equation*} 
\end{lemma}

\begin{proof}
The decomposition 
$\lambdabold+ \lambdabold'$ is chosen so that $\lambdabold$ is the part which has to be $\gcyl_{|m_i|}$-invariant and $\lambdabold'$ is the part which may not be $\gcyl_{|m_i|}$-invariant.   
Expanding now $\frac{1}{\tau_i}\Mcal_{p} \varphi$ (recall \ref{Dmismatch}) using \ref{Lmm}, \ref{Rmismatch} and \ref{ddecomp} and equating coefficients with 
$\lambdabold+ \lambdabold'$ we reduce the proof to confirming 
\begin{align*}
\mu_p +\muhat_p&= \frac{1}{\tau'_i} \phiunder_i(\sss_i) + \log \left( \frac{\tau'_i \tau_1}{2\delta_i}\right)+  \frac{\Phip(p)}{\tau'_i} -\conf(p)\,\\
|m_i|( \mu'_p+\muhat'_p) 
&= \frac{1}{\tau'_i} \frac{\partial \phiunder_i}{\partial \sss}(\sss_i) + \frac{1}{\tau'_i} \frac{\partial \Phi'}{\partial \sss}(p)  - \frac{1}{2}\frac{\partial\conf}{\partial \sss}(p).
\\
m \muhat^\circ_p 
&= \frac{1}{\tau'_i} \frac{\partial \phiunder_i}{\partial \theta}(p) + \frac{1}{\tau'_i} \frac{\partial \Phi'}{\partial \theta}(p)  - \frac{1}{2}\frac{\partial\conf}{\partial \theta}(p).
\end{align*}
Using \ref{Lphiavg}(iii) to substitute the data for $\phiunder_i$, substituting $\tau_1$ from \eqref{Etau1}, and using the rotational invariance of $\phiunder_i$ and $\conf$,  
we further reduce to 
\begin{equation}
\label{Em1}
	\begin{aligned}
	\mu_p+\muhat_p &= \frac{|m_1|}{ 2F^\phi_{1}} \left( e^{-\sum_{l=1}^{i-1}\sigma_l} - 1\right)+  \frac{\Phip(p)}		{\tau'_i} +\zeta_1 + \log \tau'_i - \conf(p) \,\\
	|m_i|(\mu'_p+\muhat'_p) &=  \frac{|m_i|}{2} \xi_i + \frac{1}{\tau'_i}\frac{\partial \Phip}{\partial \sss}(p) - \frac{1}{2} \frac{\partial \conf}{\partial \sss}(p),
	\\
	m\muhat^\circ_{p} &= \frac{1}{\tau'_i} \frac{\partial \Phi'}{\partial \theta}(p).
	\end{aligned}
\end{equation}
This follows by considering average and oscillatory parts and using that $\Phip_{i, \osc}$ is $\gcyl_{|m_i|}$-symmetric.
\end{proof}

\begin{notation}
Given $\abold = (a_i)_{i=1}^k \in \R^k, k\geq 2$, we define $\dd \abold \in \R^{k-1}$ by requesting that $(\dd \abold)_j = a_{j+1} - a_j$, $j=1, \dots, k-1$.
It is useful to think of $\dd \abold$ as a discrete derivative of $\abold$.     
\qed 
\end{notation}

\begin{corollary}[Matching Estimates for $\varphi\llbracket \zetabold^\top\rrbracket$]
\label{LmatchingE}
Let $\zetabold^\top \in \Btilde_{\Pcal^\top}$ and $\varphi = \varphi \llbracket \zetabold^\top \rrbracket$ as in \ref{dtau1}.  There is an absolute constant $C$ (independent of $\cunder$) such that for $m$ large enough (depending on $\cunder)$, the following hold, where $\mubold := (\mu_i)_{i=1}^k, \mubold': = (\mu'_i)_{i=1}^k$ and $\mu_i, \mu'_i$ are as defined as in \ref{Lmatching}:
\begin{enumerate}[label=\emph{(\roman*)}]
\item $\left|\zeta_1-\mu_1 \right| \le C$. 
\item $ \left|  \bsigma + \frac{2F^\phi_1}{|m_1|} \dd \mubold \right|_{\ell^\infty}\le C/m$.
\item $\left|\xibold-  2\mubold' \right|_{\ell^\infty} \le C/m$.
\item For any $p\in L_i$, and $\muhat_p, \muhat'_p, \muhat^\circ_p$ as in \ref{Lmatching}, 
	$|\muhat_p|+|\muhat'_p|+|\muhat^\circ_p|\leq Ce^{- \frac{m}{Ck}}$.
\end{enumerate}
\end{corollary}
\begin{proof}
Taking $i=1$ in the first equation of \ref{Lmatching} we obtain for any $p\in L_1$
\begin{align}
\mu_1 - \zeta_1 = \Phip_{\ave}(p) + \Phip_{1, \osc}(p)   - \conf( \sss_1).
\end{align}

Given $i\geq 2$, fix $p_i \in L_i$ and $p_{i-1}\in L_{i-1}$, compute 
$\mu_i - \mu_{i-1}$ using \ref{Lmatching}, multiply through by $\frac{2F^\phi_1}{|m_1|}$ and rearrange to see 
\begin{multline}
\frac{2F^\phi_1}{|m_1|}(\mu_i - \mu_{i-1}) + \sigma_i 
\\ 
=\: \frac{2F^\phi_1}{|m_1|}\left( \frac{\Phip_{\ave}(p_i)}{\tau'_i} - \frac{\Phip_{\ave}(p_{i-1})}{\tau'_{i-1}} + \log \frac{\tau'_i}{\tau'_{i-1}} + \conf(\sss_i) - \conf(\sss_{i-1})\right)\\
  + \: \frac{2F^\phi_1}{|m_1|}\left( \frac{\Phip_{i, \osc}(p_i)}{\tau'_i} - \frac{\Phip_{i-1, \osc}(p_{i-1})}{\tau'_{i-1}}\right)+
  O\left( \frac{\cunder^2}{m^3}\right).
\end{multline}

Next, from \ref{Lmatching} we have for $i=1, \dots, k$ and any $p_i \in L_i$
\begin{align}
2 \mu'_i - \xi_i = \frac{1}{|m_i|}\left(\frac{2}{\tau'_i} \frac{\partial \Phip_{\ave}}{\partial \sss}(p_i) + \frac{2}{\tau'_i}\frac{\partial \Phip_{i, \osc}}{\partial \sss}(p_i) -\frac{\partial \conf}{\partial \sss} (\sss_i)\right).
\end{align}

 Using \ref{LPhip} to estimate the terms involving $\Phi'$, using \ref{Ltauratio} to bound the $\tau'$ terms, and using the boundedness of the $\sss_i$ from Lemma \ref{Lrldest} and uniform bounds on $\conf$ on compact sets,  we deduce the 
 inequalities  $\left|\zeta_1-\mu_1 \right| \le C$,  $ \left|  \bsigma + \frac{2F^\phi_1}{|m_1|} \dd \mubold \right|_{\ell^\infty}\le C/m$, and $\left|\xibold-  2\mubold' \right|_{\ell^\infty} \le C/m$, which together complete the proof of (i)-(iii).  For (iv), we use the definition of $\muhat_p, \muhat'_p, \muhat^\circ_p$ in \ref{Lmatching} and estimate the oscillatory terms using \ref{Lsol} and arguing as in the proof of \ref{LPhipest}. 
 \end{proof}

\begin{lemma}
\label{Ltauratio}
For $\Phi$ as in Definition \ref{Lphiavg} and $1\leq j < i \leq k$, we have
\begin{align*}
\frac{\tau'_i}{\tau'_j} = \frac{\phi(\sss_i)}{\phi(\sss_j)}  \big( e^{\sum_{l=j}^{i-1} \sigma_l} \big)  \Sim_{1+C\frac{\log k}{k}} 1. 
\end{align*}
\end{lemma}
\begin{proof}
The first equality follows from \ref{Lphiavg}(i).  
We have then
\begin{equation}
\label{Etauratio1}
\begin{aligned}
\log \frac{\tau'_i}{\tau'_{j}} &= \log\frac{\phi(\sss_i)}{\phi(\sss_{j})} + \sum_{l=j}^{i-1}\sigma_{l}=  O\left(\frac{\log k }{k}\right) + O\left(\frac{k \cunder}{m}\right),
\end{aligned}
\end{equation}
where the estimates follow from Lemma \ref{Lrldest}, Definition \ref{dtau1} and \ref{dParam}.
\end{proof}

\subsection*{$\zetabold^\perp$ dislocations}
 \nopagebreak
 
When not all $|m_i|$'s are equal we need to expand our families of LD solutions.  
We first determine the spaces $\valperp[L_i]$ in some simple cases 
and define configurations for the corresponding families. 

\begin{lemma}
\label{Rlp} 
Given $L_i = L[\sss_i;m_i]$ as in \ref{dL} and $m_i \in \{\pm m, \pm 2m, \pm 3m\}$,  
$\lambdabold^\perp_i = (\lambda^\perp_p)_{p\in L_i}\in \valperp[L_i]$ as defined in \ref{dVal} has the form determined by the following, 
except that $b=0$ when $\sss_i=0$. 
\begin{enumerate}[label=\emph{(\alph*)}]
\item For $m_i = \pm m$: $\lambda^\perp_p = 0$  $\forall p \in L_i$. 
\item For $m_i = -2m$: $\lambda^\perp_{p_i} =  a \left. d \theta\right|_{p_{i}}$ for some $a \in \R$, where
$p_i: =( e^{i \frac{\pi}{2m}}, \sss_i)$.
\item For $m_i= 2m$: 
$\lambda^\perp_{p_{i\pm}} =  \pm (a+ b  \left. d \sss\right|_{p_{i\pm}})$ for some $a, b \in \R$, where
$p_{i\pm} := ( e^{i(\frac{\pi}{2m}\pm \frac{\pi}{2m})}, \sss_i)$.
\item 
For $m_i = 3m$: $\lambda^\perp_{p_{i+}} = 2( a+ b \left. d\sss\right|_{p_{i+}})$ and 
$\lambda^\perp_{p_{i-}} = - a - b \left. d \sss\right|_{p_{i-}} + c \left. d\theta\right|_{p_{i-}}$ for some $a, b, c \in \R$, where
$p_{i+} := ( (1, 0), \sss_i )$,  $p_{i-} :=(  e^{i\frac{2\pi}{3m}}, \sss_i )$
\item For $m_i =-3m$: $\lambda^\perp_{p_{i+}} = 2( a+ b \left. d\sss\right|_{p_{i+}})$ and 
$ \lambda^\perp_{p_{i-}} = - a - b \left. d \sss\right|_{p_{i-}} + c \left. d\theta\right|_{p_{i-}}$ for some $a, b, c \in \R$, where  $p_{i+} := ( e^{i \frac{\pi}{m}}, \sss_i)$, $
	 p_{i-} :=(  e^{i\frac{\pi}{3m}}, \sss_i)$.
\end{enumerate}
\end{lemma}

\begin{proof}
This follows immediately using the definitions of $L_i$ in \ref{dL}, $\valperp[L_i]$ in \ref{dVal}, and the symmetries. 
\end{proof}

\begin{definition}[Configurations with $\zetabold^\perp$]
\label{dLtilde}
 Given $\zetabold = \zetabold^\top+ \zetabold^\perp \in \Btilde_{\Pcal}$, 
we define 
a $\gcyl_m$-invariant $\taubold = \taubold\llbracket \zetabold \rrbracket = \taubold\llbracket\zetabold; \kcir, \mbold\rrbracket : \Ltilde \rightarrow \R$, 
which is a perturbation of $\taubold \llbracket \zetabold^\top\rrbracket$ defined in \ref{dtau1}, 
by 
$\Ltilde = L\llbracket \zetabold\rrbracket = \bigcup_{i=1}^k \Ltilde_i $,
where $\Ltilde_i = L_i\llbracket \zetabold\rrbracket    := \gcyl_m p_{i}$        if $m_i =\pm m,-2m$ 
and $\Ltilde_i = L_i\llbracket \zetabold\rrbracket := \Ltilde_i^+ \disjun \Ltilde^-_i$ 
with $\Ltilde^{\pm}_i := \gcyl_m p_{i\pm}$ otherwise, 
and 
where $p_{i}$, $\tau_{p_{i}}$, $p_{i\pm}$, and $\tau_{p_{i\pm}}$ are defined as follows (with $\tau_1$ and $\tau'_i$ as in \ref{dtau1}); 
except we have $\xitilde_i =0$ when $\sss_i=0$.   
\begin{enumerate}[label=\emph{(\alph*)}]
\item  For $m_i = \pm m$: 
$p_i := ( e^{i( \frac{\pi}{2m} \pm \frac{\pi}{2m})}, \sss_i)$ and $\tau_{p_{i}} = \tau_1 \tau'_i$. 
\item  For $m_i= -2m$: 
$p_i: =( e^{ \frac{i}{2m}(\pi + \xitildecir_i)}, \sss_i)$ and $\tau_{p_{i}} = \tau_1 \tau'_i$, where $\xitildecir_i \in \Pcal^\perp_i \simeq \R $.
\item  For $m_i = 2m$: $p_{i\pm} := ( e^{i(\frac{\pi}{2m}\pm \frac{\pi}{2m})}, \sss_i\pm  \xitilde_i)$, $\tau_{p_{i\pm}} =e^{\pm \sigmatilde_i} \tau_1 \tau'_i$, 
where $(\sigmatilde_i, \xitilde_i)\in \Pcal^\perp_i \simeq \R^2$.
\item  For $m_i = 3m$:  $p_{i+} := ( (1, 0), \sss_i +2 \xitilde_i)$,  $p_{i-} :=(  e^{i(\frac{2\pi}{3m} +
\xitildecir_i)}, \sss_i - \xitilde_i)$, and
$\tau_{p_{i\pm}} =e^{\pm \sigmatilde_i} \tau_1 \tau'_i$, 
where $(\sigmatilde_i, \xitilde_i, \xitildecir_i)\in \Pcal^\perp_i \simeq \R^3$.
\item   For $m_i = -3m$:  $p_{i+} := ( e^{i \frac{\pi}{m}}, \sss_i +2 \xitilde_i)$, 
$ p_{i-} :=(  e^{i(\frac{\pi}{3m} + \xitildecir_i)}, \sss_i - \xitilde_i)$, and $\tau_{p_{i\pm}} =e^{\pm \sigmatilde_i} \tau_1 \tau'_i$, 
where $(\sigmatilde_i, \xitilde_i, \xitildecir_i)\in \Pcal^\perp_i \simeq \R^3$.  
\end{enumerate}
\end{definition}

Note that $\taubold\llbracket \zetabold\rrbracket$ in \ref{dLtilde} is equal to $ \taubold\llbracket \zetabold^\top\rrbracket$ as defined in \ref{dtau1} 
when $\zetabold^\perp = \zerobold$.
In order to keep the presentation simple we assume now the following.  

\begin{assumption}
\label{Ambold}
We assume in the rest of the article that $k, \mbold, m$ are as in \ref{dL} satisfying \ref{Amk}, 
and furthermore that 
$m_i \in \{m, -m, -2m\}$ $\forall i\in \{1, \dots, k\}$. 
\end{assumption}

\begin{definition}[The spaces {$\skernelv_{\sym}[\Ltilde]$} and {$\skernel_{\sym}[\Ltilde]$}]
\label{dkernelsym}
Given $\Ltilde$ as in \ref{dLtilde} we define  
\begin{align*}
\skernelv_\sym[\Ltilde] = \bigoplus_{i=1}^k \skernelv_\sym[\Ltilde_i] = \bigoplus_{i=1}^k \skernelv^\top_\sym[\Ltilde_i] \oplus \skernelv^\perp_{\sym}[\Ltilde_i],
\end{align*}
where
$ \skernelv^\top_\sym[\Ltilde_i]  := \vspan\{ V_i  , V'_i\}_{i=1}^k$ 
with $V_i , V'_i \in  C^\infty_{\sym[m]}(\cyl_I)$ defined by requesting that they are supported on $D^\chi_{\Ltilde_i}(2\delta_i)$ and $\forall p\in \Ltilde_i$ they satisfy   
\begin{equation}
\label{EVV}
\begin{aligned}
V_i :=& \Psibold[ \delta_i, 2\delta_i; \dbold^\chi_{p}](  \phiunder[ 1, 0 ; \sss(p)], 0), \quad 
\\
V'_i :=&  \Psibold[ \delta_i, 2\delta_i; \dbold^\chi_{p}]( \phiunder[ 0, 1 ; \sss(p) ], 0),
\end{aligned}  
\quad \text{ on } D^\chi_{p}(2\delta_i); 
\end{equation}
and $\skernel^\perp_\sym[\Ltilde_i] := \vspan\{V^\circ_i : i \in \{1, \dots, k\} \text{ with } m_i = -2m \}$, 
with $V^\circ_i \in C^\infty_{\sym[m]}(\cyl_I)$ 
defined by requesting that 
it is supported on $D^\chi_{\Ltilde_i}(2\delta_i)$ and 
\begin{align*}
V^\circ_i =  \Psibold[ \delta_i, 2\delta_i; \dbold^\chi_{p_i}]( u^\circ_i - \phiunder[u^\circ_i(p_i), \partial_{\sss}u^\circ_i(p_i );  \sss_i], 0) 
\qquad \text{ on } D^\chi_{p_i}( 2\delta_i), 
\end{align*}
where $u^\circ_i$ is the solution of the Dirichlet problem $\Lcal_{\chi} u^\circ_i=0$ on $D^\chi_{p_i}(3 \delta_i )$ 
with boundary data on $\partial D^\chi_{p_i}(3 \delta_i)$ given by $u^\circ_i( e^{i\theta}, \sss) = \sin (\theta - \theta(p_i))$, 
where $\theta(p_i): = \frac{1}{2m}(3\pi + \xitildecir_i)$ (recall \ref{dLtilde}).
\end{definition}

\begin{lemma}
\label{LVker}
The spaces $\skernelv_\sym\llbracket \zetabold \rrbracket$ 
(as in \ref{dkernelsym}) 
satisfy parts (i)-(v) of Assumption \ref{aK}.

\end{lemma}
\begin{proof}
Items (i) and (ii) of \ref{aK} follow from Definition \ref{dkernelsym}, the definition of $\delta_i$ in \ref{ddeltai}, and Lemma \ref{LGdiff}, where we note we can bound $\frac{\partial \conf}{\partial \sss}$ on $\Omega[\sss_i;m_i]$ by a constant depending on $k$ using that $\sss_k < C\log k$ from \ref{Lrldest}. 

Next, observe that $\Ecal_L: \skernel_{\sym}[\Ltilde] \rightarrow \val_{\sym}[\Ltilde]$ splits as a direct sum of maps, $\Ecal_L = \bigoplus_{i=1}^k ( \Ecal^\top_{\Ltilde_i} \oplus \Ecal^\perp_{\Ltilde_i})$, where $\Ecal^\top_{\Ltilde_i} : \skernel^\top_{\sym}[\Ltilde_i]\rightarrow \val^\top[\Ltilde_i]$ and $\Ecal^\perp_{\Ltilde_i }: \skernel^\perp_{\sym}[\Ltilde_i] \rightarrow \val^\perp[\Ltilde_i]$.

Next, by the definitions and using that $g = e^{2\conf} \chi$ (recall \ref{Aimm}) we see that $\Ecal^\top_{\Ltilde_i}$ is invertible for each $i=1, \dots, k$ and moreover that 
\begin{align}
(\Ecal^\top_{\Ltilde_i})^{-1}\left((a+b d\sss)_{p\in \Ltilde_i}\right) =  a V_i + b V'_i.
\end{align}
Now fix $i \in \{1, \dots, k\}$ with $m_i = -2m$.  Since $\Ecalunder_{p_i} V^\circ_i = u^\circ_i(p_i) + d_{p_i} u^\circ_i = \frac{\partial u^\circ_i}{\partial \theta}(p_i)d\theta$ (recall the definitions of $V^\circ_i$ and $u^\circ_i$ in \ref{dkernelsym}), it follows from the definition of $u^\circ_i$ that $\Ecal^\perp_{\Ltilde_i}$ is an isomorphism as well. 

  Next, it is easy to see from \eqref{EVV} that the estimates 
\begin{align*}
\| V_i: C^{j}_{\sym[m]}(\cyl, \chitilde\, ) \| \leq C(j),  \qquad
\| V_i' : C^{j}_{\sym[m]}(\cyl, \chitilde\,) \| \leq C(j)
\end{align*}
hold for $i\in \{1, \dots, k\}$.  Now using \ref{Lconf} to convert these estimates to estimates where the norm $\chitilde$ is replaced first with $\chi$ and then with $g = e^{2\conf} \chi$, we conclude that $\|( \Ecal^\top_{\Ltilde_i})^{-1} \| \le C(k) m^{2+ \beta}$, and analogously, $\|(\Ecal^\perp_{\Ltilde_i})^{-1}\| \leq C(k) m^{2+\beta}$, where we have extended the notation for norms in \ref{aK}(iv) in the obvious way.  Then using from \ref{ddeltai} that $\delta_{p_i} = \frac{1}{9m}e^{-2\conf(p_i)}$ and combining the preceding, we conclude \ref{aK}(iv) holds.  Finally, \ref{aK}(v) holds from the preceding and Taylor's theorem. 
\end{proof}

We now define the full family of LD solutions we use. 
Note that when $\zetabold^\perp \ne \zerobold$ the singular sets of the LD solutions are $(\sbold, \mbold)$-rotational (recall \ref{dLmbold}), 
but not $(\sbold, \mbold)$-symmetric, perturbations of the ones with $\zetabold^\perp = \zerobold$.   

\begin{definition}[LD solutions $\varphi\llbracket \zetabold\rrbracket$]
\label{dtau2}
Given $\zetabold = \zetabold^\top+ \zetabold^\perp \in \Btilde_{\Pcal}$, 
we define using \ref{LsymLD} the LD solution $\varphi = \varphi\llbracket \zetabold\rrbracket = \varphi [ \taubold\llbracket \zetabold\rrbracket]$, 
where $\taubold\llbracket \zetabold\rrbracket$ is as in \ref{dLtilde}.
\end{definition}

\begin{definition}
\label{dZcal}
Given $p\in \cyl$, we define $\Zcal_p : \val[p]\rightarrow \R^3$ by $\Zcal_p( \mu + \mu' d\sss + \mu^\circ d\theta) = (\mu, \mu', \mu^\circ) $. 
\end{definition}

\subsection*{Green's functions on $\cyl$}
 \nopagebreak

In order to study the effect of the $\zetabold^\perp$ parameters on the mismatch of an LD solution $\varphi\llbracket \zetabold\rrbracket$ defined in \ref{dtau2}, 
we will first study a $\gcyl_m$-invariant LD solution $\Phi_m$ whose singular set is a single $L[\sssunder; m]$.  
In particular in \ref{Lxitildecir}, we estimate $\Ecalunder_p \Phi_m$ for certain points $p \in \Lpar[\sssunder]\setminus L[\sssunder; m]$.  
Later in \ref{Cphim} and \ref{Lmatchingperp}, this will be used to compare the mismatch of $\varphi\llbracket \zetabold \rrbracket$ to that of $\varphi\llbracket \zetabold^\top\rrbracket$.

Given $m \geq 2$, $\sssunder \in \R$, consider the $\gcyl_m$-invariant LD solution $\Phi_m [ \taubold]$ (recall \ref{LsymLD}), 
where $\forall p \in L[\sssunder; m]$, $\taubold$ takes the value $1$.  
Because of the symmetries, there is a function  $G_{m} : \R^2 \setminus L_m \rightarrow \R$ uniquely determined by the condition 
\begin{align}
\label{Eglift}
\Phi_m \circ \Thetacyl (\theta, \sss) = G_m (m \theta, m(\sss - \sssunder)) = G_m( \thetatilde, \ssshat\, ),
\end{align}
where $L_m := \{ (2\pi k, m(\sssunder \pm \sssunder)) : k \in \Z\} \subset \R^2$, 
$\thetatilde : = m \theta$ and $\shat := m(\sss - \sssunder)$.  
Note that $G_m \in C^\infty(\R^2 \setminus L_m)$ and satisfies $\Lchitilde G_m = 0$, where $\Lchitilde = \Delta_{\chitilde} + m^{-2} V$ and  $\chitilde = d \, \ssshat^2 + d\thetatilde^2$.

\begin{lemma}
\label{LGconv}
There exist $\phi_m \in C^\infty(\R^2)$, depending only on $\shat$ and satisfying $\Lchitilde \phi_m = 0$, 
such that on compact subsets of $\R^2 \setminus L_\infty$, where $L_\infty: = \{(2\pi k, 0):k\in \Z\}$, 
the functions $G_m+\phi_m$ converge smoothly to the singly periodic harmonic function $G_\infty : \R^2 \setminus L_\infty \rightarrow \R$ defined 
\cite{Ammari} by 
\begin{align*}
\Glap (\thetatilde, \ssshat) =  \frac{1}{2} \log \left( \sin^2 \frac{\thetatilde}{2} + \sinh^2 \frac{\ssshat}{2}\right).
 \end{align*}
\end{lemma}

\begin{proof}
We first consider the average parts $G_{m, \ave}$ and $G_{\infty, \ave}$.  
It follows from the vertical balancing lemma \ref{Lvbal} that $G_{m, \ave}$ and $G_{\infty, \ave}$ have the same derivative jump at $\ssshat =0$, that is (recall \ref{Npartial})
\begin{align*}
( \partial_+ G_{m, \ave} + \partial_- G_{m, \ave})(0) = (\partial_+ G_{\infty, \ave} + \partial_- G_{\infty, \ave})(0).
\end{align*}
For each $m\in \N$, there is a unique $\phi_m \in C^\infty(\R^2)$ such that $\phi_m$ depends only on $\ssshat$, $\Lchitilde\phi_m = 0$, and 
\begin{align*}
G_{m, \ave}(0) +\phi_m(0) &= G_{\infty, \ave}(0),\\
(\partial_+G_{m, \ave} - \partial_-G_{m, \ave})(0) + \partial \phi_m (0) &= 
(\partial_+G_{\infty, \ave} - \partial_-G_{\infty, \ave})(0).
\end{align*}
It follows that $G_{m, \ave} + \phi_m$ converges smoothly to $G_{\infty, \ave}$ on compact subsets of $\R^2 \setminus \{\shat = 0\}$. 

In the rest of this proof, denote $r: = \dbold^\chitilde_{L_m}$, $D_m: = D^{\chitilde}_{L_m}(1/5)$, and $D_\infty: = D^{\chitilde}_{L_\infty}(1/5)$.
We define $\Gtilde_m : = \sum_{k=0}^\infty \Gtilde^{(k)}_m$, for $\Gtilde^{(k)}_m$ defined as follows:  $\Gtilde^{(0)}_m \in C^\infty( D_m\setminus L_m)$ is defined by requesting that $\Lchitilde \Gtilde^{(0)}_m=0$, $\Gtilde^{(0)}_m$ depends only on $r$, $\Gtilde^{(0)}_m$ vanishes on $\partial D_m$, and $\Gtilde^{(0)}_m =  \log (5 r) +O((\frac{r}{m})^2\log \frac{r}{m})$; $\Gtilde^{(1)}_m \in C^{2, \beta}(D_m)$, on $D_\infty$ is the solution of the Dirichlet problem
\begin{align*}
\Delta_{\chitilde} \Gtilde^{(1)}_m &= - m^{-2} (V- V(0)) \Gtilde^{(0)}_m \quad \text{on} \quad  D_\infty\\
\Gtilde^{(1)}_m &= 0 \quad \text{on} \quad \partial D_\infty
\end{align*}
and on $D_m \setminus D_\infty$ solves the analogous Dirichlet problem with $V(0)$ replaced with $V(-2\sssunder m)$; 
and for $k\in \N$,  $\Gtilde^{(k+1)}_m \in C^{2, \beta}( D_m)$ is the solution of the Dirichlet problem 
\begin{align*}
\Delta_{\chitilde} \Gtilde^{(k+1)}_m &= - m^{-2} V \Gtilde^{(k)}_m \quad \text{on} \quad  D_m\\
\Gtilde^{(k+1)}_m &= 0 \quad \text{on} \quad \partial D_m.
\end{align*}
By the preceding and standard regularity theory, it follows that $\Gtilde_m \in C^{\infty}(D_m\setminus L_m)$, that $\Lchitilde \Gtilde_m =0$ on $D$, and that on each compact subset of $D_\infty\setminus L_\infty$, $\Gtilde_m$ converges smoothly to $\log(5 r)$ as $m\rightarrow \infty$.

We next define decompositions $G_m := \Ghat_m + G'_m$ and $G_\infty: = \Ghat_{\infty} + G'_{\infty}$, where $\Ghat_m, \Ghat_{\infty}$ are defined by requesting that they are supported on $D_m$ and $D_\infty$ respectively and satisfy
\begin{align*}
\Ghat_m = \Psibold[\textstyle{\frac{1}{10}}, \textstyle{\frac{1}{5}}; r]( \Gtilde_m, 0), \quad 
\Ghat_\infty &= \Psibold[\textstyle{\frac{1}{10}}, \textstyle{\frac{1}{5}}; r]( \log(5r), 0).
\end{align*}
Clearly $\Ghat_m$ converges smoothly to $\Ghat_\infty$ on compact subsets of $\R^2 \setminus L_\infty$.  
Since 
\begin{align*}
\Lchitilde G'_{m, \osc} = - \Lchitilde \Ghat_{m, \osc} \quad \text{and} \quad
\Delta_{\chitilde} G'_{\infty, \osc} = - \Delta_{\chitilde} \Ghat_{\infty, \osc},
\end{align*}
by using the smooth convergence of $\Ghat_m$ to $\Ghat_{\infty}$, separation of variables (arguing as in \ref{Lsol}, including using the exponential decay away from $L_m$ of the oscillatory modes), it follows that $\Ghat'_{m, \osc}$ smoothly converges to $\Ghat'_{\infty, \osc}$ on compact subsets of $\R^2$.
Combined with the preceding analysis of the average parts, this completes the proof.
\end{proof}

\begin{remark}
Although we do not need it here, one can compute $G_{\infty, \ave}(\ssshat) = \frac{|\ssshat|}{2} - \log 2$. 
\qed 
\end{remark}

\begin{lemma}
\label{Lxitildecir}
For $\xitildecir \in \R$ with $|\xitildecir |$ small, the following hold as $m\rightarrow \infty$. 
\begin{enumerate}[label=\emph{(\roman*)}]
\item $\Phi_m(  e^{\frac{i}{m}(\pi+\xitildecir)}, \sssunder) = \Phi_m(e^{\frac{i\pi}{m}}, \sssunder) +O(|\xitilde^\circ|^2). $
\item $\frac{\partial \Phi_m}{\partial \stilde}( e^{\frac{i}{m}(\pi+\xitildecir)}, \sssunder) = \frac{\partial \Phi_m}{\partial \stilde}( e^{\frac{i \pi}{m}}, \sssunder) + O(|\xitildecir|).$
\item $\frac{\partial \Phi_m}{\partial \thetatilde}( e^{\frac{i}{m}(\pi+\xitildecir)}, \sssunder) = \xitildecir ( - \textstyle{\frac{1}{4}} + o(1)) + O(|\xitildecir|^2)$.
\end{enumerate}
\end{lemma}
\begin{proof}
For item (i), we have via Taylor's theorem
\begin{equation*}
\Phi_m(  e^{\frac{i}{m}(\pi+\xitildecir)}, \sssunder) \: = \: G_m( \pi + \xitildecir, 0) 
= \: G_m(\pi, 0) + \xitildecir \frac{\partial G_m}{\partial \thetatilde}( \pi, 0) + O(|\xitildecir|^2)
\: = \:  \Phi_m(e^{\frac{i\pi}{m}}, \sssunder) + O(|\xitildecir|^2),
\end{equation*}
where we have used that $\frac{\partial G_m}{\partial \thetatilde}(\pi, 0) = 0$ by symmetry.  This proves (i).  Next, we have
\begin{align*}
\frac{\partial \Phi_m}{\partial \stilde}( e^{\frac{i}{m}(\pi+\xitildecir)}, \sssunder) = 
\frac{\partial \Phi_m}{\partial \stilde}( e^{\frac{i\pi}{m}}, \sssunder) + \xitildecir \frac{\partial^2 G_m}{\partial \stilde \partial \thetatilde}(\pi, 0) + O(|\xitildecir|^2),
\end{align*}
and item (ii) follows from this and \ref{LGconv}.  Finally, in similar fashion we have
\begin{align*}
\frac{\partial \Phi_m}{\partial \thetatilde}( e^{\frac{i}{m} (\pi+\xitildecir)}, \sssunder) = \frac{\partial G_m}{\partial \thetatilde}(\pi +\xitildecir, 0) 
= 
\xitildecir \frac{\partial^2 G_m}{\partial \thetatilde^2}( \pi, 0) + O(|\xitildecir|^2),
\end{align*}
where we have used that $\frac{\partial G_m}{\partial \thetatilde}(\pi, 0) = 0$ by symmetry.
Item (iii) then follows by using Lemma \ref{LGconv} and the direct calculation $\frac{\partial^2 \Glap}{\partial \thetatilde^2}(\pi, 0) = -1/4$. 
 \end{proof}
 
 \subsection*{Matching estimates}
 \nopagebreak
 
 In the last part of this section, given $\zetabold = \zetabold^\top + \zetabold^\perp \in \Btilde_{\Pcal}$, 
we will need to compare the LD solutions $\varphi\llbracket \zetabold^\top\rrbracket$ and $\varphi \llbracket \zetabold \rrbracket$ defined in \ref{dtau1} and \ref{dtau2} respectively.  
To avoid confusion, we will denote hereafter $\varphitilde = \varphi\llbracket \zetabold \rrbracket$ and $\varphi = \varphi\llbracket \zetabold^\top \rrbracket$.  
It will be useful to consider the decompositions
\begin{align*}
\varphitilde = \sum_{i=1}^k \varphitilde_i, \qquad \qquad \qquad 
\varphi = \sum_{i=1}^k \varphi_i, 
\end{align*}
where $\varphitilde_i, \varphi_i$ are $\gcyl_m$-invariant and have singular sets $\Ltilde_i$ and $L_i$.

\begin{cor}
\label{Cphim}
Suppose $i\in \{1, \dots, k\}$ and $m_i = -2m$.  Then (recall \ref{dZcal})
\begin{align*}
\Zcal_{\ptilde_i} \Mcal_{\ptilde_i} \varphitilde_i - \Zcal_{p_i} \Mcal_{p_i} \varphi_i = 
\tau_i \left( O(|\xitildecir_i|^2),   mO( |\xitildecir_i|), - \frac{m}{4}\xitildecir_i(1 + o(1)+ O(|\xitildecir_i|))\right).
\end{align*}
\end{cor}
\begin{proof}
We have $\varphitilde_i = \varphitilde_{i+} + \varphitilde_{i-}$ and $\varphi_i = \varphi_{i+} + \varphi_{i-}$, where $\varphitilde_{i+}, \varphitilde_{i-},  \varphi_{i+},  \varphi_{i-}$ are $\grot_m$-invariant and the singular sets of $\varphitilde_{i+}$ and $\varphi_{i+}$ are respectively $\grot_m \ptilde_i$ and $\grot_m p_i$.  By symmetry, we have
\begin{align*}
\Zcal_{\ptilde_i} \Mcal_{\ptilde_i} \varphitilde_{i+} = \Zcal_{p_i} \Mcal_{p_i} \varphi_{i+}. 
\end{align*}
Next, from Lemma \ref{Lxitildecir}, we have (recall \ref{DVcal})
\begin{align*}
\Zcal_{\ptilde_i} \Ecalunder_{\ptilde_i} \varphitilde_{i-} - \Zcal_{p_i} \Ecalunder_{p_i} \varphi_{i-} 
= \tau_i \big( O(|\xitildecir_i|^2),  m O( |\xitildecir_i|), - \frac{m}{4}\xitildecir_i(1+ o(1)+ O(|\xitildecir_i|))\big).
\end{align*}
The conclusion now follows by combining these equations.
\end{proof}

\begin{lemma}
\label{Lmphicomp}
Given $i, j \in \{1, \dots, k\}$ with $i\neq j$, the following hold.
\begin{enumerate}[label=\emph{(\roman*)}]
\item $\Zcal_{\ptilde_i} \Ecalunder_{\ptilde_i} ( \varphitilde_j - \varphi_j) = \tau_j (O(e^{-\frac{m}{Ck}}) , mO(e^{-\frac{m}{Ck}}) , mO(e^{-\frac{m}{Ck}}) )$.
\item $\Zcal_{\ptilde_i} \Ecalunder_{\ptilde_i} \varphi_j - \Zcal_{p_i} \Ecalunder_{p_i} \varphi_j = \tau_j  (O(e^{-\frac{m}{Ck}}), mO(e^{-\frac{m}{Ck}}), mO(e^{-\frac{m}{Ck}}))$.
\end{enumerate}
\end{lemma}
\begin{proof}
Note first that $\varphitilde_j = \varphi_j$ if $|m_j|=m$, so (i) holds trivially in that case.  Now suppose that $m_j = -2m$.  Since $\varphitilde_{j,\ave} = \varphi_{j,\ave}$, we need only establish (i) when $\varphitilde_j - \varphi_j$ is replaced with $\varphitilde_{j,\osc}- \varphi_{j,\osc}$.  The required estimate in (i) now follows from Lemma \ref{Lsol} and arguing as in the proof of \ref{LPhipest}, using \ref{Lrldest}(ii) to see that $|\sss_j - \sss_i|> \frac{|j-i|}{Ck}$.  Item (ii) follows in similar fashion from \ref{Lsol}. 
\end{proof}

 \begin{lemma}[Matching Estimates for $\varphi\llbracket \zetabold\rrbracket$]
 \label{Lmatchingperp}
  Let $\zetabold = \zetabold^\top+\zetabold^\perp \in \Btilde_{\Pcal}$, $\varphitilde = \varphi\llbracket \zetabold\rrbracket$, 
$\varphi = \varphi\llbracket \zetabold^\top \rrbracket$ be as in \ref{dtau2} and \ref{dtau1}.  
For each $i\in \{1, \dots, k\}$ such that $m_i=-2m$, the following holds. 
 \begin{multline*}
\frac{1}{\tau_i} \Zcal_{\ptilde_i} \Mcal_{\ptilde_i} \varphitilde = \frac{1}{\tau_i}\Zcal_{p_i} \Mcal_{p_i} \varphi 
\: + \: \Big( O( |\xitildecir_i|^2 + e^{-\frac{m}{Ck}}) \: , \:  
mO( |\xitildecir_i| + e^{-\frac{m}{Ck}}) \: , \:  
\\ 
- \frac{m}{4}\xitildecir_i(1 + o(1)) + O(|\xitildecir_i|^2+ m e^{-\frac{m}{Ck}})\Big).
 \end{multline*}
 \end{lemma}

 \begin{proof}
 We have $ \Zcal_{\ptilde_i} \Mcal_{\ptilde_i}\varphitilde = \Zcal_{p_i} \Mcal_{p_i}\varphi + (I) +(II)+(III)$, where 
 \begin{equation*}
 \begin{gathered}
  (I) := \Zcal_{\ptilde_i} \Mcal_{\ptilde_i} \varphitilde_i - \Zcal_{p_i} \Mcal_{p_i} \varphi_i, \quad
 (II) := \sum_{j\neq i} \Zcal_{\ptilde_i} \Ecalunder_{\ptilde_i} ( \varphitilde_j - \varphi_j), \\
 (III):= \sum_{j\neq i } \Zcal_{\ptilde_i} \Ecalunder_{\ptilde_i} \varphi_j- \Zcal_{p_i} \Ecalunder_{p_i} \varphi_j. 
 \end{gathered}
 \end{equation*}
The conclusion now follows from dividing through by $\tau_i$, combining \ref{Cphim} and \ref{Lmphicomp} to estimate (I), (II), and (III), 
and using from \ref{Ltauratio} that $\tau_j/\tau_i = \tau'_j/\tau'_i = O(1)$.
 \end{proof}

\begin{definition}
\label{dZ}
Let $\zetabold \in \Btilde_{\Pcal}$.  Define a linear map $Z_{\zetabold} : \val_{\sym}[\Ltilde] \rightarrow \Pcal$ by requesting that $Z_\zetabold$ has the direct sum decomposition $Z_\zetabold = Z^\top_\zetabold \oplus (\bigoplus_{i=1}^k Z^{\perp i}_{\zetabold})$, where $Z^\top_{\zetabold} : \valtop[\Ltilde] \rightarrow \Pcal^\top$ and $Z^{\perp i}_{\zetabold} : \valperp[\Ltilde_i] \rightarrow \Pcal^\perp_i$ are defined as follows: given $\lambdabold^\top = ( \tau_p ( \mu_p +|m_i| \mu'_p d\sss))_{p\in \Ltilde}$, define
\begin{align*}
Z^\top_{\zetabold}( \lambdabold^\top) = \bigg( \mu_1, -\frac{2F^\phi_1}{|m_1|} \dd \mubold,  2 \mubold'\bigg),
\end{align*}
where $\mubold = (\mu_i)_{i=1}^k$, $\mubold'= (\mu'_i )_{i=1}^k$ are such that $ \forall p \in \Ltilde_i$, $\mu_p = \mu_i, \mu'_p = \mu'_i$.  
Given $i\in \{1,\dots, k\}$, we define $Z^{\perp i}_{\zetabold}$ to be the trivial map if $|m_i|=m$, and if $m_i = -2m$ we define 
\begin{align*}
Z_{\zetabold}^{\perp i} ( \lambdabold^\perp_i) := - 4\mutilde^\circ_{\ptilde_i}, 
\qquad \text{ where }  \lambdabold^\perp_i = (\tau_i m \mutilde^\circ_{\ptilde} \, d\theta)_{\ptilde \in \Ltilde_i}.
\end{align*}
\end{definition}

\begin{prop}
\label{PZ}
Let $\zetabold = \zetabold^\top+\zetabold^\perp \in \Btilde_{\Pcal}$ and $\varphitilde = \varphi\llbracket\zetabold \rrbracket$ be as in \ref{dtau2}.  There is an absolute constant $C$ (independent of $\cunder$) such that for $m$ large enough (depending on $\cunder$), the map $Z_{\zetabold}$ defined in \ref{dZ} satisfies (recall \ref{dParam})
\begin{align}
\label{EZ}
\zetabold - Z_{\zetabold}( \Mcal_{L\llbracket \zetabold\rrbracket} \varphitilde) \in C \Btilde^1_{\Pcal}.  
\end{align} 
\end{prop}

\begin{proof}
Define $\muboldtilde = (\mutilde_i)_{i=1}^k$, $\muboldtilde'= (\mutilde'_i)_{i=1}^k$, and $\muboldtilde^\circ = (\mutilde_i^\circ)_{i=1}^k$ by requesting that for $i = 1, \dots, k$, 
\begin{align*}
\frac{1}{\tau_i} \Zcal_{\ptilde_i} \Mcal_{\ptilde_i} \varphitilde = ( \mutilde_i, |m_i|\mutilde'_i, m \mutilde_i^\circ).
\end{align*}
 Note that $\mutilde^\circ_i = 0$ when $|m_i|=1$ by symmetry.  By the definitions, \eqref{EZ} is equivalent to the following inequalities, where the final one holds only for those $i$ where $m_i = -2m$: 
 \begin{multline}
 \label{EZ1}
 | \zeta_1 - \mutilde_1| < C, \qquad
 \left|  \bsigma + \frac{2F^\phi_1}{|m_1|} \dd \muboldtilde \right|_{\ell^\infty}\le {C}/{m}, \quad
\\ 
 |\xibold - 2\muboldtilde'| \leq {C}/{m}, \qquad
 |\xitildecir_i + 4\mutilde^\circ_i|/ \leq C e^{-{m}/{\const}}.
 \end{multline}
The conclusion now follows from combining the estimates in \ref{LmatchingE} and \ref{Lmatchingperp} and taking $\const$  large enough in terms of $k$ and the constant $C$ in \ref{Lmatchingperp}. 
\end{proof}

 \subsection*{Main results of Part II} 
 \nopagebreak

\begin{lemma}[Estimates on the LD solutions]
\label{LLD}
Let $\varphi \llbracket \zetabold \rrbracket$ be as in Definition \ref{dtau2}.  Then 
\begin{enumerate}[label=\emph{(\roman*)}]
\item $m k \Sim_{C(\cunder)} | \log \tau_i|$, and $C(\cunder)>1$ depends only on $\cunder$.
\item $\tau_1\llbracket\zetabold; \kcir, \mbold\rrbracket \Sim_{C(\cunder)} \tau_1\llbracket\zerobold; \kcir, \mbold\rrbracket$, and $C(\cunder)>1$ depends only on $\cunder$.
\item On $\cyl \setminus \bigsqcup_{p\in L} D^{\Sigma, \chi}_p(\tau^{2\alpha}_p)$ we have $\varphi> \tau_1 cm k$ for some $c> 0$. 
\item For $i\in \{1,\dots, k\}$, $\| \frac{1}{\tau_i} \varphi: C^{2, \beta}( D^\chi_{p_i}(2 \delta_i)\setminus D^\chi_{p_i}(\delta_i),  \chitilde) \| \le Cm k$.
\item $\| \varphi : C^{3, \beta}( \Sigma \setminus \bigsqcup_{p\in L} D^{\Sigma, \chi}_p (\tau^{2\alpha}_p), \chitilde )\| \le C\tau_1(m k+ (\tau^{2\alpha}_{\min})^{-3-\beta}|\log \tau^{2\alpha}_{\min}|)$. 
\end{enumerate}
\end{lemma}
\begin{proof}
 (i) follows from the definitions of $\tau_i$ in  \eqref{Etau1} and \ref{dLtilde}, using \ref{Lrldest} and \ref{Ltauratio}. For (ii), we denote for convenience in this proof $\phi = \phi[\bsigmaunderslash+ \bsigmaunder:\kcir, \mbold]$ and $\phi'=\phi[\bsigmaunderslash:\kcir, \mbold]$.  We have
 \begin{align*}
\left|  \log \frac{\tau_1[\zetabold; \kcir, \mbold]}{\tau_1[\zerobold;\kcir, \mbold]}\right|  = \left| \zeta_1 +\frac{|m_1|}{2} \frac{ F^{\phi'}_1 - F^{\phi}_1}{F^{\phi'}_1 F^\phi_1} \right| \leq Ck \cunder,
 \end{align*}
where the equality uses \eqref{Etau1} and the estimate uses
 \ref{PODEest}, \ref{Lrldest}(i), and \ref{dParam}.  This establishes (ii). 
 
For items (iii)-(v), note that it suffices to prove each estimate when $\varphi\llbracket \zetabold\rrbracket$ is replaced with $\varphi\llbracket \zetabold^\top\rrbracket$ as defined in \ref{dtau1}, since the former is a small perturbation of the latter (recall \ref{dLtilde} and \ref{dtau2}). 

Estimating $\Phip$ using Lemma \ref{LPhipest} and using \ref{Lgreen}, \ref{Ltauratio}, and \eqref{Etau1} to bound $\Ghat$, we have
\begin{align*}
|\Ghat | < \alpha C m  k \quad \text{and} \quad
|\Phip| < C \quad \text{on} \quad 
\cyl \setminus \bigsqcup_{p\in L} D^{\Sigma, \chi}_p( \tau^{2\alpha}_p). 
\end{align*} 
On the other hand, it is easy to see from Definition \ref{ddecomp}, \ref{Lrldest}(iii), and \ref{Lphiavg}(ii) that there is an absolute constant $c> 0$ such that $\Phat > cm k$, so (iii) follows from the decomposition (recall \ref{ddecomp}) $\varphi\llbracket \zetabold^\top\rrbracket = \tau_1\Phi = \tau_1(\Phat + \Phip + \Ghat)$ by taking $m$ large enough and $\alpha$ small enough.

We next prove (iv).  By \ref{Lphiavg}(iii), \ref{ddecomp}, and \ref{dtau1}, on the domain under consideration we have 
\begin{align*}
\frac{1}{\tau_i}\varphi = \phiunder \bigg[{\frac{|m_1|}{2F^\phi_1}}(e^{-\sum_{l=1}^{i-1} \sigma_l}), {\frac{|m_i|}{ 2}}\xi_i; \sss_i\bigg] + \frac{1}{\tau'_i} \Phip + \Ghat_i.
\end{align*}
The estimate in (iv) now follows from this decomposition using \ref{Lrldest} and \ref{Lode} to estimate the $\phiunder$ term, \ref{LPhip}(i) and \ref{Ltauratio} to estimate $\frac{1}{\tau'_i} \Phip$, and \ref{LEest}(i) to estimate $\Ghat_i$.

For (v), by Lemmas \ref{Lgreen} and  \ref{LEest}, we have 
\begin{align}
\label{Eghatcirc}
\left\| \Pp : C^{3, \beta}_\sym\left( \cyl\setminus D^\chi_{L}(\tau^{2\alpha}_{\min}), \chi\right)\right\| \leq C (\tau^{2\alpha}_{\min})^{-3-\beta}| \log \tau^{2\alpha}_{\min} | .
\end{align}
Combined with the preceding estimates, this completes the proof of (v). 
\end{proof}

\begin{lemma}
\label{Ldiffeoshr}
There exists a family of diffeomorphisms $\Fcal^\Sigma_{\zetabold}: \Sigma \rightarrow \Sigma, \zetabold \in \Btilde_{\Pcal}$ satisfying \ref{Adiffeo}\ref{Aa}-\ref{Ab}.
\end{lemma}

\begin{proof}
The proof is essentially the same as the first part of the proof \cite[Lemma 6.7]{kapmcg}, but we give the details for completeness.  Let $\zetabold \in \domzb$.  For ease of notation, denote the positive $\sss$-coordinates of the circles $\Lpar\llbracket\zerobold\rrbracket$ by $\sbold$ and likewise the coordinates of the circles in $\Lpar\llbracket\zetabold\rrbracket$ by $\sbold'$.  We define $\Fcal^\Sigma_{\zetabold}: \Sigma \rightarrow \Sigma$ to be an $O(2)\times \Z_2$ covariant diffeomorphism satisfying $\Fcal^\Sigma_{\zetabold} (X_\Sigma (p, \sss) ) = X_\Sigma( p, f_\zetabold (\sss))$, where $f_\zetabold \in C^\infty(\R)$ is a diffeomorphism satisfying $f_\zetabold(\sss) = \sss'_i - \sss_i + \sss$ on $(\sss_i - 5\delta, \sss_i + 5\delta)$ for each $i=1, \dots, k$.  By choosing $f_\zetabold$ carefully, we can ensure \ref{Adiffeo}\ref{Aa} and \ref{Ab} hold. 
\end{proof}

\begin{theorem}[Theorem \ref{TB}] 
\label{Trldldgen}
Given a background as in \ref{background} satisfying \ref{Aimm},  
$\kcir \geq \kcirmin$ (recall \ref{Dkmin}), 
and ${\mbold}\in \{m,-m,-2m\}^k$ where $k=\lceil  \kcir/2\rceil $, 
there are positive constants $\cunder, \widehat{m}$ depending only on $\kcir$ such that if 
$m>\widehat{m}$ (implying \ref{Ambold}), 
then \ref{Azetabold} holds with 
$\zetabold \in \Btilde_{\Pcal} := \cunder \Btilde^1_{\Pcal}$ as in \ref{dParam},
$\Fcal^\Sigma_\zetabold$ as in \ref{Ldiffeoshr}, 
$L\llbracket \zetabold \rrbracket$ and $\taubold\llbracket \zetabold\rrbracket$ as in \ref{dLtilde}, 
$\varphi\llbracket \zetabold \rrbracket$ as in \ref{dtau2}, 
$\delta_p\llbracket \zetabold \rrbracket$ as in \ref{ddeltai}, 
$\skernelv_\sym\llbracket \zetabold\rrbracket$ as in \ref{dkernelsym}, 
and $Z_\zetabold$ as in \ref{dZ}. 
\end{theorem} 

\begin{proof}
Clearly $\Pcal$ as defined in \ref{dParam} is finite dimensional and $\domzb \subset \Pcal$ is compact and convex.

We now check the properties \ref{Azetabold}\ref{Aa}-\ref{AZ}: \ref{Aa}-\ref{Ab} follow from \ref{Ldiffeoshr}.  Next, we verify that the LD solutions $\varphi\llbracket \zetabold\rrbracket$ satisfy \ref{con:one}:  the smallness of $\tau_{\max}$ in \ref{con:one}(i) follows from \ref{LLD}(i), and \ref{con:L} holds from \eqref{Edelta} and taking $m$ large enough.
Convention \ref{con:one}(ii) follows from \ref{LLD}(i) by taking $m$ large enough, and \ref{con:one}(iii) follows from \ref{Ltauratio} and \ref{LLD}(i) also by taking $m$ large enough. 

We will prove \ref{con:one}(iv)-(v) by suitably modifying the estimates in \ref{LLD}(iv)-(v).  For (iii), first note that by \ref{LGdiff} and \ref{ddeltai}, $\partial D^{\Sigma, g}_{p_i}(\delta_{p_i}) \subset D^{\Sigma, \chi}_{p_i}(2\delta)$.  Then, using \ref{Lconf}, we can switch the metric with which the norm on the left hand side of \ref{LLD}(iv) is computed with respect to from $\chitilde$ to $\chi$ and then from $\chi$ to $g = e^{2\conf} \chi$ at the cost of multiplying the right hand side by powers of $m$ and constants depending on the norms of $\conf$.  \ref{con:one}(iv) then follows because we can ensure that any polynomial in $m$ of bounded degree is bounded by $\tau_{p_i}^{-\alpha/9}$ by taking $m$ large enough and using \ref{LLD}(i). 

\ref{con:one}(v) follows in an analogous way: first, using the smallness of $\tau_{p_i}$ and the boundedness of $\conf$ and its derivatives in the $\chi$ metric in the vicinity of $L$, we have that $D^{\Sigma, \chi}_{p_i}(\tau^{2\alpha}_{p_i}) \subset D^{\Sigma, g}_{p_i}( \tau_{p_i}^{\alpha})$.  Next, note that the estimate in \ref{con:one}(v) holds when $\Sigma$ in the domain is replaced with $\Omega: = \cyl_{[-\sss_k-3/m, \sss_k+3/m]}$ by using Lemma \ref{Lconf} to convert the estimate in \ref{LLD}(v) to one where the metric $\chitilde$ is replaced with $g$ at the cost of powers of $m$ and constants depending on the norms of $\conf$ on $\Omega$.  Finally, on $\Sigma \setminus \Omega$, note that $\varphi = \tau_1(\Phat + \Phip)$, so using the exponential decay of $\Phip$ away from $\Lpar$ from Lemma \ref{LPhipest} and that $\Phip$ satisfies $\Lchi \Phat = 0$ on $\Sigma \setminus \Omega$ we conclude the estimate \ref{con:one}(v) on $\Sigma\setminus \Omega$.  Next, \ref{con:one}(vi) follows from \ref{LLD}(iii) and that $\varphi = \tau_1 \Phi$ using the smallness of $\tau_{\max}$.  This finishes the verification of \ref{con:one} and thus the verification of \ref{Azetabold}\ref{Ac}.

Next, the uniformity condition \ref{Azetabold}\ref{Atau} follows from \ref{LLD}(ii) and \ref{Ltauratio}.  
Finally, the prescribed unbalancing condition \ref{Azetabold}\ref{AZ} follows from Proposition \ref{PZ} by taking $\cunder$ large enough in terms of the constant $C$ in \ref{PZ}.
 \end{proof}

We now construct embedded minimal doublings of $\Sigma$ by combining Theorems \ref{Trldldgen} and \ref{Ttheory}:

\begin{theorem}[Theorem \ref{TC}] 
\label{Tconstruct}
With the same assumptions and notation as in Theorem \ref{Trldldgen}
there are $\zetaboldhatunder = (\zetaboldhat, \breve{\kappaunderbold}) \in \domzb \times \Btilde_{\val\llbracket \zerobold \rrbracket}$ (recall \ref{dParam}) 
and $\upphihat \in C^\infty(M\llbracket \zetaboldhatunder \rrbracket)$ (recall \ref{dkappatilde}) 
satisfying 
$ 
\| \upphihat\|_{2, \beta, \gamma, \gamma'; M\llbracket \zetaboldhatunder \rrbracket } \le 
\tau_{\max}^{1+\alpha/4} 
$ 
(recall \ref{D:norm}), 
such that the normal graph $\Mhat : = (M\llbracket \zetaboldhatunder\rrbracket)_\upphihat$ is a $\groupcyl$-invariant embedded closed minimal 
doubling over $\Sigma$ in $N$ (recall \ref{Ddoubling})
of genus $2g_\Sigma-1+|L|$ where $g_\Sigma$ is the genus of $\Sigma$ and $|L| = |L\llbracket \zetaboldhat\rrbracket|$ is as in \ref{RLcard}. 
For each fixed $\kcir$, the surfaces $\Mhat$ converge in the sense of varifolds as $m\rightarrow \infty$ to $2\Sigma$. 
\end{theorem}

\begin{proof}
Since $\Sigma$ is closed and embedded (recall \ref{cLker}) and Assumption \ref{Azetabold} holds by \ref{Trldldgen}, we may apply Theorem \ref{Ttheory} to conclude the existence of $\Mhat$ as above, for all large enough $m$.  
$\Mhat$ has the claimed genus because the construction connects two copies of $\Sigma$ by $|L|$ bridges.
\end{proof}

\begin{remark}
\label{RSph2}
Theorem \ref{Tconstruct} applies also in the case studied in \cite{kapmcg} 
for the background $(\Sigma, N, g)$ with $\Sigma=\Sph^2 \subset N=\Sph^3$, providing new minimal doublings even for that background,    
because of the ability to prescribe $m_i \in \{m, -m, -2m\}$ $\forall i\in \{1, \dots, k\}$, 
whereas the doublings in \cite{kapmcg} had all $m_i=m$.  
\qed 
\end{remark}


\section*{Part III: New Minimal Surfaces and Self-shrinkers via Doubling}


\section{Doubling the Spherical Shrinker and the Angenent Torus}
\label{S:SphShr}

\begin{definition}
\label{dshrinker}
We call the minimal hypersurfaces in $ \big( \R^{n+1}, e^{- \frac{|x|^2}{2n}} \delta_{ij}\big)$ \emph{self-shrinkers}.  
\end{definition}
The following well-known lemma catalogs several equivalent characterizations of self-shrinkers:

\begin{lemma}[cf. {\cite[Section 1]{cmcompact}}]
\label{Lssminimal}
Let $\Sigma^n \subset \R^{n+1}$ be a smooth oriented hypersurface.  The following are equivalent. 
\begin{enumerate}[label=\emph{(\roman*)}]
\item $H = \frac{\langle x, \nu \rangle}{2}$.
\item The one-parameter family of hypersurfaces $\Sigma_t : \Sigma \times (-\infty, 0] \rightarrow \R^{n+1}$ defined by $\Sigma_t ( p, t) = \sqrt{ - t} p$ flows by mean curvature. 
\item $\Sigma \subset \left( \R^{n+1}, e^{- \frac{|x|^2}{2n}} \delta_{ij}\right)$ is minimal. 
\item $\Sigma \subset \R^{n+1} $ is a critical point for the area (or volume) induced by the Gaussian metric $e^{-\frac{|x|^2}{2n}} \delta_{ij}$. 
\end{enumerate}
\end{lemma}
In this section we consider the ambient Riemannian three-manifold in the background is taken to be $(N, g) = (\R^3, e^{-|x|^2/4} \delta)$.

\subsection*{The spherical shrinker}
\nopagebreak

By \ref{Lssminimal}(i) $\Sphshr:=\Sph^2(2)$ is a self-shrinker, 
and is clearly $O(2)\times \Z_2$-invariant in the sense of \ref{Aimm}.  
The Jacobi operator is \cite[Lemma C.2]{kapshrinker} 
$$ 
\Lcal_{\Sphshr} := e \left( \Delta_{\Sph^2(2)} +1\right) =  \frac{e}{4} \left(\Delta_{\Sph^2(1)} + 4\right).  
$$ 
Note that $\ker \Lcal_{\Sphshr}$ is trivial since $4$ is not an eigenvalue of $\Delta_{\Sph^2(1)}$. 
Next, note that  $X_{\Sphshr} : \cyl \rightarrow \Sigma$ and $\conf$ defined by
\begin{align}
\label{Esphimm} 
X_{\Sphshr} ( p, \sss) = 2(\sech \sss \, p, \tanh \sss), \quad
e^{2\conf(\sss)} = {4}{e}^{-1} \sech^2 \sss
\end{align}
are as in the conclusion of Lemma \ref{LAconf}, and that $V = 4 \sech^2 \sss$ (recall \ref{dLchi}). 

\begin{lemma}
\label{Lphiee}
\label{Lphis}
$\phie$ and $\phiend$ satisfy the following (recall \ref{dHflux} and \ref{AV2}):
\begin{enumerate}[label=\emph{(\roman*)}]
\item $\phie$ is strictly decreasing on $[0, \infty)$, and has a unique root $\sss^{\phie}_{\mathrm{root}} \in (0, \infty)$.
\item  $\phiend(0)<0$, $\phiend$ is strictly increasing on $[0, \infty)$, and has a unique root $\sss^{\phiend}_{\mathrm{root}} \in (0, \sss^{\phie}_{\mathrm{root}} )$. 
\end{enumerate}
\end{lemma}
\begin{proof}
A straightforward consequence of the fact that $4$ is between the first two nonzero eigenvalues ($2$ and $6$ respectively) of the Laplacian on $\Sphtwo(1)$.
\end{proof}

\begin{theorem}[Doublings of the spherical shrinker $\Sphshr$]
\label{Tmainsph}
Given any integer $\kcir \geq 2$, any $m\in \N$  large enough depending only on $\kcir$, and any $\mbold \in \{m, -m, -2m\}^{\lceil \kcir/2\rceil}$, there is a $\gcyl_m$-invariant doubling of $\Sphshr$ as a self-shrinker for the mean curvature flow containing one catenoidal bridge close to each singularity of one of a family of $\group_m$-invariant LD solutions as in Theorem \ref{Trldldgen} whose singularities concentrate on $\kcir$ parallel circles, with the number of singularities and their alignment at each circle prescribed by $\mbold$.  Moreover, as $m\rightarrow \infty$ with fixed $\kcir$, the corresponding doublings converge in the appropriate sense to $\Sphshr$ covered twice. 
\end{theorem}
\begin{proof}
It follows that $\kcirmin = 2$ by combining Lemma \ref{Rk1}(ii) and \ref{Lphis}(ii).
The discussion above shows that \ref{Aimm} holds, so the existence of the doublings follows immediately from Theorem \ref{Tconstruct}.
\end{proof}

\subsection*{The Angenent torus}
\nopagebreak

 In \cite{angenent}, Angenent constructed an embedded and $O(2)\times \Z_2$-invariant (in the sense of \ref{Aimm}(ii)) self-shrinking torus, which we denote in this subsection by $\Tor$.  

\begin{lemma}
\label{Langric}
$\Ric(\nu, \nu)>0$ on $\Tor$. 
\end{lemma}
\begin{proof}
We have (see e.g. the proof of \cite[Proposition C.2]{kapshrinker})
\begin{align*}
\Ric(\nu, \nu) = e^{- \frac{|x|^2}{4}} \left( 1 + \frac{(x\cdot \nu_0)^2}{16} - \frac{|x|^2}{16}\right),
\end{align*}
where above $x$ and $\nu_0$ are the position vector field and the Euclidean unit normal to $\Tor$ and the norms and dot product are computed with respect to the Euclidean metric.  From \cite[Proposition 2.1]{mollertorus} (see also \cite{Berchenko}), we have that $\max_{x\in \Tor} |x| < 3.4$ and the conclusion follows. 
\end{proof}

\begin{theorem}[Doublings of the Angenent torus $\Tor$]
\label{Tmaintor}
There exists  $\kcirmin \in \N$ such that if $\kcir \geq \kcirmin$, $m\in \N$ is large enough depending only on $\kcir$, and $\mbold \in \{m, -m, -2m\}^{\lceil \kcir/2\rceil}$, there is a $\gcyl_m$-invariant doubling of $\Tor$ as a self-shrinker for the mean curvature flow containing one catenoidal bridge close to each singularity of one of a family of $\group_m$-invariant LD solutions as in Theorem \ref{Trldldgen} whose singularities concentrate on $\kcir$ parallel circles, with the number of singularities and their alignment at each circle prescribed by $\mbold$.  Moreover, as $m\rightarrow \infty$ with fixed $\kcir$, the corresponding doublings converge in the appropriate sense to $\Tor$ covered twice. 
\end{theorem}
\begin{proof}
In order to apply Theorem \ref{Tconstruct}, we need only check that \ref{Aimm} holds.  \ref{Aimm}(i)-(ii) hold by the discussion above.   
It follows from \ref{Langric} that $|A|^2 +\Ric(\nu, \nu)>0$ on $\T$ and therefore that \ref{Aimm}(iii) holds.  
Finally, it was checked in \cite[Theorem 2.7]{mollertorus} that the intersection of $\ker \Lcal_{\Tor}$ with the set of $\grouprotcyl$-invariant functions on $\Tor$ is trivial.  
\end{proof}

\begin{remark}
Although we have not done so here, it would be interesting to determine the minimum number $\kcirmin$ of circles (recall \ref{Dkmin}) associated to the doublings of $\T$ in Theorem \ref{Tmaintor}.
\qed 
\end{remark}

%
%
\section{Doubling the Catenoid}
\label{S:Cat}
\nopagebreak

In this section, let $(N, g)$ be Euclidean three-space and $\Sigma$ be the Euclidean catenoid $\mathbb{K}$ parametrized by 
$X_{\mathbb{K}} : \cyl \rightarrow \R^3$, where $X_{\mathbb{K}} (p, \sss) = (\cosh \sss \, p, \sss)$.  
Clearly $(\Sigma, N, g)$ is $O(2)\times \Z_2$-invariant in the sense of \ref{Aimm}(ii), and $X_{\mathbb{K}}$ satisfies \ref{LAconf}(ii) with $I = \R$.  
Moreover, $V$ and $\conf$ as in \ref{Nconf} and \ref{dLchi} satisfy
\begin{align}
V(\sss) = 2\sech^2 \sss, \qquad \qquad e^{\conf(\sss)} = \cosh \sss.
\end{align}

\begin{remark}
\label{Rlinconf} 
The linearized operator $\Lcal_{\Sph^2} = \Delta_{\Sph^2} + 2$ of an equatorial sphere $\Sph^2 \subset \Sph^3$ is conformally related to $\Lcal_\Sigma$ by
\[ 
\Lcal_\Sigma = \Delta_\Sigma + | A|^2 = \frac{|A|^2}{2} \left( \Delta_{\nu^*g_{\Sph^2}} +2\right),
\] 
where $\nu : \Sigma \rightarrow \Sph^2$ is the Gauss map, so RLD and LD studied in \cite{kapmcg} can be pulled back by $\nu$ to LD solutions on $\Sigma$.  
Because of this, we may use results from sections \ref{S:RLD} and \ref{S:LDs} in this section.
\qed
\end{remark}

\begin{definition}[{\cite[2.18]{kapmcg}}]
\label{dphie}
Define $\phie \in C^\infty_{|\sss|}(\cyl)$ and $\phio \in C^\infty_\sss(\cyl)$ by
\begin{align}
\phie(\sss) = 1-\sss \tanh \sss, \qquad \qquad \phio(\sss) = \tanh \sss.
\end{align}
\end{definition}

\begin{lemma}[{\cite[2.19]{kapmcg}}]
\label{Lphie}
$\phie$ and $\phio$ are even and odd in $\sss$ respectively and satisfy $\Lcal_\chi \phie = 0$ and $\Lcal_\chi \phio = 0$.  $\phie$ is strictly decreasing on $[0, \infty)$ and has a unique root $\sroot\in (0, \infty)$.  $\phio$ is strictly increasing in $\R$.  The Wronskian $W[ \phie, \phio]$ satisfies
\[ W[\phie, \phio](\sss): = \phie(\sss)\partial \phio(\sss) - \partial \phie(\sss) \phio(\sss) = 1.\]
\end{lemma}
\begin{proof}
Straightforward calculation using Definition \ref{dphie} and \eqref{ELchirot}.  
\end{proof}

\begin{remark}
\label{rH}
By straightforward computations (recall Lemma \ref{Lphie}),
\begin{align*} 
H(\sss) =  
 \big( F_+^{\phio}(\sbar) - F\big) \phio(\sbar) \, \phie(\sss) 
+ 
\big( - F_+^{\phie}(\sbar) + F\big) \phie(\sbar) \, \phio(\sss). 
\end{align*} 
Note also that when $\sss \geq 0$, $H[F; \sssunder](\sss)  = \phiunder[1, F; \sssunder](\sss)$ (recall \ref{dauxode}).
\qed 
\end{remark}

\begin{notation}
\label{Nab}
Given $\phat[F_1; \bsigmaunder]$ as in \ref{Pexist}, we define for $i \in \{1, \dots, k-1\}$ numbers $A_i, B_i$ by 
$  
\quad \phat[F_1; \bsigmaunder] = A_i \phie + B_i \phio \quad \text{on} \quad \cyl_{[\sss_{i}, \sss_{i+1}]}.
$  
\qed 
\end{notation}

In contrast to the situation for the smooth at the ends $\mathbb{K}$-RLD solution $\phat[ \bsigmaunder : \kcir ]$, 
$\phat[ F; \bsigma]$ as defined in \ref{Pexist} above becomes severely distorted on $\cyl_{(\sss_k, \sss_{k+1})}$ as $F \nearrow a_{k, \bsigmaunder}$.   
The following lemma makes this precise. 

\begin{lemma}
\label{Llastring}
Let $\phat = \phat[ F_1; \bsigmaunder]$ be as in \ref{Pexist}, where $F \in [ a_{k+1, \bsigmaunder}, a_{k, \bsigmaunder})$.  The following hold.
\begin{enumerate}[label=\emph{(\roman*)}]
\item $\lim_{F \nearrow a_{k, \bsigmaunder}} \sss_i = \sss_i[ a_{k, \bsigmaunder}; \bsigmaunder]$ for $i=1, \dots, k$ and $\lim_{F \nearrow a_{k, \bsigmaunder}} \sss_{k+1} = \infty$. 
\item $\lim_{F \nearrow a_{k, \bsigmaunder}} \phat(\sss_{k+1}) = 0$ and $\lim_{F \nearrow a_{k, \bsigmaunder}}\frac{ \phat(\sss_{k+1})}{\phat(\sss_{k})} = 0$.  
\end{enumerate}
\end{lemma}
\begin{proof}
(i) follows immediately from Proposition \ref{Pexist}(iii).  Since by \ref{dF} $F^{\phat}_+(\sss) = \partial ( \log \phat)$ on any domain on which $\phat$ is smooth, we have by integrating on $(\sss_k, \sss_{k+1})$ that 
\begin{align}
\label{Ephiratioint}
\log \bigg( \frac{\phat(\sss_{k+1})}{\phat(\sss_k)}\bigg) = - \int_{\sss_k}^{\sss_{k+1}} F^\phat_-(\sss) \, d\sss.
\end{align}
Reparametrizing the integral in \eqref{Ephiratioint} by $\big(F^\phat_-|_{[\sss_i, \sss_{i+1}]}\big)^{-1}$ (recall \ref{LFmono}(i)), we have
\begin{multline}
\label{Ephiratio2}
\log\bigg(\frac{\phat(\sss_{i+1})}{\phat(\sss_{i})}\bigg) =  (I)+(II), \qquad\text{where} 
\\ 
(I):= -\int_{-F^\phat_{k+}}^{0} \frac{ f}{V(\sss(f))+f^2}dF, 
\qquad 
(II):= -\int_{0}^{F^\phat_{k+1-}} \frac{ f}{V( \sss(f))+f^2}df. 
\end{multline}
Note that (I) and (II) have opposite signs.  To estimate (I), recall from the proof of \ref{Pexist} that $F^\phat_+(\sss) < F^{\phiend}_+(\sss)$ on  $(\sss_k, \sss')$, where $\sss' \in (\sss_k, \sss_{k+1})$ is defined by requesting that $F^{\phat}_+(\sss') = 0$.  Using \ref{AV2}, we conclude that $F^{\phat}_+(\sss) < CV(\sss)$ on $(\sss_k, \sss')$, and from this we estimate  $|(I)|< C F^\phat_{k+}$.  For (II), we estimate 
\begin{equation}
\begin{aligned}
|(II)| &> \int_0^{F^\phat_{k+1-}} \frac{ f}{V(\sss')+ f^2} d f = \frac{1}{2} \log \bigg( 1+ \frac{F^{\phat}_{k+1-}}{V(\sss')}\bigg).
\end{aligned}
\end{equation} 
Since $\lim_{F \nearrow a_{k, \bsigmaunder}} \sss' = \infty$, it follows that $\lim_{F \nearrow a_{k, \bsigmaunder}}V(\sss') = 0$ and we conclude the proof of (ii).  
\end{proof}

\begin{lemma}
\label{Lcatend}
Let $\bsigmaunder = (\bsigma, \xibold) \in \ell^1 \left( \R^\N \right) \oplus \ell^\infty\left( \R^\N\right)$ satisfy $|\xibold|_{\ell^\infty}<\frac{1}{10}$.  There exist constants $\epsilonunder_1>0, C_1> 0$, depending only on $|\bsigma|_{\ell^1}$, such that for $k\in \N$, we have (recall \ref{Pexist})
$a_{k, \bsigmaunder}+ \epsilonunder_2/k^2< a_{k-1, \bsigmaunder}$ and on $[ a_{k, \bsigmaunder} , a_{k, \bsigmaunder}+ \epsilonunder_2/ k^2]$ we have  $\frac{ \partial A_k [F; \bsigmaunder]}{\partial F} \Sim_{C_1} -k$ (Recall also $A_k[ a_{k, \bsigmaunder}; \bsigmaunder] = 0$).
\end{lemma}

\begin{remark}
Lemma \ref{Lcatend} is similar to \cite[Lemma 7.4]{kapmcg}, except that in the present case, we are interested in the behavior of $A_k[F; \bsigmaunder]$ to the right of $a_{k, \bsigmaunder}$ instead of to the left, as was the case in \cite{kapmcg}. 
\qed 
\end{remark}

\begin{proof}
We omit the proof because it is almost identical to the proof of \cite[Lemma 7.4]{kapmcg}. 
\end{proof}

\begin{definition}
\label{dphattau}  
Let $\epsilonunder_2: = \epsilonunder_1/C_1>0$ with $\epsilonunder_1$ and $C_1>0 $ as in the statement of Lemma \ref{Lcatend}.  
Given $\kcir \in 2\N$ and $k: = \kcir/2$, $\bsigmaunder$ as in \ref{Lcatend}, 
and $\tautilde \in (- \epsilonunder_2/k, 0]$, we define $\phat[ \bsigmaunder, \tautilde : \kcir] := \phat[ F; \bsigmaunder]$, 
where $F \in [a_{k, \bsigmaunder}, a_{k, \bsigmaunder}+\epsilonunder_2/k^2]$ and $A_k[F; \bsigmaunder] = \tautilde$.  If $\tautilde = 0$, we may suppress $\tau$ and write $\phat[ \bsigmaunder, 0 :\kcir] = \phat[\bsigmaunder: \kcir]$.
\end{definition}

By modifying the proof of \ref{Lcatend} and statement of \ref{dphattau}, 
we can analogously define for $\tautilde \in (\epsilonunder_2/k, 0]$ RLD solutions $\phat[\bsigmaunder, \tautilde :\kcir]$ and $\kcir$ odd, $\kcir>1$.  
Moreover, by a straightforward modification of the statement and proof of \ref{Lphiavg}, 
we construct for $\kcir \in \N$ satisfying $\kcir >1$ LD solutions $\Phi\llceil\bsigmaunder, \tautilde: \kcir, \mbold\rrfloor$ 
whose average is a multiple of $\phat[\bsigmaunderslash+ \bsigmaunder, \tautilde : \kcir] $.  
For each $\tautilde \in (- \epsilonunder_2/k, 0]$, 
we define LD solutions $\varphi\llbracket \zetabold^\top\rrbracket$ and $\varphi\llbracket \zetabold\rrbracket$ 
as in \ref{dtau1} and \ref{dtau2} but with the modified definition of $\Phi\llceil\bsigmaunder, \tautilde: \kcir, \mbold\rrfloor$.

Because $\Sigma=\KK$ is noncompact, we must modify the definition of the initial surfaces (recall \ref{Dinit}), and we will need the following. 

\begin{definition}
\label{Dcore} 
We define $\catcore \subset  \Sigma=\KK$  to be the convex hull in the $\chi$ or $g$ metric (recall \ref{Ecyl}) of 
$\Lpar[\sss_k+1]$, and also $\catend: = \Sigma \setminus \catcore$. 
\end{definition}

\begin{definition}
\label{Dinitcat} 
Given $\varphi = \varphi\llbracket \zetabold \rrbracket$ as above and $\kappaunderbold$ as in \ref{dalpha} we define the smooth initial surface 
$$
M = M[\varphi, \kappaunderbold]:= 
\graph^{\R^3}_{\Omega}\big( \varphigl_+\, \big) \bigcup \graph^{\R^3}_{\Omega}\big(-\varphigl_-\, \big) \bigcup 
\bigsqcup_{p\in L} \cat[p, \tau_p, \kappaunder_p ], 
$$ 
where $ \Omega : =  \Sigma \setminus \disjun_{p\in L} D^{\Sigma}_p( 9 \tau_p)$ and the functions 
$\varphigl_{\pm} = \varphigl_{\pm} [ \varphi, \kappaunderbold] : \Omega \rightarrow \R$
are defined as follows:
\begin{enumerate}[label = \emph{(\roman*)} ]
\item On $\catcore$, $\varphigl_{\pm}$ is defined as in (i)-(ii) of \ref{Dinit}. 
\item On $\catend$,  
$\varphigl_{\pm} := \Psibold[1, 2; \dbold^{\Sigma,\chi}_{\Lpar[\sss_k]}](\varphi, \varphiend^\pm)$, 
where $\varphiend^{\pm} \in C^{\infty}(\catend)$ are the unique functions whose graphs 
$\graph^{\R^3}_{\catend}(\pm \varphiend^{\pm})$ over $\catend$ are catenoidal ends with vertical axes and initial values
\[ \varphiend^{\pm}(p_k) = \varphi_{\ave}(p_k), \qquad \quad
\frac{\partial \varphiend^\pm}{\partial \sss}(p_k) = \lim_{\sss \searrow \sss_k} \frac{\partial \varphi_{\ave}}{\partial \sss}.
\] 
\end{enumerate}
We define also $\Mend := \graph^{\R^3}_{\catend}(\varphigl_+)\cup \graph^{\R^3}_{\catend}(\varphigl_-)$ 
\\ $\phantom1$ \hfill 
and
$\quad \Mcore := \graph^{\R^3}_{\catcore}(\varphigl_+)\cup \graph^{\R^3}_{\catcore}(\varphigl_-)$.
\end{definition}

We need now to update the definition of the global norms to deal with the ends. 

\begin{definition}
\label{D:normcat}
For $k\in \N, \betahat \in (0, 1), \gammahat \in \R$, we define 
\begin{align*}
\| u \|_{k, \betahat, \gammahat, \gammahat'; M} : = \| u \|_{k, \betahat, \gammahat, \gammahat'; \Mcore} + \| u : C^{k, \betahat}(\Mend, \frac{1}{2} |A|^2 g )\| ,
\end{align*}
where the first term on the right hand side is as in \ref{D:norm} and $|A|$ above is the length of the second fundamental form on $\Mend$. 
\end{definition}

\begin{lemma}
\label{LglobalHcat}
$\| H - J_M(w^+, w^-)\|_{0, \beta, \gamma-2, \gamma'-2; M} \le \tau_{\max}^{1+\alpha/3}$.
\end{lemma}

\begin{proof}
Arguing as in the proof of \ref{LglobalH} and using that the graphs of $\pm \varphiend^{\pm}$ have zero mean curvature, we need only estimate the mean curvature on the transition region in \ref{Dinitcat}(ii).  We have via \ref{Dinitcat}(ii) that
\[ 
\varphigl_\pm = \varphi +  \Psibold[1, 2; \dbold^{\Sigma,\chi}_{\Lpar[\sss_k]}](0, \varphiend^{\pm}- \varphi) \qquad 
\text{on} \quad 
\catend \cap D^{\Sigma,\chi}_{\Lpar[\sss_k]}(2).
\] 
Using  \ref{con:one}(v) and the initial values in \ref{Dinitcat}(ii), note that  
$\| \varphiend^\pm - \varphi : C^k(\catend \cap D^{\Sigma,\chi}_{\Lpar[\sss_k]}(2) ) \| \le  \tau_{\max}^{3/2}.$ 
It now follows expanding $H'_{\pm}$ in linear and higher order terms as in the proof of \ref{Lgluingreg} that 
\begin{align*}
\| H'_{\pm} : C^{0, \beta}(\catend \cap D^{\Sigma,\chi}_{\Lpar[\sss_k]}(2) , g)\| \le  \tau_{\max}^{3/2},
\end{align*}
where $H'_{\pm}$ denotes the pushforward of the mean curvature of the graph of $\varphigl_\pm$ to $\mathbb{K}$ by $\Pi_{\mathbb{K}}$.  This concludes the proof. 
\end{proof}

\begin{definition}
We define smooth surfaces $\Sigma_\pm := \catcore \cup \graph_{\catend}( \varphi'_{\pm})$, 
where $\varphi'_{\pm} \in C^{\infty}(\catend)$ are defined by requesting that 
$\varphi'_{\pm} := \pm \varphiend$ on $\catend \setminus D^{\Sigma}_{\Lpar[\sss_k]}(2)$ and 
$ 
\varphi'_{\pm} := \Psibold \big[ 1, 2; \dbold^{\Sigma,\chi}_{\Lpar[\sss_k]} \big]  \big(0, \pm \varphiend \big) 
$ 
on $\catend$. 
\end{definition}

\begin{remark}
Note that in the metric $h: = \frac{1}{2} |A|^2 g$ 
the ends of $\Sigma_\pm$ are isometric to spherical caps with the poles removed.
\qed 
\end{remark}

We next modify the definition of $\RMa          $ to deal with the ends.  
For this, let $E\in C^{0, \beta}_{\sym[m]}(M)$, and let $E'_{\pm}$ be as in \ref{Edecom}.  Using \ref{cLker}, \ref{aK}, and that $h$ is very close to the round metric on $\Sph^2$, there are unique $u'_\pm \in C^{2, \beta}_{\sym[m]}(\Sigma_\pm)$ and $w^{\pm}_{E, 1}\in \skernel_{\sym[m]}[L]$ such that
\begin{align}
\label{Eupcat} 
(\Delta_{h}+2)u'_\pm = \frac{1}{2}|A|^2 \big( E'_{\pm} + w^{\pm}_{E, 1}\big) 
\quad \text{on}
\quad  \Sigma_\pm.
\end{align}

\begin{notation}
\label{Npmcat}
If $f^{\pm}$ are functions supported on $\Sigma_{\pm} \setminus \disjun_{p\in L} D^\Sigma_p(b \tau_p)$, we define $J_M(f^+, f^-)$ to be the function on $M$ supported on $M \setminus \bigsqcup_{p\in L} D^{\R^3}_p(9\tau_p)$ defined by $f^+\circ \Pi_{\Sigma_+}$ on the graph of $\varphi^{gl}_{+}$ and by $f^- \circ \Pi_{\Sigma_-}$ on the graph of $- \varphi^{gl}_-$.
\qed 
\end{notation}

Note in particular \ref{Eupcat} implies  
\begin{equation}
\begin{gathered}
\Lcal_\Sigma u'_{\pm} = E'_{\pm} + w^{\pm}_{E, 1} \quad \text{on} \quad \catcore \quad \text{and}\\ \quad
\Lcal_{M} J_M(u'_+, u'_-)  = E'_{\pm} \quad \text{on} \quad  \Mend\setminus D^{\R^3}_{\Lpar[\sss_k]}(3).
\end{gathered}
\end{equation}
We define $\RMa          $ exactly as in \ref{DRMappr}, except using the modified definitions of $u'_{\pm}$ and $J_M$ just discussed.  Further, we define $\Rcal_M$ as in the statement of \ref{Plinear} and $\Rcal'_M$ as in the proof of \ref{Plinear2}.

Define $(u, w^+_H, w^-_H) = - \Rcal'_M(H - J_M(w^+, w^-))$.  Using \ref{LglobalHcat}, the proof of \ref{Plinear}, \ref{D:normcat}, and \eqref{Eupcat}, it is not difficult to see (using separation of variables to estimate $u$ on the ends) that
\begin{align}
\label{Euhcat}
\| w^\pm_H: C^{0, \beta}(\Sigma, g)\| + \| u\|_{2, \beta, \gamma, \gamma'; M} \le \tau_{\max}^{1+ \alpha/4}. 
\end{align}

We next modify the estimates of the quadratic terms.  Given $\phi \in C^{2, \beta}(M)$ with $\| \phi\|_{2, \beta, \gamma, \gamma'; M}\le \tau_{\max}^{1+\alpha/4}$, we have by arguing as in the proof of \ref{Lquad} and using \ref{D:normcat} that
\begin{align}
\label{Equadcat}
\|H_\phi - H- \Lcal_M \phi\|_{0, \beta, \gamma-2, \gamma'-2; M} \le \tau^{3/2}_{\max}. 
\end{align}
Finally, define $(u_Q, w^+_Q, w^-_Q) = - \Rcal'_M(H_\phi - H - \Lcal_M \phi)$.  Arguing as above, we have
\begin{align}
\label{Eqestcat}
\| w^{\pm}_Q : C^{0, \beta}(\Sigma, g)\| + \| u_Q\|_{2, \beta, \gamma, \gamma'; M} \le \tau^{4/3}_{\max}. 
\end{align}

\begin{lemma}
\label{Ldiffeocat}
There exists a family of diffeomorphisms $\Fcal^{\mathbb{K}}_{\zetabold}: \mathbb{K}\rightarrow \mathbb{K}$, $\zetabold \in \Btilde_{\Pcal}$ satisfying \ref{Azetabold}\ref{Aa}-\ref{Ab}.
\end{lemma}
\begin{proof}
We omit the proof, which is very similar to the proof of \ref{Ldiffeoshr}.
\end{proof}

\begin{theorem}
\label{Tmaincat}
Given any integer $\kcir \geq 2$, and $m \in \N$ large enough determining only on $\kcir$, any $\mbold \in \{m, -m, -2m\}^{\lceil \kcir/ 2\rceil}$, and any $\tautilde \in (- \epsilonunder_2/k, 0]$ (recall \ref{dphattau})  there is a $\group_m$-invariant minimal doubling of $\mathbb{K}$ containing one catenoidal bridge close to each singularity of one of a family of $\group_m$-invariant LD solutions as in Theorem \ref{Trldldgen} whose singularities concentrate on $\kcir$ parallel circles, with the number of singularities and their alignment at each circle prescribed by $\mbold$.  Moreover, the doublings are embedded, have four ends, and as $m\rightarrow \infty$ with fixed $\kcir$, converge in the appropriate sense to $\mathbb{K}$ covered twice.
\end{theorem}

\begin{proof}
We apply the steps of the proofs of Theorems \ref{Ttheory} and \ref{Trldldgen}---with small modifications because $\mathbb{K}$ is noncompact.  
We first check that Assumption \ref{Aimm} holds, except for the condition in \ref{Aimm}(i) that $\Sigma$ is closed.  
Clearly \ref{Aimm}(ii)-(iii) hold.  
Using the Gauss map, we can conformally identify $\mathbb{K}$ with a twice punctured sphere (recall \ref{Rlinconf}), 
and therefore \ref{Aimm}(iv) holds when considering solutions which extend to the poles of the sphere. 

By a straightforward modification of the arguments in the proof of \ref{Trldldgen}, 
Assumption \ref{Azetabold} holds, where $\Fcal^{\mathbb{K}}_{\zetabold}$ are as in \ref{Ldiffeocat} and the isomorphisms $Z_\zetabold$ are as in \ref{dZ}.  
We may then apply the steps in the proof of Theorem \ref{Ttheory} except that we use the estimates \eqref{Euhcat} and \eqref{Equadcat} 
to replace items (2) and (4) in the proof of \ref{Ttheory}.  
This concludes the proof.
\end{proof}

\begin{remark}
\label{Rcat2}
Note that Theorem \ref{Tmaincat} produces a one-parameter family of doublings, 
with the parameter $\tautilde \in (- \epsilonunder_2/k, 0]$ (as in  \ref{dphattau}) governing the latitudes of the outermost circles where catenoidal bridges are placed.  
\end{remark}

\begin{remark}
\label{Rcatk1}
It is possible to construct doublings of $\mathbb{K}$ with $\kcir=1$.  
However, these would necessarily be immersed but not embedded because the corresponding LD solutions would be negative on the ends of $\mathbb{K}$ 
(recall \ref{Pexist2}, \ref{dphie}, and \ref{Nab}), so we do not study these examples in detail. 
\qed 
\end{remark}


\section{Doubling the critical catenoid}
\label{S:ccat}
\begin{definition}
\label{dfbms}
Let $(M^n, g)$ be a Riemannian manifold and $\Omega \subset M$ be a domain with smooth boundary.  
A smooth, properly immersed (in the sense that intersections with compact subsets of $\Omega$ are compact) 
submanifold $\Sigma^k \subset \Omega$ is a \emph{free boundary minimal submanifold} if its mean curvature vanishes, 
$\partial \Sigma \subset \partial \Omega$, and $\Sigma$ is orthogonal to $\partial \Omega$ along $\partial \Sigma$.
\end{definition}

Let $\B^3:=\{\, x\in\R^3 \, : \, |x|\le1\,\}$ equipped with the standard Euclidean metric. 
By standard calculations the linearized equation for free boundary minimal surfaces in $\B^3$ 
at a free boundary minimal surface $\Sigma$ in $\Omega:= \B^3$ defined as in \ref{dfbms}, 
with unit normal (smooth) field $\nu$
and unit outward conormal field $\eta$ along $\partial \Sigma$, 
is given (see for example \cite[2.25, (2.31) and (2.41)]{kapli}) by the boundary value problem
\begin{equation} 
\left\{ \,\, 
\label{Ejacfb}
\begin{aligned}
\Delta u +   |A|^2 u \, &= \,\, 0 \quad \text{on} \quad \Sigma, \\
-\frac{\partial u}{\partial \eta} + u \quad &= \,\, 0 \quad \text{on} \quad \partial \Sigma. 
\end{aligned}
\right. 
\end{equation} 

\begin{definition}
\label{Dcc}Define $I: = (-\sroot, \sroot)$ (recall \ref{Lphie}) and an immersion $X_{\ccat} : \cyl_{I} \rightarrow \R^3$ by
\begin{equation}
\label{Eccat}
X_{\ccat} =  \frac{ \sech \sroot}{ \sroot}\left. X_{\cat}\right|_{\cyl_I }.
 \end{equation}
 We call the image of $X_{\ccat}$ the \emph{critical catenoid} and denote it by $\ccat$.
\end{definition}

It is easy to check that $\ccat$ is a free boundary minimal surface in $\B^3$.  Moreover, using \eqref{Ecatmetric}
\begin{align}
\label{Eccatmetric}
e^{\conf(\sss)}= \frac{ \sech \sroot}{ \sroot}  \cosh \sss, \quad V(\sss) = 2\sech^2 \sss,
\end{align}
and it is straightforward to see that assumptions \ref{Aimm}(ii),(iii) hold. 

When $\Sigma = \ccat$, we have by \eqref{Eccat}  that  \eqref{Ejacfb} is equivalent to
\begin{align}
\label{Ejacccat}
\begin{cases}
\Lcalchi u &= 0 \hfill \quad \text{on} \quad \cyl_{I}\\
\sroot \frac{\partial u}{\partial \eta} &=  u \hfill \quad \text{on} \quad \Lpar[ \sroot].
\end{cases}
\end{align}

\begin{lemma}
\label{Lkercc} 
There are no nontrivial $\grouprotcyl$-invariant solutions of \eqref{Ejacfb} on $\ccat$.  
\end{lemma}
\begin{proof}
This was checked in Lemma 3.18 and Remark 3.20 of \cite{kapli}. 
\end{proof}

\begin{definition}
\label{dphibd}
Define a rotationally invariant function $\phicrit \in C^\infty_{\sss}(\cyl)$ by 
\begin{align}
\label{Ephibd}
\phicrit(\sss) = (\sech^2\sroot - \tanh^2\sroot) \phie(\sss) - \partial \phie(\sroot) \phio(\sss).
\end{align}
\end{definition}

\begin{lemma}
\label{Lphibd}
$\phicrit$ is strictly increasing on $[0, \sroot]$, has a unique root $\sss^{\phicrit}_{\mathrm{root}} \in (0, \sroot)$, and satisfies $\sroot F^\phicrit_+(\sroot) = 1$.
\end{lemma}
\begin{proof}
Straightforward computation from \ref{dphibd}.
\end{proof}

\begin{remark}
\label{rphibd}
 $\phicrit$ appeared also in \cite[equation 3.19]{kapli}, although there it was called $\phi_{Robin}$. 
Note also that \ref{Lphibd} shows that the construction fails for $\kcir=1$ because \ref{RL}\ref{RLF} is violated (see \ref{R:RLD}). 
\qed 
\end{remark}

\begin{lemma}
\label{Lphi:bd} For a function $\phi \in C^1_{|\sss|}(\cyl_I)$, the Robin boundary condition in \eqref{Ejacccat} is equivalent to 
the flux condition $\sroot F^\phi_+(\sroot) = 1$.
\end{lemma} 
\begin{proof} 
This is immediate from the symmetries and the definition of $F^\phi_+$ in \ref{dF}.
\end{proof} 

\begin{definition}
\label{Dccatrld}
We say $\phi \in C^0_{|\sss|}(\cyl_I)$ is a $\ccat$-RLD solution if $\phi$ is an RLD solution in the sense of Definition \ref{RL} which satisfies also 
the condition $\sroot F^\phi_+(\sroot) = 1$.

\end{definition}
By Definition \ref{Dccatrld} and Lemma \ref{Lphi:bd}, 
it follows that any $\ccat$-RLD solution $\phi$ coincides with a constant multiple of $\phicrit$ on $\cyl_{[\sss^\phi_k, \sroot]}$. 

In contrast to the situation for the RLD solutions established in Proposition \ref{Pexist}, the number of possible parallel circles of a $\ccat$-RLD solution with $\bsigmaunder^\phi = \zerobold$ is limited:

\begin{lemma}
\label{Lccb}
Suppose $\phi$ is a $\ccat$-RLD solution satisfying $\bsigmaunder^\phi = \zerobold$.  Then $\kcir [ \sbold^\phi] \leq 3$.
\end{lemma}
\begin{proof}
Suppose first that $\sss_1>0$.
Let $\phat = \phat[F; \zerobold]$ be as in \ref{Pexist}, where in this proof $F: = F^{\phicrit}_+(\sroot)$.  
A numerical calculation establishes that $\sss^{\phat}_2[F; \zerobold] \approx 2.414>\sroot$.  The result then follows from the flux monotonicity and \ref{Pexist}(i). 

Next suppose that $\sss_1 = 0$.
Let $\phi = \phicheck[ F; \zerobold]$ (recall \ref{Pexist2}). It follows from Lemma \ref{LHmono}(i) that $\sss^{\phicheck}_2[F; \zerobold]> \sss^{\phat}_2[ F; \zerobold] \approx 2.414>\sroot$, where $\sss^{\phat}_2[ F; \zerobold]$ is as in the above paragraph.  Using again the flux monotonicity this completes the proof. 
\end{proof}

\begin{prop}[$\ccat$-RLD existence and uniqueness]
\label{Pexistccat}
The following hold.
\begin{enumerate}[label=\emph{(\roman*)}]

\item (Two parallel circles) Given $\bsigmaunder = \xi \in \R$ satisfying $|\xi|<1$, there is a unique unit $\ccat$-RLD solution $\phat = \phat [ \bsigmaunder:2]$ satisfying $\kcir[\sbold^\phat] = 2$ and $\bsigmaunder^\phat = \xi$.  Moreover $\sss^\phat_1 \in (\sss^\phicrit_{\mathrm{root}}, \sroot)$. 
\item (Three parallel circles) There exists $\epsilonunder_1>0$ such that for all $\bsigmaunder = (\sigma, \xi) \in (- \epsilonunder_1, \infty)\times (- \epsilonunder_1, \epsilonunder_1)$ there is a unique unit $\ccat$-RLD solution $\phat =  \phat[ \bsigmaunder:3]$ satisfying $\kcir[\sbold^\phat] = 3$ and $\bsigmaunder^\phat = (\sigma, \xi)$. 
\end{enumerate}
\end{prop}

\begin{proof}
We first prove (i).  By \ref{LFmono}, \ref{Lphie}, and \ref{Lphibd}, the function $u: (\sss^\phicrit_{\mathrm{root}}, \sroot) \rightarrow (0, \infty)$ 
defined by $u = F^{\phicrit}_+/F^\phie_-$ is a strictly decreasing diffeomorphism.  
Therefore, there is a unique $\sss_1 \in (\sss^\phicrit_{\mathrm{root}}, \sroot)$ such that $u(\sss_1) = \frac{1+\xi}{1-\xi}$; equivalently
\begin{align}
\label{Exiccat}
\xi = \frac{F^{\phicrit}_+(\sss_1) - F^\phie_-(\sss_1)}{F^\phicrit_+(\sss_1) + F^\phie_-(\sss_1)}.
\end{align}

 By \ref{rsigbij}, \eqref{Exiccat}, and \ref{Pexist}, $\phat[2: \bsigmaunder ] := \phat[F^{\phie}_-(\sss_1); \bsigmaunder']|_{\cyl_I}$, where $\bsigmaunder' = (\, \zerobold, (\xi, 0, 0, \dots)\, )$ (recall \ref{dsigma}), is an  $\ccat$-RLD solution satisfying the conditions in (i).  The uniqueness is clear.

Proof of (ii): We first consider the case where $(\sigma, \xi) = (0, 0)$.  Note first that $\sroot \approx 1.1997$.  Given $F>0$, denote in this proof $\phi = \phicheck[F; \zerobold]$ (recall the notation of \ref{Rphicheck}), where we recall that $\phi$ satisfies $\phi = \phie + F \phio$ on $\cyl_{[0, \sss_1^\phi]}$ (so that in particular $F^\phi_+(0) = F$) and $F_-^\phi(\sss_1^\phi) = F$.  By numerical computations, we have the following:
\begin{align*}
\sss_1^{\phicheck[.9; \zerobold]} &\approx 1.109, \quad F^\phicrit_+( \sss_1^{\phicheck[.9; \zerobold]}) \approx 1.152,\\
\sss_1^{\phicheck[1; \zerobold]} &\approx 1.157, \quad F^\phicrit_+( \sss_1^{\phicheck[1; \zerobold]}) \approx .902.
\end{align*}
Differentiating the equation $F_-^\phi(\sss_1^\phi) = F$ implicitly with respect to $F$ and using \ref{LFmono}, we conclude that $\frac{\partial \sss^\phi_1}{\partial F}>0$.  In combination with the flux monotonicity \ref{LFmono} applied to $\phicrit$, this and the preceding numerical calculations show that there is a unique $F>0$ such that $F^\phi_+(\sss_1^\phi) = F^{\phicrit}_+(\sss_1^\phi)$.  This concludes the proof of (ii) in the case where $(\sigma, \xi) = (0, 0)$.  

The general case follows from the smooth dependence of $\sss^{\phicheck[F;  \bsigmaunder]}_1$ on $\bsigmaunder = (\sigma, \xi)$, the fact that $\frac{\partial \sss^{\phicheck[F; \bsigmaunder]}_1}{\partial F}>0$ and the flux monotonicity by taking $\epsilonunder_1>0$ small enough, in similar fashion to the case discussed above.
\end{proof}

\begin{figure}[h]
\centering
\begin{tikzpicture}[
  declare function={
    func(\s)= 
      (\s <-.747)*.948*( -.389*(1-\s*tanh(\s))-1.199*tanh(\s))+
   and(\s >=-.747, \s < .747)*( 1-\s*tanh(\s))+
    (\s >= .747)*.948*( -.389*(1-\s*tanh(\s))+1.199*tanh(\s))
 ;
  func2(\s)=      (\s <-1.15684)*.87526*( -.389*(1-\s*tanh(\s))-1.199*tanh(\s))+
         and(\s >=-1.15684, \s <=0)*( 1-\s*tanh(\s)-.966*tanh(\s))+
   and(\s >=0, \s < 1.15684)*( 1-\s*tanh(\s)+.966*tanh(\s))+
    (\s >= 1.15684)*.87526*( -.389*(1-\s*tanh(\s))+1.199*tanh(\s))
    ;
 }
]
\begin{axis}[
legend pos=outer north east,
  axis x line=middle, axis y line=left,
  ymin=.45, ymax=1.25, ytick={.4, .5, .6, .7, .8, .9, 1, 1.1, 1.2}, ylabel=$y$,
  xmin=-1.199, xmax=1.199, 
  xtick={-1.19, -.5, 0,  .5, 1.19},
  xticklabels={$-\sroot$, $-.5$, $0$, $.5$, $\sroot$},
   xlabel=$\sss$,
   ylabel=,
  samples=630
]

\addplot[black, thick, domain=-2:2, 
]{func(x)};
\addlegendentry{$\phat[\zerobold:2]$}
\addplot[black, thick, dashed, domain=-2:2, 
]{func2(x)};
\addlegendentry{$\phat[\zerobold:3]$}
\end{axis}
\end{tikzpicture} 
\caption{Profiles of the $\ccat$-RLD solutions $\phat[\zerobold:2]$ and $\phat[\zerobold:3]$. }
\end{figure}

\subsection*{LD solutions}
\nopagebreak

\begin{assumption}
\label{Amcc}
We assume $\kcir \in \{2, 3\}$, $m\in \N$ is as large as needed in terms of $\kcir$, $\mbold = (\pm m)$ when $\kcir =2$, and $\mbold = (\pm m, \pm m)$ or $\mbold = (\pm m, -2m)$ when $\kcir = 3$. 
\end{assumption}

Now that we are equipped with $\ccat$-RLD solutions, we can apply the analysis of Section \ref{S:LDs}---with only small, mostly notational modifications, to construct and estimate LD solutions corresponding to the RLD solutions just constructed in Proposition \ref{Pexistccat}.  For brevity, we remark only that the obvious modification of Lemma \ref{Lphiavg}---which constructs LD solutions from RLD solutions---holds because by Lemma \ref{Lkercc}, the boundary value problem \eqref{Ejacfb} has trivial kernel on $\ccat$.  The remaining estimates and decompositions of the corresponding LD solutions hold essentially exactly as in Section \ref{S:LDs}.

\subsection*{Initial surfaces}
\nopagebreak

To construct the initial surfaces and later also to perturb the initial surfaces, it will be useful to deform a surface which meets $\partial \B^3$ orthogonally without leaving the ball.  
To do this, we adopt an approach from \cite{kapwiygul} and introduce an auxiliary metric $g_A$ which makes the boundary $\Sph^2 = \partial \B^3$ totally geodesic.  
For numbers $\underline{r}, \overline{r}$ satisfying $0< \underline{r}< \overline{r} < 1$ which we will fix later, we define
\begin{align*}
g_A = \Omega^2 g, \quad
\text{where} \quad 
\Omega := \Psibold\left[ \underline{r}, \overline{r}; \dbold^g_0\right](1, 0) +
\frac{1}{\dbold^g_0} \Psibold\left[\underline{r}, \overline{r}; \dbold^g_0\right](0, 1).
\end{align*}

For the purposes of the following discussion, let $S$ be a properly embedded surface in $\B^3$; later we will take either $S = \ccat$ or $S$ to be an initial surface defined below. 

Note that the unit normal to $\partial S$ with respect to $g_A$ which points in the same direction as $\nu$ is $(\Omega \circ X)^{-1} \nu$.  Now denote $X: S\rightarrow \R^3$ the inclusion map.  Given $\uutilde \in C^2(S)$, we define the perturbation $X_{\uutilde}: \ccat \rightarrow \R^3$ by $\uutilde$ of $\ccat$ by
\begin{align*}
X_{\uutilde}(p) = \exp^{\B^3, g_A}_{X(p)} \left( \frac{\uutilde(p) \nu(p)}{ (\Omega \circ X)(p)}\right).
\end{align*}
For $\uutilde$ sufficiently small, $X_\uutilde$ is an immersion, and then we denote the corresponding Euclidean normal by $\nu_\uutilde$.

On a neighborhood of $\partial S$ in $S$ we define the function $\sigma := \dbold^{g}_{\partial S}$; near $\partial \ccat$, we can take $\sigma$ to be a coordinate on $S$ whose associated coordinate vector field $\partial_\sigma$ is then the inward pointing unit conormal to $S$ along $\partial S$.  We define also the boundary angle function $\Theta[\uutilde] :\partial S \rightarrow \R$ by 
\begin{align}
\label{Ethetacc}
\Theta[\uutilde]: = g( X_{\uutilde}, \nu_{\uutilde}). 
\end{align}

It is shown in \cite{kapwiygul} that the condition $\Theta[\uutilde] = 0$ is equivalent to the condition that $\uutilde$ satisfies the Neumann condition $\uutilde_{, \sigma} =0$.

Next, let $\widetilde{\Lcal}$ denote the linearized operator associated to the Euclidean mean curvature of $X_\uutilde$ computed at $\uutilde =0$.  The following lemma from \cite{kapwiygul} relates $\widetilde{\Lcal}$ to the usual Jacobi operator $\Lcal_{S}$ on $S$ and relates the equation $\partial_\sigma \uutilde|_{\partial S} = 0$ to a Robin boundary condition (recall \ref{Ejacfb}) for an associated function $u$.

\begin{lemma}[{\cite[Lemma 5.19]{kapwiygul}}]
\label{Lconvccat}
Given $\uutilde \in C^2(S)$, if we define $u \in C^2(S)$ by $u : = ( \Omega\circ X)^{-1} \uutilde$, then
\begin{enumerate}[label = \emph{(\roman*)} ]
\item $\widetilde{\Lcal} \uutilde = \Lcal_{S} u$.
\item $\left. \partial_\sigma \uutilde\right|_{\partial S } = ( \partial_\sigma +1) \left. u \right|_{\partial S}.$
\end{enumerate} 
\end{lemma}

\begin{definition}[The initial surfaces]
\label{Dinitccat}
Given $\varphi$, $\kappaunderbold$, and $\varphi^{gl}_\pm$ as in \ref{Dinit} we define $\widetilde{\varphi}^{gl}_\pm =(\Omega \circ X) \varphi^{gl}_\pm$ (recall \ref{Lconvccat}).  We then define the smooth initial surface $M = M[ \varphi, \kappaunderbold]$ in the same way as in \ref{Dinit}, except that we replace the graphs by 
\begin{align*}
\graph^{\R^3, g_A}_\Omega\big( \widetilde{\varphi}^{gl}_+\big) \quad \text{and} \quad
\graph^{\R^3, g_A}_\Omega\big( -\widetilde{\varphi}^{gl}_-\big).
\end{align*}
\end{definition}

\begin{convention}
\label{crunder}
We now fix $\underline{r}$ and $\overline{r}$ so that $\underline{r}$ is large enough that $\bigcup_{p\in L} D^{\Sigma, g}_p(4\delta'_p)$ is contained in the set where $g_A$ coincides with the Euclidean metric.  Note this is possible from \ref{Pexistccat}.
\end{convention}

\begin{lemma}
\label{LglobalHccat}
$\| H - J_M(w^+, w^-)\|_{0, \beta, \gamma-2, \gamma'-2; M} \le \tau_{\max}^{1+\alpha/3}$.
\end{lemma}
\begin{proof}

Because $g_A$ only differs from being Euclidean outside the ball $D^{\R^3, g}_0(\underline{r})$, by convention \ref{crunder}, and repeating the estimates in the proof of \ref{Dinit}, we need only estimate the Euclidean mean curvature portions of the graphs
\begin{align*}
\graph^{\R^3, g_A}_\Omega\big( \widetilde{\varphi}^{gl}_+\big) \quad \text{and} \quad
\graph^{\R^3, g_A}_\Omega\big( -\widetilde{\varphi}^{gl}_-\big)
\end{align*}
outside this ball.

By using Lemma \ref{Lconvccat} and arguing as in the proof of \cite[Lemma 7.8]{kapwiygul} we can bound these terms and the proof is complete.
\end{proof}

We conclude this subsection with a discussion of perturbations of the initial surfaces.  If $\phi \in C^1(M)$ is appropriately small, 
we denote $M_\phi = \graph_M^{\R^3, g_A}( \widetilde{\phi})$, where $\widetilde{\phi} = (\Omega \circ X) \phi$, and here $X: M \rightarrow \R^3$ is the inclusion map.  
We have the following estimate (recall \ref{Lquad} on the nonlinear terms of the mean curvature of $M_\phi$:

\begin{lemma}
\label{Lquadcc}
If $M$ is as in \ref{Dinitccat} and 
$\phi\in C^{2,\beta}(M)$ 
satisfies $\|\phi\|_{2,\beta,\gamma, \gamma';M} \, \le \, \tau_{\max}^{1+\alpha/4} $,  
then $M_\phi$ is well defined as above, 
is embedded, 
and if 
$H_\phi$ is the Euclidean mean curvature of $M_\phi$ pulled back to $M$
and $H$ is the mean curvature of $M$, then we have 
$$
\|\, H_\phi-\, H - \Lcal_M \phi \, \|_{0,\beta,\gamma-2, \gamma'-2;M}
\, \le \, C \, \tau_{\min}^{-\alpha/2}
\|\, \phi\, \|_{2,\beta,\gamma, \gamma';M}^2.
$$
\end{lemma}
\begin{proof}
Although $M_\phi$ is defined as the normal graph of $\widetilde{\phi}$ with respect to the auxiliary metric $g_A$, Lemma \ref{Lconvccat} shows that the linear terms are given by $\widetilde{\Lcal} \, \widetilde{\phi} = \Lcal_M \phi$.  The proof then essentially the same as that of \ref{Lquadcc} (see also \cite[Lemma 7.8]{kapwiygul}), so we omit the details. 
\end{proof}

\subsection*{The linearized equation on the initial surfaces}
\nopagebreak

Because the linearized equation on the initial surfaces is a boundary value problem, 
we need to modify the definition of $\RMa          $ in Definition \ref{DRMappr}: 
we will define 
$\RMa          (E, E^\partial) = (u_1, w^+_{E, 1}, w^{-}_{E, 1}, E_1, E^\partial_1)$ 
for given $(E, E^\partial) \in C^{0, \beta}(M) \times C^{0, \beta}(\partial M)$,  
where $u_1$ will be an approximate solution to the linearized equation modulo $\skernel[L]$, 
that is the boundary value problem (recall \ref{NT}\ref{N:A})  
\begin{equation*}
\label{ELcalccat}
\begin{cases}
\Lcal_M u=E+ J_M(w^+_E, w^-_E) \\
(\partial_\sigma+1)|_{\partial M} u = E^\partial 
\end{cases},
\quad \text{where} \quad
   w^\pm_E
 \in \skernel[L],
\end{equation*}
$w^\pm_{E, 1}$ are the $\skernel[L]$ terms,
and $E_1, E^\partial_1$ are the approximation errors defined by
\begin{equation}
\label{EEone2}
\begin{aligned}
E_1&:= \Lcal_M u_1 - E - J_M(w^{+}_{E, 1}, w^-_{E, 1}), \\
E^\partial_1 &: = (\partial_\sigma+1)|_{\partial M} u_1 - E^\partial.
\end{aligned}
\end{equation}

Before proceeding with the definition, we need to modify the definition of $J_M$ from \ref{Npm} and define an analogous operator $J_{\partial M}$ for the boundary. 

\begin{notation}
\label{Npm2}
If $f^+$ and $f^-$ are functions supported on $\Stildep$ (recall \eqref{EStildep}), we define $J_M(f^+, f^-)$ to be the function on $M$ supported on $(\left.\Pi^{g_A}_{\ccat}\right|_M)^{-1}\Stildep$ defined by $f^+\circ \Pi^{g_A}_{\ccat}$ on $\graph^{\R^3, g_A}_\Omega\big( \widetilde{\varphi}^{gl}_+\big)$
 and by $f^- \circ \Pi^{g_A}_{\ccat}$ on the $\graph^{\R^3, g_A}_\Omega\big( -\widetilde{\varphi}^{gl}_-\big)$.
 
If $f^{\partial}_+$ and $f^\partial_-$ are functions defined on $\partial \ccat$, we define $J_{\partial M}(f^\partial_+, f^\partial_-)$ 
to be the function defined on $\partial M$ defined by $f^{\partial}_+\circ \Pi^{g_A}_{\ccat}$ on 
$\graph_{\partial \ccat}^{\R^3, g_A}(\widetilde{\varphi}^{gl}_+)$ and by $f^\partial_-\circ \Pi^{g_A}_{\ccat}$ on $\graph_{\partial \ccat}^{\R^3, g_A}( - \widetilde{\varphi}^{gl}_-)$.
\qed 
\end{notation}

We follow the discussion before Definition \ref{DRMappr} with the following small modifications: just after the definition of $E'_{\pm}$ in \eqref{Edecom}, we define $E^\partial_{\pm} \in C^{0, \beta}(\partial \ccat)$ by requesting that
\begin{align}
\label{Ejpm}
J_{\partial M}(E^\partial_+, E^\partial_-) = E^\partial.
\end{align}

We then replace the equation  \eqref{Eup} defining $u'_{\pm} \in C^{2, \beta}(\Sigma)$ and $w^{\pm}_{E, 1}$ with the equation
\begin{align}
\label{Eupcc}
\begin{cases}
\Lcal_\Sigma u'_{\pm} = E'_{\pm} + w^{\pm}_{E,1}\\
(\partial_\sigma + 1)|_{\partial \ccat} u'_{\pm} = E^{\partial}_{\pm}
\end{cases}
\quad
\text{and} 
\quad
\forall p\in L  \quad \Ecalunder_p u'_{\pm}=0.
\end{align} 

We now define $\RMa          = (u_1, w^+_{E, 1}, w^{-}_{E, 1}, E_1, E^\partial_1)$,
where $u_1, w^+_{E, 1}$, and $w^{-}_{E, 1}$ are defined as in \ref{DRMappr}, and $E_1, E^\partial_1$ are defined as in \eqref{EEone2}.

We are now ready to state and prove an appropriately modified version of \ref{Plinear2} in the present setting.  Note that in the statement below we only need to solve with homogeneous boundary value data because of the way we perturb using the auxiliary metric. 

\begin{prop}
\label{Plinearcc}
Recall that we assume that \ref{con:alpha},  \ref{cLker}, \ref{aK}, \ref{con:one}, and \ref{con:b} hold. 
A linear map 
$
\Rcal_M: C^{0,\beta}(M) \to C^{2,\beta}(M) \times \skernel[L] \times \skernel[L] 
$ 
can be defined then 
by 
$$
\Rcal_M E
:=
(u,w^+_E,  w^-_E) 
:=
\sum_{n=1}^\infty(u_n, w^+_{E,n}, w^-_{E,n})
\in C^{2,\beta}(M) \times \skernel[L]\times\skernel[L]
$$
for 
$E\in C^{0,\beta}(M)$, 
where 
$\{(u_n,w^+_{E,n}, w^-_{E,n},E_n, E^\partial_{n})\}_{n\in \N}$
is defined inductively for $n\in \N$ by 
$$
(u_n,w^+_{E,n}, w^-_{E,n},E_n, E^\partial_n) := - \RMa           (E_{n-1}, E^\partial_{n-1}) 
\qquad\quad
E_0:=-E, \quad E^\partial_0 = 0. 
$$
Moreover the following hold.
\begin{enumerate}[label=\emph{(\roman*)}]
\item $\Lcal_M u = E + J_M(w^+_E, w^-_E)$ and $(\partial_\sigma + 1)|_{\partial M} u = 0$.
\item $ \| u \|_{2, \beta, \gamma, \gamma'; M} \le C(b) \delta^{-4-2\beta}_{\min}|\log \tau_{\min}|  \| E\|_{0, \beta, \gamma-2, \gamma'-2; M}$.
\item $\| w^{\pm}_E : C^{0, \beta}(\Sigma, g)\| \le C \delta^{\gamma-4-2\beta}_{\min}
\| E\|_{0, \beta, \gamma-2, \gamma'-2; M}$.
\end{enumerate}
\end{prop}
\begin{proof}
We need only check that $(\partial_\sigma + 1)|_{\partial M} u = 0$.  Using \eqref{EEone2} and \eqref{Ejpm} and pulling back to $\partial \ccat$ we have 
\begin{align}
\label{Ee1sm}
 E^\partial_{1\pm} = ( (\Pi^{g_A}_{\partial \ccat})^* \partial_\sigma +1) u'_{1\pm} =  ( (\Pi^{g_A}_{\partial \ccat})^* \partial_\sigma - \partial_\sigma )|_{\partial \ccat} u'_{1\pm},
\end{align} 
where the second equality follows by
combining with \eqref{Eupcc} and using that $E^\partial_0 = 0$.

It follows by a straightforward inductive argument that 
\begin{align}
\label{Eeism}
( ( \Pi^{g_A}_{\partial \ccat})^* \partial_\sigma +1) u'_{i\pm} = E^\partial_{i\pm} - E^\partial_{i-1\pm} 
\quad \text{and} \quad
E^\partial_{i\pm} = (( \Pi^{g_A}_{\partial \ccat})^* \partial_\sigma- \partial_\sigma) u'_{i\pm}.
\end{align}
We have then for any $n\in \N$
\begin{align*}
(( \Pi^{g_A}_{\partial \ccat})^* \partial_\sigma +1) \sum_{i=1}^n u'_{i\pm} = E^\partial_{n\pm}.
\end{align*}
Estimating the smallness of $E^\partial_{1\pm}$ using \eqref{Ee1sm} and inductively estimating $E^\partial_{n\pm}$ using \eqref{Eeism}, we conclude that for any $n\in \N$, 
\[ \bigg\|  (\partial_\sigma +1)|_{\partial M} \sum_{i=1}^n u_i : C^{1, \beta}(\partial M, g)\bigg\| < 2^{-n},\]
and from this we conclude that $(\partial_\sigma + 1)|_{\partial M} u = 0$.
\end{proof}

\subsection*{The main theorem}
\nopagebreak

\begin{theorem}
\label{Tmainccat}
Let $\kcir \in \{2, 3\}$.  For all $m\in \N$ sufficiently large and $\mbold =\pm m$ in the case $\kcir = 2$, and $\mbold = (\pm m, -2m)$ or $\mbold = (\pm m, \pm m)$ in the case $\kcir = 3$, 
there is a $\gcyl_m$-invariant doubling of $\ccat$ as a free boundary minimal surface in $\B^3$ with four boundary components.  It contains one catenoidal bridge close to each singularity of one of a family of $\group_m$-invariant LD solutions as in Theorem \ref{Trldldgen} whose singularities concentrate on $\kcir$ parallel circles, with the number of singularities and their alignment at each circle prescribed by $\mbold$.  Moreover, as $m\rightarrow \infty$ with fixed $\kcir$, the corresponding doublings converge in the appropriate sense to $\ccat$ covered twice. 
\end{theorem}

\begin{proof}
The structure of the proof is the same as that of \ref{Tmainsph}, except that Theorem \ref{Ttheory} cannot be applied directly because of the boundary and the free boundary condition.  However, we can still carry out steps (1)-(6) in the proof of \ref{Ttheory}, where we use \ref{Plinearcc} instead of \ref{Plinear2} and \ref{Lquadcc} to estimate the quadratic terms instead of \ref{Lquad}.  We then conclude a fixed point of the map $\Jcal$ in \eqref{Tfp}.  
It follows as in \ref{Ttheory} that $(M\llbracket \zetaboldhatunder \rrbracket)_{\upphihat}$ is smooth and minimal; 
moreover  $(M\llbracket \zetaboldhatunder \rrbracket)_{\upphihat}$ intersects $\partial \B^3$ orthogonally 
because $\upphihat$ satisfies the Robin boundary condition $(\partial_\sigma+1)|_{M} \upphihat = 0$ (recall \ref{Plinearcc}(i), 
the discussion just below \eqref{Ethetacc}, and \ref{Lconvccat}).  
\end{proof}


\appendices
\section*{Appendices}
\section{Fermi coordinates}
\label{Sfermi}

In this appendix we define a modification of the standard exponential map we call \emph{Fermi exponential map},  
and we collect some facts about the corresponding Fermi coordinates in Lemma \ref{Lgauss}, 
most of which can be found for example in \cite{Gray}.

\begin{definition}[Fermi exponential map]
\label{dexp}
We assume given 
a hypersurface $\Sigma^n$ in a Riemannian manifold $(N^{n+1}, g)$ and a unit normal $\nu_p \in T_p N$ at some $p\in\Sigma$. 
For $\delta>0$ we define 
$$
\Dhat^{\Sigma, N, g}_p (\delta) := \{v+  z\nu_p \, : \, v\in D_0^{T_p\Sigma, \left. g\right|_p} (\delta) \subset T_p\Sigma, \,   z\in (-\delta,\delta) \, \} \subset T_p N. 
$$
For small enough $\delta$, the map $\exp^{\Sigma,N,g}_p : \Dhat^{\Sigma, N, g}_p (\delta) \to N$, defined by 
$$
\exp^{\Sigma,N,g}_p (v+  z\nu_p ) := \exp^{N,g}_q(  z\nu_v) 
\quad \forall \, v+  z\nu_p \in \Dhat^{\Sigma, N, g}_p (\delta) 
\quad \text{with} \quad v\in T_p\Sigma,  
$$  
where 
$q:=\exp^{\Sigma,g}_p(v)$ 
and $\nu_v\in T_qN$ is the unit normal to $\Sigma$ at $q$ 
pointing to the same side of $D_p^{\Sigma,\delta}(\delta)$ (which is two-sided) as $\nu_p$, 
is a diffeomorphism onto its image 
which we will denote by 
$D^{\Sigma, N, g}_p (\delta)\subset N$. 
We define the \emph{injectivity radius $\inj^{\Sigma, N, g}_p$ of $(\Sigma,N,g)$ at $p$} to be the supremum of such $\delta$'s.  
Finally when $\delta< \inj^{\Sigma, N, g}_p$ we define on $D^{\Sigma, N, g}_p (\delta)$ the following.  
\begin{enumerate}[label=\emph{(\alph*)}]
\item 
$\Pi_\Sigma : D^{\Sigma, N, g}_p (\delta)  \to \Sigma\cap D^{\Sigma, N, g}_p (\delta)$   
is the nearest point projection in $(D^{\Sigma, N, g}_p (\delta) \, , g)$.  
Alternatively $\Pi_\Sigma$ corresponds through $\exp^{\Sigma,N,g}_p$ to orthogonal projection to $T_p\Sigma$ in $(T_p N, \left. g \right|_p)$.  
\item 
$\zz: D^{\Sigma, N, g}_p (\delta)  \to (-\delta,\delta)$ is the signed distance from 
$\Sigma\cap D^{\Sigma, N, g}_p (\delta)$   
in $(D^{\Sigma, N, g}_p (\delta) \, , g)$.  
Alternatively $\zz \circ \exp^{\Sigma,N,g}_p (v) \, \nu_p$ is the orthogonal projection of $v$ to $\left< \nu_p \right>$ in $(T_p N, \left. g \right|_p)$ 
$\forall v\in \Dhat^{\Sigma, N, g}_p (\delta)$.  
\item 
A foliation by the level sets $\Sigma_z :=\zz^{-1}(z)\subset D^{\Sigma, N, g}_p (\delta) $ for $z\in(-\delta,\delta)$.  
\item 
Tensor fields 
$g^{\Sigma,\zz}$, $A^{\Sigma,\zz}$ and $B^{\Sigma,\zz}$ 
by requesting that on each level set $\Sigma_z$ 
they are equal to 
the first and second fundamental forms and Weingarten map of $\Sigma_z$ respectively. 
\end{enumerate} 
\end{definition}

\begin{remark}
\label{R:complete}
Note if $\Sigma$ and $N$ are both complete with respect to $g$ in \ref{dexp}  
and $\Sigma$ is two-sided,  
then $\exp^{\Sigma,N,g}_p $ is well defined on $T_pN$ by the same definition, 
even in the case $\inj^{\Sigma, N, g}_p<\infty$.   
\qed 
\end{remark}

\begin{example}[Clifford torus,  cf. {\cite[p. 263-264]{kapouleas:clifford}}]
\label{exClifford}
We identify $\R^4$ with $\C^2$ 
and let $N:=\Sph^3 
\subset \C^2$,  
$\T := \{ (z_1, z_2) \in \C^2: |z_1| = |z_2| = 1/\sqrt{2}\} \subset \Sph^3$ be the Clifford torus, 
and $p = (1/\sqrt{2}, 1/\sqrt{2}) \in \T$. 
There is then a linear isomorphism $\breve{E}:\R^3\to T_p\Sph^3$ such that the map 
$\tildeE := \exp^{\T,\Sph^3,g}_p \circ \breve{E} :\R^3\to \Sph^3$ 
(called $\Phi$ in \cite{kapouleas:clifford}) satisfies 
$$
\begin{gathered}
\tildeE(\xx, \yy, \zz) = 
\left( \sin(\zz + \textstyle{\frac{\pi}{4}}) e^{\sqrt{2}  \xx i } \, , \, \cos(\zz+ \textstyle{\frac{\pi}{4}}) e^{\sqrt{2} \yy i } \right) \in \Sph^3 \subset \C^2
\\ 
\text{and } 
\tildeE^* g = (1+\sin 2\zz) d\xx^2 + (1- \sin 2\zz) d\yy^2 + d\zz^2.
\end{gathered}
$$
\end{example}

\begin{example}[Cylindrical Fermi coordinates about $\Spheq$ in $\Sph^3$]
\label{exSph}
Let $N:=\Sph^3 \subset \R^4$,  
$\Spheq$ be the equatorial two-sphere in the round three-sphere $\Sph^3$,  
and $p = (0, 0, 1, 0)\in \Spheq$.  
There is then a ``spherical coordinates parametrization'' $\breve{E}:\R^3\to T_p\Sph^3$ such that the map 
$\tildeE := \exp^{\Spheq,\Sph^3,g}_p \circ \breve{E} :\R^3\to \Sph^3$ 
(which is equivalent to the map $\Theta$ in  \cite[(2.2)]{kap})  
satisfies 
\begin{align*}
\tildeE( \rr, \theta, \zz) = & (  \sin \rr \cos \theta \cos \zz,  \sin \rr \sin \theta \cos\zz, \cos \rr  \cos \zz, \sin \zz), 
\\ 
\tildeE^* g = & \cos^2 \zz\left( d \rr^2 + \sin^2 \rr d \theta^2\right) + d\zz^2,
 \end{align*}
 and the only nonvanishing Christoffel symbols in the $(\rr, \theta, \zz)$ coordinates are  
\begin{equation*}
\begin{gathered}
\Gamma_{\rr \zz}^{\rr}  = \Gamma_{\zz \rr}^{\rr}= \Gamma_{\theta \zz}^{\theta} = \Gamma_{\zz \theta}^{\theta} =   - \tan \zz, \quad
\Gamma_{\theta \theta}^\rr = - \sin \rr \cos \rr, \\
\Gamma_{\rr\theta}^{\theta} = \Gamma_{\theta \rr}^{\theta} = \cot \rr, \quad
\Gamma_{\rr\rr}^{\zz} = \cos \zz \sin \zz, \quad
\Gamma_{\theta \theta}^{\zz} = \sin^2 \rr \sin \zz \cos \zz .
\end{gathered}
\vspace{-.27in}
\end{equation*} 
\qed
\end{example}

\begin{lemma}[Properties of Fermi coordinates]
\label{Lgauss}
Assuming $\delta< \inj^{\Sigma, N, g}_p$ as in \ref{dexp}, and with the same notation, the following hold on 
$D^{\Sigma, N, g}_p (\delta)\subset N$. 
\begin{enumerate}[label=\emph{(\roman*)}]
\item $g_{\zz \zz} =1$ and $\nabla_{\partial_\zz}  \partial_\zz = 0$.
\item $g = \gpar+ d\zz^2$. 
\item $\Lie_{\partial_\zz} \gpar = - 2 A^{\Sigma,\zz}$.
\item $\Lie_{\partial_\zz} B^{\Sigma,\zz} = B^{\Sigma,\zz}\circ B^{\Sigma,\zz} - \Rend_{\partial_\zz}$ and 
$\Lie_{\partial_\zz} A^{\Sigma,\zz} = - \left( A^{\Sigma,\zz}* A^{\Sigma,\zz} + \Rm_{{\partial_\zz}}\right)$.
\item \label{g-v} 
$\gpar = \Pi^*_\Sigma g^\Sigma - 2 \zz \Pi^*_\Sigma A^\Sigma + \zz^2 \Pi^*_\Sigma\left( A^\Sigma * A^{\Sigma} + \Rm^\Sigma_\nu \right) + \zz^3 h^{\mathrm{err}}$, 
where $h^{\mathrm{err}}$ is a smooth symmetric two-tensor on 
$D^{\Sigma, N, g}_p (\delta)\subset N$. 
\end{enumerate}
\end{lemma}

\begin{proof}
(i) follows immediately from Definition \ref{dexp}.  
Next we compute  
\begin{align}
\label{E2ff}
\left( \Lie_{\partial_\zz} g\right)_{ij} = 
g_{ij, \zz} =  \langle \nabla_{\partial_i} {\partial_\zz}, \partial_j\rangle+ \langle \partial_{i},  \nabla_{\partial_j}{\partial_\zz}\rangle,
\end{align}
where the indices $i,j$ refer to the $\Sigma$ exponential coordinates. 
With (i), this implies $g_{i\zz, \zz} = \frac{1}{2}g_{\zz\zz, i} = 0$ and (ii) follows, since $g_{i \zz} =  \delta_{i\zz}$ on $\Sigma$.  (iii) follows from \eqref{E2ff} and (ii). Next note that any $X$ satisfying $[X, {\partial_\zz}] =0$ satisfies $\nabla_{\partial_\zz} X = \nabla_X {\partial_\zz} = -BX$; then 
\begin{align*}
\left(\nabla_{\partial_\zz} B\right) X = 
-\nabla_{\partial_\zz} \nabla_X {\partial_\zz}  - B \left(\nabla_{\partial_\zz} X\right) = - \Rend({\partial_\zz}, X){\partial_\zz} + B^2 X.
\end{align*}
The first equation of (iv) follows after noting that 
$\Lie_{\partial_\zz} B^{\Sigma,\zz}= \nabla_{\partial_\zz} B - (\nabla {\partial_\zz})\circ B^{\Sigma,\zz} + B^{\Sigma,\zz} \circ \nabla {\partial_\zz}$  
and the second equation 
follows from the first by lowering an index and using (iii). 
(v) follows via the preceding parts and Taylor's theorem.
\end{proof}

\begin{remark}
\label{rFermiH}
Straightforward calculations using \ref{Lgauss} recover the usual formulas for the first variations 
of volume $dV^{\Sigma,\zz}$ and mean curvature $H^{\Sigma,\zz}$ along  
the parallel surfaces $\Sigma_\zz$:
\begin{equation*} 
\begin{aligned}
\Lie_{\partial_\zz} dV^{\Sigma,\zz}  &= \left(\ddiv_{\Sigma_\zz} {\partial_\zz}\right) dV^{\Sigma,\zz}  
= H^{\Sigma,\zz} dV^{\Sigma,\zz},   
\\
\qquad \qquad 
\Lie_{{\partial_\zz}} H^{\Sigma,\zz} &= \Lie_{\partial_\zz} \tr B^{\Sigma,\zz} 
= \tr \left( \Lie_{\partial_\zz} B^{\Sigma,\zz}\right) = | B^{\Sigma,\zz} |^2 + \Ric({\partial_\zz}, {\partial_\zz}). 
\qquad \!\!\! \square
\end{aligned}
\end{equation*} 
\end{remark}

\begin{lemma}
\label{Lda}
Let $\Sigma$ be a two-sided hypersurface in a Riemannian manifold $(N,g)$ and $\Omega \subset \Sigma$ a precompact domain.  For $u\in C^1(\Omega)$ with $\| u : C^1(\Omega, g)\|$ small, the pullback of the area form $d\sigma_u$ on $\text{\emph{Graph}}^{N, g}_\Omega(u)$ by $X^{N, g}_{\Omega, u} : \Omega \rightarrow \text{\emph{Graph}}^{N, g}_\Omega(u)$ 
(recall \ref{NT}\ref{dgraph}) satisfies
\begin{equation*}
\big((X^{N, g}_{\Omega, u})^* d\sigma_u\big) \: = 
\:\Big(\,  1 - u H+  {\frac{1}{2}} |\nabla u|^2 - {\frac{u^2}{2}} \big(|A^\Sigma|^2 + \Ric(\nu, \nu)-H^2\big) 
+O \big( |u|^3 + |u||du|^2_{g} \big) \:\Big)\: d \sigma ,
\end{equation*}
where $d\sigma$ is the Riemannian area form on $\Sigma$. 
\end{lemma}

\begin{proof}
From the definitions and \ref{Lgauss}\ref{g-v}, we have 
\begin{align}
\label{Egpull}
(X^{N,g}_{\Omega, u})^* g = g^\Sigma -2 u A^\Sigma+ du \otimes du + u^2 (A^\Sigma * A^\Sigma+ \Rm_\nu )+ O(|u|^3).
\end{align}
For any square matrix $M$, recall that 
\begin{align*}
\det ( I + M) = 1 + \tr M + \frac{1}{2}( (\tr M)^2 - \tr M^2) + O( |M|^3),
\end{align*}
where $I$ is the identity matrix.  From this and \eqref{Egpull}, it follows that 
\begin{equation*}
\det \big((X^{N,g}_{\Omega, u})^* g\big)
= \det g^\Sigma \: \Big( 1-2u H+  |\nabla u|^2 
- u^2(|A^\Sigma|^2+\Ric(\nu, \nu)-2H^2) + O( |u|^3 + |u||du|^2_g)\Big). 
\end{equation*}
By taking square roots and using that $\sqrt{1+x} = 1+\frac{1}{2}x-\frac{1}{8}x^2 + O(x^3)$ for $x$ near zero, the conclusion now follows. 
\end{proof}

\begin{lemma}
\label{LAomega}
Let $\Sigma, N, g, \Omega$, and $u$ be as in \ref{Lda}.  If moreover $u \in C^2(\overline{\Omega})$ and $\partial \Omega$ is smooth, then
\begin{multline*}
|\text{\emph{Graph}}^{N,g}_{\Omega}(u)| = 
|\Omega| 
- \int_{\Omega} u H d\sigma
- \frac{1}{2}\int_{\Omega} u \Lcal_\Sigma u \, d\sigma 
\\
+ \frac{1}{2} \int_{\partial \Omega} u \frac{\partial u}{\partial \eta}ds 
+ \frac{1}{2} \int_{\Omega} u^2 H^2 d\sigma 
+\int_{\Omega}O(|u|^3+|u||du|^2_g) d\sigma.
\end{multline*}
\end{lemma}

\begin{proof}
This follows from integrating $(X^{N,g}_{\Omega, u})^* d\sigma_u$ over $\Omega$ via \ref{Lda} and integrating $\int_{\Omega}\frac{1}{2}|\nabla u|^2 d\sigma$ by parts. 
\end{proof}


\section{Perturbations of graphs}
\label{A:tilt}

\begin{definition}[Vector fields and sliding]
\label{Dvslide}
We assume given a Riemannian manifold $(\Sigma, g)$, an open set $\Omegain \subset \Sigma$, 
and a vector field $V$ defined on a domain containing $\Omegain$ satisfying $V_p \in \dom(\exp^{\Sigma, g})$ for each $p \in \Omegain$.
We define then $\DDD_V=\DDD^{\Sigma, g}_{V, \Omega} : \Omegain \rightarrow \Sigma$ by $\DDD_V := \exp^{\Sigma, g} \circ V|_{\Omega} = \pert^{\Sigma, g}_V I^\Sigma_{\Omegain}$, 
where $I^{\Sigma}_{\Omegain} : \Omegain \rightarrow \Sigma$ is the inclusion (recall \ref{NT}\ref{Dpertimm}).  
We also define $\Omegatilde_V : = \Omega \cap \DDD_V(\Omega)$. 
\end{definition}

\begin{lemma}
\label{LVmvt}
If $\Omegain, V$, and $\DDD_V$ are as in \ref{Dvslide} and $f \in C^\infty(\Omegatilde_V)$, then 
\begin{align*}
	\| f \circ \DDD_V - f : C^k(\Omegatilde_V)\|  \leq
	C(k) \| f: C^{k+1}(\Omegain, g)\| \| V : C^k( \Omegain, g)\|.
\end{align*}
If $\DDD_V$ is moreover a diffeomorphism and $\| V : C^k(\Omega, g)\|$ is small enough, then additionally
\begin{align*}
	\| f \circ \DDD^{-1}_V - f : C^k(\Omegatilde_V)\|  \leq
	C(k) \| f: C^{k+1}(\Omegain, g)\| \| V : C^k( \Omegain, g)\|.
\end{align*}
\end{lemma}
\begin{proof}
This is a consequence of the mean value theorem and a straightforward induction argument.
\end{proof}

\begin{assumption}
\label{Aufermi}
We now assume given the following:
\begin{enumerate}[label=(\roman*)]
\item A two-sided hypersurface $\Sigma$ with a choice of a unit normal $\nu$ in a Riemannian manifold $(N, g)$.
\item A domain $\Omega : = D^{\Sigma, g}_{p}(\delta) \subset \Sigma$ for some $p\in \Sigma$ and $\delta>0$ satisfying $2\delta < \inj^{\Sigma, N, g}_p$ (recall \ref{dexp}).
\item A function $u \in C^\infty(\Omega)$ with $\| u : C^k(\Omegain,g)\|$ as small as needed in terms of $\delta$. 
\item A vector field $\Vtilde$ along $\graph^{N,g}_{\Omegain}(u)$ with $\| (X^{N, g}_{\Omegain, u})^*\Vtilde : C^k(\Omegain, g)\|$ as small as needed in terms of $\delta$. 
\end{enumerate}
\end{assumption}

\begin{definition}
\label{dVtildedecomp}
We define a decomposition $\Vtilde = \Vtilde^\top + \Vtilde^\perp$ by requesting that $\Vtilde^\perp = \langle \Vtilde, \partial_z\rangle \partial_z$, where $z$ is the signed distance from $\Sigma \cap D^{\Sigma, N, g}_p(2\delta)$ in $D^{\Sigma, N, g}_{p}(2\delta)$ as in \ref{dexp}. 
\end{definition}

\begin{lemma}
\label{LVestf}
Given $u$ and $\Vtilde$ as in \ref{Aufermi}, there is a vector field $V$ on $\Omega$ uniquely determined by $\DDD_V = \Pi_\Sigma \circ \pert_{\Vtilde} X^{N,g}_{\Omegain, u}$ and a function $w : \DDD_V(\Omegain) \rightarrow \R$ uniquely determined by $\pert^{N,g}_{\Vtilde} X_{\Omegain, u}^{N, g} = X^{N,g}_{\DDD_V(\Omegain), w} \circ \DDD_V$.  In other words, the diagram
 \begin{equation}
\label{ediag3}
\begin{tikzcd}
N \arrow[r, rightarrow, "\exp^{N, g}\circ\Vtilde"] 
               & N \arrow[d, rightarrow, "\Pi_\Sigma"']\\
\Omegain \arrow[r, rightarrow, "{\DDD_V}"] \arrow[u, rightarrow, "X^{N, g}_{\Omegain, u}"] & \DDD_V(\Omegain) \arrow[u, bend right=25, pos=.45,  "X^{N, g}_{\DDD_V(\Omegain), w}"']
\end{tikzcd}
\end{equation} 
commutes. 
Moreover, the following hold.
\begin{enumerate}[label=\emph{(\roman*)}]
\item $ \| V : C^k(\Omegain, g) \| \leq C(k) \| (X^{N,g}_{\Omega, u})^*\Vtilde^\top : C^k(\Omega, g)\|$. 
\item $\| w - (X^{N,g}_{\Omegain, u})^*\langle \Vtilde , \partial_z\rangle - u : C^k(\Omegatilde_V )\|\lem C(k) \big\| (X^{N, g}_{\Omegain, u})^* \Vtilde^\top : C^k(\Omegain, g) \big\| $ 
\\  
$\phantom1$ \hfill 
$ 
\cdot \big( \| u : C^{k+1}(\Omegain, g)\|+ \| (X^{N, g}_{\Omegain, u})^* \Vtilde : C^{k+1}(\Omegain, g)\| \big)$. 
\end{enumerate}
\end{lemma}
\begin{proof}
By the smallness assumptions in \ref{Aufermi}, we may assume that $\pert_{\Vtilde} X^{N, g}_{\Omega, u}(\Omega) \subset D^{\Sigma, N, g}_{p}(2\delta)$.  
Since $\Pi_\Sigma : D^{\Sigma, N, g}_{p}(2\delta) \rightarrow D^{\Sigma,g}_{p}(2\delta)$ is smooth and $\exp^{\Sigma, g}$ is invertible on $D^{\Sigma, g}_p(2\delta)$, 
the stated condition on $V$ is equivalent to 
\begin{align}
\label{EVdef}
V = (\exp^{\Sigma, g})^{-1}\circ \Pi_\Sigma \circ \pert^{N, g}_{\Vtilde}X^{N, g}_{\Omega, u}.
\end{align}
The estimate (i) follows from \eqref{EVdef} and the fact that the differential of the exponential map at $0$ is the identity.  Now combining (i) with the smallness assumption on $\Vtilde$ in \ref{Aufermi}(iv) and the implicit function theorem, it follows that $\DDD_V$ is a diffeomorphism, so in particular $w$ is uniquely determined by the equation $X^{N, g}_{\DDD_V(\Omega), w} = \pert_{\Vtilde}X^{N, g}_{\Omega, u} \circ \DDD^{-1}_V$.  
 
Next, note from \eqref{ediag3} that $w \circ \DDD_V = z \circ \exp^{N, g} \circ \Vtilde \circ X^{N, g}_{\Omegain, u}$, where $z$ is the signed distance function from $\Sigma$ as in \ref{dexp}(b). 
From this and the fact that the differential of the exponential map at zero is the identity, it follows that
\begin{equation} 
\label{Ewestm}
\| w\circ \DDD_V - (X^{N,g}_{\Omegain, u})^*\langle \Vtilde, \partial_z\rangle - u : C^k(\Omegatilde )\| 
\lem 
C(k)  \| (X^{N, g}_{\Omegain, u})^* \Vtilde^\top : C^k(\Omegain, g)\| \| (X^{N, g}_{\Omegain, u})^* \Vtilde : C^k(\Omegain, g)\|.
\end{equation} 
The conclusion follows from this by using Lemma \ref{LVmvt} in conjunction with item (i). 
\end{proof}

\begin{corollary}[Graphs over graphs]
\label{Cwgraph}
Let $\Sigma, N, g, \delta$, and $u$ be as in \ref{Aufermi}.  
Fix a function $v \in C^\infty(\Omega)$, and define a vector field $\Vtilde$ along $\graph^{N,g}_{\Omega}(u)$ by 
$\Vtilde = ( v \circ X^{N,g}_{\Omega, u}) \nu_{u}$, where $\nu_u$ is the unit normal to $\graph^{N,g}_{\Omega}(u)$ 
which has positive inner product with $\partial_z$.  
Then \ref{Aufermi}(iv) holds provided $\| v : C^k(\Omega, g)\|$ is small enough in terms of $\delta$.  
Moreover, the function $w$ in \ref{LVestf} satisfies
$$
\|w - v - u : C^k(\Omegatilde_V) \| \lem C(k) \: \| u : C^{k+1}(\Omega, g) \|^2 \: \| v: C^k(\Omega,g)\|.
$$
\end{corollary}

\begin{proof}
First observe that in Fermi coordinates
\begin{align*} 
\nu_{u} = \frac{ \partial_z - \nabla^{g^{\Sigma, u}}u}{\sqrt{1+|du|^2_{g^{\Sigma, u}}}},
\end{align*}
where $g^{\Sigma, u}$ is as in \ref{dexp}(d). It follows from this and the definitions that
\begin{equation*}
\begin{gathered}
\| (X^{N, g}_{\Omega, u})^*\Vtilde : C^k(\Omega, g) \| \leq C(k) \| v: C^k(\Omega, g)\|  \| u : C^{k+1}(\Omega)\|, \\
\| (X^{N, g}_{\Omega, u})^* \langle \Vtilde, \partial_z\rangle - v : C^k(\Omega, g)\|
\leq C(k) \| v : C^k(\Omega, u)\| \| u : C^{k+1}(\Omega, g)\|^2.
\end{gathered}
\end{equation*}
The conclusion follows from combining these estimates with \ref{LVestf}(ii). 
\end{proof}

\subsection*{Tilted graphs}
\nopagebreak

In this part, we study tilting rotations $\RRR_\kappa$ defined in \ref{dRk}.    Given vector spaces $\Etwo, \Ethree$ as in \ref{dRk}, choose orientations for $\Etwo$ and $\Ethree$ and further identify $\Ethree$ with $\R^3$ by choosing an orthonormal frame.

\begin{lemma}
\label{Lrot} $\RRR_\kappa$ depends smoothly on $\kappa$.  Moreover, the following hold. 
\begin{enumerate}[label=\emph{(\roman*)}]
\item For $\kappa \neq 0$, $\RRR_{\kappa}$ is the right-handed rotation of angle $\theta_{\kappa}$ about $\kvec$, where $\theta_{\kappa}: = \arctan |\kappa|$, $|\kappa | := \sup_{|v| = 1} \kappa(v)$, and $\{ \kvec, \kvec^\perp\}$ is the positively oriented orthonormal frame for $\R^2$ defined by requesting that $\kappa = | \kappa | \langle \kvec^\perp, \cdot \rangle$. 
\item For any $\vec{w}\in \R^3$, $\RRR_\kappa (\vec{w}) = (\cos \theta_\kappa ) \vec{w} +( \sin \theta_\kappa ) \kvec\times \vec{w} + (1-\cos \theta_\kappa) \langle \vec{w}, \kvec\rangle \kvec.$
\end{enumerate}
\begin{proof}
By \ref{dRk} we have $\RRR_\kappa = \exp(\frac{\theta_\kappa}{|\kappa|}K_\kappa)$, where $\exp : \mathfrak{so}(3)\rightarrow SO(3)$ is the exponential map  and $K_\kappa \in \mathfrak{so}(3)$ is defined by requesting that $K_\kappa v = (\kappa(e_2), -\kappa(e_1), 0) \times v$ for $v\in \R^3$, where here $\times$ is the cross product.  Since $K_\kappa$ and $\frac{\theta_\kappa}{|\kappa|}$ depend smoothly on $\kappa$, the smoothness of $\RRR_\kappa$ follows.
By properties of the exponential map, $\RRR_\kappa$ is a right-handed rotation of angle $\theta_\kappa$ about vector $\frac{1}{|\kappa|}( \kappa(e_2), -\kappa(e_1), 0)$, which is $\kvec$ since clearly $\kvec^\perp = \frac{1}{|\kappa|}(\kappa(e_1), \kappa(e_2), 0)$. 
 (ii) is easy to check and is known as Rodrigues' formula. 
\end{proof}
\end{lemma}

We now specialize the results of \ref{LVestf} to the case where $(N, g)= (\R^3, \delta_{ij})$, $\Sigma = \R^2$, $\delta>0$ is fixed (recall \ref{Aufermi}) 
and $\Vtilde$ is induced by a tilting rotation $\RRR_\kappa$ in the following sense. 

\begin{lemma}
\label{Lvtildeest}
Given $\kappa$ as in \ref{dRk}, there is a vector field $\Vtilde$ along $\graph^{N, g}_{\Omega}(u)$ uniquely defined by 
\begin{equation}
\label{Erotpert}
\pert_{\Vtilde}X^{N, g}_{\Omega, u} = \RRR_\kappa \circ X^{N,g}_{\Omega, u}.
\end{equation}
Moreover, item (i) below holds, and if $|\kappa|<1/10$, then items (ii)-(iv) below hold, where  $R_\Omega$ is the smallest radius such that $\Omega \subset D_0(R_\Omega)$. 
\begin{enumerate}[label=\emph{(\roman*)}]
\item $\Vtilde \circ X^{N, g}_{\Omega, u} = (\cos \theta_\kappa-1) X^{N, g}_{\Omega, u} + ( \sin \theta_\kappa) \kvec \times X^{N, g}_{\Omega, u} + (1-\cos \theta_\kappa) \langle X^{N, g}_{\Omega, u}, \kvec\rangle \kvec. $
\item $\| (X^{N,g}_{\Omega, u})^* \Vtilde : C^k(\Omega, g)\| \leq C(k) R_\Omega |\kappa|$.
\item $\| (X^{N,g}_{\Omega, u})^* \Vtilde^\top : C^k(\Omega, g)\| \leq C(k) (R_\Omega+1)|\kappa| ( |\kappa| +  \| u : C^{k}(\Omega, g)\| )$.
\item $\| (X^{N,g}_{\Omega, u})^* \langle \Vtilde, \partial_z\rangle - \kappa : C^k(\Omega, g)\| \leq C(k)|\kappa|^2 \| u : C^k(\Omega, g)\|$.
\end{enumerate}
\end{lemma}
\begin{proof}
Because $\exp^{N, g}_p \vec{w} = p + \vec{w}$ for any $p \in N= \R^3$ and any $\vec{w} \in T_p N = T_p \R^3$, the condition \eqref{Erotpert} is equivalent to item (i) by   \ref{Lrot}(ii).  Items (ii) and (iii) then follow from the definitions by estimating (i). 

Next, using (i) we compute
\begin{align}
\label{Evtildekvec}
(X^{N,g}_{\Omega, u})^* \langle \Vtilde, \partial_z\rangle = (\cos \theta_\kappa-1)u+ (\sin \theta_\kappa) \langle \kvec \times X^{N,g}_{\Omega, g} , \partial_z\rangle. 
\end{align}
Note that $\langle \kvec \times X^{N,g}_{\Omega, g} , \partial_z\rangle = \langle X^{N,g}_{\Omega, u}, \kvec^\perp\rangle$.  Item (iv) follows by estimating \eqref{Evtildekvec} using this and the fact that $\kappa = |\kappa| \langle \kvec^\perp, \cdot \rangle$.
\end{proof}

\begin{corollary}
\label{Ltiltgap}
If $|\kappa|$ is small enough in terms of $R_\Omega$ and $\delta$, then \ref{Aufermi}(iv) holds.  Moreover, the function $\kappa_u : =w$ in  \ref{LVestf} satisfies 
$$ 
 \| u_{\kappa}- u - \kappa : C^k(\Omegatilde_V)\| \le C(k)(1+R_\Omega) \left( \| u: C^{k+1}(\Omega)\|+ |\kappa| \right)^3.
$$ 
\end{corollary}

\begin{proof}
The smallness assumption \ref{Aufermi}(iv) on $\Vtilde$ follows from \ref{Lvtildeest}(ii) by taking $|\kappa|$ small enough.  
Therefore, the assumptions of \ref{LVestf} apply, so in particular $\DDD_V$ is a diffeomorphism.  
Using that the exponential map in Euclidean space amounts to addition, 
we conclude from \eqref{ediag3} that (recall \ref{NT}\ref{dgraph}) 
$\DDD_V + (u_\kappa\circ \DDD_V ) \partial_z = \Vtilde\circ ( I_\Omega^{\R^3} + u_\kappa\partial_z)+ u \partial_z$,  
which implies 
\begin{equation*}
\begin{aligned} 
u_\kappa - u - \kappa =& \: (I) + (II) + (III), \quad \text{where} 
\\ 
(I):=& \: u \circ \DDD^{-1}_V - u, 
\\ 
(II):=& \: (X^{N, g}_{\Omega, u})^* \langle \Vtilde, \partial_z\rangle \circ \DDD^{-1}_V - \kappa \circ \DDD_V^{-1},
\\     
(III) :=& \: \kappa \circ \DDD_V^{-1} - \kappa. \qquad \qquad 
\end{aligned}
\end{equation*}

The estimate now follows by combining the preceding with \ref{LVmvt} and the estimates in \ref{Lvtildeest}. 
\end{proof}

\section{Mean curvature with respect to a perturbed metric}
\label{S:A1}
Let $(N^n, g)$ be a Riemannian manifold with Levi-Civita connection $\nabla$.
\begin{definition}
\label{dhgdiff}
We define a Christoffel-inspired operator  
\begin{equation*}
\begin{gathered} 
\Ccal : C^\infty(\emph{\text{Sym}}^2(TN))\rightarrow C^\infty(\emph{\text{Sym}}^2(TN) \otimes T^*N)  
\\ 
\text{by} \quad  
2(\Ccal T)(X,Y, Z) = (\nabla_XT)(Y,Z) +(\nabla_Y T)(X,Z) - (\nabla_Z T)(X,Y).
\end{gathered} 
\end{equation*}
\end{definition}

\begin{remark}
\label{rebin}
The operator $\Ccal$ above was defined in \cite[Section 6.b]{BergerEbin}, although there it was denoted by $\square$. 
\qed 
\end{remark}

Fix another Riemannian metric $\ghat$ on $N$ and define $h: = \ghat - g$.  We denote various quantities when defined with respect to $\ghat$ with a hat. By a calculation \cite[Lemma A.2]{BrendleRicci} using the Koszul formula,
\begin{equation}
\label{Ekoszuldiff}
\begin{aligned}
  \ghat (\widehat{\nabla}_X Y - \nabla_X Y, Z)= 
 (\Ccal h)(X, Y, Z)
 \qquad \text{for all} \quad X,Y,Z \in T_p N.
\end{aligned}
\end{equation}

\begin{lemma}[Mean curvature under a change of metric]  
\label{LApert}
Let $S\subset N$ be a two-sided hypersurface with unit normal field $\nu$.  
\begin{enumerate}[label=\emph{(\roman*)}]
\item $\nuhat =  (\nu- \beta^{\widehat{\sharp}})/| \nu - \beta^{\widehat{\sharp}}|_{\ghat}$ and 
$| \nu - \beta^{\widehat{\sharp}}|^2_{\ghat} = 1+ \sigma - | \beta|_g^2 -\alphahat(\beta^\sharp, \beta^\sharp)$,
\item $\displaystyle{
	| \nu - \beta^{\widehat{\sharp}}|_{\ghat}\widehat{A}^S =  
	 A^S+ \text{\emph{Sym}}\left( A^S *_g \alpha + \nabla^S \beta\right) - \frac{1}{2}\alphatilde- (\Ccal \alpha) \intprod \beta^{\widehat{\sharp}}}$,
 \item   $	| \nu - \beta^{\widehat{\sharp}}|_{\ghat} \widehat{H}^S 
= H^S+ \ddiv_{S,g} \beta - \frac{1}{2} \tr_{S, g} \alphatilde+ 
\big\langle \text{\emph{Sym}}\left(\nabla^S \beta\right) - \frac{1}{2}\alphatilde, \alphahat \big\rangle_g$ 
\\
$\phantom1$ 
\hfill 
$-\tr_{S, g}( (\Ccal^S\alpha) \intprod \beta^{\widehat{\sharp}}) -\langle  (\Ccal^S\alpha) \intprod \beta^{\widehat{\sharp}},  \alphahat\rangle_g$,
\end{enumerate}
where the symmetric two-tensor fields $\alpha$, $\alphatilde$, and $\alphahat$, 
differential one-form $\beta$, vector fields  $\beta^\sharp$ and $\beta^{\widehat{\sharp}}$, and function $\sigma$, 
are defined by requesting that for $p\in S$, $X, Y\in T_pS$,
\begin{align*} 
\begin{gathered}
\alpha(X,Y) = h(X,Y), 
\qquad 
\beta(X) = h(X, \nu_p),
\qquad 
\sigma(p) = h(\nu_p, \nu_p),
\\ 
\alphatilde(X,Y) = (\nabla_{\nu_p} h)(X, Y), 
\qquad
g(\beta^\sharp, X) = \beta(X), 
\qquad 
 \ghat( \beta^{\widehat{\sharp}},  X) = \beta(X)
\\ 
\text{and} \qquad
\alphahat(X, Y) = \ghat(X^\flat, Y^\flat) - g(X,Y),
\end{gathered}
\end{align*}
where here $X^\flat$ and $Y^\flat$ are computed with respect to $g$. 
Moreover we have, 
$\betashat = \beta^\sharp +( \alphahat \intprod \beta^\sharp)^\sharp$, 
and (in any local coordinates) 
$\alphahat_{ij}= \ghat^{kl}g_{ik}g_{jl} - g_{ij} =   \sum_{p=1}^\infty (-1)^p \alpha_{i k_1} g^{k_1 l_1}
 \alpha_{l_1 k_2} g^{k_2 l_2} \cdots  \alpha_{l_{k-2}k_{p-1}} g^{k_{p-1} l_{p-1}}\alpha_{l_{p-1}j}$.
\end{lemma}

\begin{proof}
Given $X\in T_pS$, note that $\ghat( \nu  - \beta^{\widehat{\sharp}}, X) = 0$. Therefore $\hat{\Pi}\nu = \beta^{\widehat{\sharp}}$,  where $\hat{\Pi}$ is the $\ghat$-orthogonal projection onto $T_p S$, and (i)  follows, where the formula for $|\nu - \beta^{\widehat{\sharp}}|^2_{\ghat}$ is a direct calculation.
Next we compute (where in this proof we write $A$ in place of $A^S$ since no confusion will arise)
\begin{equation*}
\begin{aligned}
|\nu - \betashat|_{\ghat}  \widehat{A}(X,X) &= \ghat( \nabla_X X+ \widehat{\nabla}_X X- \nabla_X X , \nu - \betashat)\\
&=\ghat(\nu, \nu - \betashat)A(X,X) +(\Ccal h)(X, X, \nu) - (\Ccal h)(X, X, \betashat)\\
&=  |\nu - \beta^{\widehat{\sharp}}|^2_{\ghat}A(X,X)+(\Ccal h)(X, X, \nu) - (\Ccal h)(X, X, \betashat),
\end{aligned}
\end{equation*}
where the second and third equalities use \eqref{Ekoszuldiff} and that $\ghat(\betashat, \nu-\betashat) = 0$.  Using \ref{dhgdiff}, we calculate
\begin{multline*} 
(\Ccal h)(X, X, \nu) = (\nabla_X h)(X, \nu )-\frac{1}{2} (\nabla_\nu h )(X,X)\\
  = X( h(X, \nu)) - h( \nabla_X X, \nu )- h(X, \nabla_X \nu)  - \frac{1}{2}\alphatilde(X,X)\\
= X(\beta(X)) - \beta\left( \nabla_X^S X\right) - \sigma A(X,X)+ \alpha( X, B(X)) -\frac{1}{2} \alphatilde(X,X)\\
= (\nabla^S_X\beta)(X) - \sigma A(X,X) +\left( A*_g\alpha\right)(X,X)-\frac{1}{2} \alphatilde(X,X).
\end{multline*}
Using \eqref{Ekoszuldiff} and that $\widehat{\nabla} - \nabla = \widehat{\nabla}^S - \nabla^S + \widehat{A} \nuhat - A \nu$, we find
 \begin{align*}
(\Ccal h)(X, X, \betashat)  &= (\Ccal^S \alpha)(X, X, \betashat) - \ghat(\nu, \betashat)A(X,X) \\
 &=  (\Ccal^S \alpha)(X,X, \betashat)- \beta(\betashat)A(X,X).
\end{align*}
Substituting these items above and simplifying using (i) establishes (ii). 
Taking the trace of (ii) with respect to $\ghat_S$ and simplifying (note in particular that $\tr_{S, \ghat} (A*_g\ghat) = \tr_{S, g} A = H$) establishes (iii).

Finally, let $[g]$ and $[\alpha]$ denote matrix representations of $g_S$ and $\alpha$ and note that
\begin{align*}
 ([g]+[\alpha])^{-1} =  ( \text{Id}+[g]^{-1} [\alpha])^{-1}[g]^{-1}
= \sum_{k=0}^\infty (-1)^k \left( [g]^{-1} [\alpha]\right)^k [g]^{-1},
\end{align*}
which implies the coordinate expression for $\alphahat$ and the identity $\betashat = \beta^\sharp + (\alphahat \intprod \beta^\sharp)^\sharp$.
\end{proof}

\begin{remark}
\label{rmars}
The proof of Lemma \ref{LApert} above is self-contained and done independently from \cite{mars}, but we note that (i) and (ii) are consistent with  results therein. 
\qed
\end{remark}

\begin{remark}
\label{rconfh}
When $\ghat = e^{2w} g$ for some $w \in C^\infty(N)$, it follows that 
\begin{align*}
  \beta = \beta^{\widehat{\sharp}} = 0,\quad 
   |\nu - \beta^{\widehat{\sharp}}|_{\ghat} = e^{w}, \quad
A^S *_g \ghat = e^{2w}A^S, \quad
 \alphatilde = 2 e^{2w}\nu(w) g,
 \end{align*}
 and  \ref{LApert}(i) reduces to the  usual transformation rule $\widehat{A} = e^{w} \left( A- ( \partial_\nu w) g\right)$ for the second fundamental form under a conformal change of metric. 
\qed
 \end{remark}

\begin{remark}
\label{rHpert}
When $\ghat$ is the ambient metric in a local system of Fermi coordinates about a hypersurface $\Sigma$ as in \ref{dexp}, we define $g = \left.g\right|_\Sigma+ d\zz^2$ and $S = \Sigma_\zz$, a parallel hypersurface. We have by  \ref{Lgauss}
 \begin{equation*}
 \begin{gathered}
 \ghat = \left. \ghat\right|_{\Sigma_{\zz}} + d\zz^2, \quad
\nuhat =  \nu = \partial_\zz,\quad   \sigma = 0, 
\quad\beta = 0,\\
\alpha =  - 2 \zz A^\Sigma  +  \zz^2 \left( A^\Sigma * A^\Sigma + \Rm_\nu\right) +O(\zz^3),\\
 \alphatilde = -2 A^\Sigma+2\zz(A^\Sigma*A^\Sigma+ \Rm_\nu)  + O(\zz^2),
 \end{gathered}
 \end{equation*}
so that \ref{LApert}(iii) implies the usual formula for the mean curvature of $\widehat{H}^{\Sigma_\zz}$  (note that $H^{\Sigma_\zz} = H^{\Sigma}$):
\begin{multline*}
  \widehat{H}^{\Sigma_\zz} =H^{\Sigma} - \frac{1}{2} \tr_\Sigma \alphatilde + \frac{1}{2} \langle \alpha , \alphatilde\rangle + O\left(\zz^2\right)
\\ 
= H^\Sigma+ \left(|A^\Sigma|^2+\Ric(Z, Z)\right)\zz + O\left(\zz^2\right).
 \qed
\end{multline*}
\end{remark}

\begin{corollary}
\label{Lhasmp}  
$\widehat{H}^S - H^S  - \widetilde{\sigma} = (H^S + \widetilde{\sigma}) \, ((1+\widehat{\sigma})^{-1/2}-1) $ 
\\ $\phantom{kk}$ \hfill 
$ 
+(1+\widehat{\sigma})^{-1/2} \left( \langle \mathrm{Sym}\left(\nabla^S \beta\right), \alphahat\rangle_g  -
\tr_{S,g}( (\Ccal^S \alpha) \intprod \betashat) - \langle (\Ccal^S \alpha) \intprod \betashat,  \alphahat\rangle_g\right),
$ 
\\ 
where here $\widetilde{\sigma}: = \ddiv_{S, g} \beta - \frac{1}{2} \tr_{S, g} \alphatilde - \frac{1}{2} \langle \alphahat, \alphatilde\rangle_{g}$ and $\widehat{\sigma}: = \sigma -\beta(\beta^{\widehat{\sharp}}) $.
\end{corollary}
\begin{proof}
This follows immediately from dividing through \ref{LApert}(iii) by $| \nu - \beta^{\widehat{\sharp}}|_{\ghat} = (1+\widehat{\sigma})^{1/2}$ (recall \ref{LApert}(i)) and subtracting $H^S + \widetilde{\sigma}$ from both sides. 
\end{proof}

\begin{corollary}
\label{LHpertsmall}
Suppose $\alpha, \beta$, and $\sigma$ all have small enough $C^k(S, g)$ norm in terms of $k$.  
Then 
\begin{multline*} 
\| \widehat{H}^S: C^k(S, g) \| 
\lem C(k) \left( \|H^S : C^k(S, g)\| +   \|  \beta: C^{k+1}(S, g)\| \right. 
\\
\left. \, + \| \tr_{S, g} \alphatilde: C^k(S, g)\| + \big( 1+ \| \alphatilde : C^k(S, g)\| \big) \,         \| \alpha : C^{k+1}(S, g)\|          \right) .
\end{multline*} 
\end{corollary}

\begin{proof}
From the definition of $\beta^{\widehat{\sharp}}$ and the coordinate expression for $\alphahat$, it follows that $\| \beta^{\widehat{\sharp}} : C^k(S, g)\| \le C(k) \| \beta: C^k(S,g)\|(1 +  \| \alpha: C^k(S,g) \|)$.  Using the notation in the proof of \ref{Lhasmp}, we have then
\begin{align*}
\begin{gathered}
\| w: C^{k}(S, g)\| \le C(k) \left( \| \sigma:C^k(S, g) \|+  \| \beta:  C^k(S, g)\|^2 \right),\\
\| \langle \text{{Sym}}\left(\nabla^S \beta\right), \alphahat \rangle_g: C^k(S, g) \| \le C(k) \| \alpha : C^k(S, g) \| \| \beta: C^{k+1}(S, g)\| , \\
\| \langle \alphatilde, \alphahat \rangle_g : C^k(S, g) \| \le C(k) \| \alpha : C^k(S, g) \| \|  \alphatilde: C^k(S, g)\|.
\end{gathered}
\end{align*}
Using the preceding, we also estimate
\begin{multline*}
\| \tr_{S,g}( (\Ccal^S \alpha) \intprod \betashat) : C^k(S, g) \|+ \| \langle ( \Ccal^S \alpha) \intprod \betashat, \alphahat \rangle_g: C^k(S, g) \|  
\\ 
\lem C(k) \|  \alpha : C^{k+1}(S, g) \| \| \beta: C^k(S, g) \| .
\end{multline*}
Combining the estimates with the expansion in \ref{Lhasmp} completes the proof.
\end{proof}

\begin{lemma}  
\label{Llaplace}
 Let $u\in C^2(S)$ and $X$ be a vector field on $S$.
	\begin{enumerate}[label=\emph{(\roman*)}]
	\item $\nablahat u = \nabla u + (\alphahat \intprod \nabla u)^\sharp$.
	\item $\widehat{\ddiv} X = \ddiv X + \tr_{S, g}( (\Ccal^S \alpha) \intprod X) + \langle \alphahat, (\Ccal^S \alpha) \intprod X\rangle_g$.
	\item $\widehat{\Delta} u = \Delta u + (\ddiv_{S, g} \alphahat)(\nabla u) +  (\tr_{S, g} \alphahat ) \Delta u+ 
\tr_{S, g}( ( \Ccal^S \alpha) \intprod ( \nabla u + ( \alphahat\intprod \nabla u)^\sharp))$ \\
$+ \langle \alphahat, ( \Ccal^S \alpha) \intprod ( \nabla u + ( \alphahat\intprod \nabla u)^\sharp)\rangle_g$.
\item As long as $\alpha$ has small enough $C^{k+1}(S, g)$ norm in terms of $k$, then
$\| \widehat{\Delta} u - \Delta u : C^k(S, g) \| \le C(k) \| \alpha : C^{k+1}(S, g)\| \| u : C^{k+2}(S, g)\| $.
\end{enumerate}
\end{lemma}
\begin{proof}
(i) and (ii) follow from the following calculations in coordinates: 
\begin{align*}
(\widehat{\nabla}u)^i &= \ghat^{ij} u_j 
= g^{ij} u_j +g^{ik}g^{jl}  \alphahat_{kl} u_j , \\
\widehat{\ddiv} X 
&= \ghat^{ij} \ghat( \nabla_{\partial_i} X + \nablahat_{\partial_i} X - \nabla_{\partial_i} X , \partial_j)\\
&= \ghat^{ij} ( \ghat( \nabla_{\partial_i} X, \partial_j) + (\Ccal^S \alpha)(X, \partial_i, \partial_j))\\
&= \ddiv X+ (g^{ij} + \alphahat^{ij}) (\Ccal^S\alpha)(X, \partial_i, \partial_j).
\end{align*}
(iii) follows by combining (i) and (ii) and observing that
\[ 
\ddiv( \alphahat \intprod \nabla u) = (\ddiv \alphahat) (\nabla u ) + (\tr_{S, g} \alphahat) \Delta u.
\] 
Finally, (iv) follows immediately from estimating (iii). 
\end{proof}

\begin{lemma}
\label{Ljacdif}
Suppose $\alpha$, $\beta$, and $\sigma$ all have small enough $C^{k+1}(S, g)$ norm in terms of $k$.  Then
\begin{multline*}
\big\| \, |\widehat{A}^S|^2_{\ghat} - |A^S|^2_{g} \, \big\|_{C^k} 
\lem C(k)  \Big( \, \big( \| \sigma \|_{C^k} + \| \beta\|_{C^{k+1}} + \| \alpha \|_{C^k} \big) \,  \| A^S\|_{C^k}^2\\ 
+ \, \big( \| \beta \|_{C^{k+1}} + \| \alphatilde\|_{C^k} \big) \| A^S \|_{C^k} \, + \,   \big( \| \beta\|_{C^{k+1}} + \| \alphatilde\|_{C^k} \big)^2 \, \Big),
\end{multline*}
where here $\| \cdot \|_{C^k}$ is short hand for the $C^k(S, g)$ norm. 
\end{lemma}

\begin{proof}
We first compute using \ref{LApert}(ii) that 
\begin{align*}
|\widehat{A}^S|^2_g - |A^S|^2_g = - ( \sigma - |\beta|^2_g - \alphahat(\beta^\sharp, \beta^\sharp)) |\widehat{A}^S|^2_g + 2\langle A, T\rangle_g + |T|^2_g, 
\end{align*}
where here $T :=  \text{Sym}\left( A^S *_g \alpha + \nabla^S \beta\right) - \frac{1}{2}\alphatilde- (\Ccal^S \alpha) \intprod \beta^{\widehat{\sharp}}$.  
Using this and the assumptions, we estimate
\begin{multline*}
\big\| \, |\widehat{A}^S|^2_g - |A^S|^2_g \big\|_{C^k} \lem C(k)  \: \Big( \: \big(  \| \sigma \|_{C^k} + \| \beta\|_{C^{k+1}} + \| \alpha \|_{C^k} \big) \, \| A^S\|_{C^k}^2 
\\ 
+ \, \big( \| \beta \|_{C^{k+1}} + \| \alphatilde\|_{C^k} \big) \, \| A^S \|_{C^k} \, + \, \big( \| \beta\|_{C^{k+1}} + \| \alphatilde\|_{C^k} \big)^2 \, \Big).
\end{multline*}
Next, we compute 
\begin{align*}
|\widehat{A}^S|^2_{\ghat} - |\widehat{A}^S|^2_{g}  = 2\langle \widehat{A}^S *_g \alphahat, \widehat{A}^S \rangle_g + | \widehat{A}^S *_g \alphahat |^2_g,
\end{align*}
and using this and \ref{LApert} to estimate $\| \widehat{A}^S\|_{C^k}$ we conclude 
\begin{align*}
\| \, |\widehat{A}^S|^2_{\ghat} - |\widehat{A}^S|^2_{g}\|_{C^k} &\le C(k) \| \alpha \|_{C^k} \| \widehat{A}^S\|_{C^k}^2\\
&\le C(k) \| \alpha \|_{C^k} ( \| A^S \|_{C^k} + \| \alphatilde\|_{C^k} + \| \beta\|_{C^{k+1}} )^2.
\end{align*} 
 By the triangle inequality, combining these estimates finishes the proof.
\end{proof}

\begin{lemma}
\label{Lgcomp}
Let $u$ be a $C^k$ tensor field on $N$ and let $\epsilon> 0$.  If $\| h  : C^k(N, g)\|$ is small enough in terms of $k$ and $\epsilon$, then
\begin{align}
\label{Euggh}
\| u : C^k(N, \ghat) \| \Sim_{1+\epsilon} \| u :C^k(N, g)\|  .
\end{align}
\end{lemma}
\begin{proof}
We first consider the case where $k=0$.  Since $\ghat(u, u) - g(u, u)$ is a sum of $g$-inner products of contractions of $u$ and $h$, we have that
\begin{align*}
\left| \| u : C^0(N, \ghat)\|^2 - \| u :C^0(N, g)\|^2 \right| \le C \| h : C^0(N, g)\| \| u : C^0(N, g)\|^2, 
\end{align*}
which implies \eqref{Euggh} when $k=0$, where we have used that $\| h : C^0(N, g)\|$ is small. 

Using that the Christoffel symbols of $\widehat{\nabla}$ and $\nabla$ satisfy
$\widehat{\Gamma}^{k}_{ij} - \Gamma^k_{ij} = \Gamma^{l}_{ij} h_{lk} + (\Ccal h)_{ijk},$
substituting into the formula for the components of $\widehat{\nabla} u$, estimating, and using \eqref{Euggh} when $k=0$, we find
\begin{align}
\label{Eugh2}
\| \widehat{\nabla} u : C^0(N, \ghat)\| \le (1+C \| h : C^1(N, g)\| ) \| u : C^1(N, g)\| .
\end{align}
Interchanging the roles of $\ghat$ and $g$ in \eqref{Eugh2} and using \eqref{Eugh2} also to estimate $\| h: C^1(N, \ghat)\|$, we have
\begin{align*}
\| \nabla u : C^0(N, g)\| &\le (1+C \| h : C^1(N, \ghat)\| ) \| u : C^1(N, \ghat)\| \\
&\le ( 1+ C\| h: C^1(N, g)\| )\| u: C^1(N, \ghat)\| . 
\end{align*}
With the preceding, this proves \eqref{Euggh} when $k=1$, and the result for general $k$ follows inductively. 
\end{proof}

\section{Weighted decay estimates}
\label{S:A4}
We prove a weighted estimate on surfaces for solutions of inhomogeneous linear equations which is analogous to estimates in other gluing constructions, 
e.g. in \cite{kap, kapouleas:wente, breiner:kapouleas:high}.  
The proof relies on analogous estimates in the Euclidean setting established in \cite[Proposition C.1(i)]{breiner:kapouleas:high}.  
\begin{lemma}
\label{Lwe}
Given a closed Riemannian Surface $ (\Sigma^2, g)$, $V\in C^\infty(\Sigma)$, $\beta\in (0, 1)$, and $\gamma \in (1, 2)$, there exists $\epsilonunder>0$ such that for any $\epsilon \in (0,\epsilonunder]$ and any $p \in \Sigma$, there is a linear map 
$\Rcal^\Sigma_p : C^{0, \beta}( \overline{D^\Sigma_p(\epsilon)}) \rightarrow C^{2, \beta}(\overline{D^\Sigma_p(\epsilon)})$ so that if $E\in C^{0, \beta}(\overline{D^\Sigma_p(\epsilon)})$ and $u = \Rcal^\Sigma_p (E)$, then 
\begin{enumerate}[label=\emph{(\roman*)}]  
\item $(\Delta_g +V) u = E$.
\item $u(p) = d_p u = 0$. 
\item $\| u: C^{2, \beta}(\overline{D^\Sigma_p(\epsilon)}, \dbold^\Sigma_p, g, (\dbold^\Sigma_p)^\gamma \| \le C \| E : C^{0, \beta}(\overline{D^\Sigma_p(\epsilon)}, \dbold^\Sigma_p, g, (\dbold^\Sigma_p)^{\gamma-2})\| $. 
\end{enumerate}
\end{lemma}
\begin{proof}
In \cite[Proposition C.1(i)]{breiner:kapouleas:high}, it was shown that the conclusion of Lemma \ref{Lwe} holds when $(\Sigma^2, g) = (\R^2, \delta_{ij})$ and $V = 0$. 
By identifying $\overline{D^{\Sigma, g}_p(\epsilon)}$ with $\overline{D^{T_p\Sigma,  g_p}_0(\epsilon)}$ using the exponential map, considering the Euclidean Laplacian $\Delta_{g_p}$ and its corresponding right inverse $\Rcal: C^{0, \beta}(\overline{ D^\Sigma_p(\epsilon)}) \rightarrow C^{2, \beta}(\overline{D^\Sigma_p(\epsilon)})$ from \cite[Proposition C.1(i)]{breiner:kapouleas:high},  taking $\epsilonunder$ small enough, and rescaling $\overline{D^\Sigma_p(\epsilon)}$ to be of unit size, we may assume that $\mathrm{Id} - \Lcal \Rcal = (\Delta_{g_p} - \Lcal) \Rcal$ has operator norm less than $1$.  We then define $\Rcal^\Sigma_p = \Rcal \sum_{n=0}^\infty ( \mathrm{Id} - \Lcal \Rcal)^n = \Rcal \sum_{n=0}^\infty ( (\Delta - \Lcal) \Rcal)^n$.
Item (i) follows by inspection, and (ii) and (iii) follow from the corresponding items established in \cite[Proposition C.1(i)]{breiner:kapouleas:high}.
\end{proof}

\begin{corollary}
\label{Cdecay}
Let $\Sigma, V, \beta, \epsilonunder$, and $p$ be as in \ref{Lwe}.  There exists $C>0$ such that for any $\epsilon \in (0, \epsilonunder]$ and any $v \in C^{2, \beta}(\overline{D^\Sigma_p(\epsilon)})$ satisfying $(\Delta_g + V) v= 0$ on $\overline{D^\Sigma_p(\epsilon)}$ and $\Ecalunder_p v = 0$, the following estimate holds. 
\begin{align*}
\| v : C^{2, \beta}( \overline{D^\Sigma_p(\epsilon)} , \dbold^\Sigma_p, g, (\dbold^\Sigma_p)^\gamma)\| 
\leq 
C \| v : C^{2, \beta}( \overline{\partial D^\Sigma_p(\epsilon)}, \dbold^\Sigma_p, g, (\dbold^\Sigma_p)^\gamma)\|.
\end{align*}
\end{corollary}
\begin{proof}
Suppose $\epsilon$ and $v$ are as above.  Define $\hat{v} \in C^{2, \beta}(\overline{D^\Sigma_p(\epsilon)}\setminus \overline{D^\Sigma_p(\epsilon/2)})$ to be the radial extension of $v|_{\partial \overline{D^{\Sigma}_p(\epsilon)}}$ and define $\varphi \in C^{2, \beta}(\overline{D^\Sigma_p(2\epsilon)})$ by $\varphi = \Psibold[ \epsilon, \epsilon/2 ; \dbold^\Sigma_{p}]( \hat{v}, 0)$.  By the definitions, the following estimate holds. 
\begin{align}
\label{Ephiv}
\| \varphi : C^{2, \beta}( \overline{D^\Sigma_p(\epsilon)}, \dbold^\Sigma_p,  g ) \| 
\leq 
C 
\| v : C^{2, \beta}( \partial \overline{D^\Sigma_p(\epsilon)}, \dbold^\Sigma_p,  g)\|.
\end{align}
By applying Lemma \ref{Lwe} with $E = (\Delta_g + V) \varphi$, there exists $u \in C^{2, \beta}(\overline{D^\Sigma_p(\epsilon)})$ satisfying $(\Delta_g + V)(u- \varphi) = 0$,  $\Ecalunder_p u =0$, and 
\begin{align*}
\| u: C^{2, \beta}(\overline{D^\Sigma_p(\epsilon)}, \dbold^\Sigma_p, g, (\dbold^\Sigma_p)^\gamma \|
& \leq C
 \| \varphi : C^{2, \beta}( \overline{D^\Sigma_p(\epsilon)}, \dbold^\Sigma_p, g, (\dbold^\Sigma_p)^{\gamma})\|\\
 &\leq C \| v : C^{2, \beta}( \partial \overline{D^\Sigma_p(\epsilon)}, \dbold^\Sigma_p, g, (\dbold^\Sigma_p)^\gamma)\|,
\end{align*}
where the last inequality follows from \eqref{Ephiv} and the definitions.

On the other hand, by the definition of $\varphi$ and the proof of Lemma \ref{Lwe} we have that the restriction of $v-\varphi+u$ to $\partial \overline{D^\Sigma_p(\epsilon)}$ 
is a linear combination of constants and first harmonics.  
Since $\Ecalunder_p ( v- \varphi + u) = 0$, it follows from the smallness of $\epsilon$ that $v = \varphi - u$.  
The claimed estimate now follows from this and the preceding. 
\end{proof}

\bibliographystyle{abbrv}
\bibliography{bibliography}

\end{document}